\documentclass[a4paper,leqno]{amsart}

\usepackage{latexsym}
\usepackage[english]{babel}
\usepackage{fancyhdr}
\usepackage[mathscr]{eucal}
\usepackage{amsmath}
\usepackage{mathrsfs}
\usepackage{amsthm}
\usepackage{amsfonts}
\usepackage{amssymb}
\usepackage{amscd}
\usepackage{bbm}
\usepackage{graphicx}
\usepackage{graphics}
\usepackage{latexsym}
\usepackage{color}
\usepackage{calc}

\newcommand{\ud}{\mathrm{d}}

\newcommand{\ii}{\mathrm{i}}
\newcommand{\cH}{\mathcal{H}}

\newcommand{\op}{
  \mathop{
    \vphantom{\bigoplus} 
    \mathchoice
      {\vcenter{\hbox{\resizebox{\widthof{$\displaystyle\bigoplus$}}{!}{$\boxplus$}}}}
      {\vcenter{\hbox{\resizebox{\widthof{$\bigoplus$}}{!}{$\boxplus$}}}}
      {\vcenter{\hbox{\resizebox{\widthof{$\scriptstyle\oplus$}}{!}{$\boxplus$}}}}
      {\vcenter{\hbox{\resizebox{\widthof{$\scriptscriptstyle\oplus$}}{!}{$\boxplus$}}}}
  }\displaylimits 
}

\theoremstyle{plain}
\newtheorem{theorem}{Theorem}[section]
\newtheorem{lemma}[theorem]{Lemma}
\newtheorem{corollary}[theorem]{Corollary}
\newtheorem{proposition}[theorem]{Proposition}

\theoremstyle{definition}

\newtheorem{remark}[theorem]{Remark}
\newtheorem*{remark*}{Remark}

\numberwithin{equation}{section}

\begin{document}

\title[Quantum geometric confinement/transmission in Grushin cylinder]
{Quantum geometric confinement and dynamical transmission in Grushin cylinder}
\author[M.~Gallone]{Matteo Gallone}
\address[M.~Gallone]{Mathematics Department ``F.~Enriques'', University of Milan \\ via C.~Saldini 50 \\ Milano 20133 (Italy).}
\email{matteo.gallone@unimi.it}
\author[A.~Michelangeli]{Alessandro Michelangeli}
\address[A.~Michelangeli]{Institute of Applied Mathematics and Hausdorff Center for Mathematics, University of Bonn \\ Endenicher Allee 60 \\ 
Bonn 53115  (Germany).}
\email{michelangeli@iam.uni-bonn.de}
\author[E.~Pozzoli]{Eugenio Pozzoli}
\address[E.~Pozzoli]{Inria, Sorbonne Universit\'e, Universit\'e de Paris, CNRS, Laboratoire Jacques-Louis Lions, Paris (France). \\ and Institut de Math\'ematiques de Bourgogne, UMR 5584, CNRS, Universit\'e Bourgogne Franche-Comt\'e, F-21000 Dijon (France)}
\email{eugenio.pozzoli@inria.fr, eugenio.pozzoli@u-bourgogne.fr}


\begin{abstract}
We classify the self-adjoint realisations of the Laplace-Beltrami operator minimally defined on an infinite cylinder equipped with an incomplete Riemannian metric of Grushin type, in the class of metrics yielding an infinite deficiency index. Such realisations are naturally interpreted as Hamiltonians governing the geometric confinement of a Schr\"{o}dinger quantum particle away from the singularity, or the dynamical transmission across the singularity. In particular, we characterise all physically meaningful extensions qualified by explicit local boundary conditions at the singularity. Within our general classification we retrieve those distinguished extensions previously identified in the recent literature, namely the most confining and the most transmitting one.
\end{abstract}

\date{\today}




\maketitle

\vspace{-1cm} 

\tableofcontents

\section{Introduction, setting, main results}
\label{sec:intro}

\subsection{Grushin structures and geometric quantum confinement}~


The study of a quantum particle on degenerate Riemannian manifolds, and the problem of the purely geometric confinement away from the singularity locus of the metric, as opposite to the dynamical transmission across the singularity, has recently attracted a considerable amount of attention in relation to Grushin structures and to the induced confining effective potentials on cylinder, cone, and plane (as in the works \cite{Nenciu-Nenciu-2009,Boscain-Laurent-2013,Boscain-Prandi-Seri-2014-CPDE2016,Prandi-Rizzi-Seri-2016,Boscain-Prandi-JDE-2016,Franceschi-Prandi-Rizzi-2017,GMP-Grushin-2018,Boscain-Neel-2019,PozzoliGru-2020volume,Boscain-Beschastnnyi-Pozzoli-2020,GM-Grushin3-2020,GMP-heat-2021}), as well as, more generally, on two-step two-dimensional almost-Riemannian structures \cite{Boscain-Laurent-2013,Boscain-Beschastnnyi-Pozzoli-2020,IB-2021}, or also generalisations to almost-Riemannian structures in any dimension and of any step, and even to sub-Riemannian geometries, provided that certain geometrical assumptions on the singular set are taken \cite{Franceschi-Prandi-Rizzi-2017,Prandi-Rizzi-Seri-2016}. On a related note, a satisfactory interpretation of the heat-confinement in the Grushin cylinder is known in terms of Brownian motions \cite{Boscain-Neel-2019} and random walks \cite{Agra-Bosca-Neel-Rizzi-2018}.

Underlying such analyses there is a natural problem of \emph{control of essential self-adjointness or lack thereof}, whence also a natural problem of identification, classification, and analysis of self-adjoint extensions, for the minimally defined Laplace-Beltrami operator on manifold.

In this work we focus on the paradigmatic class of \emph{quantum models on Grushin cylinder}: the latter is a two-dimensional manifold built upon $\mathbb{R}\times\mathbb{S}^1$ with an incomplete Riemannian metric both on the right and the left open half-cylinder $\mathbb{R}^{\pm}\times\mathbb{S}^1$, and a singularity of the metric along the separation circle among the two halves. 

For such models, the \emph{geometric quantum confinement} in each half-cylinder corresponds to the essential self-adjointness of the Laplace-Beltrami operator on its minimal domain of smooth functions supported away from the singularity. The \emph{quantum transmission} between the two half-cylinders corresponds instead to the lack of essential self-adjointness, in which case the type of transmission is governed by a self-adjoint extension of the Laplace-Beltrami operator.

In the literature the regimes of confinement and transmission have been recently identified (see Theorem \ref{thm:Halpha_esa_or_not} below), but with no classification of the possible different protocols of transmissions, namely of the self-adjoint extensions of the Laplace-Beltrami operator. In this work, announced in 2020, we complete such programme, and study the family of inequivalent self-adjoint realisations of the differential operator by means of the general extension theory of Kre\u{\i}n, Vi\v{s}ik, and Birman \cite{GMO-KVB2017}.

Let us start with fixing the notation and setting up the general problem. Let us denote by $(x,y)$ a generic point in $\mathbb{R}_x\times\mathbb{S}^1_y$ and let us define
\begin{equation}
 M^{\pm}\;:=\;\mathbb{R}^{\pm}\times\mathbb{S}^1\,,\qquad\mathcal{Z}\;:=\;\{0\}\times\mathbb{S}^1\,,\qquad M\;:=\;M^+\cup M^-\,.
\end{equation}
We consider the family $\{M_\alpha\equiv(M,g_\alpha)\,|\,\alpha\in\mathbb{R}\}$ of Riemannian manifolds with metric
\begin{equation}\label{eq:galphaeverywhere}
  g_\alpha\;:=\;\ud x\otimes\ud x+\frac{1}{|x|^{2\alpha}}\,\ud y\otimes\ud y\,,
\end{equation}
that is, with global orthonormal frame
\begin{equation}\label{eq:frame}
\{X_1,X_2^{(\alpha)}\}\;=\;\left\{ 
\begin{pmatrix}
 1 \\ 0
\end{pmatrix},\;
\begin{pmatrix}
 0 \\ |x|^{\alpha}
\end{pmatrix}
\right\}\;\equiv\;\Big\{\frac{\partial}{\partial x},|x|^{\alpha}\frac{\partial}{\partial y}\Big\}.
\end{equation}
One refers to $M_1$ as the standard two-dimensional \emph{Grushin cylinder} \cite[Chapter 11]{Calin-Chang-SubRiemannianGeometry}, and more generally to $M_\alpha$ with $\alpha\neq 1$ as a \emph{Grushin(-type) cylinder}. In particular, $M_0$ is a juxtaposition of two Euclidean half-cylinders.

It is easily seen  \cite{Agrachev-Boscain-Sigalotti-2008,Boscain-Laurent-2013,Pozzoli_MSc2018} that $M_\alpha$ is a hyperbolic manifold whenever $\alpha>0$, with Gaussian (sectional) curvature
\begin{equation}
 K_\alpha(x,y)\;=\;-\frac{\,\alpha(\alpha+1)\,}{x^2}\,.
\end{equation}
One has the Lie bracket\index{Lie bracket}
\begin{equation}
 [X_1,X_2^{(\alpha)}]\;=\;\begin{pmatrix}
 0 \\ \alpha|x|^{\alpha-1}
\end{pmatrix},
\end{equation}
 and moreover, when $\alpha\in\mathbb{N}$ the fields $X_1,X_2^{(\alpha)}$ are smooth and define an \emph{almost-Riemannian structure}\index{almost-Riemannian structure} on $\mathbb{R}\times\mathbb{S}^1=M^+\cup\mathcal{Z}\cup M^-$, for a rigorous definition of which one may refer to \cite[Section 1]{Agrachev-Boscain-Sigalotti-2008} or \cite[Section 7.1]{Prandi-Rizzi-Seri-2016}: indeed the Lie bracket generating condition\index{Lie bracket generating condition} 
\begin{equation}
 \dim\mathrm{Lie}_{(x,y)}\,\mathrm{span}\{X_1,X_2^{(\alpha)}\}\;=\;2\qquad\forall(x,y)\in\mathbb{R}^2
\end{equation}
is satisfied in this case. For $\alpha\in\mathbb{R}\setminus\mathbb{N}$ the field $X_2^{(\alpha)}$ is \emph{not} smooth and in addition the structure is \emph{not} Lie-bracket-generating.

To each $M_\alpha$ one naturally associates the Riemannian volume form
\begin{equation}\label{eq:volumeform}
 \mu_\alpha\;:=\;\mathrm{vol}_{g_\alpha}\;=\;\sqrt{\det g_\alpha}\,\ud x\wedge\ud y\;=\;|x|^{-\alpha}\,\ud x\wedge\ud y
\end{equation}
and the corresponding Laplace-Beltrami operator
\begin{equation}\label{eq:Deltamualpha}
 \Delta_{\mu_\alpha}\;=\;\frac{\partial^2}{\partial x^2}+|x|^{2\alpha}\frac{\partial^2}{\partial y^2}-\frac{\alpha}{|x|}\,\frac{\partial}{\partial x}\,,
\end{equation}
as follows from \eqref{eq:frame} and \eqref{eq:volumeform}, through the formula
\begin{equation*}
\begin{split}
 \Delta_{\mu_\alpha}\;&=\;\mathrm{div}_{\mu_\alpha}\nabla \;=\;X_1^2+X_2^2+(\mathrm{div}_{\mu_\alpha}X_1)X_1+(\mathrm{div}_{\mu_\alpha}X_2^{(\alpha)})X_2^{(\alpha)}\,.
\end{split}
\end{equation*}

For any fixed $\alpha$, the manifold $M_\alpha$ is geodesically incomplete, and more precisely all geodesics passing through a generic point $(x_0,y_0)\in M$ reach $\mathcal{Z}$ (see, e.g., \cite[Theorem 2.2]{GMP-Grushin-2018}, or also, for the special case $\alpha=1$ only, \cite[Sect.~11.2]{Calin-Chang-SubRiemannianGeometry} or \cite[Sect.~3.1]{Boscain-Laurent-2013}).


Let us now consider the problem of whether, depending on the parameter $\alpha$ measuring the singularity of  the metric, a quantum particle on $M_\alpha$ exhibits purely geometric confinement in each of the two halves $M^{\pm}$, or instead undergoes a transmission between them across $\mathcal{Z}$. Notably, for the classical counterpart of the same problem there is only one scenario: the geodesics reach $\mathcal{Z}$ and hence the classical particle is \emph{never} confined inside a half-cylinder.

One thus wants to study for which $\alpha$, in the Hilbert space
\begin{equation}\label{eq:Halphaspace}
 \cH_\alpha\;:=\;L^2(M,\ud\mu_\alpha)\,,
\end{equation}
understood as the completion of $C^\infty_c(M)$ (the space of smooth and compactly supported functions on $M$) with respect to the scalar product
\begin{equation}
 \langle \psi,\varphi\rangle_{\alpha}\;:=\;\iint_{(\mathbb{R}\setminus\{0\})\times\mathbb{S}^1}\overline{\psi(x,y)}\,\varphi(x,y)\,\frac{1}{|x|^{\alpha}}\,\ud x\,\ud y\,,
\end{equation}
the `\emph{minimal free Hamiltonian}'
\begin{equation}\label{Halpha}
 H_\alpha\;:=\;-\Delta_{\mu_\alpha}\,,\qquad\mathcal{D}(H_\alpha)\;:=\;C^\infty_c(M)
\end{equation}
is or is not essentially self-adjoint.

In the latter case, since $H_\alpha$ is evidently a densely defined, symmetric, lower semi-bounded operator in $\cH_\alpha$ (symmetry in particular follows from Green's identity),
it admits an infinity of self-adjoint extensions, each of which has a domain of self-adjointness characterised by suitable \emph{boundary conditions} at $\mathcal{Z}$. For a generic such extension $\widetilde{H}_\alpha$, Schr\"{o}dinger's unitary flow $e^{-\ii t \widetilde{H}}$ evolves the quantum particle's wave-function so as to reach the boundary $\mathcal{Z}$ in finite time. Generically, $\widetilde{H}_\alpha$ induces a \emph{transmission} (a transfer of mass, hence of $L^2$-norm) across $\mathcal{Z}$, or simply an evolution that, while preserving the left/right confinement, couples the particle with the side of the boundary the particle travels against. The boundary conditions of self-adjointness encode a physical interaction of the boundary with the interior. 

On the other hand, if  $H_\alpha$ is already essentially self-adjoint on $C^\infty_c(M)$, then it is natural to argue that the dynamics generated by its closure $\overline{H_\alpha}$ exhibits \emph{geometric quantum confinement} within $M$. In fact,
 \begin{equation}\label{eq:decomp+-}
  L^2(M,\ud\mu_\alpha)\;\cong\;L^2(M^-,\ud \mu_\alpha)\oplus L^2(M^+,\ud \mu_\alpha)\,,
 \end{equation}
 and if we define $H_\alpha^{\pm}$ acting on $L^2(M^{\pm},\ud \mu_\alpha)$ in complete analogy to \eqref{Halpha} with domain $C^\infty_c(M^{\pm})$, then with respect to the decomposition \eqref{eq:decomp+-} one has
 \begin{equation}\label{eq:H-HpHm}
  H_\alpha\;=\;H_\alpha^-\oplus H_\alpha^+\,.
 \end{equation}
Thus, $H_\alpha$ is essentially self-adjoint if and only if so too are both $H_\alpha^+$ and $H_\alpha^-$, in which case $\overline{H_\alpha}=\overline{H_\alpha^-}\oplus\overline{H_\alpha^+}$ as a direct orthogonal sum of self-adjoint operators, where the operator closure $\overline{H_\alpha}$ (resp., $\overline{H_\alpha^\pm}$) is the unique self-adjoint extension of $H_\alpha$ (resp. $H_\alpha^\pm$), and the propagators satisfy
 \begin{equation}
  e^{-\ii t \overline{H_\alpha}}\;=\;e^{-\ii t \overline{H_\alpha^-}}\oplus e^{-\ii t \overline{H_\alpha^+}}\,,\qquad \forall\,t\in\mathbb{R}\,.
 \end{equation}
 Therefore, for any initial datum $\psi_0\in\mathcal{D}(\overline{H_\alpha})$ with support only within $M^+$, the \emph{unique} solution $\psi\in C^1(\mathbb{R}_t,L^2(M,\ud\mu_\alpha))$ to the Cauchy problem
 \begin{equation}
  \begin{cases}
   \;\ii\partial_t \psi \!\!&=\;\overline{H_\alpha}\,\psi \\
   \;\psi|_{t=0}\!\!&=\;\psi_0
  \end{cases}
 \end{equation}
 remains for all times supported (``confined'') in $M^+$ \emph{with no need of declaration of non-trivial boundary conditions at $\mathcal{Z}$}, whence the qualification of the confinement as \emph{purely geometric}. The quantum particle initially prepared, say, in the right open half-cylinder never crosses the $y$-axis towards the left half-cylinder. \emph{For all times} the quantum particle's wave-function need not be characterised by boundary conditions at $\mathcal{Z}$ -- pictorially, the quantum particle stays permanently away from $\mathcal{Z}$, no quantum information reaches $\partial M^+$ or escapes from $M^+$.

 In this respect, the geometric quantum confinement problem has the following answer.

 \begin{theorem}[Quantum confinement vs transmission in Grushin cylinder, \cite{Boscain-Laurent-2013,Boscain-Prandi-JDE-2016,GMP-Grushin-2018}]\label{thm:Halpha_esa_or_not}~
  \begin{itemize}
   \item[(i)] If $\alpha\in(-\infty,-3]\cup[1,+\infty)$, then the operator $H_\alpha$ is essentially self-adjoint.
   \item[(ii)] If $\alpha\in(-3,-1]$, then the operator $H_\alpha$ is not essentially self-adjoint and it has deficiency index $2$.
   \item[(iii)] If $\alpha\in(-1,1)$, then the operator $H_\alpha$ is not essentially self-adjoint and it has infinite deficiency index.
  \end{itemize}
 \end{theorem}

 In the present work we study the \emph{non-trivial regime of transmission}, namely lack of self-adjointness with infinite deficiency index, and of actual \emph{singularity of the Grushin metric}. Thus, we consider $\alpha\in[0,1)$.

 In fact, the case $\alpha=0$ corresponds to the ordinary Laplacian minimally defined in each of the two halves of the Euclidean cylinder, to the right and to the left of the singularity region at $x=0$. The discussion of this case is completely analogous to the self-adjointness problem for the minimally defined Laplacian on a half-plane (see, e.g., \cite[Chapt.~9]{Grubb-DistributionsAndOperators-2009}) and our analysis for generic $\alpha\in[0,1)$ includes it. Moreover, in retrospect it will be clear how the conceptual scheme of our analysis is the very same also for the counterpart regime $\alpha\in(-1,0)$, although of course new explicit computations need be worked out.
 
%
 
 \subsection{Scheme of our analysis. Main results}\label{sec:scheme-and-main-results}~
 
 The infinity of the deficiency index of $H_\alpha$ when $\alpha\in[0,1)$ leaves room for a huge variety of inequivalent self-adjoint realisations of the free Hamiltonian. Each one provides a different mechanism how the quantum particle `reaches' or `crosses' the singularity region $\mathcal{Z}$.

 As is typical also in other contexts where minimally defined operators suggested by physical modelling have \emph{infinite} deficiency index \cite{MO-2016,MO-2017,M2020-BosonicTrimerZeroRange}, a large part of the extensions of $H_\alpha$ when $\alpha\in[0,1)$, albeit unambiguous (i.e., self-adjoint), do not have any plausible physical interpretation, like all those extensions identified by \emph{non-local} boundary conditions, i.e., when the behaviour of the wave function around a point $(0,y_0)\in\mathcal{Z}$ depends also on the behaviour around $\mathcal{Z}$ in regions away from $(0,y_0)$.
 
 Our first main result (Theorem \ref{thm:H_alpha_fibred_extensions} below) is indeed an explicit classification of the physically meaningful sub-family of `local' self-adjoint extensions of $H_\alpha$, characterising their boundary conditions at the singularity of the Grushin cylinder, and hence the mechanism of transmission of the quantum particle across the singularity.

 In this class, we identify the \emph{only} extension that actually induces \emph{geometric confinement} of the particle away from $\mathcal{Z}$ (hence confinement inside either half-cylinder), as well as the extension that in a suitable sense \emph{maximises the transmission} across $\mathcal{Z}$ -- customarily referred to as the \emph{bridging extension}. This reproduces by alternative means the recent analysis on Grushin cylinder by Boscain and Prandi \cite{Boscain-Prandi-JDE-2016}, where a `bridging extension' was identified for the first time.

 Our second main result is in fact a classification of the whole family of self-adjoint extensions of a convenient, unitarily equivalent version of $H_\alpha$, that we shall call $\mathscr{H}_\alpha$, essentially obtained from $H_\alpha$ by a re-scaling in $x$ plus a Fourier transform in the compact variable $y$. Such a transformation naturally leads to the $\alpha$-independent Hilbert space
 \begin{equation}\label{eq:previewhilberth}
  \cH\;=\;\bigoplus_{k\in\mathbb{Z}}L^2(\mathbb{R},\ud x)\;\cong\;\ell^2(\mathbb{Z},L^2(\mathbb{R},\ud x))
 \end{equation}
 and to the study of $\mathscr{H}_\alpha$ in each Fourier mode $k$. The self-adjoint extension problem in the $(x,k)$-coordinates turns out to be structurally much more manageable, for the adjoint of $\mathscr{H}_\alpha$ has the form of a direct sum 
 \begin{equation}\label{eq:previewHalphastar}
  \mathscr{H}_\alpha^*\;=\;\bigoplus_{k\in\mathbb{Z}} A_\alpha(k)^*
 \end{equation}
 for suitable symmetric operators $A_\alpha(k)$ on $L^2(\mathbb{R},\ud x)$, where clearly the symbol of the adjoint in the two sides of \eqref{eq:previewHalphastar} refers, respectively, to the Hilbert spaces $\cH_\alpha$ and $L^2(\mathbb{R},\ud x)$. This allows for a characterisation of the self-adjoint extensions of $\mathscr{H}_\alpha$ as suitable \emph{restrictions} of the operator \eqref{eq:previewHalphastar}.

 We establish such a characterisation both in its full generality (Theorem \ref{thm:Halphageneralext}), thus covering the whole family of extensions, and for a sub-class of extensions whose boundary conditions of self-adjointness are formulated \emph{separately in each mode $k$} in the form of constraints on the behaviour of the elements of the domain of $\mathscr{H}_\alpha^*$ when $x\to 0^\pm$, thus from both sides of the singularity (formulas \eqref{eq:Halphageneralext}-\eqref{eq:fibredS}, Theorems \ref{thm:bifibre-extensions}, \ref{thm:bifibre-extensionsc0c1}, and \ref{prop:g_with_Pweight}). For the latter sub-class we use the self-explanatory name of `\emph{fibred extensions}', each $L^2$-space in the Hilbert direct sum \eqref{eq:previewhilberth} being one `\emph{fibre}'.

 For generic fibred extensions, the self-adjointness constraints do not have an equally clean and simple counterpart in the $(x,y)$ variables, due to the non-local character of the inverse Fourier transform needed to go back from $\mathscr{H}_\alpha$ to the original $H_\alpha$. However, a special sub-class that we call `\emph{uniformly fibred extensions}' display the feature of having in a sense the same type and magnitude of boundary condition in each mode $k$, and this yields finally to the above-mentioned \emph{local} boundary conditions at fixed $y$ as $x\to 0$ which characterise the the `physical', most relevant extensions (Theorem \ref{thm:classificationUF}).

 From this perspective, our analysis is organised in two levels. After setting up the problem (Sect.~\ref{sec:preparatory_direct-integral}), the first level (Sections \ref{sec:fibre-extensions} through \ref{sec:bilateralfibreext}) is the study of the self-adjointness problem fibre by fibre, of $k$-dependent, densely defined, symmetric differential operator of Schr\"{o}dinger type $A_\alpha(k)$ on $L^2(\mathbb{R})$. To this aim we use the Kre\u{\i}n-Vi\v{s}ik-Birman extension theory, which is particularly suited since the differential operator in each fibre is semi-bounded. This requires the identification of the ingredients of the theory, namely the precise Sobolev regularity and short-scale behaviour of the functions in the domain of the closure $\overline{A_\alpha(k)}$, the characterisation of its Friedrichs extensions $A_{\alpha,F}(k)$ and its inverse $(A_{\alpha,F}(k))^{-1}$, and the characterisation of the deficiency space $\ker A_\alpha(k)^*$.

 The second level of our analysis (Sections \ref{sec:genextscrHa} through \ref{sec:proof_xy_Euclidean}) is devoted instead to re-assembling the information on each fibre in order to produce the classes of fibred and uniformly fibred extensions of $\mathscr{H}_\alpha$. The study of the latter, which as said produces eventually the physically relevant, local extensions, is particularly troublesome, not much for the standard operation of taking the direct sum of self-adjoint operators from each fibre, but rather because of the necessity to obtain some kind of \emph{uniformity over all the modes $k$}, in order to unfold back the Fourier transform that initially led from $H_\alpha$ to $\mathscr{H}_\alpha$. This is non-trivial because the self-adjointness condition on each fibre is in a sense highly non-uniform in $k$. To convey a flavour of the somewhat odd line of reasoning that we are forced to follow (see Section \ref{sec:uniformlyfirbredext}), let us point out that we construct the uniformly fibred extensions of $\mathscr{H}_\alpha$ by restricting $\mathscr{H}_\alpha^*$ to functions $g$ given by an expression of the form
 \begin{equation}\label{eq:oddbutconvenient}
  g\;=\;\varphi+G_0+G_1
 \end{equation}
 where \emph{none} of the three canonical summands $\varphi$, $G_0$, $G_1$ actually belongs to $\mathcal{D}(\mathscr{H}_\alpha^*)$, but only their sum does, due to cancellations on which we lack any explicit control! Yet, \eqref{eq:oddbutconvenient} is the most practical expression to export the boundary conditions of self-adjointness, cleanly formulated in terms of $G_0$ and $G_1$ as $x\to 0^\pm$, by means of an inverse Fourier transform back to the original problem in the $(x,y)$-variables.

 Whereas the above-mentioned main results contained in Theorems  \ref{thm:bifibre-extensions}, \ref{thm:bifibre-extensionsc0c1}, \ref{prop:g_with_Pweight}, \ref{thm:Halphageneralext}, and \ref{thm:classificationUF} require additional preparation that we defer to the main body of this work, in this introduction we present our first main result, namely the classification of the local, physical extensions.

 As is going to be done throughout, motivated by the fact that transmission across the singularity region $\mathcal{Z}$ is characterised by a specific behaviour as $x\to 0^\pm$, let us canonically express the elements of $L^2(M,\ud\mu_\alpha)$ with respect to the decomposition \eqref{eq:decomp+-} as
 \begin{equation}
  f\;=\;f^-\oplus f^+\;\equiv\;\begin{pmatrix} f^- \\ f^+ \end{pmatrix}\,,\qquad f^\pm(x)\,:=\,f(x)\quad\textrm{for }x\in\mathbb{R}^\pm\,,
 \end{equation}
 thus with $f^\pm \in L^2(M^\pm,\ud \mu_\alpha)$.

 The first important observation is that $H_\alpha^*$ is decomposed with respect to \eqref{eq:decomp+-}.

 \begin{proposition}\label{prop:adjoint_on_M} 
  Let $\alpha\geqslant 0$. The adjoint of $H_\alpha$ with respect to the Hilbert space $L^2(M,\ud\mu_\alpha)$ is the differential operator
   \begin{equation}
    H_\alpha^*\;=\;(H_\alpha^-)^*\oplus(H_\alpha^+)^*
   \end{equation}
    where $(H_\alpha^\pm)^*$, the adjoint of $H_\alpha^\pm$ in $L^2(M^\pm,\ud\mu_\alpha)$, is the differential operator whose domain and action are given by
   \begin{equation}
   \begin{split}
     \mathcal{D}((H_\alpha^\pm)^*)\;&=\;\big\{f^\pm\in L^2(M^\pm,\ud\mu_\alpha)\,\big|\,-\Delta_{\mu_\alpha}f^\pm\in L^2(M^\pm,\ud\mu_\alpha)\big\} \\
     (H_\alpha^\pm)^*\,f^\pm\;&=\;-\Delta_{\mu_\alpha}f^\pm\,.
   \end{split}
   \end{equation}
 \end{proposition}

 Next, we describe the special sub-class of self-adjoint restrictions of $(H_\alpha^\pm)^*$, hence extensions of $H_\alpha$, characterised by local boundary conditions.

    \begin{theorem}\label{thm:H_alpha_fibred_extensions}
 Let $\alpha\in[0,1)$. The operator $H_\alpha$ admits, among others, the following families of self-adjoint extensions in $L^2(M,\ud\mu_\alpha)$:
 \begin{itemize}
  \item \underline{Friedrichs extension}: $H_{\alpha,F}$;
  \item \underline{Family $\mathrm{I_R}$}: $\{H_{\alpha,R}^{[\gamma]}\,|\,\gamma\in\mathbb{R}\}$;
  \item \underline{Family $\mathrm{I_L}$}: $\{H_{\alpha,L}^{[\gamma]}\,|\,\gamma\in\mathbb{R}\}$;
  \item \underline{Family $\mathrm{II}_a$} with $a\in\mathbb{C}$: $\{H_{\alpha,a}^{[\gamma]}\,|\,\gamma\in\mathbb{R}\}$;
  \item \underline{Family $\mathrm{III}$}: $\{H_{\alpha}^{[\Gamma]}\,|\,\Gamma\equiv(\gamma_1,\gamma_2,\gamma_3,\gamma_4)\in\mathbb{R}^4\}$.
 \end{itemize}
 Each operator belonging to any such family is a restriction of $H_\alpha^*$, and hence its differential action is precisely $-\Delta_{\mu_\alpha}$. The domain of each of the above extensions is characterised as the space of the functions $f\in L^2(M,\ud\mu_\alpha)$ satisfying the following properties.
  \begin{itemize}
  \item[(i)] \underline{Integrability and regularity}:
  \begin{equation}\label{eq:DHalpha_cond1}
  \sum_{\pm}\;\iint_{\mathbb{R}_x^\pm\times\mathbb{S}^1_y}\big|(\Delta_{\mu_\alpha}f^\pm)(x,y)\big|^2\,\ud\mu_\alpha(x,y)\;<\;+\infty\,.
 \end{equation}
  \item[(ii)] \underline{Boundary condition}: The limits
 \begin{eqnarray}
  f_0^\pm(y)\!\!&:=&\!\!\lim_{x\to 0^\pm}f^\pm(x,y) \label{eq:DHalpha_cond2_limits-1}\\
  f_1^\pm(y)\!\!&:=&\!\!\pm(1+\alpha)^{-1}\lim_{x\to 0^\pm}\Big(\frac{1}{\:|x|^\alpha}\,\frac{\partial f(x,y)}{\partial x}\Big) \label{eq:DHalpha_cond2_limits-2}
  \end{eqnarray}
 exist and are finite for almost every $y\in\mathbb{S}^1$; depending on the considered type of extension, and for almost  every $y\in\mathbb{R}$, they satisfy
 \begin{eqnarray}
  f_0^\pm(y)\,=\,0 \qquad \quad\;\;& & \textrm{if }\;  f\in\mathcal{D}(H_{\alpha,F})\,, \label{eq:DHalpha_cond3_Friedrichs}\\
  \begin{cases}
   \;f_0^-(y)= 0  \\
   \;f_1^+(y)=\gamma f_0^+(y)
  \end{cases} & & \textrm{if }\;  f\in\mathcal{D}(H_{\alpha,R}^{[\gamma]})\,, \\
   \begin{cases}
   \;f_1^-(y)=\gamma f_0^-(y) \\
   \;f_0^+(y)= 0 
  \end{cases} & & \textrm{if }\;  f\in\mathcal{D}(H_{\alpha,L}^{[\gamma]}) \,, \label{eq:DHalpha_cond3_L}\\
     \begin{cases}
   \;f_0^+(y)=a\,f_0^-(y) \\
   \;f_1^-(y)+\overline{a}\,f_1^+(y)=\gamma f_0^-(y)
  \end{cases} & & \textrm{if }\;  f\in\mathcal{D}(H_{\alpha,a}^{[\gamma]})\,, \label{eq:DHalpha_cond3_IIa} \\
   \begin{cases}
   \;f_1^-(y)=\gamma_1 f_0^-(y)+(\gamma_2+\ii\gamma_3) f_0^+(y) \\
   \;f_1^+(y)=(\gamma_2-\ii\gamma_3) f_0^-(y)+\gamma_4 f_0^+(y)
  \end{cases} & & \textrm{if }\;  f\in\mathcal{D}(H_{\alpha}^{[\Gamma]})\,. \label{eq:DHalpha_cond3_III}
 \end{eqnarray} 
 \end{itemize} 
     Moreover,
 \begin{equation}\label{eq:traceregularity}
  f_0^\pm \in H^{s_{0,\pm}}(\mathbb{S}^1, \ud y)\qquad\textrm{ and }\qquad f_1^\pm\in H^{s_{1,\pm}}(\mathbb{S}^1,\ud y)
 \end{equation}
 with
 \begin{itemize}
 	\item $s_{1,\pm}=\frac{1}{2}\frac{1-\alpha}{1+\alpha}$\qquad\qquad\qquad\qquad\qquad\; for the Friedrichs extension,
 	\item $s_{1,-}=\frac{1}{2}\frac{1-\alpha}{1+\alpha}$, $s_{0,+}=s_{1,+}=\frac{1}{2}\frac{3+\alpha}{1+\alpha}$ \quad for extensions of type $\mathrm{I}_R$,
 	\item  $s_{1,+}=\frac{1}{2}\frac{1-\alpha}{1+\alpha}$, $s_{0,-}=s_{1,-}=\frac{1}{2}\frac{3+\alpha}{1+\alpha}$ \quad for extensions of type $\mathrm{I}_L$,
 	\item $s_{1,\pm}=s_{0,\pm}=\frac{1}{2}\frac{1-\alpha}{1+\alpha}$ \qquad\qquad\qquad \;\;\;\,for extensions of type $\mathrm{II}_a$,
 	\item $s_{1,\pm}=s_{0,\pm}=\frac{1}{2}\frac{3+\alpha}{1+\alpha}$ \qquad\qquad\qquad \;\; for extensions of type $\mathrm{III}$.
 \end{itemize}
\end{theorem}

 The requirement \eqref{eq:DHalpha_cond1} amounts to saying that all the considered extensions are contained in $H_\alpha^*$. Each of the requirements \eqref{eq:DHalpha_cond3_Friedrichs}-\eqref{eq:DHalpha_cond3_III} then expresses the corresponding condition of self-adjointness.

 The common feature of all such extensions is that their boundary conditions as $x\to 0$ have the \emph{same} form uniformly in $y\in\mathbb{R}$. In this precise sense, those are \emph{local} extensions.

 It is also clear that the Friedrichs extension, as well as type-$\mathrm{I_R}$ and type-$\mathrm{I_L}$ extensions, are reduced with respect to the Hilbert space decomposition \eqref{eq:decomp+-}: each such operator is the orthogonal sum of two self-adjoint operators, respectively on $L^2(M^+,\ud\mu_\alpha)$ and $L^2(M^-,\ud\mu_\alpha)$, characterised by independent boundary conditions at the singularity region $\mathcal{Z}$ from the right and from the left. On the contrary, type-$\mathrm{II}_a$ (with $a\neq 0$) and type-$\mathrm{III}$ extensions are not reduced \emph{in general}: the boundary condition couples the behaviour as $x\to 0^+$ and $x\to 0^-$.

The left-right reducibility
\begin{equation}\label{eq:reducedext}
 \widetilde{H}_\alpha\;\cong\;\widetilde{H}_\alpha^-\oplus\widetilde{H}_\alpha^+\,.
\end{equation}
of the extension $\widetilde{H}_\alpha=H_{\alpha,F}$, or $\widetilde{H}_\alpha=H_{\alpha,R}^{[\gamma]}$, or $\widetilde{H}_\alpha=H_{\alpha,L}^{[\gamma]}$, results in a decoupled independent Schr\"{o}dinger evolution of the two components $f^+$ and $f^-$ of the solution $f\in C^1(\mathbb{R}_t,L^2(M,\ud\mu_\alpha))$ to the Cauchy problem
 \begin{equation}
  \begin{cases}
   \;\ii\,\partial_t f \!\!&=\;\widetilde{H}_\alpha\,f \\
   \;f|_{t=0}\!\!&=\;u_0\;\in\;\mathcal{D}(\widetilde{H}_\alpha)\,.
  \end{cases}
 \end{equation}
This means that, separately on each half-cylinder,
\begin{equation}\label{eq:reducedevol}
 f^\pm(t)\;=\;e^{-\ii t \widetilde{H}_\alpha^\pm}u_0^\pm\,,
\end{equation}
with no exchange between left and right at the interface $\mathcal{Z}$.

The picture is then the following.
\begin{itemize}
 \item Friedrichs extension $H_{\alpha,F}$: geometric quantum confinement on each half of the Grushin cylinder, with no interaction of the particle with the boundary and no dynamical transmission between the two halves. 
 \item Type-$\mathrm{I_R}$ and type-$\mathrm{I_L}$ extensions: no dynamical transmission across $\mathcal{Z}$, but possible non-trivial interaction of the quantum particle with the boundary respectively from the right or from the left, with geometric quantum confinement on the opposite side. (Thus, for instance, a quantum particle governed by $H_{\alpha,R}^{[\gamma]}$ may `touch' the boundary from the right, but not from the left, and moreover it cannot trespass the singularity region.)
 \item Type-$\mathrm{II}_a$ and type-$\mathrm{III}$ extensions: in general, dynamical transmission through the boundary.
\end{itemize}

Among the latter group of extensions, a special status is deserved by the Laplace-Beltrami realisation
\begin{equation}\label{eq:HB}
 H_{\alpha,B}\;:=\;H_{\alpha,a}^{[\gamma]}\quad\textrm{ with $a=1$ and $\textstyle\gamma=0$}\,.
\end{equation}
In this case the boundary condition \eqref{eq:DHalpha_cond3_IIa} takes the form
\begin{equation}\label{eq:bridging_conditions}
 \begin{split}
  \lim_{x\to 0^-}f(x,y)\;&=\;\lim_{x\to 0^+}f(x,y) \\
  \lim_{x\to 0^-}\Big(\frac{1}{\:|x|^\alpha}\,\frac{\partial f(x,y)}{\partial x}\Big)\;&=\;\lim_{x\to 0^+}\Big(\frac{1}{\:|x|^\alpha}\,\frac{\partial f(x,y)}{\partial x}\Big)
 \end{split}
\end{equation}
for almost every $y\in\mathbb{S}^1$. Formula \eqref{eq:bridging_conditions} expresses the \emph{continuity} across the singularity region $\mathcal{Z}$, along (almost) any horizontal direction, both of a generic $f\in\mathcal{D}(H_{\alpha,B})$ and of the partial derivative in $x$ of $f$, when such a derivative is suitably weighted with the $|x|^{-\alpha}$-weight. It is easily seen by inspection of \eqref{eq:DHalpha_cond3_Friedrichs}-\eqref{eq:DHalpha_cond3_III} that no other boundary condition of self-adjointness allows for such a two-fold continuity for any other weight.

Quantum-mechanically, \eqref{eq:HB}-\eqref{eq:bridging_conditions} are interpreted as the continuity of the spatial probability density of the particle in the region around $\mathcal{Z}$ and of the momentum in the direction orthogonal to $\mathcal{Z}$, defined with respect to the weight $|x|^{-\alpha}$ induced by the metric. This occurrence corresponds to the `optimal' transmission across the boundary $\mathcal{Z}$, with no discrepancy in spatial density and momentum between left and right: the dynamics generated by $H_{\alpha,B}$ develops the best `bridging' between the left and the right side of the Grushin cylinder. For this reason $H_{\alpha,B}$ shall be referred to as the `\emph{bridging extension}' of $H_\alpha$. It is precisely the bridging extension introduced by Boscain and Prandi in \cite[Proposition 3.11]{Boscain-Prandi-JDE-2016}, which we recover here \emph{as a distinguished element of our general classification}.

One further observation on Theorem \ref{thm:H_alpha_fibred_extensions} (see also Remark \ref{rem:regularitydeficiencyspace} for a more explicit comment on this point) concerns the regularity \eqref{eq:traceregularity} of the boundary functions $f_0$ and $f_1$ in terms of which the various conditions of self-adjointness are expressed. In fact, \eqref{eq:DHalpha_cond2_limits-1}-\eqref{eq:DHalpha_cond2_limits-2} define so-called `\emph{trace maps}'
\[
\begin{split}
& \gamma_0^\pm:\mathcal{D}(\widetilde{H}_\alpha)\cap L^2(M^\pm,\ud\mu_\alpha)\to H^{s_0,\pm}(\mathbb{S}^1)  \\
& \gamma_1^\pm:\mathcal{D}(\widetilde{H}_\alpha)\cap L^2(M^\pm,\ud\mu_\alpha)\to H^{s_1,\pm}(\mathbb{S}^1)\,,
\end{split}
\]
(actually, concrete examples of what one customarily refers to as `\emph{abstract trace maps}' -- see, e.g., \cite[Sect.~2]{Posilicano-2014-sum-trace-maps}), where $\widetilde{H}_\alpha$ stands for one of the considered extensions of $H_\alpha$. Noticeably, although we do not carry this comparison further on here, and we defer it to a subsequent study, our \eqref{eq:traceregularity} is completely consistent with the abstract analysis developed recently by Posilicano \cite{Posilicano-2014-sum-trace-maps} of the trace space, and hence also, isomorphically speaking, of the deficiency space, of the operator $H_\alpha$.

Once again it is worth underlying that also in the language of \cite{Posilicano-2014-sum-trace-maps}, namely the framework of direct sum of trace maps, the hard part of the job that remains to be done, and that we completed here, for the classification of the (local) extensions of $H_\alpha$, is the passage from the `natural' direct sum setting, namely the description of the restrictions of the direct sum operator $\mathscr{H}_\alpha^*=\bigoplus_{k\in\mathbb{Z}}(A_\alpha(k))^*$, to the original Grushin setting, namely the corresponding descriptions of the restrictions of $H_\alpha^*$.

 Last, we should like to emphasise that the present analysis, characterising the relevant Hamiltonians of quantum transmission across the singularity of a Grushin-type cylinder, naturally opens new related and significant questions which have \emph{already} come on top of the agenda of this segment of research on self-adjoint Laplace-Beltrami operators on degenerate Riemannian manifolds.
 
 Two directions, in particular, are worth mentioning. First, each such Hamiltonian encoding a protocol of transmission, or at least of non-trivial interaction with the boundary, it is of relevance to investigate the spectral and scattering properties of the various admissible protocols. This includes the study of transmission and reflection coefficients (the fraction of flux of particles, initially shot along the cylinder from the far infinity against $\mathcal{Z}$, and which trespass the singularity locus or re-bounce backwards), as well as the bound states of each admissible Hamiltonian. A systematic analysis in this direction has been recently completed in \cite{GM-Grushin3-2020}.
 
 The other notable direction, more analytic in nature, is the study of the qualitative (dispersive, in particular) properties of the heat and the Schr\"{o}dinger flow generated by the self-adjoint operators classified in Theorem \ref{thm:H_alpha_fibred_extensions} -- in the case of the heat flow, the problem can be re-phrased in terms of Brownian motion on the Grushin-type cylinder, its possible stochastic completeness, recurrence, etc., in the spirit of \cite{Fukushima-Oshima-Takeda,Boscain-Prandi-JDE-2016}. A general discussion on this problem for the Grushin-type cylinder, with first numerical evidence, have been recently developed in \cite{GMP-heat-2021}.

\bigskip

 \textbf{Notation.} Besides all the standard functional-analytic and operator-theoretic notation adopted in this work, let us specify the following symbols and conventions.
 
 \begin{center}
\begin{tabular}{ ccl } 
 $\mathbb{R}^+$ & & $(0,+\infty)$, \emph{open} right half-line \\ 
 $\mathbb{R}^-$ & & $(-\infty,0)$, \emph{open} left half-line  \\ 
 $\mathring{K}$ & & interior of the subset $K\subset \mathbb{R}$ \\
 $\langle x\rangle$ & & $\sqrt{1+x^2}$ \\ 
 $\mathbbm{1}$ & & identity operator, acting on the space that is clear from the context \\
 $\mathbb{O}$ & & zero operator, acting on the space that is clear from the context \\
 $\mathbf{1}_K$ & & characteristic function of the set $K$ \\
 $\langle\cdot,\cdot\rangle$ & & Hilbert scalar product, anti-linear in the first entry \\
 $\delta_{k,\ell}$ & & Kronecker delta \\
 $V^\perp$ & & Hilbert orthogonal complement of the subspace $V$ \\
 $\dotplus$ & & direct sum between vector spaces \\
 $\oplus$ & & (if referred to operators) reduced direct sum of operators \\
 $\oplus$ & & (if referred to vector spaces) Hilbert orthogonal direct sum \\
 $\boxplus$ & & Hilbert orthogonal direct sum of non-closed subspaces \\
 $\sim$ & & identity up to the leading order and/or to a multiplicative constant.
\end{tabular}
\end{center}

\section{Preparatory materials}\label{sec:preparatory_direct-integral}

\subsection{Unitary equivalence to a constant-fibre orthogonal sum structure}~

Following the same steps we made in \cite{GMP-Grushin-2018}, let us introduce a natural, unitarily equivalent re-formulation of the problem of the self-adjoint extensions of $H_\alpha$ in $L^2(M,\ud\mu_\alpha)$, where $M=(\mathbb{R}\setminus\{0\})\times\mathbb{S}^1$ and $\ud\mu_\alpha=|x|^{-\alpha}\ud x\,\ud y$.

We recall that $H_\alpha$ is reduced with respect to the decomposition \eqref{eq:decomp+-} -- see \eqref{eq:H-HpHm} above -- hence it is natural to manipulate $H_\alpha^+$ and $H_\alpha^-$ separately.

We shall map $L^2(M^\pm,\ud\mu_\alpha)$ unitarily onto the space
\begin{equation}\label{eq:defHpm-uptoiso}
\cH^\pm\;:=\;\bigoplus_{k\in\mathbb{Z}} L^2(\mathbb{R}^\pm,\ud x)\;\cong\;\ell^2(\mathbb{Z},L^2(\mathbb{R}^\pm,\ud x))\;\cong\;L^2(\mathbb{R}^\pm,\ud x)\otimes\ell^2(\mathbb{Z})
\end{equation}
(with obvious canonical isomorphisms in the r.h.s.~of \eqref{eq:defHpm-uptoiso}).

We first apply the unitary transformation
\begin{equation}\label{eq:unit1}
\begin{split}
 U_\alpha^\pm:L^2(\mathbb{R}^\pm\times\mathbb{S}^1,|x|^{-\alpha}\ud x\ud y)&\stackrel{\cong}{\longrightarrow}L^2(\mathbb{R}^\pm\times\mathbb{S}^1,\ud x\ud y)\,, \\
 f &\; \mapsto\;\phi\;:=\;  |x|^{-\frac{\alpha}{2}}f
\end{split}
\end{equation}
(thus restoring the standard Euclidean metric by removing the weight), 
and then the further unitary transformation
\begin{equation}\label{eq:unit2}
 \mathcal{F}_2^{\pm}:L^2(\mathbb{R}^\pm\times\mathbb{S}^1,\ud x\ud y)\stackrel{\cong}{\longrightarrow}
 L^2(\mathbb{R}^\pm,\ud x)\otimes\ell^2(\mathbb{Z})\;=:\;\cH^\pm\,,
\end{equation}
consisting of the discrete Fourier transform in the $y$-variable only, that is, the mapping
\begin{equation}\label{eq:defF2}
 \begin{split}
  &\phi\;\mapsto\;\psi\;\equiv\;(\psi_k)_{k\in\mathbb{Z}}\,, \\
  e_k(y)\;:=\;\frac{e^{\ii k y}}{\sqrt{2\pi}}\,,&\qquad \psi_k(x)\;:=\int_0^{2\pi}\overline{e_k(y)}\,\phi(x,y)\,\ud y\,,\qquad x\in\mathbb{R}^{\pm}\,.
 \end{split}
\end{equation}
This is the customary way to re-write $\phi(x,y)=\sum_{k\in\mathbb{Z}}\psi_k(x)e_k(y)$ in the $L^2$-convergent sense. Each $\psi_k\in L^2(\mathbb{R}^\pm,\ud x)$ and $\sum_{k\in\mathbb{Z}}\|\psi_k\|_{L^2}^2<+\infty$.

Thus,
\begin{equation}
  \cH^\pm\;=\;\mathcal{F}_2^{\pm} U_\alpha^{\pm}L^2(M^\pm,\ud\mu_\alpha)
\end{equation}
with a natural \emph{`constant-fibre' orthogonal sum structure} on such space, namely,
\begin{equation}\label{L^2directDecomp}
  \cH^\pm\;=\;\bigoplus_{k\in\mathbb{Z}} \;\mathfrak{h}^\pm\,,\qquad \mathfrak{h}_\pm\;:=\;L^2(\mathbb{R}^\pm,\ud x)\,,
\end{equation}
with \emph{constant fibre} $\mathfrak{h}_\pm$ and scalar product 
\begin{equation}
 \big\langle (\psi_k)_{k\in\mathbb{Z}} , (\widetilde{\psi}_k)_{k\in\mathbb{Z}} \big\rangle_{\cH^{\pm}}=\;\sum_{k\in\mathbb{Z}}\,\int_{\mathbb{R}^\pm}\overline{\psi_k(x)}\,\widetilde{\psi}_k(x)\,\ud x\;\equiv\;\sum_{k\in\mathbb{Z}}\,\langle \psi_k,\widetilde{\psi}_k\rangle_{\mathfrak{h}^\pm}\,.
\end{equation}

Analogously, and with self-explanatory notation, $\mathcal{F}_2:=\mathcal{F}_2^-\oplus\mathcal{F}_2^+$, $U_\alpha:=U_\alpha^-\oplus U_\alpha^+$, whence $\mathcal{F}_2U_\alpha=\mathcal{F}_2^- U_\alpha^-\oplus \mathcal{F}_2^+ U_\alpha^+$, and
\begin{equation}\label{eq:Hxispace}
 \cH\;:=\;\mathcal{F}_2U_\alpha L^2(M,\ud\mu_\alpha)\;\cong\;\ell^2(\mathbb{Z},L^2(\mathbb{R},\ud x))\;\cong\;\cH^-\oplus\cH^+\;\cong\;\bigoplus_{k\in\mathbb{Z}}\;\mathfrak{h}
\end{equation}
with \emph{`bilateral' fibre}
\begin{equation}
 \mathfrak{h}\;:=\;L^2(\mathbb{R}^-,\ud x)\oplus L^2(\mathbb{R}^+,\ud x)\;\cong\;L^2(\mathbb{R},\ud x)\,.
\end{equation}

The above scheme is the discrete version of the \emph{constant-fibre direct integral structure}, the well-known natural formalism for the multiplication operator form of the spectral theorem \cite[Sect.~7.3]{Hall-2013_QuantumTheoryMathematicians}, as well as for the analysis of Schr\"{o}dinger's operators with periodic potentials \cite[Sect.~XIII.16]{rs4}. 

By means of \eqref{eq:unit1} and \eqref{eq:unit2} we obtain the operators
\begin{equation}\label{eq:tildeHalpha}
 \mathsf{H}_\alpha^\pm\;:=\;U_\alpha^\pm \,H_\alpha^\pm \,(U_\alpha^\pm)^{-1}
\end{equation}
acting on $L^2(\mathbb{R}^\pm\times\mathbb{S}^1,\ud x\ud y)$ as
\begin{equation}\label{eq:explicit-tildeHalpha}
 \begin{split}
  \mathcal{D}(\mathsf{H}_\alpha^\pm)\;&=\;C^\infty_c(\mathbb{R}^\pm_x\times\mathbb{S}^1_y) \,,\\
  \mathsf{H}_\alpha^\pm\phi\;&=\;\Big(-\frac{\partial^2}{\partial x^2}- |x|^{2\alpha}\frac{\partial^2}{\partial y^2}+\frac{\,\alpha(2+\alpha)\,}{4x^2}\Big)\phi\,,
 \end{split}
\end{equation}
as well as the operators 
\begin{equation}\label{eq:unitary_transf_pm}
 \mathscr{H}_\alpha^\pm\;:=\;\mathcal{F}^{\pm}_2\, U_\alpha^\pm \,H_\alpha^\pm \,(U_\alpha^\pm)^{-1}(\mathcal{F}_2^{\pm})^{-1}\;=\;\mathcal{F}^{\pm}_2\,\mathsf{H}_\alpha^\pm(\mathcal{F}_2^{\pm})^{-1}
\end{equation}
acting on $\cH^{\pm}$ as
\begin{equation}\label{eq:actiondomainHalpha}
  \begin{split}
  \mathcal{D}(\mathscr{H}_\alpha^\pm)\;&=\;\Big\{\psi\equiv(\psi_k)_{k\in\mathbb{Z}}\in \bigoplus_{k\in\mathbb{Z}} L^2(\mathbb{R}^\pm,\ud x)\,\Big|\,\psi\in\mathcal{F}_2^{\pm}C^\infty_c(\mathbb{R}^\pm_x\times\mathbb{S}^1_y)\Big\}, \\
    \mathscr{H}_\alpha^\pm\psi\;&=\;\Big(\Big(-\frac{\ud^2}{\ud x^2}+k^2 |x|^{2\alpha}+\frac{\,\alpha(2+\alpha)\,}{4x^2}\Big)\psi_k\Big)_{k\in\mathbb{Z}} \,.
 \end{split}
\end{equation}
Completely analogous formulas hold for $\mathsf{H}_\alpha$ and $\mathscr{H}_\alpha$, defined in the obvious way.

In particular, for each $\psi^\pm\in\mathcal{D}(\mathscr{H}_\alpha^\pm)$ the component functions $\psi^\pm_k(\cdot)$ are compactly supported in $x$ inside $\mathbb{R}^\pm$ for every $k\in\mathbb{Z}$, and moreover
\begin{equation}\label{eq:forcondii}
 \begin{split}
  \sum_{k\in\mathbb{Z}}\Big\|&\Big(-\frac{\ud^2}{\ud x^2}+k^2 |x|^{2\alpha}+\frac{\,\alpha(2+\alpha)\,}{4x^2}\Big)\psi^\pm_k \Big\|_{L^2(\mathbb{R}^\pm,\ud x)}^2 \\
  &=\;\|\mathscr{H}_\alpha^\pm\psi^\pm\|_{\cH^\pm}^2\;=\;\|(\mathcal{F}_2^{\pm})^{-1}\mathscr{H}_\alpha^\pm\mathcal{F}^{\pm}_2\phi^\pm\|_{L^2(\mathbb{R}^\pm_x\times\mathbb{S}^1_y)}^2 \\
  &=\;\Big\| \Big(-\frac{\partial^2}{\partial x^2}- |x|^{2\alpha}\frac{\partial^2}{\partial y^2}+\frac{\,\alpha(2+\alpha)\,}{4x^2}\Big)\phi^\pm  \Big\|_{L^2(\mathbb{R}^\pm_x\times\mathbb{S}^1_y)}^2\;<\;+\infty\,,
 \end{split}
\end{equation}
where $\phi^\pm=\mathcal{F}_2^\pm\psi\in C^\infty_c(\mathbb{R}^\pm_x\times\mathbb{S}^1_y)$.

%

The above construction establishes a unitarily equivalent version of the operators of interest. Thus, the self-adjointness problem for $H_\alpha^\pm$ in $L^2(M^\pm,\ud\mu_\alpha)$ is tantamount as the self-adjointness problem for $\mathscr{H}_\alpha^\pm$ in $\cH^\pm$, and the same holds for $H_\alpha$ with respect to $\mathscr{H}_\alpha$. Furthermore, when non-trivial self-adjoint extensions exist for $H_\alpha^\pm$ (resp., $H_\alpha$), they can be equivalently (and in practice more conveniently) identified as self-adjoint extensions of $\mathscr{H}_\alpha^\pm$ (resp., $\mathscr{H}_\alpha$). 


In fact, such an analysis for $\mathscr{H}_\alpha^\pm$ (resp., $\mathscr{H}_\alpha$) is naturally boiled down to the analysis of such operators \emph{on each fibre} and a subsequent recombination of the information over the whole \emph{constant-fibre orthogonal sum}.

To develop this approach, it is convenient to introduce on each fibre $\mathfrak{h}_\pm$, thus for each $k\in\mathbb{Z}$, the operators
\begin{equation}\label{eq:Axi}
 A_\alpha^\pm(k)\;:=\;-\frac{\ud^2}{\ud x^2}+k^2 |x|^{2\alpha}+\frac{\,\alpha(2+\alpha)\,}{4x^2}\,,\quad\,\mathcal{D}(A_\alpha^\pm(k))\;:=\;C^\infty_c(\mathbb{R}^\pm)\,,
\end{equation}
and similarly on $\mathfrak{h}$ we define
\begin{equation}\label{eq:Axibilateral}
 \begin{split}
  \mathcal{D}(A_\alpha(k))\;&:=\;C^\infty_c(\mathbb{R}^-)\boxplus C^\infty_c(\mathbb{R}^+) \\
  A_\alpha(k)\;&:=\;A_\alpha^-(k)\oplus A_\alpha^+(k)\,,
 \end{split}
\end{equation}
where the notation `$\boxplus$' simply indicates the direct sum of two (non-complete) subspaces of each summand of the orthogonal sum of two Hilbert spaces.

By construction the map $\mathbb{Z}\ni k\mapsto  A_{\alpha}(k)$ has values in the space of densely defined, symmetric, non-negative operators on $\mathfrak{h}$, \emph{all} with the \emph{same} domain irrespectively of $k$. In each $A_{\alpha}(k)$ the integer $k$ plays the role of a fixed parameter. Moreover, all the $A_{\alpha}(k)$'s are closable and each $\overline{A_{\alpha}(k)}$ is non-negative and with the same dense domain in $\mathfrak{h}$.

%

As non-trivial self-adjoint extensions are suitable restrictions of the adjoints, let us characterise the latter operators. As we argued already in \cite[Lemma 3.2]{GMP-Grushin-2018}, the adjoint of $\mathsf{H}_\alpha$ is the maximal realisation of the same differential operator, that is,
\begin{equation}\label{eq:HHalphaadjoint}
 \begin{split}
  \mathcal{D}((\mathsf{H}_\alpha^\pm)^*)\;&=\;\left\{
  \begin{array}{c}
   \phi\in L^2(\mathbb{R}^\pm\times\mathbb{S}^1,\ud x \ud y)\textrm{ such that} \\
   \Big(-\frac{\partial^2}{\partial x^2}- |x|^{2\alpha}\frac{\partial^2}{\partial y^2}+\frac{\,\alpha(2+\alpha)\,}{4x^2}\Big)\phi\in L^2(\mathbb{R}^\pm\times\mathbb{S}^1,\ud x \ud y)
  \end{array}
  \right\}, \\
  (\mathsf{H}_\alpha^\pm)\phi\;&=\;\Big(-\frac{\partial^2}{\partial x^2}- |x|^{2\alpha}\frac{\partial^2}{\partial y^2}+\frac{\,\alpha(2+\alpha)\,}{4x^2}\Big)\phi\,.
 \end{split}
\end{equation}
This, and the unitary equivalence \eqref{eq:unitary_transf_pm}, yields at once
\begin{equation}\label{eq:Hfstar}
  \begin{split}
  \mathcal{D}((\mathscr{H}_\alpha^\pm)^*)\;&=\;
  \left\{\!\!
  \begin{array}{c}
   \psi\equiv(\psi_k)_{k\in\mathbb{Z}}\in \bigoplus_{k\in\mathbb{Z}} L^2(\mathbb{R}^\pm,\ud x)\;\;\textrm{such that} \\ \\
   \;\displaystyle\sum_{k\in\mathbb{Z}}\Big\|\Big(-\frac{\ud^2}{\ud x^2}+k^2 |x|^{2\alpha}+\frac{\,\alpha(2+\alpha)\,}{4x^2}\Big)\psi_k \Big\|_{L^2(\mathbb{R}^\pm,\ud x)}^2\;<\;+\infty
  \end{array}
  \!\!\right\}, \\
   (\mathscr{H}_\alpha^\pm)^*\psi\;&=\;\Big(\Big(-\frac{\ud^2}{\ud x^2}+k^2 |x|^{2\alpha}+\frac{\,\alpha(2+\alpha)\,}{4x^2}\Big)\psi_k\Big)_{k\in\mathbb{Z}} \,.
 \end{split}
\end{equation}
Clearly, $\frac{\ud^2}{\ud x^2}$ is a weak derivative in \eqref{eq:Hfstar} and a classical derivative in \eqref{eq:actiondomainHalpha}.
Furthermore, with respect to the decomposition \eqref{eq:Hxispace},
\begin{equation}\label{eq:Hfstar_sum}
 (\mathscr{H}_\alpha)^*\;=\;(\mathscr{H}_\alpha^-)^*\oplus(\mathscr{H}_\alpha^+)^*\,.
\end{equation}

Analogously to \eqref{eq:Hfstar}, as we argued already in \cite[Eq.~(3.12)]{GMP-Grushin-2018}, one has
\begin{equation}\label{eq:Afstar}
 \begin{split}
  \mathcal{D}(A_{\alpha}^\pm(k)^*)\;&=\;
  \left\{\!\!
  \begin{array}{c}
   g^\pm\in L^2(\mathbb{R}^\pm,\ud x)\;\;\textrm{such that} \\
   \big(-\frac{\ud^2}{\ud x^2}+k^2 |x|^{2\alpha}+\frac{\,\alpha(2+\alpha)\,}{4x^2}\big)g^\pm\in L^2(\mathbb{R}^\pm,\ud x)
  \end{array}
  \!\!\right\}, \\
   A_{\alpha}^\pm(k)^*g^\pm\;&=\;\Big(-\frac{\ud^2}{\ud x^2}+k^2 |x|^{2\alpha}+\frac{\,\alpha(2+\alpha)\,}{4x^2}\Big) g^\pm\,,
 \end{split}
\end{equation}
and
\begin{equation}\label{eq:Afstar_sum}
 A_{\alpha}(k)^*\;=\;A_{\alpha}^-(k)^*\oplus A_{\alpha}^+(k)^*\,.
\end{equation}

\subsection{Orthogonal sum operators}~

Next, it is convenient to recall the structure of operators acting on $\cH$ (resp., on $\cH^\pm$) in the form of infinite orthogonal sum, that is, operators that are reduced by the orthogonal decomposition \eqref{eq:Hxispace} (resp., \eqref{L^2directDecomp}). By this we mean an operator $T$ for which there is a collection $(T(k))_{k\in\mathbb{Z}}$ of operators on $\mathfrak{h}$ (resp., on $\mathfrak{h}^\pm$) such that
\begin{equation}\label{eq:Tdirectint}
  \begin{split}
  \mathcal{D}(T)\;&:=\;\left\{\psi\equiv(\psi_k)_{k\in\mathbb{Z}}\in\cH\,\left|\! 
  \begin{array}{l}
   \mathrm{(i)}\quad\psi_k\in\mathcal{D}(T(k)) \;\;\forall k\in\mathbb{Z} \\
   \mathrm{(ii)}\,\displaystyle\sum_{k\in\mathbb{Z}}\big\|T(k)\psi_k\big\|_{\mathfrak{h}}^2<+\infty
  \end{array}
  \!\!\right.\right\}, \\
  T\psi\;&:=\;\big(T(k)\,\psi_k\big)_{k\in\mathbb{Z}}
 \end{split}
\end{equation}
(and analogous formulas on each half-fibre), 
the shorthand for which is
\begin{equation}\label{eq:T_direct_integral}
 T\;=\;\bigoplus_{k\in\mathbb{Z}} \,T(k)\,.
\end{equation}
Thus, $T(k)=T\upharpoonright (\mathcal{D}(T)\cap \mathfrak{h}_k)$, where $\mathfrak{h}_k$ is the fibre $\mathfrak{h}$ counted in the $k$-th position with respect to the sum  \eqref{eq:Hxispace}, and each $\mathfrak{h}_k$ is a reducing subspace for $T$.
A convenient shorthand for the above expression for $\mathcal{D}(T)$ is
\begin{equation}\label{eq:shorthandDTDTk}
 \mathcal{D}(T)\;=\;\op_{k\in\mathbb{Z}}\mathcal{D}(T(k))\,.
\end{equation}
As commented already, we write `$\boxplus$' instead of `$\oplus$' to denote that the infinite orthogonal sum involves now non-closed subspaces of $\cH$.

\begin{remark}\label{rem:Halphanotsum}
It is crucial to observe that $\mathscr{H}_\alpha$ is \emph{not} decomposable as $\bigoplus_{k\in\mathbb{Z}} A_\alpha(k)\,$ in the sense of formula \eqref{eq:Tdirectint}, and in fact
\begin{equation}\label{eq:Halphanotsum}
 \mathscr{H}_\alpha\;\varsubsetneq\;\bigoplus_{k\in\mathbb{Z}} A_\alpha(k)\,.
\end{equation}
Indeed, 
%
%
%
%
%
%
as seen in \eqref{eq:forcondii},
\[
 \sum_{k}\|A_\alpha(k)\psi_k\|_{\mathfrak{h}}^2\;=\; \textstyle\Big\| \Big(-\frac{\partial^2}{\partial x^2}- |x|^{2\alpha}\frac{\partial^2}{\partial y^2}+\frac{\,\alpha(2+\alpha)\,}{4x^2}\Big)\phi  \Big\|_{L^2(\mathbb{R}^\pm_x\times\mathbb{S}^1_y)}^2\,,
\]
where $\psi=\mathcal{F}_2\phi$,
the finiteness of which is guaranteed by $\phi\in C^\infty_c(\mathbb{R}^\pm_x\times\mathbb{S}^1_y)$ in the case when $\psi\in\mathcal{D}(\mathscr{H}_\alpha)$, but of course is also guaranteed by a much larger class of $\phi$'s that are still smooth and compactly supported in $x$, without being smooth in $y$ -- thus corresponding to $\psi$'s that do not belong to $\mathcal{D}(\mathscr{H}_\alpha)$. This is completely analogous to what we observed in \cite[Remark 2.2]{GMP-Grushin-2018}.
\end{remark}

Most relevantly for our purposes, the closure and the adjoint pass through the orthogonal sum of operators.

\begin{lemma}\label{lem:sumstar-sumclosure}
 If $T=\bigoplus_{k\in\mathbb{Z}} \,T(k)$, then
 \begin{eqnarray}
   T^*\!\!&=&\!\!\bigoplus_{k\in\mathbb{Z}} \,T(k)^* \label{eq:sumstar}\\
   \overline{T}\!\!&=&\!\!\bigoplus_{k\in\mathbb{Z}} \,\overline{T(k)}\,, \label{eq:sumclosure}
 \end{eqnarray}
 where the symbol of operator closure and adjoint clearly refers to the corresponding Hilbert spaces where the considered operators act on.
 Moreover,
 \begin{equation}\label{eq:sumkernel}
  \ker T^*\;=\;\bigoplus_{k\in\mathbb{Z}}\,\ker T(k)^*\,.
 \end{equation}
\end{lemma}

\begin{proof}
 Let $\psi\in\mathcal{D}(T^*)$: then there exists $\eta\in\cH$ such that
  \[
  \sum_{k\in\mathbb{Z}}\langle\eta_k,\xi_k\rangle_{\mathfrak{h}}\;=\;\langle \eta,\xi\rangle_{\cH}\;=\;\langle\psi,T\,\xi\rangle_{\cH}\;=\;\sum_{k\in\mathbb{Z}}\langle\psi_k,T(k)\,\xi_k\rangle_{\mathfrak{h}}\qquad\forall \xi\in\mathcal{D}(T)\,.
 \]
  By localising $\xi$ separately in each fibre $\mathfrak{h}_k$ one then deduces that for each $k\in\mathbb{Z}$ $\psi_k\in\mathcal{D}(T(k)^*)$ and $\eta_k=T(k)^*\psi_k$, whence also $\sum_{k\in\mathbb{Z}}\big\|T(k)^*\psi_k\big\|_{\mathfrak{h}}^2=\|\eta\|^2_\cH<+\infty$. This means precisely that $\psi\in\mathcal{D}(\bigoplus_{k\in\mathbb{Z}}T(k)^*)$ and $T^*\psi=(T(k)^*\psi_k)_{k\in\mathbb{Z}}=(\bigoplus_{k\in\mathbb{Z}} T(k)^*)\psi$, i.e.,  $T^*\subset\bigoplus_{k\in\mathbb{Z}} \,T(k)^*$.

  Conversely, if $\psi\in\mathcal{D}(\bigoplus_{k\in\mathbb{Z}}T(k)^*)$, then for each $k\in\mathbb{Z}$ one has $\langle T(k)^*\psi_k,\xi_k\rangle_{\mathfrak{h}}=\langle\psi_k,T(k)\,\xi_k\rangle_{\mathfrak{h}}$ $\forall\xi_k\in\mathcal{D}(T(k))$ and $\sum_{k\in\mathbb{Z}}\|T^*(k)\psi_k\|_{\mathfrak{h}}^2<+\infty$. Setting $\eta_k:=T^*(k)\psi_k$ and $\eta:=(\eta_k)_{k\in\mathbb{Z}}$ one then has that $\eta\in\cH$ and
  \[
  \langle \eta,\xi\rangle_{\cH}\;=\;\sum_{k\in\mathbb{Z}}\langle\eta_k,\xi_k\rangle_{\mathfrak{h}}\;=\;\sum_{k\in\mathbb{Z}}\langle\psi_k,T(k)\xi_k\rangle_{\mathfrak{h}}\;=\;\langle\psi,T\,\xi\rangle_{\cH}\,\qquad\forall \xi\in\mathcal{D}(T)\,.
 \]
  This means that $\psi\in\mathcal{D}(T^*)$ and $T^*\psi=\eta=(T(k)^*\psi_k)_{k\in\mathbb{Z}}=(\bigoplus_{k\in\mathbb{Z}} \,T(k)^*)\psi$, i.e.,  $T^*\supset\bigoplus_{k\in\mathbb{Z}} \,T(k)^*$.
 
  Identity \eqref{eq:sumstar} is thus established, and \eqref{eq:sumclosure} follows from applying  \eqref{eq:sumstar} to the operator $T^*$ instead of $T$. Identity \eqref{eq:sumkernel} is another straightforward consequence of \eqref{eq:sumstar}.
  \end{proof}

 Now, although $\mathscr{H}_\alpha\varsubsetneq\bigoplus_{k\in\mathbb{Z}} A_\alpha(k)$ (Remark \ref{rem:Halphanotsum}), the two operators have actually the same adjoint and the same closure.

\begin{lemma}\label{lem:Halphaadj-decomposable}
 One has
\begin{equation}\label{eq:Halphaadj-decomposable}
 \mathscr{H}_\alpha^*\;=\;\bigoplus_{k\in\mathbb{Z}} \,A_\alpha(k)^*
\end{equation}
and 
\begin{equation}\label{eq:Halphaclosure-decomposable}
 \overline{\mathscr{H}_\alpha}\;=\;\bigoplus_{k\in\mathbb{Z}} \,\overline{A_\alpha(k)}\,,
\end{equation}
i.e.,
\begin{equation}
   \begin{split}
  \mathcal{D}(\mathscr{H}_\alpha^*)\;&:=\;\left\{\psi\equiv(\psi_k)_{k\in\mathbb{Z}}\in\cH\,\left|\! 
  \begin{array}{l}
   \mathrm{(i)}\quad\psi_k\in\mathcal{D}(A_\alpha(k)^*) \;\;\forall k\in\mathbb{Z} \\
   \mathrm{(ii)}\,\displaystyle\sum_{k\in\mathbb{Z}}\big\|A_\alpha(k)^*\psi_k\big\|_{\mathfrak{h}}^2<+\infty
  \end{array}
  \!\!\right.\right\},  \\
  \mathscr{H}_\alpha^*\psi\;&:=\;\big(A_\alpha(k)^*\,\psi_k\big)_{k\in\mathbb{Z}}
 \end{split}
\end{equation}
and 
\begin{equation}
   \begin{split}
  \mathcal{D}(\overline{\mathscr{H}_\alpha})\;&:=\;\left\{\psi\equiv(\psi_k)_{k\in\mathbb{Z}}\in\cH\,\left|\! 
  \begin{array}{l}
   \mathrm{(i)}\quad\psi_k\in\mathcal{D}(\overline{A_\alpha(k)}) \;\;\forall k\in\mathbb{Z} \\
   \mathrm{(ii)}\,\displaystyle\sum_{k\in\mathbb{Z}}\big\|\overline{A_\alpha(k)}\psi_k\big\|_{\mathfrak{h}}^2<+\infty
  \end{array}
  \!\!\right.\right\},  \\
  \overline{\mathscr{H}_\alpha}\psi\;&:=\;\big(\overline{A_\alpha(k)}\,\psi_k\big)_{k\in\mathbb{Z}}\,.
 \end{split}
\end{equation}
Analogously,
 \begin{equation}
 (\mathscr{H}_\alpha^\pm)^*\;=\;\bigoplus_{k\in\mathbb{Z}} \,A_\alpha^\pm(k)^*\,,\qquad \overline{\mathscr{H}_\alpha^\pm}\;=\;\bigoplus_{k\in\mathbb{Z}} \,\overline{A_\alpha^\pm(k)}\,.
\end{equation}
Moreover,
 \begin{equation}\label{eq:Halphaadj-sumkernel}
  \ker \mathscr{H}_\alpha^*\;=\;\bigoplus_{k\in\mathbb{Z}}\,\ker A_\alpha(k)^*\,.
 \end{equation}
\end{lemma}

\begin{proof}
 On the one hand, $\mathscr{H}_\alpha^*\supset (\bigoplus_{k\in\mathbb{Z}} A_\alpha(k))^*=\bigoplus_{k\in\mathbb{Z}} A_\alpha(k)^*$ (owing to \eqref{eq:Halphanotsum} and \eqref{eq:sumstar} above).

 On the other hand, one proves the opposite inclusion, namely $\mathscr{H}_\alpha^*\subset \bigoplus_{k\in\mathbb{Z}} A_\alpha(k)^*$, following the very same argument used for the proof of $T^*\subset\bigoplus_{k\in\mathbb{Z}} \,T(k)^*$ in Lemma \ref{lem:sumstar-sumclosure}. 
 This is possible because for $\xi\in\mathcal{D}(\mathscr{H}_\alpha)$, one has $\xi_k\in C^\infty_c(\mathbb{R}\setminus\{0\})=\mathcal{D}(A_\alpha(k))$.

 Thus, explicitly, if $\psi\in\mathcal{D}(\mathscr{H}_\alpha^*)$, then there exists $\eta\in\cH$ such that
 \[
  \sum_{k\in\mathbb{Z}}\langle\eta_k,\xi_k\rangle_{\mathfrak{h}}\;=\;\langle \eta,\xi\rangle_{\cH}\;=\;\langle\psi,\mathscr{H}_\alpha\,\xi\rangle_{\cH}\;=\;\sum_{k\in\mathbb{Z}}\langle\psi_k,A_\alpha(k)\,\xi_k\rangle_{\mathfrak{h}}\qquad\forall \xi\in\mathcal{D}(\mathscr{H}_\alpha)\,.
 \]
 By localising $\xi$ separately in each fibre $\mathfrak{h}_k$ one then deduces that for each $k\in\mathbb{Z}$ $\psi_k\in\mathcal{D}(A_\alpha(k)^*)$ and $\eta_k=A_\alpha(k)^*\psi_k$, whence also $\sum_{k\in\mathbb{Z}}\big\|A_\alpha(k)^*\psi_k\big\|_{\mathfrak{h}}^2=\|\eta\|^2_\cH<+\infty$. This means that $\psi\in\mathcal{D}(\bigoplus_{k\in\mathbb{Z}}A_\alpha(k)^*)$ and $\mathscr{H}_\alpha^*\psi=(A_\alpha(k)^*\psi_k)_{k\in\mathbb{Z}}=(\bigoplus_{k\in\mathbb{Z}} A_\alpha(k)^*)\psi$.

 Thus, \eqref{eq:Halphaadj-decomposable} is proved. Applying \eqref{eq:sumstar} to \eqref{eq:Halphaadj-decomposable} then yields \eqref{eq:Halphaclosure-decomposable}. 
\end{proof}

\subsection{Momentum-fibred extensions. Local and non-local extensions.}\label{sec:momentum-fibred-ext}~

The technical point that is going to be crucial for us in studying the self-adjoint extensions of $\mathscr{H}_\alpha^\pm$ and $\mathscr{H}_\alpha$ is the following.

\begin{proposition}\label{prop:BextendsHalpha}
 Let $\{B(k)\,|\,k\in\mathbb{Z}\}$ be a collection of operators on the fibre space $\mathfrak{h}$ (resp., $\mathfrak{h}^\pm$) such that, for each $k$, $B(k)$ is a self-adjoint extension of $A_\alpha(k)$ (resp., $A_\alpha^\pm(k)$), and let
 \begin{equation}\label{eq:B_direct_integral}
 B\;=\;\bigoplus_{k\in\mathbb{Z}} \,B(k)\,.
\end{equation}
 Then $B$ is a self-adjoint extension of $\mathscr{H}_\alpha$ (resp., $\mathscr{H}_\alpha^\pm$).
\end{proposition}

The proof goes through reasonings that are somewhat standard, but for completeness and later discussion we sketch it here.

\begin{proof}[Proof of Proposition \ref{prop:BextendsHalpha}]
$B$ is an actual extension of $\mathscr{H}_\alpha$, because 
\[
 \mathscr{H}_\alpha\;\subset\;\bigoplus_{k\in\mathbb{Z}} \,A_\alpha(k)\;\subset\;\bigoplus_{k\in\mathbb{Z}} \,B(k)\,.
\]

It is straightforward to see that $B$ is symmetric, so in order to establish the self-adjointness of $B$ one only needs to prove that $\mathrm{ran}(B\pm\ii\mathbbm{1})=\cH$.

For generic $\eta\equiv(\eta_k)_{k\in\mathbb{Z}}\in\cH$ let us then set $\psi_k:=(B(k)+\ii\mathbbm{1})^{-1}\eta_k$ $\forall k\in\mathbb{Z}$. By construction $\psi_k\in\mathcal{D}(B(k))$,  $\|\psi_k\|_{\mathfrak{h}}\leqslant\|\eta_k\|_{\mathfrak{h}}$, and $\|B(k)\psi_k\|_{\mathfrak{h}}\leqslant\|\eta_k\|_{\mathfrak{h}}$, whence also $\sum_{k\in\mathbb{Z}}\|\psi_k\|_{\mathfrak{h}}^2<+\infty$ and  $\sum_{k\in\mathbb{Z}}\|B(k)\psi_k\|_{\mathfrak{h}}^2<+\infty$. Therefore, $\psi\equiv(\psi_k)_{k\in\mathbb{Z}}\in\mathcal{D}(B)$. Moreover, $(B+\ii\mathbbm{1})\psi=((B(k)+\ii\mathbbm{1})\psi_k)_{k\in\mathbb{Z}}=(\eta_k)_{k\in\mathbb{Z}}=\eta$. This proves that $\mathrm{ran}(B +\ii\mathbbm{1})=\cH$. Analogously, $\mathrm{ran}(B -\ii\mathbbm{1})=\cH$.
\end{proof}

Proposition \ref{prop:BextendsHalpha} provides a mechanism for constructing self-adjoint operators $B$ of the form \eqref{eq:B_direct_integral} by re-assembling, fibre by fibre in the momentum number $k$ conjugate to $y$, self-adjoint extensions of the fibre operators $A_\alpha(k)$;  by further exploiting the canonical unitary equivalence
\begin{equation}
 B\;\stackrel{\cong}{\longmapsto}\; \mathcal{F}_2^{-1} U_\alpha^{-1}\;B\:\mathcal{F}_2 U_\alpha
\end{equation}
this yields actual self-adjoint extensions of $H_\alpha$. With self-explanatory meaning, we shall refer to such extensions as `\emph{momentum-fibred extensions}', or simply `\emph{fibred extensions}'.

Thus, fibred extensions have the distinctive feature of being characterised, in position-momentum coordinates $(x,k)$, by boundary conditions on the elements $\psi$ of their domain which connect the behaviour of \emph{each} mode $\psi_k(x)$ as $x\to 0^+$ and $x\to 0^-$, with no crossing conditions between different modes. In other words, such extensions are \emph{local} in momentum -- which is another way we shall refer to them in the following -- whence their primary physical and conceptual relevance.

Evidently, $\mathscr{H}_\alpha$ (and hence $H_\alpha$) admits plenty of extensions that are \emph{non-local} in momentum, namely with boundary condition as $x\to 0^\pm$ that mixes different $k$-modes. 


It is also clear that a generic fibred extension of $\mathscr{H}_\alpha$ may or may not be reduced into a `left' and `right' component by the Hilbert space direct sum \eqref{eq:Hxispace}, whereas $\mathscr{H}_\alpha$ itself certainly is. Indeed, at the level of each fibre, the extension $B(k)$ may or may not be reduced by the sum $\mathfrak{h}=\mathfrak{h}^-\oplus\mathfrak{h}^+$ as is instead $A_\alpha(k)$ by construction (see \eqref{eq:Axibilateral} above).

In fact, the decoupling between left and right half-cylinder may hold for \emph{all} modes $k\in\mathbb{Z}$ or only for some sub-domains of $k$. In the former case, the resulting extension of $\mathscr{H}_\alpha$ is in fact a mere `juxtaposition' of two separate extensions for $\mathscr{H}_\alpha^{\pm}$ in the left/right half-cylinder.

We shall apply the above formalism and the latter considerations in Section \ref{sec:genextscrHa}, where the actual classification of the self-adjoint extensions of $\mathscr{H}_\alpha$ is discussed.

\section{Extensions of the differential operator on each half-fibre}\label{sec:fibre-extensions}

In this Section and in the next one we classify the self-adjoint extensions of the right-fibre operators $A^+_\alpha(k)$ defined in \eqref{eq:Axi} for $\alpha\in[0,1)$ and $k\in\mathbb{Z}$, with respect to the fibre Hilbert space $L^2(\mathbb{R}^+,\ud x)$.

For simplicity of notation, we shall temporarily drop the superscript `$+$' and simply write $A_\alpha(k)$ for $A^+_\alpha(k)$, and $\langle \cdot,\cdot\rangle_{L^2}$ and $\|\cdot\|_{L^2}$ for scalar products and norms taken in $L^2(\mathbb{R}^+)$, with analogous notation for the Sobolev norms. Obviously, the whole discussion can be repeated verbatim for $A^-_\alpha(k)$ in $L^2(\mathbb{R}^-)$ instead of $A^+_\alpha(k)$, with completely analogous conclusions.

As already recalled from \cite[Corollary 3.8]{GMP-Grushin-2018}, for each fixed $\alpha\in[0,1)$ and $ k\in\mathbb{Z}$ $A_\alpha(k)$ has deficiency index 1, hence admits a one-(real-)parameter family of self-adjoint extensions. We reconstruct and classify this family by means of the Kre\u{\i}n-Vi\v{s}ik-Birman extension theory \cite{GMO-KVB2017}.

When $\alpha=0$ the operator $A_\alpha(k)$ is the minimally defined, shifted Laplacian $-\frac{\ud^2}{\ud x^2}+ k^2$ on $L^2(\mathbb{R}^+)$: the family of its self-adjoint realisations is well-known (see, e.g., \cite{GTV-2012,DM-2015-halfline}) and the extension formulas that we find for $\alpha\in(0,1)$ take indeed the usual form for the extensions of the Laplacian in the limit $\alpha\downarrow 0$.

Let us observe preliminarily that not only is $A_\alpha(k)$ non-negative, but also in particular it has strictly positive lower bound for every non-zero $k$. Indeed,
\[
 \min_{x\in\mathbb{R}^+}\Big( k^2 x^{2\alpha}+\frac{\,\alpha(2+\alpha)\,}{4x^2}\Big)\;=\;(1+\alpha)\big(\textstyle{\frac{2+\alpha}{4}}\big)^{\frac{\alpha}{1+\alpha}}|k|^{\frac{2}{1+\alpha}}\;=:\;M_{\alpha,k}\,,
\]
whence
\begin{equation}\label{eq:Axibottom}
 \langle h,A_\alpha(k)h\rangle_{L^2}\;\geqslant\;M_{\alpha,k}\|h\|_{L^2}^2\qquad\forall h\in \mathcal{D}(A_\alpha(k))\,.
\end{equation}
Instead, when $k=0$ it is straightforward to see that
\begin{equation}\label{eq:Axibottom-zero}
 \inf_{h\in\mathcal{D}(A_\alpha(0))\setminus\{0\}}\frac{\langle h,A_\alpha(0)h\rangle_{L^2}}{\|h\|_{L^2}^2}\;=\;0\,.
\end{equation}

Therefore, as long as $k\neq 0$, owing to \eqref{eq:Axibottom} we can apply the Kre\u{\i}n-Vi\v{s}ik-Birman extension theory directly in the setting of a \emph{strictly positive operator}. This programme will be completed in the present Section. The special case $k=0$ is deferred to the next Section, where we highlight the main steps that need be modified -- starting from the auxiliary shifted operator $A_\alpha(0)+\mathbbm{1}$, which has again strictly positive bottom.
%
%

For convenience of notation let us set
\begin{equation}
 C_\alpha\;:=\;\frac{\,\alpha(2+\alpha)}{4}\,.
\end{equation}
Then $C_\alpha\in[0,\frac{3}{4})$. Let us also refer to
\begin{equation}\label{eq:Saxi}
 S_{\alpha,k}\;:=\; -\frac{\ud^2}{\ud x^2}+ k^2 x^{2\alpha}+\frac{C_\alpha}{x^2}
\end{equation}
as the differential operator (with no domain specification) representing the action of both $A_\alpha(k)$ and $A_\alpha(k)^*$, where the derivative is classical or weak depending on the context.

Clearly, in order to characterise the operator closure $\overline{A_\alpha(k)}$ of $A_\alpha(k)$, its Friedrichs extension $A_{\alpha,F}(k)$, as well as any other self-adjoint extension, it suffices to indicate the corresponding domains, for all such operators are restrictions of the adjoint $A_\alpha(k)^*$ and as such they all act with the action of the differential operator $S_{\alpha,k}$.

Here is the main result of this Section.

\begin{theorem}\label{thm:fibre-thm}
 Let $\alpha\in[0,1)$ and $ k\in\mathbb{Z}\!\setminus\!\{0\}$.
 \begin{itemize}
  \item[(i)] The operator closure of $A_\alpha(k)$ has domain
  \begin{equation}\label{eq:thm_Aclosure}
   \mathcal{D}(\overline{A_\alpha(k)})\;=\; H^2_0(\mathbb{R}^+)\cap L^2(\mathbb{R}^+,\langle x\rangle^{4\alpha}\,\ud x)\,.
  \end{equation}
  \item[(ii)] The adjoint of $A_\alpha(k)$ has domain
  \begin{equation}
   \begin{split}
    \mathcal{D}(A_\alpha(k)^*)\;&=\;\left\{\!\!
  \begin{array}{c}
   g\in L^2(\mathbb{R}^+)\;\;\textrm{such that} \\
   \big(-\frac{\ud^2}{\ud x^2}+ k^2 x^{2\alpha}+\frac{\,\alpha(2+\alpha)\,}{4x^2}\big)g\in L^2(\mathbb{R}^+)
  \end{array}
  \!\!\right\} \\
   &=\;\mathcal{D}(\overline{A_\alpha(k)})\dotplus\mathrm{span}\{\Psi_{\alpha,k}\}\dotplus\mathrm{span}\{\Phi_{\alpha,k}\}\,,
   \end{split} 
  \end{equation}
   where $\Phi_{\alpha,k}$ and $\Psi_{\alpha,k}$ are two smooth functions on $\mathbb{R}^+$ explicitly defined, in terms of modified Bessel functions, respectively by formula \eqref{eq:Phi_and_F} and by formulas \eqref{eq:value_of_W}, \eqref{eq:Green}, \eqref{eq:newRGalphaforus}, and \eqref{eq:defPsi} below. Moreover,
   \begin{equation}\label{eq:kerAxistar-in-thm}
  \ker A_\alpha(k)^*\;=\;\mathrm{span}\{\Phi_{\alpha,k}\}\,.
 \end{equation}
   \item[(iii)] The Friedrichs extension of $A_\alpha(k)$ has operator domain
   \begin{equation}\label{eq:thmAFoperator}
    \begin{split}
     \mathcal{D}(A_{\alpha,F}(k))\;&=\;\big\{g\in\mathcal{D}(A_\alpha(k)^*)\,\big|\,g(x)\,\stackrel{x\downarrow 0}{=}\,g_1x^{1+\frac{\alpha}{2}}+o(x^{\frac{3}{2}})\,,\; g_1\in\mathbb{C}\big\} \\
     &=\;\mathcal{D}(\overline{A_\alpha(k)})\dotplus\mathrm{span}\{\Psi_{\alpha,k}\}
    \end{split}
   \end{equation}
    and form domain
   \begin{equation}\label{eq:thmAFform}
    \mathcal{D}[A_{\alpha,F}(k)]\;=\;H^1_0(\mathbb{R}^+)\cap L^2(\mathbb{R}^+,\langle x\rangle^{2\alpha}\,\ud x)\,.
   \end{equation}
    Moreover, $A_{\alpha,F}(k)$ is the only self-adjoint extension of $A_\alpha(k)$ whose operator domain is entirely contained in $\mathcal{D}(x^{-1})$, namely the self-adjointness domain of the operator of multiplication by $x^{-1}$. 
   \item[(iv)]  The self-adjoint extensions of $A_\alpha(k)$ in $L^2(\mathbb{R}^+)$ form the family
   \[
\{ A_\alpha^{[\gamma]}(k)\,|\,\gamma\in\mathbb{R}\cup\{\infty\}\}\,.    
   \]
 The extension with $\gamma=\infty$ is the Friedrichs extension, and for generic $\gamma\in\mathbb{R}$ one has
 \begin{equation}
  \mathcal{D}(A_\alpha^{[\gamma]}(k))\,=\,\big\{g\in\mathcal{D}(A_\alpha(k)^*)\,\big|\,g(x)\,\stackrel{x\downarrow 0}{=}\,g_0 x^{-\frac{\alpha}{2}}+\gamma g_0x^{1+\frac{\alpha}{2}}+o(x^{\frac{3}{2}})\,,\; g_0\in\mathbb{C}\big\}\,.\!\!\!\!\!\!\!\!\!\!
  \end{equation}
 \end{itemize}
\end{theorem}

Concerning the spaces indicated in \eqref{eq:thm_Aclosure} and \eqref{eq:thmAFform}, let us recall that by definition and by a standard Sobolev embedding
\begin{equation}\label{eq:H10}
 \begin{split}
  H^1_0(\mathbb{R}^+)\;&=\;\overline{C^\infty_c(\mathbb{R}^+)}^{\|\,\|_{H^1}} \\
  &=\;\{\varphi\in L^2(\mathbb{R}^+)\,|\,\varphi'\in L^2(\mathbb{R}^+)\textrm{ and }\varphi(0)=0\}\,,
 \end{split}
\end{equation}
and
\begin{equation}\label{eq:H20}
 \begin{split}
  H^2_0(\mathbb{R}^+)\;&=\;\overline{C^\infty_c(\mathbb{R}^+)}^{\|\,\|_{H^2}} \\
  &=\;\{\varphi\in L^2(\mathbb{R}^+)\,|\,\varphi',\varphi''\in L^2(\mathbb{R}^+)\textrm{ and }\varphi(0)=\varphi'(0)=0\}\,.
 \end{split}
\end{equation}

The proof of Theorem \ref{thm:fibre-thm} requires an amount of preparatory material that is presented in Sections \ref{subsec:homog_problem}-\ref{subsec:distinguished} and will be finally completed in Section \ref{subsec:proof_of_fibrethm}.

\subsection{Homogeneous differential problem: kernel of $A_\alpha(k)^*$}\label{subsec:homog_problem}~

Let us characterise the kernel of the adjoint $A_\alpha(k)^*$.

To this aim, we make use of the modified Bessel functions $K_\nu$ and $I_\nu$ \cite[Sect.~9.6]{Abramowitz-Stegun-1964}, that are two explicit, linearly independent, smooth solutions to the modified Bessel equation
\begin{equation}\label{eq:modifiedBessEq}
z^2 w''+z w'-(z^2+\nu^2) w \;=\; 0\,,\qquad z\in\mathbb{R}^+
\end{equation}
with parameter $\nu\in\mathbb{C}$. In particular, in terms of $K_{\frac{1}{2}}$ and $I_{\frac{1}{2}}$ we define the functions
\begin{equation}\label{eq:Phi_and_F}
\begin{split}
\Phi_{\alpha,k}(x)\;&:=\; \sqrt{x}\,K_{\frac{1}{2}}\big({\textstyle\frac{|k|}{1+\alpha}}\,x^{1+\alpha}\big)\,,\\
F_{\alpha,k}(x)\;&:=\; \sqrt{x}\,I_{\frac{1}{2}}\big({\textstyle\frac{|k|}{1+\alpha}}\,x^{1+\alpha}\big)\,.
\end{split}
\end{equation}
Explicitly, as can be deduced from \cite[Eq.~(10.2.4), (10.2.13), and (10.2.14)]{Abramowitz-Stegun-1964},
\begin{equation}\label{eq:Phi_and_F_explicit}
\begin{split}
\Phi_{\alpha,k}(x)\;&:=\; {\textstyle\sqrt{\frac{\pi(1+\alpha)}{2|k|}}}\,x^{-\frac{\alpha}{2}}\,e^{-\frac{|k|}{1+\alpha}x^{1+\alpha}}\,, \\
F_{\alpha,k}(x)\;&:=\; {\textstyle\sqrt{\frac{2(1+\alpha)}{\pi|k|}}}\,x^{-\frac{\alpha}{2}}\,\sinh\big({\textstyle\frac{|k|}{1+\alpha}}\,x^{1+\alpha}\big)\,.
\end{split}
\end{equation}
From \eqref{eq:Phi_and_F_explicit} we obtain the short-distance asymptotics
\begin{equation}\label{eq:Asymtotics_0}
 \begin{split}
 \Phi_{\alpha,k}(x)\;&\stackrel{x\downarrow 0}{=}\; {\textstyle\sqrt{\frac{\pi (1+\alpha)}{2 |k|}}}\, x^{-\frac{\alpha}{2}} -{\textstyle\sqrt{\frac{\pi \, |k|}{2(1+\alpha)}}}\, x^{1+\frac{\alpha}{2}}+{\textstyle\sqrt{\frac{\pi |k|^3}{8(1+\alpha)^3}}}\, x^{2+\frac{3}{2}\alpha}+O(x^{3+\frac{5}{2}\alpha}) \\
 F_{\alpha,k}(x)\;&\stackrel{x\downarrow 0}{=}\;{\textstyle\sqrt{\frac{2 |k|}{(1+\alpha) \pi}}}\, x^{1+\frac{\alpha}{2} }+O(x^{3+\frac{5}{2}\alpha})\,,
 \end{split}
\end{equation}
and the large-distance asymptotics
\begin{equation}\label{eq:Asymtotics_Inf}
\begin{split}
 \Phi_{\alpha,k}(x)\;&\stackrel{x\to +\infty}{=}\;{\textstyle\sqrt{\frac{\pi (1+\alpha)}{2 |k|}}}\, e^{-\frac{|k| x^{1+\alpha}}{1+\alpha}} x^{-\frac{\alpha}{2}}(1+O(x^{-(1+\alpha)}))\,, \\
 F_{\alpha,k}(x)\;&\stackrel{x\to +\infty}{=}\;{\textstyle\sqrt{\frac{ 1+\alpha}{2 \pi |k|}}}\, e^{\frac{|k| x^{1+\alpha}}{1+\alpha}} x^{-\frac{\alpha}{2}}(1+O(x^{-(1+\alpha)}))\,,
\end{split}
\end{equation}
as well as the norm
\begin{equation}\label{eq:Phinorm}
\| \Phi_{\alpha,k} \|_{L^2}^2\;=\;\pi\,(1+\alpha)^{\frac{1-\alpha}{1+\alpha}}\,\Gamma\big({\textstyle\frac{1-\alpha}{1+\alpha}}\big)\, (2|k|)^{-\frac{2}{1+\alpha}}\,.
\end{equation}

\begin{lemma}\label{lem:kerAxistar}
 Let $\alpha\in(0,1)$ and $ k\in\mathbb{Z} \!\setminus\! \{0\}$. One has
 \begin{equation}\label{eq:kerAxistar}
  \ker A_\alpha(k)^*\;=\;\mathrm{span}\{\Phi_{\alpha,k}\}\,.
 \end{equation}
\end{lemma}

\begin{proof}
 Owing to \eqref{eq:Afstar}, a generic $h\in\ker A_\alpha(k)^*$ belongs to $L^2(\mathbb{R}^+)$ and satisfies
 \[\tag{i}
  S_{\alpha,k}\,h\;=\;-h''+ k^2 x^{2\alpha}h+C_\alpha\,x^{-2}h\;=\;0\,.
 \]
 Setting
 \[\tag{ii}
  z\;:=\;\frac{|k|}{1+\alpha}\,x^{1+\alpha}\,,\qquad w(z)\;:=\;\frac{h(x)}{\sqrt{x}}\,,\qquad\nu\;:=\;\frac{\sqrt{1+4C_\alpha}}{2(1+\alpha)}\;=\;\frac{1}{2}\,,
 \]
 the ordinary differential equation (i) takes precisely the form \eqref{eq:modifiedBessEq} with the considered $\nu$. The two linearly independent solutions $K_{\frac{1}{2}}$ and $I_{\frac{1}{2}}$ to \eqref{eq:modifiedBessEq} yield, through the transformation (ii) above, the two linearly independent solutions \eqref{eq:Phi_and_F} to (i). In fact, only $\Phi_{\alpha,k}$ is square-integrable, whereas $F_{\alpha,k}$ fails to be so at infinity (as is seen from \eqref{eq:Phinorm}-\eqref{eq:Asymtotics_Inf}). Formula \eqref{eq:kerAxistar} is thus proved.
\end{proof}

\subsection{Non-homogeneous inverse differential problem}\label{sec:non-homogeneous_problem}~

Let us now focus on the non-homogeneous problem
\begin{equation}
 S_{\alpha,k}\,u\;=\;g
\end{equation}
in the unknown $u$ for given $g$. With respect to the fundamental system $\{F_{\alpha,k},\Phi_{\alpha,k}\}$ given by \eqref{eq:Phi_and_F}, of solutions for the problem $S_{\alpha,k}\,u=0$, the general solution is given by
\begin{equation}\label{eq:ODE_general_sol}
 u\;=\;c_1 F_{\alpha,k} + c_2 \Phi_{\alpha,k} + u_{\mathrm{part}}
\end{equation}
for $c_1,c_2\in\mathbb{C}$ and some particular solution $u_{\mathrm{part}}$, i.e., $S_{\alpha,k}\,u_{\mathrm{part}}=g$.

The Wronskian
\begin{equation}
W(\Phi_{\alpha,k},F_{\alpha,k})(r)\;:=\;\det \begin{pmatrix}
\Phi_{\alpha,k}(r) & F_{\alpha,k}(r) \\
\Phi_{\alpha,k}'(r) & F_{\alpha,k}'(r)
\end{pmatrix}
\end{equation}
relative to the fundamental system $\{F_{\alpha,k},\Phi_{\alpha,k}\}$ is clearly constant in $r$, since it is evaluated on solutions to the homogeneous differential problem, with a value that can be computed by means of the asymptotics \eqref{eq:Asymtotics_0} or \eqref{eq:Asymtotics_Inf} and amounts to
\begin{equation}\label{eq:value_of_W}
W(\Phi_{\alpha,k},F_{\alpha,k})\;=\;1+\alpha\;=:\; W\,.
\end{equation}

A standard application of the method of variation of constants \cite[Section 2.4]{Wasow_asympt_expansions} shows that we can take $u_{\mathrm{part}}$ to be
\begin{equation}\label{eq:upart}
u_{\text{part}}(r) \;=\; \int_0^{+\infty} G_{\alpha,k}(r,\rho) g(\rho) \, \ud \rho\,,
\end{equation}
where 
\begin{equation}\label{eq:Green}
G_{\alpha,k}(r,\rho) \, := \,\frac{1}{W} \begin{cases}
\Phi_{\alpha,k}(r) F_{\alpha,k}(\rho)\,, \qquad \text{if }0 < \rho < r\,,\\
F_{\alpha,k}(r) \Phi_{\alpha,k}(\rho)\,, \qquad \text{if } 0 < r < \rho \,.
\end{cases}
\end{equation}

For $a\in\mathbb{R}$ and $ k\in\mathbb{Z} \!\setminus\! \{0\}$, let $R_{G_{\alpha,k}}^{(a)}$ be the integral operator acting on functions $g$ on $\mathbb{R}^+$ as
\begin{equation}
 \begin{split}
  \big(R_{G_{\alpha,k}}^{(a)}g\big)(x)\;&:=\;\int_0^{+\infty} \mathscr{G}_{\alpha,k}^{(a)}(x,\rho)\,g(\rho)\,\ud \rho\,, \\
  \mathscr{G}_{\alpha,k}^{(a)}(x,\rho)\;&:=\;x^a\,k^2\, G_{\alpha,k}(x,\rho)\,,
 \end{split}
\end{equation}
and let
\begin{equation}\label{eq:RGalphaforus}
 R_{G_{\alpha,k}}\;:=\;|k|^{-2}\,R_{G_{\alpha,k}}^{(0)}\,,
\end{equation}
whence
\begin{equation}\label{eq:newRGalphaforus}
 ( R_{G_{\alpha,k}}g)(x)\;=\;\;\int_0^{+\infty}G_{\alpha,k}(r,\rho)\,g(\rho)\,\ud \rho\,.
\end{equation}

The following property holds.

\begin{lemma}\label{lem:RGbddsa}
Let $\alpha\in(0,1)$ and $ k\in\mathbb{Z} \!\setminus\! \{0\}$.
\begin{itemize}
 \item[(i)] For each $a\in(-\frac{1-\alpha}{2},2\alpha]$, $R_{G_{\alpha,k}}^{(a)}$ can be realised as an everywhere defined, bounded operator on $L^2(\mathbb{R}^+, \ud x)$, which is also self-adjoint if $a=0$.
 \item[(ii)] When $a=2\alpha$, the operator $R_{G_{\alpha,k}}^{(2\alpha)}$ is bounded uniformly in $k$.
\end{itemize}
\end{lemma}

\begin{remark}
 For the purposes of the present Section, the thesis of Lemma \ref{lem:RGbddsa} (and therefore its proof) is overabundant, in that we do not need here the \emph{uniformity} in $k$ of the norm of $R_{G_{\alpha,k}}^{(2\alpha)}$. Instead, this information will be crucial in Subsect.~\ref{sec:control-of-tildephi}. In fact, it is to prove the boundedness claim (i) in a form that implies the $k$-uniformity of claim (ii) that we have to go through a somewhat lengthy proof.
\end{remark}

For the proof of Lemma \ref{lem:RGbddsa} it is convenient to re-write, by means of \eqref{eq:Phi_and_F_explicit} and \eqref{eq:value_of_W}, for any $ k\in\mathbb{Z} \!\setminus\! \{0\}$,
\begin{equation}
 \mathscr{G}_{\alpha,k}^{(a)}(x,\rho)\;=\;
 \begin{cases}
  \;|k|\,x^{a-\frac{\alpha}{2}}\,\rho^{-\frac{\alpha}{2}}\,e^{-\frac{|k|}{1+\alpha}x^{1+\alpha}}\,\sinh\big(\frac{|k|}{1+\alpha}\rho^{1+\alpha}\big)\,, & \textrm{ if }0<\rho<x\,, \\
  \;|k|\,x^{a-\frac{\alpha}{2}}\,\rho^{-\frac{\alpha}{2}}\,e^{-\frac{|k|}{1+\alpha}\rho^{1+\alpha}}\,\sinh\big(\frac{|k|}{1+\alpha}x^{1+\alpha}\big)\,, & \textrm{ if }0<x<\rho\,.
 \end{cases}
\end{equation}
It is also convenient to use the bound
\begin{equation}\label{eq:GleqGtilde}
 \mathscr{G}_{\alpha,k}^{(a)}(x,\rho)\;\leqslant\;\widetilde{\mathscr{G}_{\alpha,k}^{(a)}}(x,\rho)\;
\end{equation}
with
\begin{equation}\label{eq:defGtilde}
 \widetilde{\mathscr{G}_{\alpha,k}^{(a)}}(x,\rho)\;:=\;\begin{cases}
  \;|k|\,x^{a-\frac{\alpha}{2}}\,\rho^{-\frac{\alpha}{2}}\,e^{-\frac{|k|}{1+\alpha}x^{1+\alpha}}\,e^{\frac{|k|}{1+\alpha}\rho^{1+\alpha}}\,, & \textrm{ if }0<\rho<x\,, \\
  \;|k|\,x^{a-\frac{\alpha}{2}}\,\rho^{-\frac{\alpha}{2}}\,e^{-\frac{|k|}{1+\alpha}\rho^{1+\alpha}}\,e^{\frac{|k|}{1+\alpha}x^{1+\alpha}}\,, & \textrm{ if }0<x<\rho\,.
 \end{cases}
\end{equation}

\begin{proof}[Proof of Lemma \ref{lem:RGbddsa}]

$R_{G_{\alpha,k}}^{(a)}$ splits into the sum of four integral operators with non-negative kernels given by
\[
\begin{split}
	  \mathscr{G}^{++}_{\alpha,k,a}(x,\rho)\;&:=\;\mathscr{G}_{\alpha,k}^{(a)}(x,\rho)\,\mathbf{1}_{(M,+\infty)}(x)\,\mathbf{1}_{(M,+\infty)}(\rho)\,, \\
	  \mathscr{G}^{+-}_{\alpha,k,a}(x,\rho)\;&:=\;\mathscr{G}_{\alpha,k}^{(a)}(x,\rho)\,\mathbf{1}_{(M,+\infty)}(x)\,\mathbf{1}_{(0,M)}(\rho)\,, \\
  \mathscr{G}^{-+}_{\alpha,k,a}(x,\rho)\;&:=\;\mathscr{G}_{\alpha,k}^{(a)}(x,\rho)\,\mathbf{1}_{(0,M)}(x)\,\mathbf{1}_{(M,+\infty)}(\rho)\,, \\
  \mathscr{G}^{--}_{\alpha,k,a}(x,\rho)\;&:=\;\mathscr{G}_{\alpha,k}^{(a)}(x,\rho)\,\mathbf{1}_{(0,M)}(x)\,\mathbf{1}_{(0,M)}(\rho)
 \end{split}
\]
for some cut-off $M>0$.

The $(-,-)$ operator is a Hilbert-Schmidt operator on $L^2(\mathbb{R}^+)$. Indeed, owing to \eqref{eq:GleqGtilde}-\eqref{eq:defGtilde},
\[
 \begin{split}
  \mathscr{G}^{--}_{\alpha,k,a}(x,\rho)\;&\leqslant\;|k| x^{a-\frac{\alpha}{2}}\,\rho^{-\frac{\alpha}{2}}\,e^{-\frac{|k|}{1+\alpha}|x^{1+\alpha}-\rho^{1+\alpha}|}\,\mathbf{1}_{(0,M)}(x)\,\mathbf{1}_{(0,M)}(\rho) \\
  &\leqslant\;|k| x^{a-\frac{\alpha}{2}}\,\rho^{-\frac{\alpha}{2}}\,\mathbf{1}_{(0,M)}(x)\,\mathbf{1}_{(0,M)}(\rho)\,,
 \end{split}
\]
whence, for $a>-\frac{1}{2}(1-\alpha)$,
\[
 \begin{split}
  \iint_{\mathbb{R}^+\times\mathbb{R}^+}\,\ud x\,\ud\rho\,\big|\mathscr{G}^{--}_{\alpha,k,a}(x,\rho) \big|^2\;&\leqslant\;k^2\!\int_0^M\ud x\,x^{2a-\alpha}\!\int_0^M\ud\rho\,\rho^{-\alpha} \\
  &=\;\frac{k^2\,M^{2(a+1-\alpha)}}{(2a+1-\alpha)(1-\alpha)}\,.
 \end{split}
\]

Also the $(-,+)$ operator is a Hilbert-Schmidt operator on $L^2(\mathbb{R}^+)$. Indeed,
\[
  \mathscr{G}^{-+}_{\alpha,k,a}(x,\rho)\;\leqslant\;|k|\,e^{\frac{|k|}{1+\alpha}M^{1+\alpha}} x^{a-\frac{\alpha}{2}}\,\rho^{-\frac{\alpha}{2}}\,e^{-\frac{|k|}{1+\alpha}\rho^{1+\alpha}}\,\mathbf{1}_{(0,M)}(x)\,\mathbf{1}_{(M,+\infty)}(\rho)\,,
\]
whence, for $a>-\frac{1}{2}(1-\alpha)$,
\[
 \begin{split}
  \iint_{\mathbb{R}^+\times\mathbb{R}^+}&\,\ud x\,\ud\rho\,\big|\mathscr{G}^{-+}_{\alpha,k,a}(x,\rho) \big|^2 \\
  &\leqslant\;k^2\,e^{\frac{2|k|}{1+\alpha}M^{1+\alpha}}\!\int_0^M\ud x\,x^{2a-\alpha}\!\int_M^{+\infty}\!\ud\rho\,\rho^{-\alpha}\,e^{-\frac{2|k|}{1+\alpha}\rho^{1+\alpha}} \\
  &\leqslant\;k^2\,M^{-2\alpha}\,e^{\frac{2|k|}{1+\alpha}M^{1+\alpha}}\!\int_0^M\ud x\,x^{2a-\alpha}\!\int_M^{+\infty}\!\ud\rho\,\rho^{\alpha}\,e^{-\frac{2|k|}{1+\alpha}\rho^{1+\alpha}} \\
  &=\;\frac{|k|}{2}\,M^{-2\alpha}\,e^{\frac{2|k|}{1+\alpha}M^{1+\alpha}}\!\int_0^M\ud x\,x^{2a-\alpha}\!\int_{\frac{2|k|}{1+\alpha}M^{1+\alpha}}^{+\infty}\ud y\,e^{-y} \\
  &=\;\frac{|k|\,M^{2a+1-3\alpha}}{2(2a+1-\alpha)}\,.
 \end{split}
\]

With analogous reasoning,
\[
  \mathscr{G}^{+-}_{\alpha,k,a}(x,\rho)\;\leqslant\;|k|\,e^{\frac{|k|}{1+\alpha}M^{1+\alpha}}\rho^{-\frac{\alpha}{2}} x^{a-\frac{\alpha}{2}}\,e^{-\frac{|k|}{1+\alpha}x^{1+\alpha}}\,\mathbf{1}_{(M,+\infty)}(x)\,\mathbf{1}_{(0,M)}(\rho)\,,
\]
whence
\[
 \begin{split}
   \iint_{\mathbb{R}^+\times\mathbb{R}^+}&\,\ud x\,\ud\rho\,\big|\mathscr{G}^{+-}_{\alpha,k,a}(x,\rho) \big|^2 \;\leqslant\;k^2\,e^{\frac{2|k|}{1+\alpha}M^{1+\alpha}}\!\int_0^M\ud\rho\,\rho^{-\alpha}\!\int_M^{+\infty}\!\ud x\,x^{2a-\alpha}\,e^{-\frac{2|k|}{1+\alpha}x^{1+\alpha}} \\
   &=\;\frac{\,k^2M^{1-\alpha}}{1-\alpha}\,e^{\frac{2|k|}{1+\alpha}M^{1+\alpha}}\!\int_M^{+\infty}\!\ud x\,x^{2a-\alpha}\,e^{-\frac{2|k|}{1+\alpha}x^{1+\alpha}}\,.
 \end{split} 
\]
In turn, integrating by parts, and for $a\leqslant\frac{1}{2}+\frac{3}{2}\alpha$,
\[
 \begin{split}
  \int_M^{+\infty}\!\ud x&\,x^{2a-\alpha}\,e^{-\frac{2|k|}{1+\alpha}x^{1+\alpha}} \\
  &=\;\frac{M^{2a-2\alpha}}{2|k|}\,e^{-\frac{2|k|}{1+\alpha}M^{1+\alpha}}+\frac{a-\alpha}{|k|}\int_{M}^{+\infty}\!\ud x\,x^{2a-1-3\alpha}\,x^{\alpha}\,e^{-\frac{2|k|}{1+\alpha}x^{1+\alpha}} \\
  &\leqslant\;\frac{M^{2a-2\alpha}}{2|k|}\,e^{-\frac{2|k|}{1+\alpha}M^{1+\alpha}}+\frac{\,(a-\alpha)M^{2a-1-3\alpha}}{2k^2}\!\int_{\frac{2|k|}{1+\alpha}M^{1+\alpha}}^{+\infty}\ud y\,e^{-y} \\
  &=\;e^{-\frac{2|k|}{1+\alpha}M^{1+\alpha}}\Big(\frac{M^{2a-2\alpha}}{2|k|}+\frac{\,(a-\alpha)M^{2a-1-3\alpha}}{2k^2}\Big)\,.
 \end{split}
\]
Thus,
\[
  \iint_{\mathbb{R}^+\times\mathbb{R}^+}\,\ud x\,\ud\rho\,\big|\mathscr{G}^{+-}_{\alpha,k,a}(x,\rho) \big|^2 \;\leqslant\;\frac{1}{\,2(1-\alpha)}\,\big(2|k| M^{2a+1-3\alpha}+(a-\alpha)M^{2(a-2\alpha)}\big)\,,
\]
which shows that the $(+,-)$ operator is a Hilbert-Schmidt operator on $L^2(\mathbb{R}^+)$.

Last, it will be now shown, by means of a standard Schur test,\index{Schur test} that the norm of the $(+,+)$ operator is bounded by $\sqrt{AB}$, where
\[
 \begin{split}
  A\;&:=\;\sup_{x\in(M,+\infty)}\int_M^{+\infty}\!\ud\rho\,\,\mathscr{G}_{\alpha,k}^{(a)}(x,\rho)\,, \\
  B\;&:=\;\sup_{\rho\in(M,+\infty)}\int_M^{+\infty}\!\ud x\,\,\mathscr{G}_{\alpha,k}^{(a)}(x,\rho)\,.
 \end{split}
\]

Owing to \eqref{eq:GleqGtilde}-\eqref{eq:defGtilde},
\[
 \begin{split}
  A\;&\leqslant\;A_1+A_2\,, \\
  B\;&\leqslant\;B_1+B_2\,,
 \end{split}
\]
with
\[
 \begin{split}
  A_1\;&:=\;\sup_{x\in(M,+\infty)}|k|\,x^{a-\frac{\alpha}{2}}\,e^{-\frac{|k|}{1+\alpha}x^{1+\alpha}}\!\int_M^{x}\ud\rho\,\rho^{-\frac{\alpha}{2}}\,e^{\frac{|k|}{1+\alpha}\rho^{1+\alpha}}\,, \\
  A_2\;&:=\;\sup_{x\in(M,+\infty)}|k|\,x^{a-\frac{\alpha}{2}}\,e^{\frac{|k|}{1+\alpha}x^{1+\alpha}}\!\int_x^{+\infty}\!\ud\rho\,\rho^{-\frac{\alpha}{2}}\,e^{-\frac{|k|}{1+\alpha}\rho^{1+\alpha}}\,, \\
  B_1\;&:=\;\sup_{\rho\in(M,+\infty)}|k|\,\rho^{-\frac{\alpha}{2}}\,e^{-\frac{|k|}{1+\alpha}\rho^{1+\alpha}}\!\int_M^{\rho}\ud x\,x^{a-\frac{\alpha}{2}}\,e^{\frac{|k|}{1+\alpha}x^{1+\alpha}}\,, \\
  B_2\;&:=\;\sup_{\rho\in(M,+\infty)}|k|\,\rho^{-\frac{\alpha}{2}}\,e^{\frac{|k|}{1+\alpha}\rho^{1+\alpha}}\!\int_{\rho}^{+\infty}\ud x\,x^{a-\frac{\alpha}{2}}\,e^{-\frac{|k|}{1+\alpha}x^{1+\alpha}}\,.
 \end{split}
\]

Concerning $A_1$, integration by parts yields
\[
 \begin{split}
  |k|\!\int_M^{x}\ud\rho\,\rho^{-\frac{\alpha}{2}}\,e^{\frac{|k|}{1+\alpha}\rho^{1+\alpha}}\;&=\;x^{-\frac{3}{2}\alpha}\,e^{\frac{|k|}{1+\alpha}x^{1+\alpha}}-M^{-\frac{3}{2}\alpha}\,e^{\frac{|k|}{1+\alpha}M^{1+\alpha}} \\
  &\qquad\qquad\quad+\frac{3\alpha}{2}\!\int_M^x\ud\rho\,\rho^{-(1+\frac{3}{2}\alpha)}\,e^{\frac{|k|}{1+\alpha}\rho^{1+\alpha}}
 \end{split}
\]
and choosing $M\geqslant M_\circ$, where
\[
 M_\circ\;:=\;\Big(\frac{2+3\alpha}{2|k|}\Big)^{\frac{1}{1+\alpha}}
\]
is the point of absolute minimum of the function $\rho\mapsto\rho^{-(1+\frac{3}{2}\alpha)}\,e^{\frac{|k|}{1+\alpha}\rho^{1+\alpha}}$, yields
\[
 \begin{split}
  |k|\!\int_M^{x}\ud\rho\,\rho^{-\frac{\alpha}{2}}\,e^{\frac{|k|}{1+\alpha}\rho^{1+\alpha}}\;&\leqslant\;x^{-\frac{3}{2}\alpha}\,e^{\frac{|k|}{1+\alpha}x^{1+\alpha}}+\frac{3\alpha}{2}\,x^{-(1+\frac{3}{2}\alpha)}\,e^{\frac{|k|}{1+\alpha}x^{1+\alpha}}\!\int_0^x\ud\rho \\
 &=\;{\textstyle\big(1+\frac{3}{2}\alpha\big)}x^{-\frac{3}{2}\alpha}\,e^{\frac{|k|}{1+\alpha}x^{1+\alpha}}\,.
 \end{split}
\]
Therefore,
\[
 A_1\;\leqslant\;\sup_{x\in(M,+\infty)}{\textstyle\big(1+\frac{3}{2}\alpha\big)}x^{a-2\alpha}\;=\;{\textstyle\big(1+\frac{3}{2}\alpha\big)}M^{a-2\alpha}\,,
\]
the last identity being valid for $a\leqslant 2\alpha$.

Concerning $A_2$,
\[
\begin{split}
 |k|\!\int_x^{+\infty}\!\ud\rho&\,\rho^{-\frac{\alpha}{2}}\,e^{-\frac{|k|}{1+\alpha}\rho^{1+\alpha}}\;\leqslant\;|k|\,x^{-\frac{3}{2}\alpha}\!\int_x^{+\infty}\!\ud\rho\,\rho^{\alpha}\,e^{-\frac{|k|}{1+\alpha}\rho^{1+\alpha}} \\
 &=\;x^{-\frac{3}{2}\alpha}\!\int_{\frac{|k|}{1+\alpha}x^{1+\alpha}}^{+\infty}\ud y\,e^{-y}\;=\;x^{-\frac{3}{2}\alpha}\,e^{-\frac{|k|}{1+\alpha}x^{1+\alpha}}\,,
\end{split}
\]
whence, when  $a\leqslant 2\alpha$,
\[
 A_2\;\leqslant\;\sup_{x\in(M,+\infty)}x^{a-2\alpha}\;=\;M^{a-2\alpha}\,.
\]

Concerning $B_1$,
\[
 \begin{split}
  |k|\!\int_M^{\rho}\ud x&\,x^{a-\frac{\alpha}{2}}\,e^{\frac{|k|}{1+\alpha}x^{1+\alpha}}\;=\;|k|\!\int_M^{\rho}\ud x\,x^{a-\frac{3}{2}\alpha}\,x^{\alpha}\,e^{\frac{|k|}{1+\alpha}x^{1+\alpha}} \\
  &\leqslant\;|k|\int_M^{\rho}\ud x\,x^{\alpha}\,e^{\frac{|k|}{1+\alpha}x^{1+\alpha}} \times
  \begin{cases}
   \;\rho^{a-\frac{3}{2}\alpha}\,, & \textrm{ if } a\geqslant\frac{3}{2}\alpha \\
   \;M^{a-\frac{3}{2}\alpha}\,, & \textrm{ if } a<\frac{3}{2}\alpha
  \end{cases} \\
  &\leqslant\;\int_0^{\frac{|k|}{1+\alpha}\rho^{1+\alpha}}\!\ud y\;e^y\times
  \begin{cases}
   \;\rho^{a-\frac{3}{2}\alpha}\,, & \textrm{ if } a\geqslant\frac{3}{2}\alpha \\
   \;M^{a-\frac{3}{2}\alpha}\,, & \textrm{ if } a<\frac{3}{2}\alpha
  \end{cases} \\
  &\leqslant\;e^{\frac{|k|}{1+\alpha}\rho^{1+\alpha}}\times
  \begin{cases}
   \;\rho^{a-\frac{3}{2}\alpha}\,, & \textrm{ if } a\geqslant\frac{3}{2}\alpha \\
   \;M^{a-\frac{3}{2}\alpha}\,, & \textrm{ if } a<\frac{3}{2}\alpha\,,
  \end{cases} 
  \end{split}
\]
whence
\[
 \begin{split}
  B_1\;&\leqslant\;\sup_{\rho\in(M,+\infty)}\begin{cases}
   \;\rho^{a-2\alpha}\,, & \textrm{ if } a\geqslant\frac{3}{2}\alpha\,, \\
   \;\rho^{-\frac{\alpha}{2}}\,M^{a-\frac{3}{2}\alpha}\,, & \textrm{ if } a<\frac{3}{2}\alpha\,.
  \end{cases} 
 \end{split}
\]
In either case, as long as  $a\leqslant 2\alpha$,
\[
 B_1\;\leqslant\;M^{a-2\alpha}\,.
\]

Concerning $B_2$, one splits the analysis between $a\leqslant\frac{3}{2}\alpha$ and $a>\frac{3}{2}\alpha$. In the former case,
\[
 \begin{split}
  |k|\!\int_{\rho}^{+\infty}\ud x&\,x^{a-\frac{\alpha}{2}}\,e^{-\frac{|k|}{1+\alpha}x^{1+\alpha}}\;\leqslant\;\rho^{a-\frac{3}{2}\alpha}|k|\!\int_{\rho}^{+\infty}\ud x\,x^{\alpha}\,e^{-\frac{|k|}{1+\alpha}x^{1+\alpha}} \\
  &=\;\rho^{a-\frac{3}{2}\alpha}\!\int_{\frac{|k|}{1+\alpha}}^{+\infty}\ud y\,e^{-y}\;=\;\rho^{a-\frac{3}{2}\alpha}\,e^{-\frac{|k|}{1+\alpha}\rho^{1+\alpha}}\,,
 \end{split}
\]
whence, as long as $a\leqslant 2\alpha$,
\[
 B_2\;\leqslant\sup_{\rho\in(M,+\infty)}\rho^{a-2\alpha}\;\leqslant\;M^{a-2\alpha}\,.
\]

When instead $a>\frac{3}{2}\alpha$, then, integrating by parts and using $a\leqslant1+\frac{5}{2}\alpha$,
\[
 \begin{split}
  |k|\!\int_{\rho}^{+\infty}\ud x&\,x^{a-\frac{\alpha}{2}}\,e^{-\frac{|k|}{1+\alpha}x^{1+\alpha}} \\
  &=\;\rho^{a-\frac{3}{2}\alpha}\,e^{-\frac{|k|}{1+\alpha}\rho^{1+\alpha}}+\big(a-{\textstyle\frac{3\alpha}{2}}\big)\!\int_{\rho}^{+\infty}\!\ud x\,x^{a-\frac{3}{2}\alpha-1}\,e^{-\frac{|k|}{1+\alpha}x^{1+\alpha}} \\
  &\leqslant\;\rho^{a-\frac{3}{2}\alpha}\,e^{-\frac{|k|}{1+\alpha}\rho^{1+\alpha}}+\big(a-{\textstyle\frac{3\alpha}{2}}\big)\rho^{a-\frac{5}{2}\alpha-1}\!\int_{\rho}^{+\infty}\!\ud x\,x^{\alpha}\,e^{-\frac{|k|}{1+\alpha}x^{1+\alpha}} \\
  &=\;\rho^{a-\frac{3}{2}\alpha}\,e^{-\frac{|k|}{1+\alpha}\rho^{1+\alpha}}+\big(a-{\textstyle\frac{3\alpha}{2}}\big)\rho^{a-\frac{5}{2}\alpha-1}|k|^{-1}\!\int_{\frac{|k|}{1+\alpha}\rho^{1+\alpha}}^{+\infty}\ud y\,e^{-y} \\
  &=\;e^{-\frac{|k|}{1+\alpha}\rho^{1+\alpha}}\big(\rho^{a-\frac{3}{2}\alpha}+(a-{\textstyle\frac{3\alpha}{2}})\,|k|^{-1}\rho^{a-\frac{5}{2}\alpha-1}\big)\,,
 \end{split}
\]
whence
\[
 \begin{split}
  B_2\;&\leqslant\sup_{\rho\in(M,+\infty)}\big(\rho^{a-2\alpha}+(a-{\textstyle\frac{3\alpha}{2}})\,|k|^{-1}\rho^{a-3\alpha-1}\big) \\
  &\leqslant\;M^{a-2\alpha}\big(1+(a-{\textstyle\frac{3\alpha}{2}})\,(|k|M^{1+\alpha})^{-1}\big)\,.
 \end{split}
\]

This completes the proof of the boundedness, via a Schur test,\index{Schur test} of the $(+,+)$ operator.

Summarising, with the above choice of the cut-off $M\geqslant M_\circ$, and intersecting all the above restrictions of $a$ in terms of $\alpha$, that is, $-\frac{1}{2}(1-\alpha)\leqslant a\leqslant 2\alpha$, it is established that there is an overall constant $Z_{a,\alpha}>0$ such that
\[
 \begin{split}
  \big\|R_{G_{\alpha,k}}^{(a)}\big\|^2_{\mathrm{op}}\;&\leqslant\;Z_{a,\alpha}\Big(k^2 M^{2(a+1-\alpha)}+|k|\,M^{2a+1-3\alpha}+M^{2a-4\alpha} \\
  &\qquad\qquad\quad +M^{2a-4\alpha}(|k|\,M^{1+\alpha})^{-1}\Big)\,.
 \end{split}
\]

This yields the statement of boundedness of part (i). The self-adjointness of $R_{G_{\alpha,k}}=|k|^{-2}\,R_{G_{\alpha,k}}^{(0)}$ is clear from \eqref{eq:Green}: the adjoint $R_{G_{\alpha,k}}^*$ has kernel $\overline{G_{\alpha,k}(\rho,r)}$, but $G$ is real-valued and $G_{\alpha,k}(\rho,r)=G_{\alpha,k}(r,\rho)$, whence indeed $R_{G_{\alpha,k}}^*=R_{G_{\alpha,k}}$. Thus, part (i) is proved.

As for part (ii), for the cut-off one chooses $M= M_\circ$ when $a=2\alpha$. In this case,
\[
 \begin{split}
  |k|\,M^{1+\alpha}\;&=\;{\textstyle 1+\frac{3}{2}\alpha}\,, \\
  |k|\,M^{2a+1-3\alpha}\;&=\;|k|\,M^{1+\alpha}\;=\;{\textstyle 1+\frac{3}{2}\alpha}\,,\\
  k^2 M^{2(a+1-\alpha)}\;&=\;\big(|k|\,M^{1+\alpha}\big)^2\;=\;\big( {\textstyle 1+\frac{3}{2}\alpha}\big)^2\,,
 \end{split}
\]
implying that there is an updated constant $\widetilde{Z}_{a,\alpha}>0$ such that 
\[
 \big\|R_{G_{\alpha,k}}^{(2\alpha)}\big\|_{\mathrm{op}}\;\leqslant\;\widetilde{Z}_{a,\alpha}
\]
uniformly in $k$. Thus, also part (ii) is proved.
\end{proof}

A relevant consequence of Lemma \ref{lem:RGbddsa} is the following large-distance decaying behaviour of a generic function of the form $R_{G_{\alpha,k}}u$.

\begin{corollary}\label{cor:RGtoWeightedL2} Let $\alpha\in(0,1)$ and $ k\in\mathbb{Z}\!\setminus\!\{0\}$. Then
\begin{equation}
	\mathrm{ran}\,R_{G_{\alpha,k}} \;\subset\; L^2(\mathbb{R}^+,\langle x\rangle^{4\alpha} \ud x)\,.
\end{equation}
\end{corollary}

\begin{proof}
By Lemma \ref{lem:RGbddsa} we know that both $x^{2\alpha} R_{G_{\alpha,k}}$ and $R_{G_{\alpha,k}}$ are bounded in $L^2(\mathbb{R}^+,\ud x)$. Therefore, for any $u\in L^2(\mathbb{R}^+,\ud x)$ one has that both $R_{G_{\alpha,k}} u$ and $x^{2\alpha}R_{G_{\alpha,k}} u$ must belong to $L^2(\mathbb{R}^+,\ud x)$, whence indeed $R_{G_{\alpha,k}} u\in L^2(\mathbb{R}^+,(1+x^{4\alpha}) \ud x)$.
\end{proof}

Moreover, we recognise that $R_{G_{\alpha,k}}$ inverts a self-adjoint extension of $A_\alpha(k)$.

\begin{lemma}\label{eq:RGinvertsExtS}
 Let $\alpha\in(0,1)$ and $ k\in\mathbb{Z}\!\setminus\! \{0\}$. There exists a self-adjoint extension $\mathscr{A}_\alpha(k)$ of $A_\alpha(k)$ in $L^2(\mathbb{R}^+)$ which has everywhere defined and bounded inverse and such that $\mathscr{A}_{\alpha}(k)^{-1}=R_{G_{\alpha,k}}$.
\end{lemma}

\begin{proof}
 $R_{G_{\alpha,k}}$ is bounded and self-adjoint (Lemma \ref{lem:RGbddsa}), and by construction satisfies $S_{\alpha,k}\,R_{G_{\alpha,k}}\,g=g$ $\forall g\in L^2(\mathbb{R}^+)$. Therefore, $R_{G_{\alpha,k}} g=0$ for some $g\in L^2(\mathbb{R}^+)$ implies $g=0$, i.e., $R_{G_{\alpha,k}}$ is injective. Then $R_{G_{\alpha,k}}$ has dense range ($(\mathrm{ran}\,R_{G_{\alpha,k}})^\perp=\ker R_{G_{\alpha,k}}$). As a consequence (see, e.g., \cite[Theorem 1.8(iv)]{schmu_unbdd_sa}), $\mathscr{A}_\alpha(k):=R_{G_{\alpha,k}}^{-1}$ is self-adjoint. One thus has $R_{G_{\alpha,k}}=\mathscr{A}_{\alpha}(k)^{-1}$ and from the identity $A_\alpha(k)^*R_{G_{\alpha,k}}=\mathbbm{1}$ on $L^2(\mathbb{R}^+)$ one deduces that for any $h\in\mathcal{D}(\mathscr{A}_\alpha(k))$, say, $h=R_{G_{\alpha,k}} g=\mathscr{A}_\alpha(k)^{-1} g$ for some $g\in L^2(\mathbb{R}^+)$, the identity $A_\alpha(k)^*h=\mathscr{A}_\alpha(k)h$ holds. This means that $A_\alpha(k)^*\supset\mathscr{A}_\alpha(k)$, whence also $\overline{A_\alpha(k)}=A_\alpha(k)^{**}\subset\mathscr{A}_{\alpha}(k)$, i.e., $\mathscr{A}_{\alpha}(k)$ is a self-adjoint extension of $A_\alpha(k)$.
\end{proof}

We conclude this Subsection by examining the function
\begin{equation}\label{eq:defPsi}
 \Psi_{\alpha,k}\::=\;R_{G_{\alpha,k}}\Phi_{\alpha,k}\,.
\end{equation}

We prove the following useful asymptotics.

\begin{lemma}\label{lem:Psi_asymptotics}
 Let $\alpha\in(0,1)$ and $ k\in\mathbb{Z}\!\setminus\! \{0\}$. Then
 \begin{equation}\label{eq:Psi_asymptotics}
  \Psi_{\alpha,k}(x)\;\stackrel{x\downarrow 0}{=}\;{\textstyle\sqrt{\frac{2|k|}{\,\pi(1+\alpha)^3}}}\;\|\Phi_{\alpha,k}\|_{L^2}^2\:x^{1+\frac{\alpha}{2}}+o(x^{\frac{3}{2}})\,.
 \end{equation}
\end{lemma}

\begin{proof}
Owing to \eqref{eq:Green} and \eqref{eq:newRGalphaforus},
\[
	\Psi_{\alpha,k}(x)\;=\;\frac{1}{W} \Big( \Phi_{\alpha,k}(x)\!\int_0^x \! F_{\alpha,k}(\rho) \Phi_{\alpha,k}(\rho) \ud \rho + F_{\alpha,k}(x)\!\int_x^{+\infty} \!\! \Phi_{\alpha,k}(\rho)^2 \ud \rho \Big)\,.
\]
By means of \eqref{eq:Asymtotics_0} one then finds
\[
	\Phi_{\alpha,k}(x)\!\int_0^x \! F_{\alpha,k}(\rho) \Phi_{\alpha,k}(\rho) \ud \rho \;\stackrel{x\downarrow 0}{=}\; {\textstyle \sqrt{\frac{\pi(1+\alpha)}{8 |k|}} }x^{-\frac{\alpha}{2} + 2}+o(x^3)\;\stackrel{x\downarrow 0}{=}\;o(x^{\frac{3}{2}})
\]
(having explicitly used that $\alpha \in (0,1)$), and
\[
 \begin{split}
  F_{\alpha,k}(x)\!\int_x^{+\infty} \!\! \Phi_{\alpha,k}(\rho)^2 \ud \rho\:&\stackrel{x\downarrow 0}{=}\: F_{\alpha,k}(x)\Big(\|\Phi_{\alpha,k}\|_{L^2}^2-\int_0^x\!\Phi_{\alpha,k}(\rho)^2 \ud \rho\Big) \\
  &\stackrel{x\downarrow 0}{=}\:\;{\textstyle\sqrt{\frac{2|k|}{\,\pi(1+\alpha)}}}\;\|\Phi_{\alpha,k}\|_{L^2}^2\:x^{1+\frac{\alpha}{2}}+O(x^{2-\frac{\alpha}{2}})\,.
 \end{split}
\]
The latter quantity is leading, and using the expression \eqref{eq:value_of_W} for $W$ yields \eqref{eq:Psi_asymptotics}.
\end{proof}

In fact, using \eqref{eq:value_of_W},  \eqref{eq:Green}, and \eqref{eq:newRGalphaforus} as in the proof above, and using the explicit expression \eqref{eq:Phi_and_F_explicit} for $\Phi_{\alpha,k}$ and $F_{\alpha,k}$, one finds
\begin{equation}
 \begin{split}
  \!\!\!\!\!\!\!\!\Psi_{\alpha,k}(x)\;=\;{\textstyle\sqrt{\frac{\pi(1+\alpha)}{2|k|^3}}}&\bigg(x^{-\frac{\alpha}{2}}e^{-\frac{|k|}{1+\alpha}x^{1+\alpha}}\!\int_0^x\ud\rho\,\rho^{-\alpha}\sinh{\textstyle(\frac{|k|}{1+\alpha}\rho^{1+\alpha})}\,e^{-\frac{|k|}{1+\alpha}\rho^{1+\alpha}} \\
  &\quad +x^{-\frac{\alpha}{2}}\sinh{\textstyle(\frac{|k|}{1+\alpha}x^{1+\alpha})}\!\int_x^{+\infty}\!\!\ud\rho\,\rho^{-\alpha}\,e^{-\frac{2|k|}{1+\alpha}\rho^{1+\alpha}}\bigg),
 \end{split}
\end{equation}
or also, with a change of variable $\rho\mapsto|k|^{\frac{1}{1+\alpha}}\rho$,
 \begin{equation}\label{eq:explicitPsika}
 \begin{split}
  \Psi_{\alpha,k}(x)\;&=\;{\textstyle\sqrt{\frac{\pi(1+\alpha)}{2}}}|k|^{-\frac{5+\alpha}{2(1+\alpha)}}\times \\
  &\qquad \times\bigg(x^{-\frac{\alpha}{2}}e^{-\frac{|k|}{1+\alpha}x^{1+\alpha}}\!\int_0^{x|k|^{\frac{1}{1+\alpha}}}\ud\rho\,\rho^{-\alpha}\sinh{\textstyle(\frac{\rho^{1+\alpha}}{1+\alpha})}\,e^{-\frac{\,\rho^{1+\alpha}}{1+\alpha}} \\
  &\qquad\qquad\quad +x^{-\frac{\alpha}{2}}\sinh{\textstyle(\frac{|k|}{1+\alpha}x^{1+\alpha})}\!\int_{x|k|^{\frac{1}{1+\alpha}}}^{+\infty}\!\!\ud\rho\,\rho^{-\alpha}\,e^{-\frac{2\rho^{1+\alpha}}{1+\alpha}}\bigg)\,.
 \end{split}
\end{equation}
However, we will not need such an explicit expression for $\Psi_{\alpha,k}$ until Subsect.~\ref{sec:q0q1}.

\subsection{Operator closure $\overline{A_\alpha(k)}$}~

The next fundamental ingredient for the Kre\u{\i}n-Vi\v{s}ik-Birman extension scheme is the characterisation of the operator closure $\overline{A_\alpha(k)}$ of $A_\alpha(k)$. 

In this Subsection we establish the following result.

\begin{proposition}\label{prop:domAclosure}
 Let $\alpha\in(0,1)$ and $ k\in\mathbb{Z}\!\setminus\! \{0\}$. Then
 \begin{equation}
  \mathcal{D}(\overline{A_\alpha(k)})\;=\;H^2_0(\mathbb{R}^+)\cap L^2(\mathbb{R}^+,\langle x\rangle^{4\alpha}\,\ud x)\,.
 \end{equation}
\end{proposition}

Here $H^2_0(\mathbb{R}^+)$ is the space \eqref{eq:H20} and, by definition,
\begin{equation}
 \mathcal{D}(\overline{A_\alpha(k)})\;=\;\overline{C^\infty_c(\mathbb{R}^+)}^{\|\,\|_{A_\alpha(k)}}\,,
\end{equation}
where the norm $\|\cdot\|_{A_\alpha(k)}$ is defined by
\begin{equation}
\begin{split}
 \|\varphi\|_{A_\alpha(k)}^2\;&:=\;\|-\varphi''+ k^2 x^{2\alpha}\varphi+C_\alpha x^{-2}\varphi\|_{L^2(\mathbb{R}^+)}^2+\|\varphi\|_{L^2(\mathbb{R}^+)}^2 \\
 &\qquad\forall\varphi\in\mathcal{D}(A_\alpha(k))=C^\infty_c(\mathbb{R}^+)\,.
\end{split}
\end{equation}

We prove Proposition \ref{prop:domAclosure} in several steps.

First, in the same spirit of \cite[Lemma 5.1]{MG_DiracCoulomb2017}, we produce a useful representation of $\mathcal{D}(A_\alpha(k)^*)$ based on the differential nature \eqref{eq:Afstar}  of the adjoint $A_\alpha(k)^*$.

\begin{lemma}\label{lem:odedecomp}
 Let $\alpha\in(0,1)$ and $ k\in\mathbb{Z}\!\setminus\! \{0\}$.
 \begin{itemize}
  \item[(i)] For each $g \in \mathcal{D}(A_\alpha(k)^*)$ there exist uniquely determined constants $a_0^{(g)}, a_\infty^{(g)} \in \mathbb{C}$ and functions
\begin{equation}\label{eq:b0binfty}
\begin{split}
b_0^{(g)}(x) \;&:=\; \frac{1}{W} \int_0^x F_{\alpha,k}(\rho) (A_\alpha(k)^* g)(\rho) \, \ud \rho\,, \\
b_\infty^{(g)} (x) \;&:= \; - \frac{1}{W} \int_0^x \Phi_{\alpha,k}(\rho) (A_\alpha(k)^* g)(\rho) \, \ud \rho
\end{split}
\end{equation}
on $\mathbb{R}^+$ such that
\begin{equation}\label{eq:G-ODE-Decomposition}
g \;=\; a_0^{(g)} F_{\alpha,k} + a_\infty^{(g)} \Phi_{\alpha,k} + b_\infty^{(g)} F_{\alpha,k} + b_0^{(g)} \Phi_{\alpha,k}\,,
\end{equation}
with $\Phi_{\alpha,k}$ and $F_{\alpha,k}$ defined in \eqref{eq:Phi_and_F} and $W=-(1+\alpha)$ as in \eqref{eq:value_of_W}.
\item[(ii)] The functions $b_0^{(g)}$ and $b_\infty^{(g)}$ satisfy the properties
\begin{eqnarray}
 & & b_0^{(g)},b_\infty^{(g)}\in AC(\mathbb{R}^+)\,, \label{eq:b_areAC} \\
 & & b_0^{(g)}(x)\:\stackrel{x\downarrow 0}{=}\: o(1)\,,\quad b_\infty^{(g)}(x)\:\stackrel{x\downarrow 0}{=}\: o(1)\,, \label{eq:b_vanishes}\\
 & & b_\infty^{(g)}(x) F_{\alpha,k}(x) + b_0^{(g)} (x) \Phi_{\alpha,k}(x) \:\stackrel{x\downarrow 0}{=}\: o(x^{\frac{3}{2}})\,.\label{eq:sdasimpt}
\end{eqnarray}
 \end{itemize}
\end{lemma}

\begin{proof}
(i) Let $h:=A_\alpha(k)^*g=S_{\alpha,k}\, g$. As already observed at the beginning of Section \ref{sec:non-homogeneous_problem}, $g$ can be expressed in terms of $h$ by the standard representation
\[
	g=A_0 F_{\alpha,k} + A_\infty \Phi_{\alpha,k}+ \Theta^{(h)}_\infty F_{\alpha,k} + \Theta^{(h)}_0 \Phi_{\alpha,k}
\]
for some constants $A_0, A_\infty \in \mathbb{C}$ determined by $h$ and some $h$-dependent functions explicitly given, as follows from \eqref{eq:ODE_general_sol}, \eqref{eq:upart}, and \eqref{eq:Green}, by
\[
	\begin{split}
		\Theta_0^{(h)}(x)\;&:=\; \frac{1}{W} \int_0^x F_{\alpha,k}(\rho) h(\rho)\, \ud \rho\,,\\
		\Theta_\infty^{(h)} (x) \;&:=\; \frac{1}{W} \int_x^{+\infty} \!\!\Phi_{\alpha,k}(\rho) h(\rho) \, \ud \rho\,.
	\end{split}
\]
Comparing the latter formulas with \eqref{eq:b0binfty}-\eqref{eq:G-ODE-Decomposition}, one deduces 
\[
	\begin{split}
		\Theta^{(h)}_0(x)\;&=\;b_0^{(g)}(x)\,,\\
		\Theta^{(h)}_\infty(x)\;&=\;\frac{1}{W} \int_x^{+\infty}\!\! \Phi_{\alpha,k}(\rho) (A_\alpha(k)^* g)(\rho) \, \ud \rho \\
		&=\;W^{-1} \,\langle \Phi_{\alpha,k}, A_\alpha(k)^* g \rangle_{L^2(\mathbb{R}^+)}+b_\infty^{(g)}(x)\,.
	\end{split}
\]
So \eqref{eq:G-ODE-Decomposition} is proved upon setting
\[
\begin{split}
	a_0^{(g)}\;&:=\;A_0+W^{-1} \,\langle \Phi_{\alpha,k}, A_\alpha(k)^* g \rangle_{L^2(\mathbb{R}^+)}\,, \\
	a_\infty^{(g)} \;&:=\; A_\infty\,.
	\end{split}
\]

(ii) Since $\Phi_{\alpha,k}$, $F_{\alpha,k}$ and $A_\alpha(k)^*g$ are all square-integrable on the interval $[0,x]$, the integrand functions in \eqref{eq:b0binfty} are $L^1$-functions on $[0,x]$: this proves \eqref{eq:b_areAC} and justifies the simple estimates
\[
 \begin{split}
  |b_0^{(g)}(x)| \;&\lesssim\; \Vert F_{\alpha,k} \Vert_{L^2((0,x))} \Vert A_\alpha^*(k) g \Vert_{L^2((0,x))} \:\stackrel{x\downarrow 0}{=}\: o(1)\,, \\
  |b_\infty^{(g)}(x)| \;&\lesssim\; \Vert \Phi_{\alpha,k} \Vert_{L^2((0,x))} \Vert A_\alpha^*(k) g \Vert_{L^2((0,x))}\:\stackrel{x\downarrow 0}{=}\: o(1)\,,
 \end{split}
\]
so \eqref{eq:b_vanishes} is proved too. Last, one finds
\[
 \begin{split}
  |b_\infty^{(g)}(x) F_{\alpha,k}(x) | \;&\lesssim\; x^{1+\frac{\alpha}{2}} \Big(\int_0^x \!\rho^{-\alpha} \, \ud \rho \Big)^{\frac{1}{2}} \Vert h \Vert_{L^2((0,x))} \;\lesssim\; x^{\frac{3}{2}} o(1) \;=\; o(x^{\frac{3}{2}})\,, \\
  |b_0^{(g)}(x) \Phi_{\alpha,k}(x)| \;&\lesssim\; x^{-\frac{\alpha}{2}}\! \int_0^x \rho^{1+\frac{\alpha}{2}} |h(\rho)| \ud \rho \;\leqslant\; x \Vert h \Vert_{L^2((0,x))} x^{\frac{1}{2}} \;=\; o(x^{\frac{3}{2}})\,,
 \end{split}
\]
and \eqref{eq:sdasimpt} follows.
\end{proof}

\begin{remark}
 As is evident from the proof of Lemma \ref{lem:odedecomp}, the decomposition \eqref{eq:G-ODE-Decomposition} is valid for a generic solution $g$ to $S_{\alpha,k}\, g=h$, irrespectively of whether $g$ belongs to $\mathcal{D}(A_\alpha(k)^*)$ (i.e., irrespectively of whether $h$ is square-integrable), it is just a consequence of general facts of the theory of linear ordinary differential equations. It is only in part (ii) of Lemma \ref{lem:odedecomp} that we explicitly used $g\in\mathcal{D}(A_\alpha(k)^*)$ (i.e., $h\in L^2(\mathbb{R}^+)$). 
\end{remark}

Next, proceeding towards the proof of Proposition \ref{prop:domAclosure}, we introduce, for any two functions in $\mathcal{D}(A_\alpha(k)^*)$, the `generalised Wronskian'
\begin{equation}
\mathbb{R}^+ \ni x \mapsto W_x(g,h) \;:=\; \det \begin{pmatrix}
g(x) & h(x) \\
g'(x) & h'(x)
\end{pmatrix}, \qquad g,h \in \mathcal{D}(A_\alpha(k)^*)
\end{equation}
and the `boundary form'
\begin{equation}\label{eq:boundaryform}
\omega(g,h)\;:=\; \langle A_\alpha(k)^* g, h \rangle_{L^2}-\langle g, A_\alpha(k)^* h \rangle_{L^2}, \qquad g,h \in \mathcal{D}(A_\alpha(k)^*).
\end{equation}
The boundary form is anti-symmetric, i.e.,
\begin{equation}
\omega(h,g)=-\overline{\omega(g,h)},
\end{equation}
and it is related to the Wronskian by
\begin{equation}\label{eq:omega-W}
\omega(g,h)\;=\;- \lim_{x \downarrow 0} W_x (\overline{g},h) \, .
\end{equation}
Indeed,
\[
	\begin{split}
		\omega(g,h)\;&=\; \int_0^{+\infty} (\overline{A_\alpha(k)^* g})(\rho) \,h(\rho) \, \ud \rho - \int_0^{+\infty} \overline{g(\rho)}\, (A_\alpha(k)^* h)(\rho) \ud \rho  \\
		&=\;\lim_{x \downarrow 0}\Big( \int_x^{+\infty} (\overline{- g''(\rho)}\, h(\rho) \,\ud \rho + \int_x^{+\infty} \overline{g(\rho)}\, h''(\rho)\, \ud \rho \Big) \\
		&=\; \lim_{x \downarrow 0} \Big(\overline{g'(x)} h(x) - \overline{g(x)} h'(x) \Big)\;=\;-\lim_{x \downarrow 0} W_x(\overline{g},h) \, .
	\end{split}
\]

It is also convenient to refer to the two-dimensional space of solutions to the differential problem $S_{\alpha,k}\,u = 0$ as the space
\begin{equation}
\mathcal{L}\;:=\; \{ u: \mathbb{R}^+ \to \mathbb{C} \,| \, S_{\alpha,k}\,u = 0 \} \;=\; \mathrm{span} \, \{\Phi_{\alpha,k},F_{\alpha,k}\}\,,
\end{equation}
where the second identity follows from what argued in the proof of Lemma \ref{lem:kerAxistar}.
As well known, $x \mapsto W_x(u,v)$ is constant whenever $u,v \in \mathcal{L}$, and this constant is zero if and only if $u$ and $v$ are linearly dependent. Clearly, any $u \in \mathcal{L}$ is square-integrable around $x=0$, as follows from the asymptotics \eqref{eq:Asymtotics_0}.

\begin{lemma}\label{lem:LV}
Let $\alpha\in(0,1)$ and $ k\in\mathbb{Z} \!\setminus\! \{0\}$.
For given $u \in \mathcal{L}$,
\begin{equation}\label{eq:LinearLV}
\begin{split}
L_u : \mathcal{D}(&A_\alpha(k)^*) \to \mathbb{C} \\
&g \mapsto L_u(g)\;:=\; \lim_{x \downarrow 0} W_x(\overline{u},g)
\end{split}
\end{equation}
defines a linear functional on $\mathcal{D}(A^*_\alpha(k))$ which vanishes on $\mathcal{D}(\overline{A_\alpha(k)})$.
\end{lemma}

\begin{proof}
The linearity of $L_u$ is obvious. Its finiteness of $L_u(g)$ is checked as follows. One decomposes (according to \eqref{eq:G-ODE-Decomposition} and using the basis of $\mathcal{L}$)
\[
 \begin{split}
  g\;&=\;a_0^{(g)} F_{\alpha,k}+a_\infty^{(g)} \Phi_{\alpha,k}+b_\infty^{(g)} F_{\alpha,k} + b_0^{(g)} \Phi_{\alpha,k}\,, \\
  u\;&=\;c_0 F_{\alpha,k}+c_\infty \Phi_{\alpha,k} \, .
 \end{split}
\]
Owing to \eqref{eq:LinearLV} it suffices to control the finiteness of $L_{F_{\alpha,k}}(g)$ and $L_{\Phi_{\alpha,k}}(g)$. By linearity
\[\tag{i}
\begin{split}
 	L_{F_{\alpha,k}}(g) \;&=\; a_0^{(g)} L_{F_{\alpha,k}}(F_{\alpha,k})+a_\infty^{(g)} L_{F_{\alpha,k}}(\Phi_{\alpha,k})+L_{F_{\alpha,k}}(b_\infty^{(g)} F_{\alpha,k}+b_0^{(g)} \Phi_{\alpha,k})\,, \\
 	L_{\Phi_{\alpha,k}}(g) \;&=\; a_0^{(g)} L_{\Phi_{\alpha,k}}(F_{\alpha,k})+a_\infty^{(g)} L_{\Phi_{\alpha,k}}(\Phi_{\alpha,k})+L_{\Phi_{\alpha,k}}(b_\infty^{(g)} F_{\alpha,k}+b_0^{(g)} \Phi_{\alpha,k}) \, .
\end{split}
\]
Moreover, obviously,
\[\tag{ii}
 \begin{split}
  L_{F_{\alpha,k}}(F_{\alpha,k})\;&=\;L_{\Phi_{\alpha,k}}(\Phi_{\alpha,k})\;=\;0\,, \\
     L_{F_{\alpha,k}}(\Phi_{\alpha,k})\;&=\;-W\;=\;-L_{\Phi_{\alpha,k}}(F_{\alpha,k})\,.
 \end{split}
\]
 The following properties are then claimed:
\[\tag{iii}
  L_{F_{\alpha,k}}\big(b_\infty^{(g)} F_{\alpha,k} + b_0^{(g)} \Phi_{\alpha,k}\big)\;=\;0\;=\;L_{\Phi_{\alpha,k}}\big(b_\infty^{(g)} F_{\alpha,k} + b_0^{(g)} \Phi_{\alpha,k}\big)\,.
\]
Plugging (ii) and (iii) into (i) the finiteness
\[
 L_{F_{\alpha,k}}(g)\;=\;-Wa_\infty^{(g)}\,,\qquad L_{\Phi_{\alpha,k}}(g)\;=\;Wa_0^{(g)}
\]
follows.

To prove (iii) one computes
\[
 \begin{split}
  &\det\begin{pmatrix}
       F_{\alpha,k} & b_\infty^{(g)} F_{\alpha,k} + b_0^{(g)} \Phi_{\alpha,k} \\
       F_{\alpha,k}' & (b_\infty^{(g)} F_{\alpha,k} + b_0^{(g)} \Phi_{\alpha,k})'
      \end{pmatrix}= \\
  &\qquad=\;F_{\alpha,k}^2(b_\infty^{(g)})'+F_{\alpha,k}(b_0^{(g)})'\Phi_{\alpha,k}+F_{\alpha,k}b_\infty^{(g)}F_{\alpha,k}'-F_{\alpha,k}'b_0^{(g)} \Phi_{\alpha,k} \\
  &\qquad=\;F_{\alpha,k}b_\infty^{(g)}F_{\alpha,k}'-F_{\alpha,k}'b_0^{(g)} \Phi_{\alpha,k}\,,
 \end{split}
\]
having used the cancellation 
\[
 F_{\alpha,k}^2(b_\infty^{(g)})'+F_{\alpha,k}(b_0^{(g)})'\Phi_{\alpha,k}\;=\;0\,,
\]
that follows from \eqref{eq:b0binfty}. Therefore, by means of the asymptotics \eqref{eq:Asymtotics_0} and \eqref{eq:b_vanishes} as $x\downarrow 0$, namely,
\[
 \begin{split}
  F_{\alpha,k}(x)\;&=\;O(x^{1+\frac{\alpha}{2}})\,,\qquad F_{\alpha,k}'(x)=O(x^{\frac{\alpha}{2}})\,,\qquad\Phi_{\alpha,k}(x)\;=\;O(x^{-\frac{\alpha}{2}})\,, \\
  b_0^{(g)}(x)\;&=\;o(1)\,,\qquad b_\infty^{(g)}(x)\;=\;o(1)\,,
 \end{split}
\]
 one concludes 
\[
  L_{F_{\alpha,k}}\big(b_\infty^{(g)} F_{\alpha,k} + b_0^{(g)} \Phi_{\alpha,k}\big)\;=\;\lim_{x\downarrow 0}\big(F_{\alpha,k}b_\infty^{(g)}F_{\alpha,k}'-F_{\alpha,k}'b_0^{(g)} \Phi_{\alpha,k} \big)\;=\;0\,.
\]
The proof of the second identity in (iii) is completely analogous.

  It remains to demonstrate that $L_u (\varphi) =0$ for $\varphi \in \mathcal{D}(\overline{A_\alpha(k)})$ and $u\in\mathcal{L}$. Although $u$ does not necessarily belong to $\mathcal{D}(A_\alpha(k)^*)$ (it might fail to be square-integrable at infinity), the function $\chi u$ surely does for $\chi \in C^\infty_0([0,+\infty))$ with $\chi(x)=1$ on $x \in[0,\frac{1}{2}]$ and $\chi(x)=0$ on $x \in[1,+\infty)$. This fact follows from \eqref{eq:Afstar} observing that $\chi u\in L^2(\mathbb{R}^+)$ and also
\[
	S_{\alpha,k}(u \chi)\;=\;\chi S_{\alpha,k}\,u- 2u'\chi'-u \chi''\;=\;-2 u'\chi' - u \chi'' \in L^2(\mathbb{R}^+)\,.
\]
The choice of $\chi$ guarantees that the Wronskians $W_x(\overline{u \chi}, g)$ and $W_x(\overline{u},g)$ coincide in a neighbourhood of $x=0$, that is, $L_{u \chi}=L_u$. Therefore, by means of \eqref{eq:boundaryform}, \eqref{eq:omega-W}, and \eqref{eq:LinearLV} one deduces
\[
	\begin{split}
		L_u(\varphi)\;&=\; L_{u \chi}(\varphi)\;=\;\lim_{x \downarrow 0} W_x(\overline{u \chi},\varphi)\;=\;-\omega(u \chi, \varphi)\\
		&=\;\langle u\chi, A_\alpha(k)^* \varphi \rangle - \langle A_\alpha(k)^*u \chi, \varphi \rangle \;=\; \langle u \chi, \overline{A_\alpha(k)} \varphi \rangle - \langle u \chi, \overline{A_\alpha(k)} \varphi \rangle \;=\;0\,,
	\end{split}
\]
which completes the proof.
\end{proof}

With this preparatory material at hand, we can characterise the space $\mathcal{D}(\overline{A_\alpha(k)})$ as follows.

\begin{lemma}\label{prop:EquivalentClosure}
Let $\alpha\in(0,1)$, $ k\in\mathbb{Z}\!\setminus\! \{0\}$, and $\varphi \in \mathcal{D}(A_\alpha(k)^*)$. The following conditions are equivalent:
\begin{itemize}
\item[(i)] $\varphi \in \mathcal{D}(\overline{A_\alpha(k)})$,
\item[(ii)] $\omega(\varphi,g)=0$ for all $g \in \mathcal{D}(A_\alpha(k)^*)$,
\item[(iii)] $L_u(\varphi)=0$ for all $u \in \mathcal{L}$,
\item[(iv)] in the decomposition \eqref{eq:G-ODE-Decomposition} of $\varphi$ one has $a_0^{(\varphi)}=a_\infty^{(\varphi)}=0$.
\end{itemize}
\end{lemma}

\begin{proof}
The implication (i) $\Rightarrow$ (ii) follows at once from
\[
	\omega(\varphi,g)\;=\; \langle A_\alpha(k)^* \varphi, g \rangle - \langle \varphi, A_\alpha(k)^* g \rangle \;=\; \langle \overline{A_\alpha(k)} \varphi, g \rangle - \langle \overline{A_\alpha(k)} \varphi,g \rangle \;=\; 0\,.
\]
For the converse implication (i) $\Leftarrow$ (ii), we observe that the property
\[
	0\;=\;\omega(\varphi,g)\;=\; \langle A_\alpha(k)^* \varphi, g \rangle - \langle \varphi, A_\alpha(k)^* g \rangle \qquad \forall g \in \mathcal{D}(A_\alpha(k)^*)
\]
is equivalent to $\langle A_\alpha(k)^* \varphi,g \rangle = \langle \varphi, A_\alpha(k)^* g \rangle$ $ \forall g \in \mathcal{D}(S^*)$, which implies that $\varphi \in \mathcal{D}(A_\alpha(k)^{**})=\mathcal{D}(\overline{A_\alpha(k)})$.

The implication (i) $\Rightarrow$ (iii) is given by Lemma \ref{lem:LV}. Let us now prove that (iii) $\Rightarrow$ (ii): thus, now $L_u(\varphi)=0$ for all $u \in \mathcal{L}$ and we want to prove that for such $\varphi$ one has $\omega(\varphi,g)=0$ for all $g \in \mathcal{D}(A_\alpha(k)^*)$. Owing to the decomposition \eqref{eq:G-ODE-Decomposition} for $g$,
\[
	\omega(\varphi,g)\;=\; a_0^{(g)} \omega(\varphi,F_{\alpha,k})+a_\infty^{(g)} \omega(\varphi,\Phi_{\alpha,k})+\omega(\varphi, b_\infty^{(g)} F_{\alpha,k})+\omega(\varphi,b_0^{(g)} \Phi_{\alpha,k}) \, .
\]
The first two summands in the r.h.s.~above are zero: indeed,
\[
	\overline{\omega(\varphi,F_{\alpha,k})} \;=\; - \omega(F_{\alpha,k},\varphi)\;=\;\lim_{x \downarrow 0} W_x(\overline{F_{\alpha,k}},\varphi)\;=\;L_{F_{\alpha,k}}(\varphi)\;=\;0\,
\]
having used in the last step the assumption that $L_u(\varphi)=0$ for all $u \in \mathcal{L}$, and analogously, $\overline{\omega(\varphi,\Phi_{\alpha,k})}=L_{\Phi_{\alpha,k}}(\varphi)=0$.
Therefore, 
\[
	\begin{split}
		\overline{\omega(\varphi,g)} \;&=\; \overline{\omega(\varphi, b_\infty^{(g)} F_{\alpha,k})}+\overline{\omega(\varphi,b_0^{(g)} \Phi_{\alpha,k})} \\
		&=\;-\omega(b_\infty^{(g)} F_{\alpha,k},\varphi)-\omega(b_0^{(g)} \Phi_{\alpha,k}, \varphi )\\
		&=\;\lim_{x \downarrow 0} \big( W_x(b_\infty^{(g)} F_{\alpha,k},\varphi)+W_x(b_0^{(g)} \Phi_{\alpha,k}, \varphi )\big)\\
		&=\;\lim_{x \downarrow 0} \big( b_\infty^{(g)}\,W_x( F_{\alpha,k},\varphi)+b_0^{(g)} \,W_x(\Phi_{\alpha,k}, \varphi )\big)\\	
		&=\;b_\infty^{(g)} L_{F_{\alpha,k}}(\varphi)+b_0^{(g)} L_{\Phi_{\alpha,k}}(\varphi)\;=\;0\,,
	\end{split}
\]
having used again the assumption (ii) in the last step (observe also that helpful cancellation $(b_\infty^{(g)})'F_{\alpha,k}\varphi+(b_0^{(g)})'\Phi_{\alpha,k}\varphi=0$ occurred in computing the determinants in the fourth step).

Properties (i), (ii), and (iii) are thus equivalent.
Last, let us establish the equivalence (i) $\Leftrightarrow$ (iv). Representing $\varphi$ according to \eqref{eq:G-ODE-Decomposition} as 
\[
 \varphi\;=\;a_0^{(\varphi)} F_{\alpha,k} + a_\infty^{(\varphi)} \Phi_{\alpha,k} + b_\infty^{(\varphi)} F_{\alpha,k} + b_0^{(\varphi)} \Phi_{\alpha,k}\,,
\]
and using the identities $W_x(F_{\alpha,k},F_{\alpha,k})=0$ and $W_x(F_{\alpha,k},\Phi_{\alpha,k})=-W$,
one has
\[
  L_{F_{\alpha,k}}(\varphi)\;=\;\lim_{x\downarrow 0} W_x(F_{\alpha,k},\varphi)\;=\;-Wa_\infty^{(\varphi)}+\lim_{x\downarrow 0}W_x(F_{\alpha,k},b_\infty^{(\varphi)} F_{\alpha,k} + b_0^{(\varphi)} \Phi_{\alpha,k})\,.
\]
The determinant in the latter Wronskian has the very same form of the determinant computed in the proof of Lemma \ref{lem:LV}: using the same cancellation $F_{\alpha,k}^2(b_\infty^{(\varphi)})'+F_{\alpha,k}(b_0^{(\varphi)})'\Phi_{\alpha,k}=0$ and the usual short-distance asymptotics we find
\[
 L_{F_{\alpha,k}}(\varphi)\;=\;-Wa_\infty^{(\varphi)}\,.
\]
In a completely analogous fashion,
\[
 L_{\Phi_{\alpha,k}}(\varphi)\;=\;Wa_0^{(\varphi)}\,.
\]
Therefore, $\varphi \in \mathcal{D}(\overline{A_\alpha(k)})$ if and only if $L_u(\varphi)=0$ for all $u \in \mathcal{L}$ (because (i) $\Leftrightarrow$ (iii)), and the latter property is equivalent to $a_0^{(\varphi)}=a_\infty^{(\varphi)}=0$.
\end{proof}

We can now characterise the short-distance behaviour of the functions in $\mathcal{D}(\overline{A_\alpha(k)})$ and of their derivative.

\begin{lemma}\label{lem:BehaviourZeroClosure}
Let $\alpha\in(0,1)$ and $ k\in\mathbb{Z}\setminus \{0\}$. If 
$\varphi \in \mathcal{D}(\overline{A_\alpha(k)})$, then $\varphi(x)=o(x^{\frac{3}{2}})$ and $\varphi'(x)=o(x^{\frac{1}{2}})$ as $x \downarrow 0$.
\end{lemma}

\begin{proof}
Owing to Lemma \ref{prop:EquivalentClosure},
\[
 \varphi\;=\;b_\infty^{(\varphi)} F_{\alpha,k} + b_0^{(\varphi)} \Phi_{\alpha,k}\,.
\]
Thus, $\varphi=o(x^{\frac{3}{2}})$ follows from  \eqref{eq:sdasimpt} of Lemma \ref{lem:odedecomp}. Moreover,
 \[
 \varphi'\;=\;\big(b_\infty^{(\varphi)} F_{\alpha,k} + b_0^{(\varphi)} \Phi_{\alpha,k}\big)'\;=\;b_\infty^{(\varphi)} F_{\alpha,k}' + b_0^{(\varphi)} \Phi_{\alpha,k}'\,,
\]
thanks to the cancellation $(b_\infty^{(\varphi)})'F_{\alpha,k}+(b_0^{(\varphi)})'\Phi_{\alpha,k}=0$ that follows from \eqref{eq:b0binfty}.
From the short-distance asymptotics \eqref{eq:Asymtotics_0} one has
\[
 \begin{split}
  F_{\alpha,k}(x)\;&=\;O(x^{1+\frac{\alpha}{2}})\,,\qquad F_{\alpha,k}'(x)=O(x^{\frac{\alpha}{2}})\,, \\
  \Phi_{\alpha,k}(x)\;&=\;O(x^{-\frac{\alpha}{2}})\,,\qquad\! \Phi_{\alpha,k}(x)'\;=\;O(x^{-(1+\frac{\alpha}{2})})\,,
 \end{split}
\]
whence
\[
 \begin{split}
  |b_\infty^{(\varphi)}(x) F_{\alpha,k}'(x)|\;&\lesssim\;x^{\frac{\alpha}{2}}\,\|A_\alpha(k)^*\varphi\|_{L^2((0,x))}\Big(\int_0^x |\rho^{-\frac{\alpha}{2}}|^2\ud \rho\Big)^{\!\frac{1}{2}} \\
  &\lesssim\;x^{\frac{1}{2}}\,\|\overline{A_\alpha(k)}\varphi\|_{L^2((0,x))}\;=\;o(x^{\frac{1}{2}})\,,
 \end{split}
\]
and also
\[
 \begin{split}
  |b_0^{(\varphi)}(x) \Phi_{\alpha,k}'(x)|\;&\lesssim\;\frac{1}{\:x^{1+\frac{\alpha}{2}}}\,\|A_\alpha(k)^*\varphi\|_{L^2((0,x))}\Big(\int_0^x |\rho^{1+\frac{\alpha}{2}}|^2\ud \rho\Big)^{\!\frac{1}{2}} \\
  &\lesssim\;x^{\frac{1}{2}}\,\|\overline{A_\alpha(k)}\varphi\|_{L^2((0,x))}\;=\;o(x^{\frac{1}{2}})\,.
 \end{split}
\]
The proof is thus completed.
\end{proof}

We are finally in the condition to prove Proposition \ref{prop:domAclosure}.

\begin{proof}[Proof of Proposition \ref{prop:domAclosure}]
 Let us first prove the inclusion
 \[\tag{*}
 H^2_0(\mathbb{R}^+)\cap L^2(\mathbb{R}^+,\langle x\rangle^{4\alpha}\,\ud x)\;\subset\;\mathcal{D}(\overline{A_\alpha(k)})\,.
 \]
 If $\varphi$ belongs to the space on the l.h.s.~of (*), then $\varphi''\in L^2(\mathbb{R})$, $x^{2\alpha}\varphi\in L^2(\mathbb{R})$, and $\varphi(x)=o({x^{\frac{3}{2}}})$ as $x\downarrow 0$, whence also $x^{-2}\varphi\in L^2(\mathbb{R})$. As a consequence, $-\varphi''+ k^2 x^{2\alpha}\varphi+C_\alpha x^{-2}\varphi\in L^2(\mathbb{R})$, i.e., owing to \eqref{eq:Afstar}, $\varphi\in\mathcal{D}(A_\alpha(k)^*)$. Representing now $\varphi$ according to \eqref{eq:G-ODE-Decomposition} as
 \[
  \varphi\;=\;a_0^{(\varphi)} F_{\alpha,k} + a_\infty^{(\varphi)} \Phi_{\alpha,k} + b_\infty^{(\varphi)} F_{\alpha,k} + b_0^{(\varphi)} \Phi_{\alpha,k}\,,
 \]
 we deduce that $a_0^{(\varphi)}=a_\infty^{(\varphi)}=0$, for otherwise the behaviour \eqref{eq:Asymtotics_0} of $\Phi_{\alpha,k}$ and $F_{\alpha,k}$ as $x\downarrow 0$ would be incompatible with $\varphi(x)=o({x^{\frac{3}{2}}})$. Instead, the component $b_\infty^{(\varphi)} F_{\alpha,k} + b_0^{(\varphi)} \Phi_{\alpha,k}$ displays the $o({x^{\frac{3}{2}}})$-behaviour, as we see from \eqref{eq:sdasimpt}. Lemma \ref{prop:EquivalentClosure} then implies $\varphi\in\mathcal{D}(\overline{A_\alpha(k)})$, which proves (*).

 Next, let us prove the opposite inclusion
 \[\tag{**}
 H^2_0(\mathbb{R}^+)\cap L^2(\mathbb{R}^+,\langle x\rangle^{4\alpha}\,\ud x)\;\supset\;\mathcal{D}(\overline{A_\alpha(k)})\,.
 \]
 Owing to Lemma \ref{eq:RGinvertsExtS} there exists a self-adjoint extension $\mathscr{A}_\alpha(k)$ of $\overline{A_\alpha(k)}$ with $\mathcal{D}(\mathscr{A}_\alpha(k))=\mathrm{ran}R_{G_{\alpha,k}}$, and owing to Corollary \ref{cor:RGtoWeightedL2} $\mathrm{ran}R_{G_{\alpha,k}}\subset L^2(\mathbb{R}^+,\langle x\rangle^{4\alpha}\,\ud x)$. Therefore, $\mathcal{D}(\overline{A_\alpha(k)})\subset L^2(\mathbb{R}^+,\langle x\rangle^{4\alpha}\,\ud x)$. It remains to prove that $\mathcal{D}(\overline{A_\alpha(k)})\subset H^2_0(\mathbb{R}^+)$. For $\varphi\in \mathcal{D}(\overline{A_\alpha(k)})\subset\mathcal{D}(A_\alpha(k)^*)$ formula \eqref{eq:Afstar} prescribes that $g:=-\varphi''+ k^2 x^{2\alpha}\varphi+C_\alpha x^{-2}\varphi\in L^2(\mathbb{R})$. As proved right above, $x^{2\alpha}\varphi\in L^2(\mathbb{R})$, whereas the property $x^{-2}\varphi\in L^2(\mathbb{R}^+)$ follows from Lemma \ref{lem:BehaviourZeroClosure}. Then by linearity $\varphi''\in L^2(\mathbb{R}^+)$, which also implies $\varphi\in H^2(\mathbb{R}^+)$ by standard arguments \cite[Remark 4.21]{Grubb-DistributionsAndOperators-2009}. Lemma \ref{lem:BehaviourZeroClosure} ensures that $\varphi(0)=\varphi'(0)=0$, and we conclude (see \eqref{eq:H20} above) that $\varphi\in H^2_0(\mathbb{R}^+)$. This completes the proof of (**).\end{proof}

\subsection{Distinguished extension and induced classification}\label{subsec:distinguished}~

In the Kre\u{\i}n-Vi\v{s}ik-Birman extension scheme one characterises all self-adjoint extensions of $A_\alpha(k)$ in terms of a \emph{reference} extension with everywhere defined bounded inverse: the Friedrichs extension $A_{\alpha,F}(k)$ is surely so, since the bottom of $A_\alpha(k)$ is strictly positive, as seen in \eqref{eq:Axibottom} above.

In fact, we have not characterised $A_{\alpha,F}(k)$ yet, which we will be able to do at a later stage, and we shall rather implement the classification scheme with respect to another distinguished extension, precisely the extension $\mathscr{A}_\alpha(k)$ determined in Lemma \ref{eq:RGinvertsExtS}. All this is only going to be temporary, and will allow us to recognise that $\mathscr{A}_\alpha(k)=A_{\alpha,F}(k)$.

When $\mathscr{A}_\alpha(k)$ is taken as a reference, the other self-adjoint extensions of $A_\alpha(k)$ constitute a one-real-parameter-family $\{ A_\alpha^{[\beta]}(k)\,|\,\beta\in\mathbb{R}\}$ (because, as recalled already from \cite[Corollary 3.8]{GMP-Grushin-2018}, the deficiency index of $A_\alpha(k)$ is 1), each element of which, according to the classification a la Kre\u{\i}n-Vi\v{s}ik-Birman \cite[Theorem 3.4]{GMO-KVB2017} and Grubb \cite[Corollary 13.12]{Grubb-DistributionsAndOperators-2009}, is given by
\begin{equation}\label{eq:temp_fibre_classif-prelim}
 \begin{split}
  \mathcal{D}(A_\alpha^{[\beta]}(k))\;&:=\;\big\{\,g=\varphi+c\beta\mathscr{A}_\alpha(k)^{-1}\Phi_{\alpha,k}+c\,\Phi_{\alpha,k}\,\big|\,\varphi\in\mathcal{D}(\overline{A_\alpha(k)})\,,\;c\in\mathbb{C}\big\}\,, \\
  A_\alpha^{[\beta]}(k)g\;&:=\;A_\alpha(k)^*g\;=\;\overline{A_\alpha(k)}\,\varphi+c\beta\Phi_{\alpha,k}\,.
 \end{split}
\end{equation}
It is also standard (see, e.g., \cite[Theorem 1]{GMO-KVB2017}) that
\begin{equation}\label{eq:Dadjoint-prelim}
 \begin{split}
  \mathcal{D}(A_\alpha(k)^*)\;&=\;\mathcal{D}(\overline{A_\alpha(k)})\dotplus\mathscr{A}_\alpha(k)^{-1}\mathrm{span}\{\Phi_{\alpha,k}\}\dotplus\mathrm{span}\{\Phi_{\alpha,k}\} \,,\\
  \mathcal{D}(\mathscr{A}_\alpha(k))\;&=\;\mathcal{D}(\overline{A_\alpha(k)})\dotplus\mathscr{A}_\alpha(k)^{-1}\mathrm{span}\{\Phi_{\alpha,k}\}\,.
 \end{split}
\end{equation}

Owing to Lemma \ref{eq:RGinvertsExtS} and to \eqref{eq:defPsi} we can re-write \eqref{eq:temp_fibre_classif-prelim} and \eqref{eq:Dadjoint-prelim} as 
\begin{equation}\label{eq:temp_fibre_classif}
 \begin{split}
  \mathcal{D}(A_\alpha^{[\beta]}(k))\;&=\;\big\{\,g=\varphi+c(\beta\, \Psi_{\alpha,k}+\Phi_{\alpha,k})\,\big|\,\varphi\in\mathcal{D}(\overline{A_\alpha(k)})\,,\;c\in\mathbb{C}\big\}\,, \\
  A_\alpha^{[\beta]}(k)g\;&=\;A_\alpha(k)^*g\;=\;\overline{A_\alpha(k)}\,\varphi+c\,\beta\,\Phi_{\alpha,k}
 \end{split}
\end{equation}
and 
\begin{equation}\label{eq:Dadjoint}
 \begin{split}
  \mathcal{D}(A_\alpha(k)^*)\;&=\;\mathcal{D}(\overline{A_\alpha(k)})\dotplus\mathrm{span}\{\Psi_{\alpha,k}\}\dotplus\mathrm{span}\{\Phi_{\alpha,k}\}\,, \\
  \mathcal{D}(\mathscr{A}_\alpha(k))\;&=\;\mathcal{D}(\overline{A_\alpha(k)})\dotplus\mathrm{span}\{\Psi_{\alpha,k}\}\,.
 \end{split}
\end{equation}

By comparing \eqref{eq:Dadjoint} with the short-range asymptotics for $\Phi_{\alpha,k}$ (formula \eqref{eq:Asymtotics_0} above), for $\Psi_{\alpha,k}$ (Lemma \ref{lem:Psi_asymptotics}), and for the elements of $\mathcal{D}(\overline{A_\alpha(k)})$ (Lemma \ref{lem:BehaviourZeroClosure}), it is immediate to deduce that for a function
\begin{equation}\label{eq:g_in_Dstar}
 g\;=\;\varphi+c_1\Psi_{\alpha,k}+c_0\Phi_{\alpha,k}\;\in\;\mathcal{D}(A_\alpha(k)^*)
\end{equation}
(with $\varphi\in \mathcal{D}(\overline{A_\alpha(k)})$ and $c_0,c_1\in\mathbb{C}$)
the limits
\begin{equation}\label{eq:limitsg0g1}
 \begin{split}
  g_0\;&:=\;\lim_{x\downarrow 0}\,x^{\frac{\alpha}{2}}g(x)\;=\;c_0{\textstyle\sqrt{\frac{\pi(1+\alpha)}{2|k|}}}\,, \\
  g_1\;&:=\;\lim_{x\downarrow 0}\,x^{-(1+\frac{\alpha}{2})}(g(x)-g_0 x^{-\frac{\alpha}{2}})\;=\;c_1{\textstyle\sqrt{\frac{2|k|}{\pi(1+\alpha)^3}}}\,\|\Phi_{\alpha,k}\|_{L^2}^2-c_0{\textstyle\sqrt{\frac{\pi|k|}{2(1+\alpha)}}} 
 \end{split}
\end{equation}
exist and are finite, and one has the asymptotics
\begin{equation}\label{eq:gatzero}
 g(x)\;\stackrel{x\downarrow 0}{=}\;  g_0x^{-\frac{\alpha}{2}}+g_1x^{1+\frac{\alpha}{2}}+o(x^{\frac{3}{2}})\,.
\end{equation}

In turn, by comparing \eqref{eq:g_in_Dstar} with \eqref{eq:temp_fibre_classif} we see that for given $\beta$ the domain of the extension $A_\alpha^{[\beta]}(k)$ consists of all those $g$'s in $\mathcal{D}(A_\alpha(k)^*)$ that, decomposed as in \eqref{eq:g_in_Dstar}, satisfy the condition
\begin{equation}\label{eq:c1betac0}
 c_1\;=\;\beta\,c_0\,.
\end{equation}
Moreover, replacing $c_0$ and $c_1$ of the expression \eqref{eq:g_in_Dstar} with $g_0$ and $g_1$ according to \eqref{eq:limitsg0g1}, the self-adjointness condition \eqref{eq:c1betac0} takes the form
\begin{equation}\label{eq:g1gammag0}
 g_1\;=\;\gamma \,g_0\,,\qquad \gamma\;:=\;\frac{|k|}{1+\alpha}\Big(\frac{\, 2 \|\Phi_{\alpha,k}\|_{L^2}^2 \,}{\pi(1+\alpha)}\, \beta-1\Big)\,.
\end{equation}
We can therefore equivalently parametrise each extension with the new real parameter $\gamma$ and write $A_\alpha^{[\gamma]}(k)$ in place of $A_\alpha^{[\beta]}(k)$, with $\beta$ and $\gamma$ linked by \eqref{eq:g1gammag0}.

We have thus proved the following.

\begin{proposition}\label{prop:temp-class}
 Let $\alpha\in(0,1)$ and $ k\in\mathbb{Z}\setminus\{0\}$. The self-adjoint extensions of $A_\alpha(k)$ in $L^2(\mathbb{R}^+)$ form the family $\{ A_\alpha^{[\gamma]}(k)\,|\,\gamma\in\mathbb{R}\cup\{\infty\}\}$. The extension with $\gamma=\infty$ is the reference extension $\mathscr{A}_\alpha(k)=R_{G_{\alpha,k}}^{-1}$, where $R_{G_{\alpha,k}}$ is the operator defined by \eqref{eq:newRGalphaforus}. For generic $\gamma\in\mathbb{R}$ one has
 \begin{equation}\label{eq:tempclass}
 \begin{split}
  A_\alpha^{[\gamma]}(k)\;&=\;A_\alpha(k)^*\Big|_{\mathcal{D}(A_\alpha^{[\gamma]}(k))}\,, \\
  \mathcal{D}(A_\alpha^{[\gamma]}(k))\;&=\;\{g\in\mathcal{D}(A_\alpha(k)^*)\,|\,g_1=\gamma g_0\}\,,
 \end{split}
 \end{equation}
 where, for each $g$, the constants $g_0$ and $g_1$ are defined by the limits \eqref{eq:limitsg0g1}.
\end{proposition}

Although the above classification is not yet in the final form we wish, it allows us to make now an important identification.

\begin{proposition}\label{eq:RGisSFinv}
 Let $\alpha\in(0,1)$ and $ k\in\mathbb{Z}\setminus\{0\}$. Then  $\mathscr{A}_\alpha(k)=A_{\alpha,F}(k)$, and hence $R_{G_{\alpha,k}}=A_{\alpha,F}(k)^{-1}$ and $\Psi_{\alpha,k}=(A_{\alpha,F}(k))^{-1}\Phi_{\alpha,k}$.
\end{proposition}

For the proof of Proposition \ref{eq:RGisSFinv} it is convenient to recall the following.

\begin{lemma}\label{lem:Fform}
 Let $\alpha\in(0,1)$ and $ k\in\mathbb{Z}\setminus\{0\}$. The quadratic form of the Friedrichs extension of $A_\alpha(k)$ is given by
 \begin{equation}\label{eq:Fform}
  \begin{split}
  \mathcal{D}[A_{\alpha,F}(k)]\;&=\;\big\{g\in L^2(\mathbb{R}^+)\,\big|\, \|g'\|_{L^2}^2+\|x^{\alpha} g\|_{L^2}^2+\|x^{-1}g\|_{L^2}^2<+\infty\big\} \,,\\
   A_{\alpha,F}(k)[g,h]\;&=\;\int_0^{+\infty}\!\!\Big(\,\overline{g'(x)}h'(x)+ k^2x^{2\alpha}\,\overline{g(x)}h(x)+C_\alpha\frac{\,\overline{g(x)}h(x)}{\,x^2}\Big)\ud x\,.
  \end{split}
 \end{equation}
 \end{lemma}

\begin{proof}
 It is a standard construction (see, e.g., \cite[Theorem 15]{GMO-KVB2017}), that follows from the fact that $\mathcal{D}[A_{\alpha,F}(k)]$ is the closure of $\mathcal{D}(A_\alpha(k))=C^\infty_c(\mathbb{R}^+)$ in the norm
 \[
  \begin{split}
   \|g\|_F^2\;:=&\;\langle g,A_\alpha(k) g\rangle_{L^2}+\langle g,g\rangle_{L^2} \\
   =&\; \|g'\|_{L^2}^2+ k^2\|x^{\alpha} g\|_{L^2}^2+C_\alpha\|x^{-1}g\|_{L^2}^2+\|g\|_{L^2}^2\,.
  \end{split}
 \]
 Then \eqref{eq:Fform} follows at once from the above formula, since $ k^2> 0$ and $C_\alpha>0$.
\end{proof}

\begin{proof}[Proof of Proposition \ref{eq:RGisSFinv}]
%
%
%

 Let $g\in\mathcal{D}(A_\alpha^{[\gamma]}(k))$ for some $\gamma\in\mathbb{R}$. 
 The short-distance expansion \eqref{eq:gatzero}, combined with the self-adjointness condition \eqref{eq:tempclass}, yields
 \[
  x^{-1}g(x)\;\stackrel{x\downarrow 0}{=}\;g_0\, x^{-(1+\frac{\alpha}{2})}+\gamma\,g_0\, x^{\frac{\alpha}{2}}+o(x^{\frac{1}{2}})\,.
 \]
 Therefore, in general (namely whenever $g_0\neq 0$) $x^{-1}g$ is \emph{not} square-integrable at zero. When this is the case, formula \eqref{eq:Fform} prevents $g$ from belonging to $\mathcal{D}[A_{\alpha,F}(k)]$.
 This shows that \emph{no} extension $ A_\alpha^{[\gamma]}(k)$, $\gamma\in\mathbb{R}$, has operator domain entirely contained in $\mathcal{D}[A_{\alpha,F}(k)]$. The latter statement does not cover  $\mathscr{A}_\alpha(k)$ ($\gamma=\infty$). Now, $A_{\alpha,F}(k)$ can be none of the $ A_\alpha^{[\gamma]}(k)$'s, $\gamma\in\mathbb{R}$, because the Friedrichs extension has indeed  operator domain inside $\mathcal{D}[A_{\alpha,F}(k)]$ -- in fact, it is the unique extension with such property.
 Necessarily the conclusion is that $A_{\alpha,F}(k)$ and $\mathscr{A}_\alpha(k)$ are the same.
%
%
%
%
%
\end{proof}

A straightforward consequence of Proposition \ref{eq:RGisSFinv} (and of its proof) is the following.

\begin{corollary}\label{cor:AF_in_x-1}
Let $\alpha\in(0,1)$ and $ k\in\mathbb{Z}\setminus\{0\}$.  
 The Friedrichs extension $A_{\alpha,F}(k)$ of $A_\alpha(k)$ is the only self-adjoint extension whose operator domain is contained in $\mathcal{D}(x^{-1})$. 
\end{corollary}

\subsection{Proof of the classification theorem on fibre}\label{subsec:proof_of_fibrethm}~

Let us collect the results of the preceding discussion and prove Theorem \ref{thm:fibre-thm}.

Clearly, the case when $\alpha=0$ and hence $A_\alpha(k)$ is (a positive shift of) the minimally defined Laplacian, is already well known in the literature (see, e.g., \cite{GTV-2012}) and in this case Theorem \ref{thm:fibre-thm} provides familiar information. In particular, the operator closure has domain $H^2_0(\mathbb{R}^+)$, the adjoint has domain $H^2(\mathbb{R}^+)$, the Friedrichs extension is the Dirichlet Laplacian and has form domain $H^1_0(\mathbb{R}^+)$, etc.

Thus, Theorem \ref{thm:fibre-thm} need only be proved when $\alpha\in(0,1)$, the regime in which the analysis of Subsections \ref{subsec:homog_problem}-\ref{subsec:distinguished} was developed.

Part (i) of Theorem \ref{thm:fibre-thm} is precisely Proposition \ref{prop:domAclosure}. Part (ii) follows from \eqref{eq:Afstar} and \eqref{eq:Dadjoint} concerning the operator domain, and from Lemma \ref{lem:kerAxistar} concerning the kernel.

Part (iv), the actual classification of extensions, is the rephrasing of Proposition \ref{prop:temp-class}, using the fact that the reference extension is $\mathscr{A}_\alpha(k)=A_{\alpha,F}(k)$ (Proposition \ref{eq:RGisSFinv}), and plugging the self-adjointness condition $g_1=\gamma g_0$ into the general asymptotics \eqref{eq:gatzero}.

In part (iii), formula \eqref{eq:thmAFoperator} for the operator domain follows from \eqref{eq:Dadjoint} (with $\mathscr{A}_\alpha(k)=A_{\alpha,F}(k)$) and from the short-range asymptotics for $\Psi_{\alpha,k}$ (Lemma \ref{lem:Psi_asymptotics}), and for the elements of $\mathcal{D}(\overline{A_\alpha(k)})$ (Lemma \ref{lem:BehaviourZeroClosure}) -- which is the same as taking formally $\gamma=\infty$ in the general asymptotics. The distinctive property of $A_{\alpha,F}(k)$ with respect to the space $\mathcal{D}(x^{-1})$ is given by Corollary \ref{cor:AF_in_x-1}.

Thus, it remains to prove \eqref{eq:thmAFform} for the form domain of $A_{\alpha,F}(k)$. The inclusion $\mathcal{D}[A_{\alpha,F}(k)]\subset H^1_0(\mathbb{R}^+)\cap L^2(\mathbb{R}^+,\langle x\rangle^{2\alpha}\,\ud x)$ follows directly from Lemma \ref{lem:Fform}, as \eqref{eq:Fform} prescribes that if $g\in\mathcal{D}[A_{\alpha,F}(k)]$, then $g',x^\alpha g,x^{-1}g\in L^2(\mathbb{R}^+)$, and the latter condition implies necessarily $g(0)=0$. Conversely, if $g\in H^1_0(\mathbb{R}^+)$ \emph{and} $g\in L^2(\mathbb{R}^+,\langle x\rangle^{2\alpha}\,\ud x)$, then $g(x)\stackrel{x\downarrow 0}{=}o(x^{\frac{1}{2}})$ and all three norms $\|g'\|_{L^2}$, $\|x^{\alpha} g\|_{L^2}$, and $\|x^{-1}g\|_{L^2}$ are finite. Owing to \eqref{eq:Fform}, $g\in \mathcal{D}[A_{\alpha,F}(k)]$.

The proof of Theorem \ref{thm:fibre-thm} is completed.

\section{Continuation: the mode $k=0$}\label{sec:zero_mode}

We discuss now how the analysis of the previous Section is to be modified when $k=0$. We follow the same conceptual scheme, but applying it now to the \emph{shifted} operator $A_\alpha(0)+\mathbbm{1}$: owing to \eqref{eq:Axibottom-zero}, such a (densely defined, symmetric) operator has strictly positive bottom.

Thus, whereas for $k\neq 0$ self-adjoint extensions were determined a la Kre\u{\i}n-Vi\v{s}ik-Birman by implementing the self-adjointness condition between regular and singular part of the domain of the adjoint
\[
 \mathcal{D}(A_\alpha(k)^*)\;=\;\mathcal{D}(\overline{A_\alpha(k)})\dotplus(A_{\alpha,F}(k))^{-1}\ker A_\alpha(k)^*\dotplus\ker A_\alpha(k)^*\,,
\]
when $k=0$ the self-adjointness condition is implemented as a restriction in the formula
\[
 \begin{split}
  \mathcal{D}(A_\alpha(0)^*+\mathbbm{1})\;&=\;\mathcal{D}(\overline{A_\alpha(0)}+\mathbbm{1})\dotplus \\
  &\qquad \dotplus(A_{\alpha,F}(0)+\mathbbm{1})^{-1}\ker (A_\alpha(0)^*+\mathbbm{1})\dotplus\ker (A_\alpha(0)^*+\mathbbm{1})\,,
 \end{split}
\]
where obviously $\mathcal{D}(A_\alpha(0)^*+\mathbbm{1})=\mathcal{D}(A_\alpha(0)^*)$ and $\mathcal{D}(\overline{A_\alpha(0)}+\mathbbm{1})=\mathcal{D}(\overline{A_\alpha(0)})$, and analogously the domain of each extension is insensitive to the shift by $\mathbbm{1}$. The main result is Theorem \ref{thm:fibre-thm_zero_mode} below.

In fact, by other means and from a different perspective, the extensions of $A_\alpha(0)$ were also determined in \cite{Bruneau-Derezinski-Georgescu-2011}: we shall therefore omit an amount of details that can be either worked out in the very same manner of Sect.~\ref{sec:fibre-extensions}, or can be found in \cite{Bruneau-Derezinski-Georgescu-2011}.

Let us start with the homogeneous problem 
\begin{equation}\label{eq:homokzero}
 0\;=\;(S_{\alpha,0}+\mathbbm{1})h\;=\;-h''+C_\alpha x^{-2} h+h\,.
\end{equation}
Setting
\[
 w(z)\;:=\;\frac{h(x)}{\sqrt{x}}\,,\qquad \nu\;:=\;\sqrt{\frac{1+4C_\alpha}{4}}\;=\;\frac{1+\alpha}{2}\,,
\]
\eqref{eq:homokzero} takes the form of the modified Bessel equation
\begin{equation}
 x^2w''+x w'-(z^2+\nu^2)w\;=\;0\,,\qquad x\in\mathbb{R}^+\,.
\end{equation}
From the two linearly independent solutions $K_\nu$ and $I_\nu$ to the latter \cite[Sect.~9.6]{Abramowitz-Stegun-1964} we therefore have that
\begin{equation}\label{eq:Phi_and_F_mode_0}
\begin{split}
	\Phi_{\alpha,0}(x)\;&:=\;\sqrt{x}\, K_{\frac{1+\alpha}{2}}(x)\,, \\
	F_{\alpha,0}(x)\;&:=\; \sqrt{x}\, I_{\frac{1+\alpha}{2}}(x)
\end{split}
\end{equation}
are two linearly independent solutions to \eqref{eq:homokzero}. In fact, only $\Phi_{\alpha,0}$ is square-integrable, as is seen from the short-distance asymptotics \cite[Eq.~(9.6.2) and (9.6.10)]{Abramowitz-Stegun-1964} 
\begin{equation}\label{eq:AsymPhi00}
\begin{split}
	\Phi_{\alpha,0}\;&\overset{x \downarrow 0}{=}\;2^{\frac{\alpha-1}{2}} \Gamma\left({\textstyle\frac{1+\alpha}{2}}\right) x^{-\frac{\alpha}{2}}-
	 \textstyle{\frac{\Gamma\left({\textstyle\frac{1-\alpha}{2}} \right)}{2^{\frac{1+\alpha}{2}}(1+\alpha)}}  x^{1+\frac{\alpha}{2}}
	  + O(x^{2-\frac{\alpha}{2}})\,,\\
	F_{\alpha,0}(x) \;&\overset{x \downarrow 0}{=}\; \textstyle{\big(2^{\frac{1+\alpha}{2}}\Gamma\left(\frac{3+\alpha}{2}\right)\big)}^{-1} x^{1+\frac{\alpha}{2}} + O(x^{3+\frac{\alpha}{2}})\,,
\end{split}
\end{equation}
and from the large-distance asymptotics \cite[Eq.~(9.7.1) and (9.7.2)]{Abramowitz-Stegun-1964} 
\begin{equation}\label{eq:Asym_Phi_F_0_Infty}
\begin{split}
	\Phi_{\alpha,0}(x)\;&\overset{x \to +\infty}{=}\; \textstyle{\sqrt{\frac{\pi}{2}}}\, e^{-x}\, (1+O(x^{-1}))\,,\\
		F_{\alpha,0}(x) \;&\overset{x \to +\infty}{=}\; \textstyle{\frac{1}{\sqrt{2 \pi}}}\, e^{x} \,(1+O(x^{-1})) \, .
\end{split}
\end{equation}

Thus, in analogy to Lemma \ref{lem:kerAxistar}, we find:
\begin{lemma}\label{lem:kerAxistarzero}
 For $\alpha\in(0,1)$,
 \begin{equation}\label{eq:kerAxistarzero}
  \ker (A_\alpha(0)^*+\mathbbm{1})\;=\;\mathrm{span}\{\Phi_{\alpha,0}\}\,.
 \end{equation}
\end{lemma}

Next, concerning the non-homogeneous problem
\begin{equation}\label{eq:Inhomog_ODE_mode_0}
	S_{\alpha,0}u + u \; = \; g
\end{equation}
in the unknown $u$ for given $g$, the Wronskian relative to the fundamental system $\{\Phi_{\alpha,0}, F_{\alpha,0}\}$ is constant in $r$ and explicitly given by
\begin{equation}
	W(\Phi_{\alpha,0},F_{\alpha,0})\; =\; \det \begin{pmatrix}
		\Phi_{\alpha,0}(r) & F_{\alpha,0}(r) \\
		\Phi_{\alpha,0}'(r) & F_{\alpha,0}'(r)	
	\end{pmatrix}
	\; =\; 1\,,
\end{equation}
as one computes based on the asymptotics \eqref{eq:AsymPhi00} or \eqref{eq:Asym_Phi_F_0_Infty}. By standard variation of constants, a particular solution to \eqref{eq:Inhomog_ODE_mode_0} is
\begin{equation}
	u_{\text{part}} (r) \; = \; \int_0^{+\infty} G_{\alpha,0}(r,\rho) g(\rho) \, \ud \rho
\end{equation}
with
\begin{equation}\label{eq:GreenMode0}
	G_{\alpha,0}(r,\rho) \; := \; \begin{cases}
		\Phi_{\alpha,0}(r) F_{\alpha,0}(\rho)\,, \qquad \text{if $0<\rho<r$}\,, \\
		F_{\alpha,0}(r) \Phi_{\alpha,0}(\rho)\,, \qquad \text{if $0< r < \rho$}\,.
	\end{cases}
\end{equation}
With the same arguments used for Lemma  \ref{lem:RGbddsa}, using now the asymptotics \eqref{eq:AsymPhi00}-\eqref{eq:Asym_Phi_F_0_Infty}, we find the following analogue (an explicit proof of which can be found also in \cite[Lemma 4.4]{Bruneau-Derezinski-Georgescu-2011}).

\begin{lemma}
Let $\alpha \in (0,1)$. Let $R_{G_{\alpha,0}}$ be the operator associated to the integral kernel \eqref{eq:GreenMode0}. $R_{G_{\alpha,0}}$ can be realised as an everywhere defined, bounded, and self-adjoint operator on $L^2(\mathbb{R}^+,\ud r)$.
\end{lemma}

Analogously to \eqref{eq:defPsi} we set 
\begin{equation}\label{eq:defPsi_zero_mode}
	\Psi_{\alpha,0}(x)\;:=\; R_{G_{\alpha,0}} \Phi_{\alpha,0}\,.
\end{equation}
The proof of Lemma \ref{lem:Psi_asymptotics} can be then repeated verbatim, with $\Phi_{\alpha,0}$ and $F_{\alpha,0}$ in place of $\Phi_{\alpha,k}$ and $F_{\alpha,k}$, so as to obtain:

\begin{lemma}\label{lem:Psi_asymptotics_mode_zero}
 For $\alpha\in(0,1)$,
 \begin{equation}\label{eq:Psi_Asymptotics_mode_zero}
	\Psi_{\alpha,0}(x)\;\overset{x \downarrow 0}{=}\; \textstyle{\big( 2^{\frac{1+\alpha}{2}}\Gamma\left(\frac{3+\alpha}{2}\right)\!\big)}^{\!-1}\, \Vert \Phi_{\alpha,0} \Vert_{L^2}^2\, x^{1+\frac{\alpha}{2}} + o(x^{\frac{3}{2}})\,.
\end{equation}
\end{lemma}

Concerning $\overline{A_\alpha(0)}$, it suffices for our purposes to import from the literature the following analogue of Lemma \ref{lem:BehaviourZeroClosure}.

\begin{lemma}\label{lem:BehaviourZeroClosure_zero_mode}
Let $\alpha \in (0,1)$. Then $\mathcal{D}(\overline{A_\alpha(0)})=H^2_0(\mathbb{R}^+)$. In particular, every $\varphi \in \mathcal{D}(\overline{A_\alpha(0)})$ satisfies $\varphi(x)=o(x^{\frac{3}{2}})$ and $\varphi'(x)=o(x^{\frac{1}{2}})$ as $x \downarrow 0$.
\end{lemma}

\begin{proof}
A direct consequence of \cite[Theorem 4.1]{Derezinski-Georgescu-2021}: in the notation therein $\overline{A_\alpha(0)}$ is the operator $L_\delta^\mathrm{min}$ with $\delta-\frac{1}{4}=C_\alpha$ (the present $\delta$ replaces the notation $\alpha$ from \cite{Derezinski-Georgescu-2021} so as not to clash with the current meaning of $\alpha$), that is, $\delta=(\frac{1+\alpha}{2})^2$; the requirement $\mathfrak{Re}\sqrt{\delta}<1$ needed for the applicability of \cite[Theorem 4.1]{Derezinski-Georgescu-2021} is therefore satisfied, since $\alpha\in(0,1)$.
\end{proof}

As a further step, repeating the argument for Lemma \ref{eq:RGinvertsExtS} one concludes that $R_{G_{\alpha,0}}^{-1}$ is a self-adjoint extension of $A_\alpha(0)+\mathbbm{1}$ with everywhere defined and bounded inverse, whose domain clearly contains $\Psi_{\alpha,0}$. Such a reference extension induces a classification of all other self-adjoint extensions in complete analogy to what discussed in Subsect.~\ref{subsec:distinguished}. Thus, \eqref{eq:temp_fibre_classif} and \eqref{eq:Dadjoint} are valid in the identical form also when $k=0$, and the short-range asymptotics for $\Phi_{\alpha,0}$ (formula \eqref{eq:AsymPhi00}), for $\Psi_{\alpha,0}$ (Lemma \ref{lem:Psi_asymptotics_mode_zero}), and for the elements of $\mathcal{D}(\overline{A_\alpha(0)})$ (Lemma \ref{lem:BehaviourZeroClosure_zero_mode}) imply that for a generic
\begin{equation}\label{eq:g_in_Dstar_zero_mode}
 g\;=\;\varphi+c_1\Psi_{\alpha,0}+c_0\Phi_{\alpha,0}\;\in\;\mathcal{D}(A_\alpha(0)^*)
\end{equation}
(with $\varphi\in \mathcal{D}(\overline{A_\alpha(0)})$ and $c_0,c_1\in\mathbb{C}$)
the limits
\begin{equation}\label{eq:limitsg0g1_zero_mode}
\begin{split}
	g_0 \; :=& \; \lim_{x \downarrow 0}  x^{\frac{\alpha}{2}} g(x) \; = \;  c_0\,\textstyle{2^{-\frac{1-\alpha}{2}} \Gamma\left( {\textstyle \frac{1+\alpha}{2}}\right)}\,, \\
	g_1\;  :=& \;\lim_{x \downarrow 0} x^{-(1+\frac{\alpha}{2})} (g(x)-g_0x^{-\frac{\alpha}{2}}) \\
	=& \; c_1 \textstyle{\big(2^{\frac{1+\alpha}{2}}\Gamma({\textstyle\frac{3+\alpha}{2}})\big)}^{\!-1} \Vert \Phi_{\alpha,0} \Vert^2_{L^2(\mathbb{R}^+)} -
	c_0 \textstyle{\big(2^{\frac{1+\alpha}{2}}(1+\alpha)\big)}^{\!-1}\Gamma({\textstyle\frac{1-\alpha}{2}})
\end{split}
\end{equation}
exist and are finite, and one has the asymptotics
\begin{equation}\label{eq:gatzero_zero_mode}
 g(x)\;\stackrel{x\downarrow 0}{=}\;  g_0x^{-\frac{\alpha}{2}}+g_1x^{1+\frac{\alpha}{2}}+o(x^{\frac{3}{2}})\,.
\end{equation}
Then, analogously to \eqref{eq:c1betac0}-\eqref{eq:g1gammag0}, the condition of self-adjointness reads again as $c_1=\beta c_0$ for some $\beta\in\mathbb{R}$, or equivalently as 
\begin{equation}\label{eq:g1gammag0_zero_mode}
 g_1\;=\;\gamma g_0\,,\qquad \gamma\;:=\;\frac{\|\Phi_{\alpha,0}\|_{L^2}^2}{2^\alpha\Gamma(\frac{1+\alpha}{2})\Gamma(\frac{3+\alpha}{2})}\Big(\beta-\frac{\Gamma(\frac{1-\alpha}{2})\Gamma(\frac{3+\alpha}{2})}{(1+\alpha)\|\Phi_{\alpha,0}\|_{L^2}^2}\Big).
\end{equation}
This yields an obvious analogue of the `temporary' classification of Prop.~\ref{prop:temp-class}, where if $A_\alpha^{[\gamma]}(0)+\mathbbm{1}$ is a self-adjoint extension of $A_\alpha(0)+\mathbbm{1}$, so is $A_\alpha^{[\gamma]}(0)$ for $A_\alpha(0)$, with $\mathcal{D}(A_\alpha^{[\gamma]}(0)+\mathbbm{1})=\mathcal{D}(A_\alpha^{[\gamma]}(0))$.

In fact, based on the very same argument of Lemma \ref{lem:Fform}, repeated now for the characterisation of the form domain of $A_{\alpha,F}(0)$, one can also reproduce the argument of Prop.~\ref{eq:RGisSFinv}, establishing the following analogue.

\begin{proposition}\label{eq:RGisSFinv_zero_mode}
 For $\alpha\in(0,1)$, one has $A_{\alpha,F}(0)+\mathbbm{1}=R_{G_{\alpha,0}}^{-1}$ and $\Psi_{\alpha,0}=(A_{\alpha,F}(0)+\mathbbm{1})^{-1}\Phi_{\alpha,0}$.
\end{proposition}

Notably, the following useful characterisation of the domain of the Friedrichs extension of $A_\alpha(0)$ is available in the literature.

\begin{proposition}
 For $\alpha\in(0,1)$,
 \begin{equation}\label{eq:Friedrichs_mode_0}
	\mathcal{D}(A_{\alpha,F}(0)) = \mathcal{D}(\overline{A_\alpha(0)}) + \mathrm{span}\{ x^{1+\frac{\alpha}{2}} P \}\,,
\end{equation}
where $P\in C^\infty_c([0,+\infty))$ with $P(0)=1$. 
\end{proposition}

\begin{proof}
 In the notation of \cite{Bruneau-Derezinski-Georgescu-2011}, the Friedrichs extension is the operator $H_m^{\theta}$ with $m^2-\frac{1}{4}=C_\alpha$, hence $m=\frac{1+\alpha}{2}\in(0,1)$, and with $\theta=\frac{\pi}{2}$ (\cite[Prop.~4.19]{Bruneau-Derezinski-Georgescu-2011}), whereas $\overline{A_\alpha(0)}$ is the operator $L_m^\mathrm{min}$. In turn, such $H_m^{\theta}$ is recognised to be the operator $L_m^{u_{\theta}}$, where $u_\theta$ is the function that for $\theta=\frac{\pi}{2}$ has the form $u_{\pi/2}(x)=x^{1+\frac{\alpha}{2}}$ (\cite[Prop.~4.17(1)]{Bruneau-Derezinski-Georgescu-2011}). With this correspondence, the formula $\mathcal{D}(L_m^{u_{\theta}})=\mathcal{D}(L_m^\mathrm{min})+\mathrm{span}\{u_\theta P\}$ (\cite[Prop.~A.5]{Bruneau-Derezinski-Georgescu-2011}) then yields precisely \eqref{eq:Friedrichs_mode_0}. 
\end{proof}

With all the ingredients collected so far, and based on a straightforward adaptation of the arguments of Subsect.~\ref{subsec:proof_of_fibrethm}, the above `temporary' classification then takes the following final form.

\begin{theorem}\label{thm:fibre-thm_zero_mode}
 Let $\alpha\in[0,1)$.
 \begin{itemize}
  \item[(i)] The adjoint of $A_\alpha(0)$ has domain
  \begin{equation}\label{eq:Dadjoint_kequalzero}
   \begin{split}
    \mathcal{D}(A_\alpha(0)^*)\;&=\;\left\{\!\!
  \begin{array}{c}
   g\in L^2(\mathbb{R}^+)\;\;\textrm{such that} \\
   \big(-\frac{\ud^2}{\ud x^2}+\frac{\,\alpha(2+\alpha)\,}{4x^2}\big)g\in L^2(\mathbb{R}^+)
  \end{array}
  \!\!\right\} \\
   &=\;H^2_0(\mathbb{R}^+)\dotplus\mathrm{span}\{\Psi_{\alpha,0}\}\dotplus\mathrm{span}\{\Phi_{\alpha,0}\}\,,
   \end{split} 
  \end{equation}
   where $\Phi_{\alpha,0}$ and $\Psi_{\alpha,0}$ are two smooth functions on $\mathbb{R}^+$ explicitly defined, in terms of modified Bessel functions, respectively by formulas \eqref{eq:Phi_and_F_mode_0}, \eqref{eq:GreenMode0}, and \eqref{eq:defPsi_zero_mode}. Moreover,
   \begin{equation}\label{eq:kerAxistar-in-thm_zero_mode}
  \ker (A_\alpha(0)^*+\mathbbm{1})\;=\;\mathrm{span}\{\Phi_{\alpha,0}\}\,.
 \end{equation}
   \item[(ii)] The self-adjoint extensions of $A_\alpha(0)$ in $L^2(\mathbb{R}^+)$ form the family
   \[
\{ A_\alpha^{[\gamma]}(0)\,|\,\gamma\in\mathbb{R}\cup\{\infty\}\}\,.    
   \]
 The extension with $\gamma=\infty$ is the Friedrichs extension $A_{\alpha,F}(0)$, whose domain is given by \eqref{eq:Friedrichs_mode_0}, and moreover $(A_{\alpha,F}(0)+\mathbbm{1})^{-1}=R_{G_{\alpha,0}}$, the everywhere defined and bounded operator with integral kernel given by \eqref{eq:GreenMode0}. For generic $\gamma\in\mathbb{R}$ one has
 \begin{equation}
  \mathcal{D}(A_\alpha^{[\gamma]}(0))\,=\,\big\{g\in\mathcal{D}(A_\alpha(0)^*)\,\big|\,g(x)\,\stackrel{x\downarrow 0}{=}\,g_0 x^{-\frac{\alpha}{2}}+\gamma g_0x^{1+\frac{\alpha}{2}}+o(x^{\frac{3}{2}})\,,\; g_0\in\mathbb{C}\big\}\,.\!\!\!\!\!\!\!\!\!\!
  \end{equation}
 \end{itemize}
\end{theorem}

\section{Bilateral-fibre extensions}\label{sec:bilateralfibreext}

In this Section we study the `doubling' of the problem considered in Sections \ref{sec:fibre-extensions} and \ref{sec:zero_mode}, namely the problem of the self-adjoint extensions in $L^2(\mathbb{R})$ of the `bilateral' differential operator
\begin{equation}
 \begin{split}
  \mathcal{D}(A_\alpha(k))\;&=\;C^\infty_c(\mathbb{R}^-)\boxplus C^\infty_c(\mathbb{R}^+) \\
  A_\alpha(k)\;&=\;A_\alpha^-(k)\oplus A_\alpha^+(k)\,,
 \end{split}
\end{equation}
already defined in \eqref{eq:Axibilateral}. The Hilbert space $L^2(\mathbb{R})$ is now canonically decomposed into the orthogonal sum 
\begin{equation}\label{eq:againdecomp}
 L^2(\mathbb{R},\ud x)\;\cong\;L^2(\mathbb{R}^-,\ud x)\oplus L^2(\mathbb{R}^+,\ud x)\,.
\end{equation}
Each $g\in L^2(\mathbb{R})$ reads therefore
\begin{equation}\label{eq:gdecomp}
 g\;=\;g^-\oplus g^+\;\equiv\;
 \begin{pmatrix}
  g^- \\ g^+
 \end{pmatrix}
\,,\qquad g^\pm(x)\,:=\,g(x)\;\;\textrm{ for }x\in\mathbb{R}^\pm\,,
\end{equation}
and 
\begin{equation}
 A_\alpha(k)g\;=\;S_{\alpha,k}g^-\oplus S_{\alpha,k}g^+\,,\qquad S_{\alpha,k}\;:=\; -\frac{\ud^2}{\ud x^2}+ k^2 |x|^{2\alpha}+\frac{C_\alpha}{x^2}\,.
\end{equation}

As $A_\alpha^\pm(k)$ has deficiency index 1 in $L^2(\mathbb{R}^\pm)$, $A_\alpha(k)$ has deficiency index 2 in $L^2(\mathbb{R})$, and therefore has a richer variety of extensions.

Among them, as commented already in Section \ref{sec:preparatory_direct-integral}, one has extensions of form
\begin{equation}\label{eq:fibre_uncoupled_ext}
 B_\alpha^-(k) \oplus B_\alpha^+(k)\,,
\end{equation}
where $B_\alpha^\pm(k)$ is a self-adjoint extension of $A_\alpha^\pm(k)$ in $L^2(\mathbb{R}^\pm)$, namely a member of the family determined in the previous Section (Theorem \ref{thm:fibre-thm}). Extensions of type \eqref{eq:fibre_uncoupled_ext} are \emph{reduced} with respect to the decomposition \eqref{eq:againdecomp} (in the usual sense of, e.g., \cite[Sect.~1.4]{schmu_unbdd_sa}): they are \emph{decoupled} self-adjoint realisations of the differential operator $S_{\alpha,k}$, with no constraint between the behaviour as $x\to 0^+$ and $x\to 0^-$. An important extension of this type is the Friedrichs extension $A_{\alpha,F}(k)$: indeed, it is straightforward to argue that
\begin{equation}\label{eq:biFriedrichs}
 A_{\alpha,F}(k)\;=\;A_{\alpha,F}^-(k)\oplus A_{\alpha,F}^+(k)\,,
\end{equation}
where $A_{\alpha,F}^\pm$ is the Friedrichs extension of $A_\alpha^\pm(k)$ in $L^2(\mathbb{R}^\pm)$, which we already described in Theorem \ref{thm:fibre-thm}(iii).

Generic extensions, instead, are not reduced as in \eqref{eq:fibre_uncoupled_ext}, and are characterised by \emph{coupled} bilateral boundary conditions. We classify them using again the convenient  Kre\u{\i}n-Vi\v{s}ik-Birman scheme \cite{GMO-KVB2017}.

Following the same steps of Sections \ref{sec:fibre-extensions} and \ref{sec:zero_mode}, we are now interested in self-adjoint \emph{restrictions} of the adjoint $A_\alpha(k)^*=A_\alpha^-(k)^*\oplus A_\alpha^+(k)^*$ (see formula \eqref{eq:Afstar_sum} above).

In order to export the `one-sided' analysis of Sections \ref{sec:fibre-extensions} and \ref{sec:zero_mode} to the present `two-sided' context, let us introduce a unique expression for the functions of relevance, $\Phi_{\alpha,k}$ and $\Psi_{\alpha,k}$, valid for the left and the right side. Thus, we set
\begin{equation}\label{eq:PhiPsitilde}
 \widetilde{\Phi}_{\alpha,k}(x)\;:=\;\Phi_{\alpha,k}(|x|)\,,\qquad\widetilde{\Psi}_{\alpha,k}(x)\;:=\;\Psi_{\alpha,k}(|x|)\,,
\end{equation}
understanding $\widetilde{\Phi}_{\alpha,k}$ and $\widetilde{\Psi}_{\alpha,k}$ both as functions on $\mathbb{R}^-$ and on $\mathbb{R}^+$, depending on the context. Let us recall that such functions are defined in \eqref{eq:Phi_and_F_explicit} and \eqref{eq:defPsi} when $k\neq 0$, and in \eqref{eq:Phi_and_F_mode_0} and \eqref{eq:defPsi_zero_mode} when $k=0$.

Let us discuss the case $k\neq 0$ first. We deduce at once, respectively from Proposition \ref{prop:domAclosure}, Lemma \ref{lem:kerAxistar}, formula \eqref{eq:defPsi}, and Proposition \ref{eq:RGisSFinv}, that
\begin{eqnarray}
 \mathcal{D}(\overline{A_\alpha(k)})\!&=&\!\big( H^2_0(\mathbb{R}^-)\boxplus H^2_0(\mathbb{R}^+)\big)\cap L^2(\mathbb{R},\langle x\rangle^{4\alpha}\,\ud x) \label{eq:Dclosureclosure}\\
 \ker A_\alpha(k)^* \!&=&\!\mathrm{span}\{\widetilde{\Phi}_{\alpha,k}\}\oplus\mathrm{span}\{\widetilde{\Phi}_{\alpha,k}\} \label{eq:bilateralkernel} \\
 A_{\alpha,F}(k)^{-1}\ker A_\alpha(k)^* \!&=&\!\mathrm{span}\{\widetilde{\Psi}_{\alpha,k}\}\oplus\mathrm{span}\{\widetilde{\Psi}_{\alpha,k}\}\,, \label{eq:bilateralAFinvKer}
\end{eqnarray}
whence also \cite[Theorem 1]{GMO-KVB2017} 
\begin{equation}\label{eq:bilateralDAxistar}
 \begin{split}
  \mathcal{D}(A_\alpha(k)^*)\;&=\;\big( H^2_0(\mathbb{R}^-)\boxplus H^2_0(\mathbb{R}^+)\big)\cap L^2(\mathbb{R},\langle x\rangle^{4\alpha}\,\ud x) \\
  &\qquad \dotplus\mathrm{span}\{\widetilde{\Psi}_{\alpha,k}\}\oplus\mathrm{span}\{\widetilde{\Psi}_{\alpha,k}\} \\
  &\qquad \dotplus\mathrm{span}\{\widetilde{\Phi}_{\alpha,k}\}\oplus\mathrm{span}\{\widetilde{\Phi}_{\alpha,k}\}\,,
 \end{split}
\end{equation}
namely, the analogue of \eqref{eq:Dadjoint}.

In the notation of \eqref{eq:gdecomp}, a generic $g\in \mathcal{D}(A_\alpha(k)^*)$ has therefore the short-range asymptotics
\begin{equation}\label{eq:gbilateralasympt}
 g(x)\;\equiv\;\begin{pmatrix} g^-(x) \\ g^+(x) \end{pmatrix}\;\stackrel{x\to 0}{=} \;\begin{pmatrix} g_0^- \\ g_0^+ \end{pmatrix} |x|^{-\frac{\alpha}{2}}+\begin{pmatrix} g_1^- \\ g_1^+ \end{pmatrix}|x|^{1+\frac{\alpha}{2}}+o(|x|^{\frac{3}{2}})
\end{equation}
for suitable $g_0^\pm,g_1^\pm\in\mathbb{C}$ given by the (thereby existing) limits
 \begin{equation}\label{eq:bilimitsg0g1}
  \begin{split}
   g_0^\pm\;&=\;\lim_{x\to 0^\pm} |x|^{\frac{\alpha}{2}}g^\pm(x)\,, \\
   g_1^\pm\;&=\;\lim_{x\to 0^\pm} |x|^{-(1+\frac{\alpha}{2})}\big(g^\pm(x)-g_0^\pm |x|^{-\frac{\alpha}{2}}\big)\,.
  \end{split}
 \end{equation}
 Formula \eqref{eq:gbilateralasympt} follows from \eqref{eq:bilateralDAxistar} and the usual short-range asymptotics for $\Phi_{\alpha,k}$, $\Psi_{\alpha,k}$, and $\mathcal{D}(\overline{A_\alpha(k)})$.


Now, the Kre\u{\i}n-Vi\v{s}ik-Birman extension theory establishes a one-to-one correspondence between self-adjoint extensions of $A_\alpha(k)$
and self-adjoint operators $T$ in Hilbert subspaces of $\ker A_\alpha(k)^*$: denoting by $A_\alpha^{(T)}\!(k)$ each such extension, and by $\mathcal{K}\subset \ker A_\alpha(k)^*$ the Hilbert subspace where $T$ acts in, $A_\alpha^{(T)}\!(k)$ is the restriction of $A_\alpha(k)^*$ to the domain \cite[Theorem 5]{GMO-KVB2017}
\begin{equation}\label{eq:Birman_formula}
 \mathcal{D}\big(A_\alpha^{(T)}\!(k)\big)\;=\;
 \left\{
 \begin{array}{c}
  g=\varphi+A_{\alpha,F}(k)^{-1}(Tv+w)+v \\
  \textrm{with} \\
  \varphi\in\big( H^2_0(\mathbb{R}^-)\boxplus H^2_0(\mathbb{R}^+)\big)\cap L^2(\mathbb{R},\langle x\rangle^{4\alpha}\,\ud x)\,, \\
  v\in\mathcal{K}\,,\quad w\in \mathrm{span}\{\widetilde{\Phi}_{\alpha,k}\}\oplus\mathrm{span}\{\widetilde{\Phi}_{\alpha,k}\}\,,\quad w\perp v
 \end{array}
 \right\}.
\end{equation}
Clearly $\dim\mathcal{K}$ can be equal to 0, 1, or 2.  The former case corresponds to taking formally `$T=\infty$' on $\mathcal{D}(T)=\{0\}$, and reproduces the Friedrichs extension. The other two cases produce the rest of the family of extensions.

All the preceding discussion has an immediate counterpart when $k=0$, based on the findings of Sect.~\ref{sec:zero_mode}. The above formulas are valid for $k=0$ too, except for \eqref{eq:Dclosureclosure}, that need be replaced with
\begin{equation}
 \mathcal{D}(\overline{A_\alpha(0)})\;=\;H^2_0(\mathbb{R}^-)\oplus H^2_0(\mathbb{R}^+)
\end{equation}
(on account of Lemma \ref{lem:BehaviourZeroClosure_zero_mode}), and except for \eqref{eq:bilateralDAxistar}, that consequently reads now
\begin{equation}\label{eq:bilateralDAxistar_zero_mode}
 \begin{split}
  \mathcal{D}(A_\alpha(0)^*)\;&=\;\mathcal{D}\big(\overline{A^-_\alpha(0)}\big)\boxplus \mathcal{D}\big(\overline{A^+_\alpha(0)}\big) \\
  &\qquad \dotplus\mathrm{span}\{\widetilde{\Psi}_{\alpha,0}\}\oplus\mathrm{span}\{\widetilde{\Psi}_{\alpha,0}\} \\
  &\qquad \dotplus\mathrm{span}\{\widetilde{\Phi}_{\alpha,0}\}\oplus\mathrm{span}\{\widetilde{\Phi}_{\alpha,0}\}\,.
 \end{split}
\end{equation}
Thus when $k=0$ formula \eqref{eq:Birman_formula} takes the form
\begin{equation}\label{eq:Birman_formula_zero_mode}
 \mathcal{D}\big(A_\alpha^{(T)}\!(0)\big)\;=\;
 \left\{
 \begin{array}{c}
  g=\varphi+(A_{\alpha,F}(0)+\mathbbm{1})^{-1}(Tv+w)+v \\
  \textrm{with} \\
  \varphi\in\mathcal{D}\big(\overline{A^-_\alpha(0)}\big)\boxplus \mathcal{D}\big(\overline{A^+_\alpha(0)}\big)\,, \\
  v\in\mathcal{K}\,,\quad w\in \mathrm{span}\{\widetilde{\Phi}_{\alpha,0}\}\oplus\mathrm{span}\{\widetilde{\Phi}_{\alpha,0}\}\,,\quad w\perp v
 \end{array}
 \right\},
\end{equation}
where now $\mathcal{K}$ is a Hilbert subspace of $\ker(A_\alpha(0)^*+\mathbbm{1})$ and $T$ is a self-adjoint operator in $\mathcal{K}$.

We can now formulate the main result of this Section.

\begin{theorem}\label{thm:bifibre-extensions}
 Let $\alpha\in[0,1)$ and $k\in\mathbb{Z}$. Each self-adjoint extension $B_\alpha(k)$ of $A_\alpha(k)$ acts as 
 \begin{equation}
  B_\alpha(k) g\;=\;S_{\alpha,k} \,g^- \oplus S_{\alpha,k} \,g^+
 \end{equation}
 on a generic $g$ of its domain, written in the notation of \eqref{eq:gdecomp} and \eqref{eq:gbilateralasympt}-\eqref{eq:bilimitsg0g1}. The family of self-adjoint extensions of $A_\alpha(k)$ is formed by the following sub-families.
 \begin{itemize}
  \item \underline{Friedrichs extension}. 
  
  It is the operator \eqref{eq:biFriedrichs}. Its domain consists of those functions in $\mathcal{D}(A_\alpha(k)^*)$ whose asymptotics \eqref{eq:gbilateralasympt} has $g_0^\pm=0$.
    \item \underline{Family $\mathrm{I_R}$}.
  
  It is the family $\{A_{\alpha,R}^{[\gamma]}(k)\,|\,\gamma\in\mathbb{R}\}$ defined, with respect to the  asymptotics \eqref{eq:gbilateralasympt}, by
  \[
   \mathcal{D}(A_{\alpha,R}^{[\gamma]}(k))\;=\;\{g\in\mathcal{D}(A_\alpha(k)^*)\,|\,g_0^-=0\,,\;g_1^+=\gamma g_0^+\}\,.
  \]
   \item \underline{Family $\mathrm{I_L}$}.
  
  It is the family $\{A_{\alpha,L}^{[\gamma]}(k)\,|\,\gamma\in\mathbb{R}\}$ defined, with respect to the  asymptotics \eqref{eq:gbilateralasympt}, by
  \[
   \mathcal{D}(A_{\alpha,L}^{[\gamma]}(k))\;=\;\{g\in\mathcal{D}(A_\alpha(k)^*)\,|\,g_0^+=0\,,\;g_1^-=\gamma g_0^-\}\,.
  \]
   \item \underline{Family $\mathrm{II}_a$} with $a\in\mathbb{C}$.
   
   It is the family $\{A_{\alpha,a}^{[\gamma]}(k)\,|\,\gamma\in\mathbb{R}\}$ defined, with respect to the  asymptotics \eqref{eq:gbilateralasympt}, by
  \[
   \mathcal{D}(A_{\alpha,a}^{[\gamma]}(k))\;=\;\left\{g\in\mathcal{D}(A_\alpha(k)^*)\left|\!
   \begin{array}{c}
    g_0^+=a\,g_0^- \,,\\
    g_1^-+\overline{a}\,g_1^+=\gamma\, g_0^-
   \end{array}\!\!
   \right.\right\}.
  \]
   \item \underline{Family $\mathrm{III}$}.
   
   It is the family $\{A_{\alpha}^{[\Gamma]}(k)\,|\,\Gamma\equiv(\gamma_1,\gamma_2,\gamma_3,\gamma_4)\in\mathbb{R}^4\}$ defined, with respect to the  asymptotics \eqref{eq:gbilateralasympt}, by
     \[
   \mathcal{D}(A_{\alpha}^{[\Gamma]}(k))\;=\;\left\{g\in\mathcal{D}(A_\alpha(k)^*)\left|\!
   \begin{array}{c}
    g_1^-=\gamma_1 g_0^-+\zeta g_0^+ \,,\\
    g_1^+=\overline{\zeta} g_0^-+\gamma_4 g_0^+ \,,\\
    \zeta:=\gamma_2+\ii\gamma_3
   \end{array}\!\!
   \right.\right\}.
  \]
 \end{itemize}
 The families $\mathrm{I_R}$, $\mathrm{I_L}$, $\mathrm{II}_a$ for all $a\in\mathbb{C}\setminus\{0\}$, and $\mathrm{III}$ are mutually disjoint and, together with the Friedrichs extension, exhaust the family of self-adjoint extensions of $A_\alpha(k)$. 
\end{theorem}

\begin{remark}
 As already observed, the extensions are operators of the form $A_\alpha^{(T)}\!(k)$ for some self-adjoint $T$ acting on a Hilbert subspace $\mathcal{K}\subset\ker A_\alpha(k)^*$ if $k\neq 0$, or $\mathcal{K}\subset(\ker A_\alpha(0)^*+\mathbbm{1})$ if $k=0$. We are going to show in the proof of Theorem \ref{thm:bifibre-extensions} that the correspondence between each of the considered family and the choice of $\mathcal{K}$ is summarised by Table \ref{tab:extensions}. Thus, extensions of type $\mathrm{I_R}$, $\mathrm{I_L}$, and  $\mathrm{II}_a$ correspond to $\dim\mathcal{K}=1$, type-$\mathrm{III}$ extensions correspond to to $\dim\mathcal{K}=2$, and the Friedrichs extension is the case with $\dim\mathcal{K}=0$.
\end{remark}

\begin{table}
\begin{center}
\begin{tabular}{|c|c|c|c|c|}
 \hline
 \!\!\begin{tabular}{c} family of \\ extensions \end{tabular}\!\! & space $\mathcal{K}$ & \begin{tabular}{c} boundary \\ conditions \end{tabular} & parameters & notes \\
 \hline
 \hline
 Friedrichs & $\{0\} \oplus \{0\}$ & $g_0^{\pm}=0$ &  & \!\!\begin{tabular}{c} bilateral \\ confining \end{tabular}\!\!\\
 \hline
  $\mathrm{I_R}$ & $\{0\}\oplus \mathrm{span}\{\widetilde{\Phi}_{\alpha,k}\}$ & $\begin{array}{c} g_0^-= 0 \\ g_1^+=\gamma g_0^+ \end{array}$ & $\gamma\in\mathbb{R}$ & \!\!\begin{tabular}{c} left \\ confining \end{tabular}\!\! \\
 \hline
   $\mathrm{I_L}$ & $\mathrm{span}\{\widetilde{\Phi}_{\alpha,k}\}\oplus \{0\}$ & $\begin{array}{c} g_1^-=\gamma g_0^- \\ g_0^+= 0 \end{array}$ & $\gamma\in\mathbb{R}$ & \!\!\begin{tabular}{c} right \\ confining \end{tabular}\!\! \\
 \hline
  \begin{tabular}{c} $\mathrm{II}_a$ \\ $a\in\mathbb{C}$ \end{tabular} & $\mathrm{span}\{\widetilde{\Phi}_{\alpha,k}\oplus a \widetilde{\Phi}_{\alpha,k}\}$ & \!\!\!\!$\begin{array}{c}
    g_0^+=a\,g_0^- \\
    g_1^-+\overline{a}\,g_1^+=\gamma\, g_0^-
   \end{array}$\!\!\!\! & $\gamma\in\mathbb{R}$ & \!\!\!\!\begin{tabular}{c} bridging \\ for $a=1$ \\ and $\gamma=0$ \end{tabular}\!\!\!\! \\
  \hline
    $\mathrm{III}$ & $\mathrm{span}\{\widetilde{\Phi}_{\alpha,k}\}\oplus \mathrm{span}\{\widetilde{\Phi}_{\alpha,k}\}$ & \!\!\!\!$\begin{array}{c}
    g_1^-=\gamma_1 g_0^-+\zeta g_0^+ \\
    g_1^+=\overline{\zeta} g_0^-+\gamma_4 g_0^+ \\
    \zeta:=\gamma_2+\ii\gamma_3
   \end{array}$\!\!\!\! & \!\!\!\!\begin{tabular}{c} $\gamma_j\in\mathbb{R}$ \\ $j=1,2,3,4$ \end{tabular}\!\!\!\! &  \\
 \hline
\end{tabular}
\medskip
\caption{Summary of all possible boundary conditions of self-adjointness for the bilateral-fibre extensions of $A_\alpha(k)$}\label{tab:extensions}
\end{center}
\end{table}

\begin{proof}[Proof of Theorem \ref{thm:bifibre-extensions}]
 Let us consider first $k\neq 0$ and let us exploit the classification formula \eqref{eq:Birman_formula} in all possible cases.

 The choice $\mathcal{K}=\{0\}\oplus \{0\}$ yields the extension with domain
 \[
  \mathcal{D}(\overline{A_\alpha(k)})\dotplus A_{\alpha,F}(k)^{-1}\ker A_\alpha(k)^*\;=\;\mathcal{D}(A_{\alpha,F}(k))\,,
 \]
namely the Friedrichs extension. Formula \eqref{eq:thmAFoperator} of Theorem \ref{thm:fibre-thm}, applied on both sides $\mathbb{R}^+$ and $\mathbb{R}^-$, then implies $g_0^+=0=g_0^-$.

 The choice $\mathcal{K}=\{0\}\oplus \mathrm{span}\{\widetilde{\Phi}_{\alpha,k}\}$ yields the extensions in the domain of which a function $g=\varphi+A_{\alpha,F}(k)^{-1}(Tv+w)+v$ is decoupled into a component $g^-$ in the domain of $A_{\alpha,F}^-(k)$ (the Friedrichs extension of $A_{\alpha}^-(k)$) and a component $g^+$ in the domain of a self-adjoint extension of $A_\alpha^+(k)$ in $L^2(\mathbb{R}^+)$. This identifies a family $\{A_{\alpha,R}^{[\gamma]}(k)\,|\,\gamma\in\mathbb{R}\}$ of extensions with
 \[
  A_{\alpha,R}^{[\gamma]}(k)\;=\;A_{\alpha,F}^-(k)\oplus A_{\alpha}^{+,[\gamma]}(k)\,,
 \]
 where $A_{\alpha}^{+,[\gamma]}(k)$ denotes here the generic extension of $A_\alpha^+(k)$, according to the classification of Theorem \ref{thm:fibre-thm}(iv), for which therefore $g_1^+=\gamma g_0^+$. The symmetric choice $\mathcal{K}=\mathrm{span}\{\widetilde{\Phi}_{\alpha,k}\}\oplus \{0\}$ is treated in a completely analogous way.

 The next one-dimensional choice is $\mathcal{K}=\mathrm{span}\{\widetilde{\Phi}_{\alpha,k}\oplus a \widetilde{\Phi}_{\alpha,k}\}$ for some $a\in\mathbb{C}$. We can exclude the case $a=0$ that yields type-$\mathrm{I_L}$ extensions already discussed above. Formula \eqref{eq:Birman_formula} is now to be specialised with
 \[
  v\in\mathcal{K}\,,\qquad w\in \mathcal{K}^\perp\cap 
  \big(\mathrm{span}\{\widetilde{\Phi}_{\alpha,k}\}\oplus\mathrm{span}\{\widetilde{\Phi}_{\alpha,k}\}\big)=\mathrm{span}\{\widetilde{\Phi}_{\alpha,k}\oplus (-\overline{a}^{-1}) \widetilde{\Phi}_{\alpha,k}\}\,.
 \]
 The generic self-adjoint operator $T$ on $\mathcal{K}$ is now the multiplication by some $\tau\in\mathbb{R}$.
 Then \eqref{eq:Birman_formula} reads
 \[
 \begin{split}
  g\;&=\;\varphi+A_{\alpha,F}(k)^{-1}\left( \tau c_0 \begin{pmatrix} \widetilde{\Phi}_{\alpha,k} \\ a \widetilde{\Phi}_{\alpha,k} \end{pmatrix} + \widetilde{c}_0 \begin{pmatrix} \widetilde{\Phi}_{\alpha,k} \\ -\overline{a}^{-1}\,  \widetilde{\Phi}_{\alpha,k}\end{pmatrix} \right)+c_0 \begin{pmatrix} \widetilde{\Phi}_{\alpha,k} \\ a \widetilde{\Phi}_{\alpha,k} \end{pmatrix} \\
  &=\;\varphi+\begin{pmatrix} (\tau c_0+\widetilde{c}_0)\widetilde{\Psi}_{\alpha,k} \\ (\tau c_0 a-\widetilde{c}_0 \,\overline{a}^{-1})\widetilde{\Psi}_{\alpha,k}  \end{pmatrix}+c_0 \begin{pmatrix} \widetilde{\Phi}_{\alpha,k} \\ a \widetilde{\Phi}_{\alpha,k} \end{pmatrix}
 \end{split}
 \]
 for generic coefficients $c_0,\widetilde{c}_0\in\mathbb{C}$.
 From the expression above we find that the limits \eqref{eq:bilimitsg0g1}, computed with the short-range asymptotics \eqref{eq:Asymtotics_0} and \eqref{eq:Psi_asymptotics} (and Lemma \ref{lem:BehaviourZeroClosure}), amount to
 \[
  \begin{split}
   g_0^-\;&=\;c_0\textstyle\sqrt{\frac{\pi(1+\alpha)}{2|k|}}\,, \\
   g_0^+\;&=\;c_0 \,a\textstyle\sqrt{\frac{\pi(1+\alpha)}{2|k|}}\,, \\
   g_1^-\;&=\;\textstyle(\tau c_0+\widetilde{c}_0)\sqrt{\frac{2|k|}{\pi(1+\alpha)^3}}\|\Phi_{\alpha,k}\|_{L^2(\mathbb{R}^+)}^2-c_0\sqrt{\frac{\pi|k|}{2(1+\alpha)}}\,, \\
   g_1^+\;&=\;\textstyle(\tau c_0 a-\widetilde{c}_0\,\overline{a}^{-1})\sqrt{\frac{2|k|}{\pi(1+\alpha)^3}}\|\Phi_{\alpha,k}\|_{L^2(\mathbb{R}^+)}^2-c_0 a \sqrt{\frac{\pi|k|}{2(1+\alpha)}}\,.
  \end{split}
 \]
 Let us stress that here the constant $\|\Phi_{\alpha,k}\|_{L^2}$ is the $L^2$-norm of $\Phi_{\alpha,k}$ on the sole positive half-line. The first two equations above yield $g_0^+=ag_0^-$. The last two yield
 \[
  \begin{split}
  g_1^- +\overline{a}\,g_1^+\;&=\;\textstyle c_0(1+|a|^2)\Big(\tau\sqrt{\frac{2|k|}{\pi(1+\alpha)^3}}\|\Phi_{\alpha,k}\|_{L^2(\mathbb{R}^+)}^2-\sqrt{\frac{\pi|k|}{2(1+\alpha)}} \Big) \\
  &=\;g_0^-\textstyle(1+|a|^2) \frac{|k|}{1+\alpha} \Big( \frac{\,2 \|\Phi_{\alpha,k}\|_{L^2(\mathbb{R}^+)}^2 \,}{\pi (1+\alpha)} \, \tau -1 \Big)\,,
  \end{split} 
 \]
 having replaced $c_0=g_0^-\sqrt{\frac{2|k|}{\pi(1+\alpha)}}$. We can also write
 \[
  g_1^- +\overline{a}\,g_1^+\;=\;\gamma\, g_0^-
 \]
 after re-parametrising the extension parameter as
 \begin{equation}\label{eq:IkTheorem51}
  \gamma\;:=\;(1+|a|^2) \frac{|k|}{1+\alpha} \Big( \frac{\,2 \|\Phi_{\alpha,k}\|_{L^2(\mathbb{R}^+)}^2 \,}{\pi (1+\alpha)} \, \tau -1 \Big)\,\in\,\mathbb{R}\,.
 \end{equation}
 
 This completes the identification of the extensions $A_{\alpha,a}^{[\gamma]}(k)$.

 The remaining choice for $\mathcal{K}$ is $\mathcal{K}=\mathrm{span}\{\widetilde{\Phi}_{\alpha,k}\}\oplus \mathrm{span}\{\widetilde{\Phi}_{\alpha,k}\}$, namely the whole $\ker A_\alpha(k)^*$. In this case formula \eqref{eq:Birman_formula} only has $v$-vectors and no $w$-vectors, and the self-adjoint $T$ is represented by a generic $2\times 2$ Hermitian matrix
 \[
  T\;=\;
  \begin{pmatrix}
   \tau_1 & \tau_2 +\ii\tau_3 \\
   \tau_2 -\ii\tau_3 &\tau_4
  \end{pmatrix},\qquad\tau_1,\tau_2,\tau_3,\tau_4\in\mathbb{R}\,.
 \]
 Then \eqref{eq:Birman_formula} reads
 \[
  \begin{split}
   g\;&=\;\varphi+A_{\alpha,\mathrm{F}}(k)^{-1}T\begin{pmatrix} c_0^-\widetilde{\Phi}_{\alpha,k} \\ c_0^+\widetilde{\Phi}_{\alpha,k} \end{pmatrix}+\begin{pmatrix} c_0^-\widetilde{\Phi}_{\alpha,k} \\ c_0^+\widetilde{\Phi}_{\alpha,k} \end{pmatrix} \\
   &=\;\varphi+\begin{pmatrix} (\tau_1 c_0^-+(\tau_2+\ii\tau_3)c_0^+) \widetilde{\Psi}_{\alpha,k} \\ ((\tau_2-\ii\tau_3)c_0^-+\tau_4c_0^+) \widetilde{\Psi}_{\alpha,k}\end{pmatrix}+\begin{pmatrix} c_0^-\widetilde{\Phi}_{\alpha,k} \\ c_0^+\widetilde{\Phi}_{\alpha,k} \end{pmatrix}
  \end{split}
 \]
 for generic coefficients $c_0^\pm\in\mathbb{C}$.
 From the expression above we find that the limits \eqref{eq:bilimitsg0g1}, computed with the short-range asymptotics \eqref{eq:Asymtotics_0} and \eqref{eq:Psi_asymptotics} (and Lemma \ref{lem:BehaviourZeroClosure}), amount to
 \[
  \begin{split}
   g_0^\pm\;&=\;c_0^\pm\textstyle\sqrt{\frac{\pi(1+\alpha)}{2|k|}}\,, \\
   g_1^-\;&=\;\textstyle c_0^-\Big(\tau_1\sqrt{\frac{2|k|}{\pi(1+\alpha)^3}}\|\Phi_{\alpha,k}\|_{L^2(\mathbb{R}^+)}^2-\sqrt{\frac{\pi|k|}{2(1+\alpha)}}\Big)+c_0^+(\tau_2+\ii\tau_3)\sqrt{\frac{2|k|}{\pi(1+\alpha)^3}}\|\Phi_{\alpha,k}\|_{L^2(\mathbb{R}^+)}^2\,, \\
   g_1^+\;&=\;\textstyle c_0^-(\tau_2-\ii\tau_3)\sqrt{\frac{2|k|}{\pi(1+\alpha)^3}}\|\Phi_{\alpha,k}\|_{L^2(\mathbb{R}^+)}^2+c_0^+\Big(\tau_4\sqrt{\frac{2|k|}{\pi(1+\alpha)^3}}\|\Phi_{\alpha,k}\|_{L^2(\mathbb{R}^+)}^2+\sqrt{\frac{\pi|k|}{2(1+\alpha)}}\Big).
  \end{split}
 \]
 Replacing $c_0^\pm=g_0^\pm\sqrt{\frac{2|k|}{\pi(1+\alpha)}}$ in the last two equations above and re-defining the extension parameters as
 \begin{equation}\label{eq:IIkTheorem51}
  \begin{split}
   \gamma_1\;&:=\;  \frac{|k|}{1+\alpha} \Big( \frac{\,2 \|\Phi_{\alpha,k}\|_{L^2(\mathbb{R}^+)}^2 \,}{\pi (1+\alpha)} \, \tau_1 -1 \Big)\,, \\
   \gamma_2+\ii\gamma_3\;&:=\; (\tau_2+\ii\tau_3)\frac{2|k|}{\pi(1+\alpha)^2}\|\Phi_{\alpha,k}\|_{L^2(\mathbb{R}^+)}^2\,, \\
   \gamma_4\;&:=\;  \frac{|k|}{1+\alpha} \Big( \frac{\,2 \|\Phi_{\alpha,k}\|_{L^2(\mathbb{R}^+)}^2 \,}{\pi (1+\alpha)} \, \tau_4 -1 \Big),,  \end{split}
 \end{equation}
 yields precisely the boundary condition that characterises the extension $A_{\alpha}^{[\Gamma]}(k)$ with $\Gamma=(\gamma_1,\gamma_2,\gamma_3,\gamma_4)$.

 Last, let us repeat the above reasonings when $k=0$, based now on the classification formula \eqref{eq:Birman_formula_zero_mode}. The only modifications needed are the replacement of $A_{\alpha,\mathrm{F}}(k)^{-1}$ with $(A_{\alpha,\mathrm{F}}(0)+\mathbbm{1})^{-1}$, and the use, instead of the short-range asymptotics given by \eqref{eq:Asymtotics_0}, \eqref{eq:Psi_asymptotics}, and Lemma \ref{lem:BehaviourZeroClosure} valid for $k\neq 0$, of the short-range asymptotics given by \eqref{eq:AsymPhi00}, \eqref{eq:Psi_Asymptotics_mode_zero}, and Lemma \ref{lem:BehaviourZeroClosure_zero_mode} valid for $k=0$.
 
 The net result concerning the extensions of type $\mathrm{II}_a$, namely the extensions $A^{[\gamma]}_{\alpha,a}(0)$, is that \eqref{eq:IkTheorem51} is replaced by 
 \begin{equation}\label{eq:I0Theorem51}
  	\gamma \;:=\;  {\frac{(1+|a|^2)\Gamma\left( \frac{1-\alpha}{2}\right)}{2^\alpha (1+\alpha)\Gamma\left( \frac{1+\alpha}{2}\right)} \Big( \frac{(1+\alpha)\Vert \Phi_{\alpha,0} \Vert_{L^2(\mathbb{R}^+)}^2}{\Gamma\left( \frac{3+\alpha}{2}\right)\Gamma\left( \frac{1-\alpha}{2}\right)}  \, \tau - 1 \Big)}\;\in\;\mathbb{R}\,.
 \end{equation}

 Analogously, concerning the extensions of type $\mathrm{III}$, namely the extensions $A^{[\Gamma]}_{\alpha}(0)$, \eqref{eq:IIkTheorem51} is now replaced by 
  \begin{equation}\label{eq:II0Theorem51}
  \begin{split}
   \gamma_1\;&:=\;  \frac{\Gamma(\frac{1-\alpha}{2})}{2^\alpha (1+\alpha)\Gamma(\frac{1+\alpha}{2})} \Big( \frac{\, (1+\alpha) \|\Phi_{\alpha,0}\|_{L^2(\mathbb{R}^+)}^2 \,}{\Gamma(\frac{3+\alpha}{2}) \Gamma (\frac{1-\alpha}{2}) } \, \tau_1 -1 \Big)\,, \\
   \gamma_2+\ii\gamma_3\;&:=\; (\tau_2+\ii\tau_3)\frac{\| \Phi_{\alpha,0} \|_{L^2(\mathbb{R}^+)}^2}{2^\alpha \Gamma(\frac{3+\alpha}2) \Gamma(\frac{1+\alpha}2)}\,, \\
   \gamma_4\;&:=\;  \frac{\Gamma(\frac{1-\alpha}{2})}{2^\alpha (1+\alpha)\Gamma(\frac{1+\alpha}{2})} \Big( \frac{\, (1+\alpha) \|\Phi_{\alpha,0}\|_{L^2(\mathbb{R}^+)}^2 \,}{\Gamma(\frac{3+\alpha}{2}) \Gamma (\frac{1-\alpha}{2}) } \, \tau_4 -1 \Big)\,.  \end{split}
 \end{equation}
 
 The proof is now completed.
\end{proof}

\begin{remark}
 The type-$\mathrm{II}_a$ extension with $a=1$ and extension parameter $\gamma=0$ is caracterised by the distinguished boundary condition
 \begin{equation}
  g_0^-\,=\,g_0^+\,,\qquad g_1^-\,=\,-g_1^+\,.
 \end{equation}
 We shall interpret this condition as the maximally transmitting, or `bridging' condition between the two sides of the bilateral fibre.
\end{remark}

Whereas Theorem \ref{thm:bifibre-extensions} expresses the various conditions of self-adjointness in terms of the representation \eqref{eq:gdecomp} and \eqref{eq:gbilateralasympt}-\eqref{eq:bilimitsg0g1} of a generic $g\in\mathcal{D}(A_\alpha(k)^*)$, that is, in terms of the short-range behaviour of $g$, for the forthcoming analysis it will be convenient to re-formulate the above classification in two further equivalent forms.

The first one refers to the representation \eqref{eq:gdecomp}, \eqref{eq:bilateralDAxistar}, and \eqref{eq:bilateralDAxistar_zero_mode}
of $g\in\mathcal{D}(A_\alpha(k)^*)$,  that is,
 \begin{equation}\label{eq:gkrepresentationc0c1}
 g\;=\;\begin{pmatrix}\widetilde{\varphi}^- \\ \widetilde{\varphi}^+\end{pmatrix}+\begin{pmatrix} c_{1}^-\widetilde{\Psi}_{\alpha,k} \\ c_{1}^+\widetilde{\Psi}_{\alpha,k} \end{pmatrix}+\begin{pmatrix} c_{0}^-\widetilde{\Phi}_{\alpha,k} \\ c_{0}^+\widetilde{\Phi}_{\alpha,k} \end{pmatrix}
\end{equation}
with $\widetilde{\varphi}^\pm\in\mathcal{D}(\overline{A^\pm_\alpha(k)})$ and $c_0^\pm,c_1^\pm\in\mathbb{C}$. Then the proof of Theorem \ref{thm:bifibre-extensions} demonstrates also the following.

\begin{theorem}\label{thm:bifibre-extensionsc0c1}
 Let $\alpha\in[0,1)$ and $k\in\mathbb{Z}$. The family of self-adjoint extensions of $A_\alpha(k)$ is formed by the following sub-families.
 \begin{itemize}
  \item \underline{Friedrichs extension}.  It is the operator \eqref{eq:biFriedrichs}. Its domain consists of those functions in $\mathcal{D}(A_\alpha(k)^*)$ whose representation \eqref{eq:gkrepresentationc0c1} has $c_0^\pm=0$.
    \item \underline{Family $\mathrm{I_R}$}.   It is the family $\{A_{\alpha,R}^{[\gamma]}(k)\,|\,\gamma\in\mathbb{R}\}$ defined, with respect to the  representation \eqref{eq:gkrepresentationc0c1}, by
  \[
   \mathcal{D}(A_{\alpha,R}^{[\gamma]}(k))\;=\;\{g\in\mathcal{D}(A_\alpha(k)^*)\,|\,c_0^-=0\,,\;c_1^+=\beta c_0^+\}\,,
  \]
  where $\beta$ and $\gamma$ are related by \eqref{eq:g1gammag0} for $k\neq 0$ and \eqref{eq:g1gammag0_zero_mode} for $k=0$.
   \item \underline{Family $\mathrm{I_L}$}.   It is the family $\{A_{\alpha,L}^{[\gamma]}(k)\,|\,\gamma\in\mathbb{R}\}$ defined, with respect to the  representation \eqref{eq:gkrepresentationc0c1}, by
  \[
   \mathcal{D}(A_{\alpha,L}^{[\gamma]}(k))\;=\;\{g\in\mathcal{D}(A_\alpha(k)^*)\,|\,c_0^+=0\,,\;c_1^-=\beta c_0^-\}\,,
  \]
  where $\beta$ and $\gamma$ are related by \eqref{eq:g1gammag0} for $k\neq 0$ and \eqref{eq:g1gammag0_zero_mode} for $k=0$.
   \item \underline{Family $\mathrm{II}_a$} with $a\in\mathbb{C}$.   It is the family $\{A_{\alpha,a}^{[\gamma]}(k)\,|\,\gamma\in\mathbb{R}\}$ defined by
%
   \[
    \mathcal{D}(A_{\alpha,a}^{[\gamma]}(k))\;=\;\left\{\!
    \begin{array}{c}
     g\in\mathcal{D}(A_\alpha(k)^*)\textrm{ with \eqref{eq:gkrepresentationc0c1} of the form } \\
      g=\begin{pmatrix}\widetilde{\varphi}^- \\ \widetilde{\varphi}^+\end{pmatrix}+\begin{pmatrix} (\tau c_0+\widetilde{c}_0)\widetilde{\Psi}_{\alpha,k} \\ (\tau c_0 a-\widetilde{c}_0 \,\overline{a}^{-1})\widetilde{\Psi}_{\alpha,k}  \end{pmatrix}+c_0 \begin{pmatrix} \widetilde{\Phi}_{\alpha,k} \\ a \widetilde{\Phi}_{\alpha,k} \end{pmatrix}
    \end{array}
    \!\right\},
   \]
  where $\tau$ and $\gamma$ are related by \eqref{eq:IkTheorem51} if $k\neq 0$, and by \eqref{eq:I0Theorem51} if $k=0$.
   \item \underline{Family $\mathrm{III}$}.   It is the family $\{A_{\alpha}^{[\Gamma]}(k)\,|\,\Gamma\equiv(\gamma_1,\gamma_2,\gamma_3,\gamma_4)\in\mathbb{R}^4\}$ defined by
   \[
     \mathcal{D}(A_{\alpha}^{[\Gamma]}(k))\;=\;\left\{\!
     \begin{array}{c}
       g\in\mathcal{D}(A_\alpha(k)^*) \textrm{ satisfying \eqref{eq:gkrepresentationc0c1} with } \\
       \begin{pmatrix}
		c_{1}^- \\ c_1^+
	\end{pmatrix} = \begin{pmatrix}
		\tau_1 & \tau_2 + \ii \tau_3 \\
		\tau_2 - \ii \tau_3 & \tau_4
	\end{pmatrix} \begin{pmatrix}
		c_0^- \\
		c_0^+
	\end{pmatrix}
     \end{array}
     \!\right\},
   \]
   where $(\tau_1,\tau_2,\tau_3,\tau_4)$ and $(\gamma_1,\gamma_2,\gamma_3,\gamma_4)$ are related by \eqref{eq:IIkTheorem51} if $k\neq 0$ and \eqref{eq:II0Theorem51} if $k=0$.   
 \end{itemize}
\end{theorem}

The second alternative for the self-adjointness conditions is in fact a very close re-phrasing of Theorem \ref{thm:bifibre-extensions}, with the same short-range parameters $g_0^\pm$ and $g_1^\pm$ and the same classification parameters $\gamma$ or $\Gamma$, except that it is referred to the following representation of $g$, which is valid identically for any $x\in\mathbb{R}\setminus\{0\}$, and not just as $|x|\to 0$.

To this aim, and also for later convenience, we shall henceforth refer to $P$ as a cut-off function in $C^\infty_c(\mathbb{R})$ such that
\begin{equation}\label{eq:Pcutoff}
 P(x)\;=\;
 \begin{cases}
  \;1\,, &\textrm{ if }\;|x|<1\,, \\
  \;0\,, &\textrm{ if }\;|x|>2\,.
 \end{cases}
\end{equation}
In fact, in the following Theorem it is enough that $P$ be smooth, compactly supported, and with $P(0)=1$, but we keep the general assumption \eqref{eq:Pcutoff} for later use.

\begin{theorem}\label{prop:g_with_Pweight}
 Let $\alpha\in[0,1)$ and let $k\in\mathbb{Z}$. Then for any $g\in\mathcal{D}(A_\alpha(k)^*)$ there exist a unique $\varphi\in\mathcal{D}(\overline{A_{\alpha}(k)})$ (in particular, $\varphi^\pm\in H^2_0(\mathbb{R}^\pm)$) and uniquely determined coefficients 
 $g_0^\pm,g_1^\pm\in\mathbb{C}$ such that
 \begin{equation}\label{eq:gwithPweight}
  g(x)\;=\;\varphi(x)+g_0\,|x|^{-\frac{\alpha}{2}}\,P(x)+g_1\,|x|^{1+\frac{\alpha}{2}}\,P(x)\qquad\forall x\in\mathbb{R}\setminus\{0\}
 \end{equation}
 in the usual notation
 \begin{equation*}
  \varphi(x)\equiv\begin{pmatrix} \varphi^-(x) \\ \varphi^+(x) \end{pmatrix},\quad g_0\equiv\begin{pmatrix} g_0^- \\ g_0^+\end{pmatrix}, \quad g_1\equiv\begin{pmatrix} g_1^- \\ g_1^+\end{pmatrix}.
 \end{equation*}
 Here $g_0^\pm$ and $g_1^\pm$ are precisely the same as in the asymptotics \eqref{eq:gbilateralasympt}-\eqref{eq:bilimitsg0g1}. Therefore, the same classification of Theorem \ref{thm:bifibre-extensions} in terms of $g_0^\pm$ and $g_1^\pm$ is applicable. 
\end{theorem}

\begin{proof}
 Let $k\neq 0$ and let us decompose $g\in\mathcal{D}(A_\alpha(k)^*)$ as $g^\pm=\widetilde{\varphi}^\pm+c_1^\pm\widetilde{\Psi}_{\alpha,k}+c_0^\pm\widetilde{\Phi}_{\alpha,k}$ with respect to the decomposition \eqref{eq:gkrepresentationc0c1}. For short, let us discuss only the component $g^+$, dropping the `$+$' superscript: the discussion for $g^-$ is completely analogous. Thus, $g=\widetilde{\varphi}+c_1\widetilde{\Psi}_{\alpha,k}+c_0\widetilde{\Phi}_{\alpha,k}$ for all $x>0$ and uniquely determined $\widetilde{\varphi}\in\mathcal{D}(\overline{A_{\alpha}(k)})$ and $c_0,c_1\in\mathbb{C}$. Let us introduce the functions
 \[
  \begin{split}
   L_{0,k}(x)\,&:=\,\Big({\textstyle\sqrt{\frac{\pi(1+\alpha)}{2|k|}}-\sqrt{\frac{\pi|k|}{2(1+\alpha)}}\,|x|^{1+\alpha}} \Big)\,P(x)\,, \\
   L_{1,k}(x)\,&:=\,{\textstyle\sqrt{\frac{2|k|}{\pi(1+\alpha)^3}}}\,\|\Phi_{\alpha,k}\|_{L^2}^2 \,P(x)
  \end{split}
 \]
 and re-write
 \[
 \begin{split}
  g\;&=\;\widetilde{\varphi}+c_1(\widetilde{\Psi}_{\alpha,k}-|x|^{1+\frac{\alpha}{2}}L_{1,k})+c_0(\widetilde{\Phi}_{\alpha,k}-|x|^{-\frac{\alpha}{2}}L_{0,k})+c_1\,|x|^{1+\frac{\alpha}{2}}L_{1,k}+c_0\,|x|^{-\frac{\alpha}{2}}L_{0,k} \\
  &=\;\varphi+\Big( c_1{\textstyle\sqrt{\frac{2|k|}{\pi(1+\alpha)^3}}}\,\|\Phi_{\alpha,k}\|_{L^2}^2-c_0{\textstyle\sqrt{\frac{\pi|k|}{2(1+\alpha)}}}\Big)\,|x|^{1+\frac{\alpha}{2}}P+c_0{\textstyle\sqrt{\frac{\pi(1+\alpha)}{2|k|}}}\,|x|^{-\frac{\alpha}{2}} P\,,
 \end{split}
 \]
 having set 
 \[
  \varphi\;:=\;\widetilde{\varphi}+c_1(\widetilde{\Psi}_{\alpha,k}-|x|^{1+\frac{\alpha}{2}}L_{1,k})+c_0(\widetilde{\Phi}_{\alpha,k}-|x|^{-\frac{\alpha}{2}}L_{0,k})\,.
 \]
 Because of the relation \eqref{eq:limitsg0g1} between $c_0,c_1$ and $g_0,g_1$, we also have
 \[
  g\;=\;\varphi+g_0\,|x|^{-\frac{\alpha}{2}}P+g_1\,|x|^{1+\frac{\alpha}{2}}P\,.
 \]
 Next, let us argue that $\varphi\in\mathcal{D}(\overline{A_{\alpha}(k)})$. First, we observe that both $|x|^{-\frac{\alpha}{2}}L_{0,k}$ and $|x|^{1+\frac{\alpha}{2}}L_{1,k}$ belong to $\mathcal{D}(A_\alpha(k)^*)$. The latter statement, owing to \eqref{eq:Afstar} and \eqref{eq:Saxi}, is proved by checking the square-integrability of $S_{\alpha,k}(|x|^{-\frac{\alpha}{2}}L_{0,k})$ and of $S_{\alpha,k}(|x|^{1+\frac{\alpha}{2}}L_{1,k})$. Since $P$ localises $L_{0,k}$ and $L_{1,k}$ around $x=0$, square-integrability must only be checked \emph{locally}. It is then routine to see that
 \[
  \begin{array}{lllll}
   -(|x|^{-\frac{\alpha}{2}}L_{0,k})''\,   +k^2 |x|^{2\alpha}(|x|^{-\frac{\alpha}{2}}L_{0,k})\,   +C_\alpha x^{-2}(|x|^{-\frac{\alpha}{2}}L_{0,k})\,, \\
   -(|x|^{1+\frac{\alpha}{2}}L_{1,k})''\,   +k^2 |x|^{2\alpha}(|x|^{1+\frac{\alpha}{2}}L_{1,k})\,   +C_\alpha x^{-2}(|x|^{1+\frac{\alpha}{2}}L_{1,k})\,,
  \end{array}
 \]
 are both square-integrable around $x=0$. As a consequence, both $(\widetilde{\Psi}_{\alpha,k}-|x|^{1+\frac{\alpha}{2}}L_{1,k})$ and $(\widetilde{\Phi}_{\alpha,k}-|x|^{-\frac{\alpha}{2}}L_{0,k})$ are elements of $\mathcal{D}(A_\alpha(k)^*)$. Therefore, owing to the representation \eqref{eq:Dadjoint}-\eqref{eq:limitsg0g1}, in order to check that such two functions also belong to $\mathcal{D}(\overline{A_{\alpha}(k)})$ it suffices to verify the limits 
 \[
 \begin{split}
  \lim_{x\to 0}|x|^{\frac{\alpha}{2}}(\widetilde{\Psi}_{\alpha,k}-|x|^{1+\frac{\alpha}{2}}L_{1,k})\;=\;\lim_{x\to 0}|x|^{\frac{\alpha}{2}}(\widetilde{\Phi}_{\alpha,k}-|x|^{-\frac{\alpha}{2}}L_{0,k})\;&=\;0 ,,\\
  \lim_{x\to 0}|x|^{-(1+\frac{\alpha}{2})}(\widetilde{\Psi}_{\alpha,k}-|x|^{1+\frac{\alpha}{2}}L_{1,k})\;=\;\lim_{x\to 0}|x|^{-(1+\frac{\alpha}{2})}(\widetilde{\Phi}_{\alpha,k}-|x|^{-\frac{\alpha}{2}}L_{0,k})\;&=\;0\,.
 \end{split}
 \]
 This is straightforward to check, thanks to the short-distance asymptotics that were chosen for $L_{0,k}$ and $L_{1,k}$ precisely so as to suitably match with the short-distance asymptotics  \eqref{eq:Asymtotics_0} of $\widetilde{\Phi}_{\alpha,k}$ and \eqref{eq:Psi_asymptotics} of $\widetilde{\Psi}_{\alpha,k}$. This finally shows that $\varphi\in\mathcal{D}(\overline{A_{\alpha}(k)})$ and establishes \eqref{eq:gwithPweight}. Of course, if conversely a function $g$ of the form \eqref{eq:gwithPweight} is given with $\varphi\in\mathcal{D}(\overline{A_{\alpha}(k)})$, unfolding the above arguments one sees that $g\in\mathcal{D}(A_\alpha(k)^*)$.

 If instead $k=0$, the same argument can be repeated decomposing now $g\in\mathcal{D}(A_\alpha(0)^*)$ as $g^\pm=\widetilde{\varphi}^\pm+c_1^\pm\widetilde{\Psi}_{\alpha,0}+c_0^\pm\widetilde{\Phi}_{\alpha,0}$ according to the decomposition \eqref{eq:bilateralDAxistar_zero_mode}, and using now the short-range asymptotics \eqref{eq:AsymPhi00}, \eqref{eq:Psi_Asymptotics_mode_zero}, and Lemma \ref{lem:BehaviourZeroClosure_zero_mode} valid for $k=0$. We omit the straightforward details. 
\end{proof}

\section{General extensions of $\mathscr{H}_\alpha$}\label{sec:genextscrHa}

Let us now come in this Section to the study of the self-adjoint extensions, in the Hilbert space \eqref{eq:Hxispace}, namely
\begin{equation}\label{eq:Hspaceonceagain}
 \cH\;\cong\;\bigoplus_{k\in\mathbb{Z}}\mathfrak{h}_k\;\cong\;\ell^2(\mathbb{Z},\mathfrak{h})\,,\qquad \mathfrak{h}_k\;\cong\;\mathfrak{h}\;\cong\;L^2(\mathbb{R}^-)\oplus L^2(\mathbb{R}^+)\,,
\end{equation}
of the operator $\mathscr{H}_\alpha$ introduced in \eqref{eq:actiondomainHalpha} for $\alpha\in(0,1)$. 
Such extensions are restrictions of $\mathscr{H}_\alpha^*$, and it was seen in Lemma \ref{lem:sumstar-sumclosure} (eq.~\eqref{eq:sumstar}) that $\mathscr{H}_\alpha^*=\bigoplus_{k\in\mathbb{Z}}A_\alpha(k)^*$, in the sense of  the general construction \eqref{eq:Tdirectint}-\eqref{eq:T_direct_integral}.

Let us start with some preliminaries (Lemmas \ref{lem:HalphaFrie-decomposable}-\ref{lem:gkkrepr}) and then present the complete variety of extensions.


Clearly, $\mathscr{H}_\alpha$ is non-negative, since so too are all the $A_\alpha(k)$'s (see \eqref{eq:Axibottom}) and $\mathscr{H}_\alpha\subset\bigoplus_{k\in\mathbb{Z}} A_\alpha(k)$, and therefore has Friedrichs extension $\mathscr{H}_{\alpha,F}$. Recall that $\mathscr{H}_\alpha\varsubsetneq\bigoplus_{k\in\mathbb{Z}} A_\alpha(k)$ (Remark \ref{rem:Halphanotsum}) and $\mathscr{H}_\alpha^*=\bigoplus_{k\in\mathbb{Z}} \,A_\alpha(k)^*$,  $\overline{\mathscr{H}_\alpha}=\bigoplus_{k\in\mathbb{Z}} \,\overline{A_\alpha(k)}$ (Lemma \ref{lem:Halphaadj-decomposable}).


 \begin{lemma}\label{lem:HalphaFrie-decomposable}
  Let $\alpha\in[0,1)$. One has
  \begin{equation}\label{eq:HalphaFriedrichs_unif-fibred}
   \mathscr{H}_{\alpha,F}\;=\;\bigoplus_{k\in\mathbb{Z}}\,A_{\alpha,F}(k)\,.
  \end{equation}
 \end{lemma}

 Lemma \ref{lem:HalphaFrie-decomposable} is an application of a general fact that for convenience we revisit here (of course in the following the identification $\mathfrak{h}_k\cong\mathfrak{h}$ for all $k$ does not play a role).

 \begin{lemma}\label{lem:sumofFriedrichs}
  Let $T=\bigoplus_{k\in\mathbb{Z}}T(k)$ be a direct sum operator acting on the Hilbert space $\cH=\bigoplus_{k\in\mathbb{Z}}\mathfrak{h}_k$, where each $T(k)$ is densely defined, symmetric, and semi-bounded from below on $\mathfrak{h}_k$, with uniform lower bound
  \[
   m\;:=\;\inf_{k\in\mathbb{Z}}\;\inf_{\substack{u\in\mathcal{D}(T(k)) \\ u\neq 0}}  \frac{\langle u,T(k)u\rangle_{\mathfrak{h}_k}}{\|u\|^2_{\mathfrak{h}_k}}\;>\;-\infty\,.
  \]
  Denote by $T_F$, resp.~$T_F(k)$, the Friedrichs extension of $T$, resp.~$T(k)$. Then
  \[
   T_F\;=\;\bigoplus_{k\in\mathbb{Z}}\,T_F(k)\,.
  \]
 \end{lemma}

 \begin{proof}
  It is clear that $\bigoplus_{k\in\mathbb{Z}}T_F(k)$ is a self-adjoint extension of $T$. To recognise it as the Friedrichs extension, it suffices to check that the \emph{operator} domain $\mathcal{D}(\bigoplus_{k\in\mathbb{Z}}T_F(k))$ is an actual subspace of the \emph{form} domain $\mathcal{D}[T]$. To this aim, let us observe that
  \[
   \begin{split}
    \mathcal{D}\Big(\bigoplus_{k\in\mathbb{Z}}T_F(k)\Big)\;&=\;\op_{k\in\mathbb{Z}}\mathcal{D}(T_F(k))\;\subset\;\op_{k\in\mathbb{Z}}\mathcal{D}[T(k)]
   \end{split}
  \]
  (the first identity is precisely \eqref{eq:shorthandDTDTk} discussed previously, and the inclusion is due to the fact that for each $k$ the Friedrichs-extension characterising property $\mathcal{D}(T_F(k))\subset\mathcal{D}[T(k)]$ holds). On the other hand, $\mathcal{D}[T]=\mathcal{D}((T-m\mathbbm{1})^{\frac{1}{2}})$ and $\mathcal{D}[T(k)]=\mathcal{D}((T(k)-m\mathbbm{1}_k)^{\frac{1}{2}})$, whence
  \[
   \begin{split}
    \mathcal{D}[T]\;&=\;\mathcal{D}\Big[\bigoplus_{k\in\mathbb{Z}}T(k)\Big]\;=\;
    \mathcal{D}\Big(\Big(\bigoplus_{k\in\mathbb{Z}}\,(T(k)-m\mathbbm{1}_k)\Big)^{\!\frac{1}{2}}\,\Big) \\
    &=\;\mathcal{D}\Big(\bigoplus_{k\in\mathbb{Z}}\,(T(k)-m\mathbbm{1}_k)^{\frac{1}{2}}\Big)\;=\;\op_{k\in\mathbb{Z}}\mathcal{D}\big((T(k)-m\mathbbm{1}_k)^{\frac{1}{2}}\big)\;=\;\op_{k\in\mathbb{Z}}\mathcal{D}[T(k)]\,.
   \end{split}
  \]
 This proves the desired inclusion.
 \end{proof}

 \begin{proof}[Proof of Lemma \ref{lem:HalphaFrie-decomposable}]
  One applies Lemma \ref{lem:sumofFriedrichs} to $\overline{\mathscr{H}_\alpha}=\bigoplus_{k\in\mathbb{Z}} \,\overline{A_\alpha(k)}$.  
 \end{proof}

There is an obvious peculiarity of the mode $k=0$ that needs be dealt with separately. Indeed, \eqref{eq:Axibottom} and \eqref{eq:Axibottom-zero} show that the bottom of the spectrum of $A_{\alpha,\mathrm{F}}(k)$ is strictly positive when $k\in\mathbb{Z}\setminus\{0\}$, explicitly
\begin{equation}\label{eq:AFbottom}
 A_{\alpha,\mathrm{F}}(k)\;\geqslant\;(1+\alpha)\big(\textstyle{\frac{2+\alpha}{4}}\big)^{\frac{\alpha}{1+\alpha}}\,\mathbbm{1}_k\,,\qquad k\in\mathbb{Z}\setminus\{0\}\,,\quad\alpha\geqslant 0\,,
\end{equation}
 and instead amounts precisely to zero when $k=0$. Thus, all Friedrichs extensions on fibre have everywhere-defined bounded inverse, but the one corresponding to $k=0$.

It is then convenient to consider a positive shift of $\mathscr{H}_\alpha$ in the zero mode only. Clearly, with $\mathbbm{1}_0$ acting as the identity in the $0$-th fibre and as the zero operator in all other fibres, the operators $\mathscr{H}_\alpha$ and $\mathscr{H}_\alpha+\mathbbm{1}_0$ have precisely the same domain, and so do the respective adjoints and the respective Friedrichs extensions.

\begin{lemma}\label{lem:psikinspaces}
 Let $\alpha\in[0,1)$. Let $(\psi_k)_{k\in\mathbb{Z}}\in\cH\cong\ell^2(\mathbb{Z},\mathfrak{h})$. Then:
 \begin{itemize}
  \item[(i)] $(\psi_k)_{k\in\mathbb{Z}}\in\ker( \mathscr{H}_\alpha+\mathbbm{1}_0)^*$ if and only if
  \begin{equation}
   \psi_k\;=\;c_{0,k}^-\widetilde{\Phi}_{\alpha,k}\oplus c_{0,k}^+\widetilde{\Phi}_{\alpha,k}
   \qquad\quad\forall k\in\mathbb{Z}
  \end{equation}
  for coefficients $c_{0,k}^\pm\in\mathbb{C}$ such that
  \begin{equation}
   \sum_{k\in \mathbb{Z}\setminus\{0\}}|k|^{-\frac{2}{1+\alpha}}|c_{0,k}^\pm|^2\:<\:+\infty\,.
  \end{equation}
  Thus, there is a natural Hilbert space isomorphism 
 \begin{equation}
  \ker( \mathscr{H}_\alpha+\mathbbm{1}_0)^*\;\cong\;\ell^2(\mathbb{Z},\mathbb{C}^2,\mu_k)
 \end{equation}
 with
 \begin{equation}\label{eq:DefMeasureMuK}
  \mu_k\,:=\,
  \begin{cases}
   |k|^{-\frac{2}{1+\alpha}}\,,
   & k\neq 0\,, \\
   1\,, & k=0\,,
  \end{cases}
 \end{equation}
 where $\ell^2(\mathbb{Z},\mathbb{C}^2,\mu_k)$ is the Hilbert space of sequences $\Big(\begin{pmatrix} c_k^- \\ c_k^+\end{pmatrix}\Big)_{k\in\mathbb{Z}}$ with obvious (component-wise) vector space structure and with scalar product
 \begin{equation}
  \left\langle \Big(\begin{pmatrix} c_k^- \\ c_k^+\end{pmatrix}\Big)_{k\in\mathbb{Z}},\Big(\begin{pmatrix} d_k^- \\ d_k^+\end{pmatrix}\Big)_{k\in\mathbb{Z}}\right\rangle_{\!\ell^2(\mathbb{Z},\mathbb{C}^2,\mu_k)}=\:\;\sum_{k\in\mathbb{Z}}\mu_k\big(\,\overline{c_k^-}\,d_k^-+\overline{c_k^+}\,d_k^+\big)\,.
 \end{equation}
  \item[(ii)] $(\psi_k)_{k\in\mathbb{Z}}\in(\mathscr{H}_\alpha^F+\mathbbm{1}_0)^{-1}\ker( \mathscr{H}_\alpha+\mathbbm{1}_0)^*$ if and only if
  \begin{equation}
   \psi_k\;=\;c_{1,k}^-\widetilde{\Psi}_{\alpha,k}\oplus c_{1,k}^+\widetilde{\Psi}_{\alpha,k}
   \qquad\quad\forall k\in\mathbb{Z}
  \end{equation}
  for coefficients $c_{1,k}^\pm\in\mathbb{C}$ such that
  \begin{equation}\label{eq:Regularityg1}
   \sum_{k\in \mathbb{Z}\setminus\{0\}}|k|^{-\frac{2}{1+\alpha}}|c_{1,k}^\pm|^2\:<\:+\infty\,.
  \end{equation}
 \end{itemize}
\end{lemma}

\begin{proof}
 Part (i) follows from $\ker \mathscr{H}_\alpha^*=\bigoplus_{k\in\mathbb{Z}}\ker A(k)^*$ (Lemma \ref{lem:Halphaadj-decomposable}, eq.~\eqref{eq:Halphaadj-sumkernel}), from $\ker (A(k)^*+\delta_{k,0}\mathbbm{1}_0)=\mathrm{span}\{\widetilde{\Phi}_{\alpha,k}\}\oplus\mathrm{span}\{\widetilde{\Phi}_{\alpha,k}\}$ (Lemmas \ref{lem:kerAxistar} and \ref{lem:kerAxistarzero}, and formula \eqref{eq:bilateralkernel}), and from $\|\Phi_{\alpha,k}\|_{L^2(\mathbb{R}^+)}^2\sim |k|^{-\frac{2}{1+\alpha}}$ for $k\neq 0$ (formula \eqref{eq:Phinorm}). Part (ii) follows from the identity
 \[
  \begin{split}
  (\mathscr{H}_\alpha^F+\mathbbm{1}_0)^{-1}\ker( \mathscr{H}_\alpha+\mathbbm{1}_0)^*\;&=\;\bigoplus_{k\in\mathbb{Z}\setminus\{0\}}\!(A_{\alpha,F}(k))^{-1}\ker A_\alpha(k)^*  \\
  &\qquad\qquad \oplus (A_\alpha(0)+\mathbbm{1}_0)^{-1}\ker(A_\alpha(k)^*+\mathbbm{1}_0)\,,
  \end{split}
  \]
 which is a consequence of Lemma \ref{lem:Halphaadj-decomposable} (eq.~\eqref{eq:Halphaadj-sumkernel}) and Lemma \ref{lem:HalphaFrie-decomposable}, from the identity
 \[
  (A_{\alpha,F}(k)+\delta_{k,0}\mathbbm{1}_0)^{-1}\ker (A_\alpha(k)^*+\delta_{k,0}\mathbbm{1}_0)\;=\;\mathrm{span}\{\widetilde{\Psi}_{\alpha,k}\}\oplus\mathrm{span}\{\widetilde{\Psi}_{\alpha,k}\}\,,
 \]
 which is a consequence of Lemmas \ref{lem:kerAxistar} and \ref{lem:kerAxistarzero}, and of Propositions  \ref{eq:RGisSFinv} and \ref{eq:RGisSFinv_zero_mode}, from the consequent identity
 \[
  \sum_{k\in\mathbb{Z}\setminus\{0\}}\left\|A_\alpha(k)^*\!\begin{pmatrix} c_{1,k}^-\widetilde{\Psi}_{\alpha,k} \\ c_{1,k}^+\widetilde{\Psi}_{\alpha,k}\end{pmatrix}\right\|_{\mathfrak{h}}^2\;=\;
  \sum_{k\in\mathbb{Z}\setminus\{0\}}\left\|\begin{pmatrix} c_{1,k}^-\widetilde{\Phi}_{\alpha,k} \\ c_{1,k}^+\widetilde{\Phi}_{\alpha,k}\end{pmatrix}\right\|_{\mathfrak{h}}^2\,,
 \]
 and again from the normalisation $\|\Phi_{\alpha,k}\|_{L^2(\mathbb{R}^+)}^2\sim |k|^{-\frac{2}{1+\alpha}}$. 
\end{proof}

 Lemma \ref{lem:psikinspaces} has the next follow-up concerning the fibre-wise structure of the domain of $\mathscr{H}_\alpha^*$. 
 Recall, to this aim, the general `canonical' representation of $\mathcal{D}(\mathscr{H}_\alpha^*)$ (see, e.g., \cite[Theorem 1]{GMO-KVB2017}):
 
 \begin{equation}\label{eq:repreDHstar}
  \begin{split}
  \mathcal{D}&(\mathscr{H}_\alpha^*)\;=\;\mathcal{D}((\mathscr{H}_\alpha+\mathbbm{1}_0)^*) \\
  &=\;\mathcal{D}(\overline{\mathscr{H}_\alpha+\mathbbm{1}_0})\dotplus(\mathscr{H}_{\alpha,F}+\mathbbm{1}_0)^{-1}\ker( \mathscr{H}_\alpha+\mathbbm{1}_0)^*\dotplus \ker( \mathscr{H}_\alpha+\mathbbm{1}_0)^* \\
  &=\;\mathcal{D}(\overline{\mathscr{H}_\alpha})\dotplus(\mathscr{H}_{\alpha,F}+\mathbbm{1}_0)^{-1}\ker( \mathscr{H}_\alpha^*+\mathbbm{1}_0)\dotplus\ker( \mathscr{H}_\alpha^*+\mathbbm{1}_0)\,.
  \end{split}
 \end{equation}

\begin{lemma}\label{lem:gkkrepr}
 Let $\alpha\in[0,1)$. Let $(g_k)_{k\in\mathbb{Z}}\in\cH\cong\ell^2(\mathbb{Z},\mathfrak{h})$. Then $(g_k)_{k\in\mathbb{Z}}\in\mathcal{D}(\mathscr{H}_\alpha^*)$ if and only if
 \begin{equation}\label{eq:gkrepresentation}
 g_k\;=\;\begin{pmatrix}\widetilde{\varphi}_k^- \\ \widetilde{\varphi}_k^+\end{pmatrix}+\begin{pmatrix} c_{1,k}^-\widetilde{\Psi}_{\alpha,k} \\ c_{1,k}^+\widetilde{\Psi}_{\alpha,k} \end{pmatrix}+\begin{pmatrix} c_{0,k}^-\widetilde{\Phi}_{\alpha,k} \\ c_{0,k}^+\widetilde{\Phi}_{\alpha,k} \end{pmatrix}\qquad\quad\forall k\in\mathbb{Z}
\end{equation}
with
\begin{equation}\label{eq:pileupcond1}
  (\widetilde{\varphi}_k)_{k\in\mathbb{Z}}\:\in\:\mathcal{D}(\overline{\mathscr{H}_\alpha})\,,\qquad  \widetilde{\varphi}_k\,\equiv\,\begin{pmatrix}\widetilde{\varphi}_k^- \\ \widetilde{\varphi}_k^+\end{pmatrix}\\
\end{equation}
and
\begin{eqnarray}\label{eq:pileupcond2}
 & &  \sum_{k\in \mathbb{Z}\setminus\{0\}}|k|^{-\frac{2}{1+\alpha}}|c_{0,k}^\pm|^2\:<\:+\infty\,, \\
 & &  \sum_{k\in \mathbb{Z}\setminus\{0\}}|k|^{-\frac{2}{1+\alpha}}|c_{1,k}^\pm|^2\:<\:+\infty\,. \label{eq:pileupcond3}
\end{eqnarray} 
\end{lemma}

\begin{proof} From the representation \eqref{eq:repreDHstar} one deduces at once that in order for $(g_k)_{k\in\mathbb{Z}}$ to belong to  $\mathcal{D}(\mathscr{H}_\alpha^*)$ it is necessary and sufficient that 
 \[
  (g_k)_{k\in\mathbb{Z}}\;=\;(\widetilde{\varphi}_k)_{k\in\mathbb{Z}}+(\psi_k)_{k\in\mathbb{Z}}+(\xi_k)_{k\in\mathbb{Z}}
 \]
 for some $(\widetilde{\varphi}_k)_{k\in\mathbb{Z}}\in\mathcal{D}(\overline{\mathscr{H}_\alpha})$, $(\psi_k)_{k\in\mathbb{Z}}\in (\mathscr{H}_\alpha^F+\mathbbm{1}_0)^{-1}\ker( \mathscr{H}_\alpha^*+\mathbbm{1}_0)$, and $(\xi_k)_{k\in\mathbb{Z}}\in \ker( \mathscr{H}_\alpha^*+\mathbbm{1}_0)$. The conclusion then follows from Lemma \ref{lem:psikinspaces}. 
\end{proof}

\begin{remark}
 We knew already from $\mathscr{H}_\alpha^*=\bigoplus_{k\in\mathbb{Z}}A_\alpha(k)^*$ and from the analysis of $A_\alpha(k)^*$ made in Section \ref{sec:bilateralfibreext} (formulas \eqref{eq:bilateralDAxistar} and \eqref{eq:bilateralDAxistar_zero_mode}) that an element in $\mathcal{D}(\mathscr{H}_\alpha^*)$ must have the form $(g_k)_{k\in\mathbb{Z}}$ with $g_k$ satisfying \eqref{eq:gkrepresentation} for some $\widetilde{\varphi}_k\in\mathcal{D}(\overline{A_\alpha(k)})$ and some $c_{0,k}^\pm,c_{1,k}^\pm\in\mathbb{C}$. However, a generic collection $(g_k)_{k\in\mathbb{Z}}$ in $\ell^2(\mathbb{Z},\mathfrak{h})$ of $g_k$'s satisfying \eqref{eq:gkrepresentation} does not necessarily belong to $\mathcal{D}(\mathscr{H}_\alpha^*)$, in particular the corresponding collection 
 $(\widetilde{\varphi}_k)_{k\in\mathbb{Z}\setminus\{0\}}$ does not necessarily belong to $\mathcal{D}(\overline{\mathscr{H}_\alpha})$. Only under the conditions prescribed by Lemma \ref{lem:gkkrepr} 
  can one pile up such $g_k$'s so as to obtain an actual element in $\mathcal{D}(\mathscr{H}_\alpha^*)$ (in fact, \eqref{eq:pileupcond1}-\eqref{eq:pileupcond3} impose some kind of \emph{uniformity} in $k$ of $\widetilde{\varphi}_k$, $c_{0,k}^\pm$, $c_{1,k}^\pm$). 
\end{remark}

\begin{remark}\label{rem:regularitydeficiencyspace}
Lemmas \ref{lem:psikinspaces}(i) and \ref{lem:gkkrepr} characterise $\ker(\mathscr{H}_\alpha^*+\mathbbm{1}_0)$, the deficiency space for $\mathscr{H}_\alpha+\mathbbm{1}_0$, which is isomorphic to the deficiency space of $H_\alpha$. By exploiting the same unitary equivalence \eqref{eq:unitary_transf_pm}, it was determined in the already-mentioned work \cite{Posilicano-2014-sum-trace-maps} by Posilicano that the deficiency space of $H_\alpha^+$ is isomorphic to $H^{-\frac{1}{2}\frac{1-\alpha}{1+\alpha}}(\mathbb{S}^1)$ -- more precisely, isomorphic to $H^{\frac{1}{2}\frac{1-\alpha}{1+\alpha}}(\mathbb{S}^1)$ or equivalently to $H^{-\frac{1}{2}\frac{1-\alpha}{1+\alpha}}(\mathbb{S}^1)$ depending on the different explicit isomorphisms (namely the different `coordinate systems', or also the different `boundary triplets') between the trace space and the deficiency space. Our analysis is thus completely consistent with that finding: indeed, $\mathcal{F}_2:H^{-\frac{1}{2}\frac{1-\alpha}{1+\alpha}}(\mathbb{S}^1)\stackrel{\cong}{\longrightarrow}\ell^2(\mathbb{Z},\mathbb{C}^2, \mu_k)$.
\end{remark}

After the above preparations, our subsequent analysis takes two separate directions. One, which we complete here in the remaining part of the present Section, is the characterisation of the \emph{whole family} of self-adjoint extensions of $\mathscr{H}_\alpha$ in $\cH$, an information that surely deserves interest per se. Another, which is the object of the next Section, is the study of the \emph{distinguished family} of extensions of $\mathscr{H}_\alpha$ produced by Prop.~\ref{prop:BextendsHalpha}. In fact, for the latter a clean and explicit description can be further obtained when going back to the physical variables $(x,y)$ -- and this turns out to be indeed the physically relevant sub-family of self-adjoint Hamiltonians on the Grushin cylinder.

\begin{theorem}\label{thm:Halphageneralext}
 Let $\alpha\in[0,1)$. There is a one-to-one correspondence $S\leftrightarrow \mathscr{H}_\alpha^S$ between the self-adjoint extensions $\mathscr{H}_\alpha^S$ of $\mathscr{H}_\alpha$ and the self-adjoint operators $S$ defined on Hilbert subspaces of $\ker (\mathscr{H}_\alpha^*+\mathbbm{1}_0)\cong\ell^2(\mathbb{Z},\mathbb{C}^2,\mu_k)$. If $S$ is any such operator, the corresponding extension $\mathscr{H}_\alpha^S$ is given by
 \begin{equation}\label{eq:Halphageneralext}
  \begin{split}
   \mathcal{D}(\mathscr{H}_\alpha^S)\;&=\;\left\{
   \begin{array}{c}
    \psi=\widetilde{\varphi}+ (\mathscr{H}_{\alpha,F}+\mathbbm{1}_0)^{-1}(Sv+w)+v \\
    \textrm{such that} \\
    \widetilde{\varphi}\in\mathcal{D}(\overline{\mathscr{H}_\alpha})\,,\quad v\in\mathcal{D}(S)\,, \\
    w\in\ker (\mathscr{H}_\alpha^*+\mathbbm{1}_0)\cap \mathcal{D}(S)^\perp
   \end{array}\right\} \\
   (\mathscr{H}_\alpha^S+\mathbbm{1}_0)\psi\;&=\;(\overline{\mathscr{H}_\alpha}+\mathbbm{1}_0)\widetilde{\varphi}+Sv+w\,.
  \end{split}
 \end{equation}
\end{theorem}

\begin{proof}
 A direct application of the Kre\u{\i}n-Vi\v{s}ik-Birman self-adjoint extension theory -- see, e.g., \cite[Theorem 5]{GMO-KVB2017}. The second formula in \eqref{eq:Halphageneralext} follows from the first as $(\mathscr{H}_\alpha^S+\mathbbm{1}_0)=(\mathscr{H}_\alpha^*+\mathbbm{1}_0)\upharpoonright\mathcal{D}(\mathscr{H}_\alpha^S)$. 
\end{proof}

Theorem \ref{thm:Halphageneralext} encompasses a huge variety of extensions, since $\mathscr{H}_\alpha$ has infinite deficiency index. The self-adjointness condition for each $\mathscr{H}_\alpha^S$ is in fact a \emph{restriction condition} on the domain $\mathscr{H}_\alpha^*$: in terms of the representation \eqref{eq:repreDHstar}, such a restriction selects, among the generic elements 
\[
 \psi\;=\;\widetilde{\varphi}+ (\mathscr{H}_{\alpha,F}+\mathbbm{1}_0)^{-1}\eta+\xi 
\]
of $\mathcal{D}(\mathscr{H}_\alpha^*)$, only those for which the vectors $\xi,\eta\in\ker (\mathscr{H}_\alpha^*+\mathbbm{1}_0)$ (customarily referred to as the `\emph{charges}' of $\psi$, see e.g.~\cite{MO-2016} and references therein) satisfy
\[ 
\begin{split}
 \xi\;&=\;v\;\in\;\mathcal{D}(S)\,, \\
 \eta\;&=\;Sv+w\,,\quad w\in\ker (\mathscr{H}_\alpha^*+\mathbbm{1}_0)\cap \mathcal{D}(S)^\perp\,.
\end{split}
\]
In this respect, the above condition produces in general a complicated, non-fibre-preserving mixing of the charge $\eta$ with respect to the charge $\xi$: such a mixing is encoded in the auxiliary operator $S$.

For physically relevant extensions the above mixing is absent instead, and the restriction condition of self-adjointness operates \emph{independently in each fibre}, namely in each momentum mode $k$. This is the case when
\begin{equation}\label{eq:fibredS}
 S\;=\;\bigoplus_{k\in\mathbb{Z}}S(k)\qquad \textrm{ on } \qquad \ker (\mathscr{H}_\alpha^*+\mathbbm{1}_0)\;=\;\bigoplus_{k}\,\ker(A_\alpha(k)^*+\delta_{k,0}\mathbbm{1})
\end{equation}
for operators $S(k)$'s each of which is self-adjoint on a (zero-, one-, two-dimensional) subspace $\mathcal{K}$ of the two-dimensional space $\ker(A_\alpha(k)^*+\delta_{k,0}\mathbbm{1})$. Extensions \eqref{eq:Halphageneralext} where $S$ is of the type \eqref{eq:fibredS} are \emph{fibred} in the sense that the self-adjointness condition is compatible with the fibre structure.

Explicitly, if $\mathscr{H}_\alpha^S$ is a fibred extension of $\mathscr{H}_\alpha$, then a generic element $(g_k)_{k\in\mathbb{Z}}$ of $\mathcal{D}(\mathscr{H}_\alpha^S)$ is such that 
\begin{equation}\label{eq:Birman_formula-fibred}
 g_k\;=\;\widetilde{\varphi}_k+(A_{\alpha,F}(k)+\delta_{k,0}\mathbbm{1})^{-1}(S(k) v_k+w_k)+v_k\,,\qquad k\in\mathbb{Z}\,,
\end{equation}
for some $\widetilde{\varphi}_k\in\mathcal{D}(\overline{A_\alpha(k)})$, $v_k\in\mathcal{D}(S(k))$, $w_k\in \ker (A_\alpha(k)^*+\delta_{k,0}\mathbbm{1})\cap\mathcal{D}(S(k))^\perp$.  Comparing \eqref{eq:Birman_formula-fibred} with \eqref{eq:Birman_formula} and \eqref{eq:Birman_formula_zero_mode} one immediately sees that the component $g_k$ belongs to the domain of the extension $A_\alpha^{(S(k))}(k)$ of $A_\alpha(k)$ (following the notation of \eqref{eq:Birman_formula} and \eqref{eq:Birman_formula_zero_mode}) with respect to the Hilbert space $\mathfrak{h}$.
Thus, fibred extensions of $\mathscr{H}_\alpha$ are precisely of the form $\bigoplus_{k\in\mathbb{Z}}B(k)$, where each $B(k)$ is a self-adjoint extension of $A_\alpha(k)$ in $\mathfrak{h}$, namely the extensions produced through the mechanism discussed in Prop.~\ref{prop:BextendsHalpha}.

\section{Uniformly fibred extensions of $\mathscr{H}_\alpha$}\label{sec:uniformlyfirbredext}

In this Section we focus on the most relevant and physically meaningful sub-class of self-adjoint extensions of $\mathscr{H}_\alpha$: those that we refer to as \emph{uniformly fibred extensions}. For such extensions we shall obtain a more explicit and convenient characterisation, namely Theorem \ref{thm:classificationUF} below, as compared to the general classification of Theorem \ref{thm:Halphageneralext}.

\subsection{Generalities and classification theorem}~

 The focus now is on extensions that on the one hand are \emph{fibred}, in the sense discussed in the end of Sect.~\ref{sec:genextscrHa}, hence reduced as the direct sum of self-adjoint extensions of $A_\alpha(k)$ on each fibre, and therefore with conditions of self-adjointness that do not couple different fibres, and which \emph{in addition} display the following kind of uniformity.

Let us recall that a generic fibred extension acts on each fibre as a generic self-adjoint realisation of $A_\alpha(k)$ that belongs to one of the families of the classification of Theorem \ref{thm:bifibre-extensions}, and is therefore parametrised (apart when it is $A_{\alpha,F}(k)$) by one real parameter or four real parameters. Such extension types and extension parameters may differ fibre by fibre, say, parameter $\gamma^{(k_1)}$ for an extension of type $\mathrm{I_R}$ or $\mathrm{I_L}$ or $\mathrm{II}_{a_k}$ on the $k_1$-th fibre, and parameters $\gamma_1^{(k_2)},\dots,\gamma_4^{(k_2)}$ for an extension of type $\mathrm{III}$ on the $k_2$-th fibre.

\emph{Uniformly} fibred extensions are those for which, fibre by fibre, the type of extension of $A_\alpha(k)$ is the same, and all have the same extension parameter(s) $\gamma$ (and $a$), or $\gamma_1,\dots,\gamma_4$.

By definition, uniformly fibred extensions can be therefore grouped into sub-families in complete analogy to those of Theorem \ref{thm:bifibre-extensions}:
\begin{itemize}
  \item \underline{Friedrichs extension}: the operator $\mathscr{H}_{\alpha,F}=\bigoplus_{k\in\mathbb{Z}}A_{\alpha,F}(k)$ (see Lemma \ref{lem:HalphaFrie-decomposable}); 
  \item \underline{Family $\mathrm{I_R}$}: operators of the form
  \begin{equation}\label{eq:HalphaR_unif-fibred}
   \mathscr{H}_{\alpha,R}^{[\gamma]}\;:=\;\bigoplus_{k\in\mathbb{Z}}A_{\alpha,R}^{[\gamma]}(k)
  \end{equation}
  for some $\gamma\in\mathbb{R}$;
  \item \underline{Family $\mathrm{I_L}$}: operators of the form
  \begin{equation}\label{eq:HalphaL_unif-fibred}
   \mathscr{H}_{\alpha,L}^{[\gamma]}\;:=\;\bigoplus_{k\in\mathbb{Z}}A_{\alpha,L}^{[\gamma]}(k)
  \end{equation}
 for some $\gamma\in\mathbb{R}$;
  \item \underline{Family $\mathrm{II}_a$} for given $a\in\mathbb{C}$: operators of the form
  \begin{equation}\label{eq:Halpha-IIa_unif-fibred}
   \mathscr{H}_{\alpha,a}^{[\gamma]}\;:=\;\bigoplus_{k\in\mathbb{Z}}A_{\alpha,a}^{[\gamma]}(k)
  \end{equation}
 for some $\gamma\in\mathbb{R}$;
  \item \underline{Family $\mathrm{III}$}: operators of the form
  \begin{equation}\label{eq:Halpha-III_unif-fibred}
   \mathscr{H}_{\alpha}^{[\Gamma]}\;:=\;\bigoplus_{k\in\mathbb{Z}}A_{\alpha}^{[\Gamma]}(k)
  \end{equation}
 for some $\Gamma\equiv(\gamma_1,\gamma_2,\gamma_3,\gamma_4)\in\mathbb{R}^4$.
 \end{itemize}

Physically, uniformly fibred extensions have surely a special status in that the boundary condition experienced as $x\to 0$ by the quantum particle governed by any such Hamiltonian has both the same form and the same `magnitude' (hence the same $\gamma$-parameter, or $\gamma_j$-parameters) irrespective of the transversal momentum, namely the quantum number $k$. 

 Mathematically, uniformly fibred extensions have a completely explicit description not only in mixed position-momentum variables $(x,k)$, namely extensions of $\mathscr{H}_\alpha$, but also in the original physical coordinates $(x,y)$, namely extensions of the symmetric operator $\mathsf{H}_\alpha=\mathcal{F}_2^{-1}\mathscr{H}_\alpha\mathcal{F}_2$ acting on $L^2(\mathbb{R}\times\mathbb{S}^1,\ud x\ud y)$, explicitly described in \eqref{eq:explicit-tildeHalpha}. 

This is the content of the main result of the present Section.

\begin{theorem}\label{thm:classificationUF}
 Let $\alpha\in[0,1)$. The densely defined, symmetric operator
 \[
  \begin{split}
   \mathsf{H}_\alpha\;&=\;\mathcal{F}_2^{-1}\mathscr{H}_\alpha\mathcal{F}_2\;=\;-\frac{\partial^2}{\partial x^2}- |x|^{2\alpha}\frac{\partial^2}{\partial y^2}+\frac{\,\alpha(2+\alpha)\,}{4x^2}\,, \\
  \mathcal{D}(\mathsf{H}_\alpha^\pm)\;&=\;C^\infty_c(\mathbb{R}^\pm_x\times\mathbb{S}^1_y)
  \end{split}
 \]
 admits, among others, the following families of self-adjoint extensions in $L^2(\mathbb{R}\times\mathbb{S}^1,\ud x \ud y):$  
 \begin{itemize}
  \item \underline{Friedrichs extension}: $\mathsf{H}_{\alpha,F}$, where $\mathsf{H}_{\alpha,F}=\mathcal{F}_2^{-1}\mathscr{H}_{\alpha,F}\mathcal{F}_2$;
  \item \underline{Family $\mathrm{I_R}$}: $\{\mathsf{H}_{\alpha,R}^{[\gamma]}\,|\,\gamma\in\mathbb{R}\}$, where $\mathsf{H}_{\alpha,R}^{[\gamma]}=\mathcal{F}_2^{-1}\mathscr{H}_{\alpha,R}^{[\gamma]}\,\mathcal{F}_2$;
  \item \underline{Family $\mathrm{I_L}$}: $\{\mathsf{H}_{\alpha,L}^{[\gamma]}\,|\,\gamma\in\mathbb{R}\}$, where $\mathsf{H}_{\alpha,L}^{[\gamma]}=\mathcal{F}_2^{-1}\mathscr{H}_{\alpha,L}^{[\gamma]}\,\mathcal{F}_2$;
  \item \underline{Family $\mathrm{II}_a$} with $a\in\mathbb{C}$: $\{\mathsf{H}_{\alpha,a}^{[\gamma]}\,|\,\gamma\in\mathbb{R}\}$, where $\mathsf{H}_{\alpha,a}^{[\gamma]}=\mathcal{F}_2^{-1} \mathscr{H}_{\alpha,a}^{[\gamma]}\,\mathcal{F}_2$;
  \item \underline{Family $\mathrm{III}$}: $\{\mathsf{H}_{\alpha}^{[\Gamma]}\,|\,\Gamma\equiv(\gamma_1,\gamma_2,\gamma_3,\gamma_4)\in\mathbb{R}^4\}$, where $\mathsf{H}_{\alpha}^{[\Gamma]}=\mathcal{F}_2^{-1} \mathscr{H}_{\alpha}^{[\Gamma]}\,\mathcal{F}_2$.
 \end{itemize}
 Each element from any such family is characterised by being the \emph{restriction} of the adjoint operator
  \begin{equation}\label{eq:HHalphaadjointagain}
  \begin{split}
    \mathcal{D}(\mathsf{H}_\alpha^*)\;&=\;\left\{\!
  \begin{array}{c}
   \phi\in L^2(\mathbb{R}\times\mathbb{S}^1,\ud x \ud y)\textrm{ such that} \\
   \Big(-\frac{\partial^2}{\partial x^2}- |x|^{2\alpha}\frac{\partial^2}{\partial y^2}+\frac{\,\alpha(2+\alpha)\,}{4x^2}\Big)\phi^\pm\in L^2(\mathbb{R}^\pm\times\mathbb{S}^1,\ud x \ud y)
  \end{array}\!
  \right\}, \\
  (\mathsf{H}_\alpha^\pm)^*\phi^\pm\;&=\;-\frac{\partial^2\phi^\pm}{\partial x^2}- |x|^{2\alpha}\frac{\partial^2\phi^\pm}{\partial y^2}+\frac{\,\alpha(2+\alpha)\,}{4x^2}\phi^\pm
  \end{split}
\end{equation}
 to the functions
 \[
  \phi\;=\;\begin{pmatrix} \phi^- \\ \phi^+ \end{pmatrix},\qquad \phi^\pm\;\in\;L^2(\mathbb{R}^\pm\times\mathbb{S}^1,\ud x \ud y)
 \]
 for which the limits
  \begin{eqnarray}
   \phi_0^\pm(y)\!\!&=&\!\!\lim_{x\to 0^\pm} |x|^{\frac{\alpha}{2}}\,\phi^\pm(x,y)\,, \label{eq:limitphi0} \\
   \phi_1^\pm(y)\!\!&=&\!\!\lim_{x\to 0^\pm} |x|^{-(1+\frac{\alpha}{2})}\big(\phi^\pm(x,y)-\phi^\pm_0(y)|x|^{-\frac{\alpha}{2}}\big) \label{eq:limitphi1} \\
   &=&\!\!\pm(1+\alpha)^{-1}\lim_{x\to 0^\pm} |x|^{-\alpha}\partial_x\big(|x|^{\frac{\alpha}{2}}\phi^\pm(x,y)\big) \nonumber
  \end{eqnarray}
 exist and are finite for almost every $y\in\mathbb{S}^1$, and satisfy the following boundary conditions, depending on the considered type of extension, for almost every $y\in\mathbb{R}$:
 \begin{eqnarray}
  \phi_0^\pm(y)\,=\,0\,, \qquad \quad\;\;& & \textrm{if }\;  \phi\in\mathcal{D}(\mathsf{H}_{\alpha,F})\,, \label{eq:DHalpha_cond3_Friedrichs-NOWEIGHTS}\\
  \begin{cases}
   \;\phi_0^-(y)= 0\,,  \\
   \;\phi_1^+(y)=\gamma \phi_0^+(y)\,,
  \end{cases} & & \textrm{if }\;  \phi\in\mathcal{D}(\mathsf{H}_{\alpha,F}^{[\gamma]})\,, \\
   \begin{cases}
   \;\phi_1^-(y)=\gamma \phi_0^-(y)\,, \\
   \;\phi_0^+(y)= 0 \,,
  \end{cases} & & \textrm{if }\;  \phi\in\mathcal{D}(\mathsf{H}_{\alpha,L}^{[\gamma]}) \,, \label{eq:DHalpha_cond3_L-NOWEIGHTS}\\
     \begin{cases}
   \;\phi_0^+(y)=a\,\phi_0^-(y)\,, \\
   \;\phi_1^-(y)+\overline{a}\,\phi_1^+(y)=\gamma \phi_0^-(y)\,,
  \end{cases} & & \textrm{if }\;  \phi\in\mathcal{D}(\mathsf{H}_{\alpha,a}^{[\gamma]})\,, \label{eq:DHalpha_cond3_IIa-NOWEIGHTS} \\
   \begin{cases}
   \;\phi_1^-(y)=\gamma_1 \phi_0^-(y)+(\gamma_2+\ii\gamma_3) \phi_0^+(y)\,, \\
   \;\phi_1^+(y)=(\gamma_2-\ii\gamma_3) \phi_0^-(y)+\gamma_4 \phi_0^+(y)\,,
  \end{cases} & & \textrm{if }\;  \phi\in\mathcal{D}(\mathsf{H}_{\alpha}^{[\Gamma]})\,. \label{eq:DHalpha_cond3_III-NOWEIGHTS}
 \end{eqnarray} 
 Moreover,
 \begin{equation}
  \phi_0^\pm \in H^{s_{0,\pm}}(\mathbb{S}^1, \ud y)\qquad\textrm{ and }\qquad \phi_1^\pm\in H^{s_{1,\pm}}(\mathbb{S}^1,\ud y)
 \end{equation}
 with
 \begin{itemize}
 	\item $s_{1,\pm}=\frac{1}{2}\frac{1-\alpha}{1+\alpha}$\qquad\qquad\qquad\qquad\qquad\; for the Friedrichs extension,
 	\item $s_{1,-}=\frac{1}{2}\frac{1-\alpha}{1+\alpha}$, $s_{0,+}=s_{1,+}=\frac{1}{2}\frac{3+\alpha}{1+\alpha}$ \;\;\; for extensions of type $\mathrm{I}_R$,
 	\item  $s_{1,+}=\frac{1}{2}\frac{1-\alpha}{1+\alpha}$, $s_{0,-}=s_{1,-}=\frac{1}{2}\frac{3+\alpha}{1+\alpha}$ \quad for extensions of type $\mathrm{I}_L$,
 	\item $s_{1,\pm}=s_{0,\pm}=\frac{1}{2}\frac{1-\alpha}{1+\alpha}$ \qquad\qquad\qquad \;\;\;\,for extensions of type $\mathrm{II}_a$,
 	\item $s_{1,\pm}=s_{0,\pm}=\frac{1}{2}\frac{3+\alpha}{1+\alpha}$ \qquad\qquad\qquad \;\; for extensions of type $\mathrm{III}$.
 \end{itemize}
\end{theorem}

\subsection{General strategy}\label{sec:genstrategy}~

The proof of Theorem \ref{thm:classificationUF} is going to require quite a detailed analysis, as we shall now explain. All the preparation is developed in Subsect.~\ref{sec:genstrategy} through \ref{sec:q0q1}, and the proof will be discussed in Subsect.~\ref{sec:proofclassifthm}.

The trivial part is of course the reconstruction of each uniformly fibred extension of $\mathscr{H}_\alpha$ through a direct sum of self-adjoint extensions of the $A_\alpha(k)$'s. Instead, the difficult part is to extract the appropriate information so as to export the boundary conditions of self-adjointness from the mixed position-momentum variables $(x,k)$ to the physical coordinates $(x,y)$. The inverse Fourier transform $\mathcal{F}_2^{-1}$ is indeed a non-local operation, and in order to `add up' the boundary conditions initially available $k$ by $k$, one needs suitable \emph{uniformity} controls in $k$.

Let $\mathscr{H}_\alpha^{\mathrm{u.f.}}$ be a uniformly fibred extension of $\mathscr{H}_\alpha$. A generic element $(g_k)_{k\in\mathbb{Z}}\in\mathcal{D}(\mathscr{H}_\alpha^{\mathrm{u.f.}})$ can be represented as in \eqref{eq:gkrepresentation} with the `summability' conditions \eqref{eq:pileupcond1}-\eqref{eq:pileupcond2} that guarantee $(g_k)_{k\in\mathbb{Z}}$ to belong to $\mathcal{D}(\mathscr{H}_\alpha^*)$ (Lemma \ref{lem:gkkrepr}), plus additional constraints among the coefficients $c_{0,k}^\pm$ and $c_{1,k}^\pm$ that guarantee that $\mathcal{D}(\mathscr{H}_\alpha^{\mathrm{u.f.}})$ is indeed a domain of self-adjointness. Actually, the latter requirement imposes \emph{stronger} summability conditions on the $c_{0,k}^\pm$'s and $c_{1,k}^\pm$'s, as we shall discuss in Subsect.~\ref{subsec:g0g1Sobolev}.

However, the above-mentioned representation \eqref{eq:gkrepresentation} for the elements of $\mathcal{D}(\mathscr{H}_\alpha^{\mathrm{u.f.}})$ is problematic when one needs to describe $\mathcal{F}_2^{-1}\mathcal{D}(\mathscr{H}_\alpha^{\mathrm{u.f.}})$, namely the same domain in $(x,y)$-coordinates (it is immediate from \eqref{eq:unitary_transf_pm} that $\mathcal{F}_2^{-1}\mathcal{D}(\mathscr{H}_\alpha^{\mathrm{u.f.}})$ is the domain of the self-adjoint extension $\mathcal{F}_2^{-1}\mathscr{H}_\alpha^{\mathrm{u.f.}}\mathcal{F}_2$ of $\mathsf{H}_\alpha=\mathcal{F}_2^{-1}\mathscr{H}_\alpha\mathcal{F}_2$).

More precisely, when applying $\mathcal{F}_2^{-1}$ to \eqref{eq:gkrepresentation}, one loses control on the self-adjointness constraint that now becomes a rather implicit condition between the $(x,y)$-functions
\begin{equation}\label{eq:badF2-1}
 \mathcal{F}_2^{-1} \left(\begin{pmatrix} c_{1,k}^-\widetilde{\Psi}_{\alpha,k} \\ c_{1,k}^+\widetilde{\Psi}_{\alpha,k} \end{pmatrix}\right)_{\!k\in\mathbb{Z}}\,,\qquad\qquad \mathcal{F}_2^{-1} \left(\begin{pmatrix} c_{0,k}^-\widetilde{\Phi}_{\alpha,k} \\ c_{0,k}^+\widetilde{\Phi}_{\alpha,k} \end{pmatrix}\right)_{\!k\in\mathbb{Z}}.
\end{equation}
Recall indeed from \eqref{eq:defF2} that
\[
 (\mathcal{F}_2^+)^{-1}(( c_{1,k}^+\widetilde{\Psi}_{\alpha,k}))_{k\in\mathbb{Z}}\;=\;\frac{1}{\sqrt{2\pi}}\sum_{k\in\mathbb{Z}}c_{1,k}\widetilde{\Psi}_{\alpha,k}(x) e^{\ii k y}\,,
\]
and similarly for the other components: on such functions of $x$ and $y$ it is not evident if differentiating or taking the limit $x\to 0$ term by term in the series in $k$ is actually justified -- and it is precisely in terms of such operations that the final boundary conditions are going to be expressed.

From another perspective, the known regularity and asymptotic properties of $\widetilde{\Psi}_{\alpha,k}$ (and, analogously, $\widetilde{\Phi}_{\alpha,k}$) may well provide the above information on the function  $(\mathcal{F}_2^+)^{-1}((\widetilde{\Psi}_{\alpha,k}))_{k\in\mathbb{Z}}$, but it is not evident how to read out useful information from $(\mathcal{F}_2^+)^{-1}(( c_{1,k}^+\widetilde{\Psi}_{\alpha,k}))_{k\in\mathbb{Z}}$ so as to finally express the boundary conditions of self-adjointness in terms of limits as $x\to 0$ of the functions in the domain and on their derivatives.

As taking the inverse Fourier transform directly on \eqref{eq:gkrepresentation} appears not to be informative in practice, we shall follow a second route inspired to the alternative representation \eqref{eq:gwithPweight} (Theorem \ref{prop:g_with_Pweight}).

Now the generic element $(g_k)_{k\in\mathbb{Z}}\in\mathcal{D}(\mathscr{H}_\alpha^{\mathrm{u.f.}})$ is represented for each $k$ as
 \begin{equation}\label{eq:gwithPweight_k}
  g_k\;=\;\begin{pmatrix} \varphi_k^- \\ \varphi_k^+ \end{pmatrix}+\begin{pmatrix} g_{0,k}^- \\ g_{0,k}^+\end{pmatrix}|x|^{-\frac{\alpha}{2}}\,P+\begin{pmatrix} g_{1,k}^- \\ g_{1,k}^+\end{pmatrix}|x|^{1+\frac{\alpha}{2}}\,P
 \end{equation}
where each $\varphi_k\in\mathcal{D}(\overline{A_\alpha(k)})$ and $P$ is the short-scale cut-off \eqref{eq:Pcutoff}.

The evident advantage of \eqref{eq:gwithPweight_k}, as compared to \eqref{eq:gkrepresentation}, is that computing
\begin{equation}
 \phi\;:=\;\mathcal{F}_2^{-1}(g_k)_{k\in\mathbb{Z}}
\end{equation}
and using the linearity of $\mathcal{F}_2^{-1}$ yields \emph{formally}
\begin{equation}\label{eq:afterF2-1}
 \phi(x,y)\;=\;\varphi(x,y)+g_1(y)|x|^{1+\frac{\alpha}{2}}P(x)+g_0(y)|x|^{-\frac{\alpha}{2}}P(x)
\end{equation}
with
\begin{eqnarray}
 \varphi\!\!&:=&\!\!\mathcal{F}_2^{-1}(\varphi_k)_{k\in\mathbb{Z}} \,,\label{eq:F2-1phi}\\
 g_0\!\!&:=&\!\!\mathcal{F}_2^{-1}(g_{0,k})_{k\in\mathbb{Z}}\,, \label{e1:F2-1g0}\\
 g_1\!\!&:=&\!\!\mathcal{F}_2^{-1}(g_{1,k})_{k\in\mathbb{Z}} \label{eq:F2-1g1}\,.
\end{eqnarray}

In \eqref{eq:afterF2-1} the function $\varphi$ is expected to retain the regularity in $x$ and the fast vanishing properties, as $x\to 0$, of each $\varphi_k$, and hence $\varphi$ is expected to be a sub-leading term when taking $\lim_{x\to 0}\phi(x,y)$ and $\lim_{x\to 0}\partial_x\phi(x,y)$; on the other hand, the regularity and short-distance behaviour in $x$ of the other two summands in the r.h.s.~of \eqref{eq:afterF2-1} are immediately read out, unlike the situation with the functions \eqref{eq:badF2-1}. Moreover, and most importantly, since $\mathscr{H}_\alpha^{\mathrm{u.f.}}$ is a \emph{uniformly fibred} extension, the boundary condition of self-adjointness in \eqref{eq:gwithPweight_k} (namely a condition among those listed in the third column of Table \ref{tab:extensions}) takes the same form, with the same extension parameter, irrespective of $k$, and therefore is immediately exported, in the same form and with the same extension parameter, between $g_0(y)$ and $g_1(y)$ for almost every $y\in\mathbb{S}^1$.

The above reasoning paves the way to a classification of the family of uniformly fibred extensions of $H_\alpha$ in terms of explicit boundary conditions as $x\to 0$.

Clearly, so far \eqref{eq:afterF2-1} is only formal: one must guarantee that \eqref{eq:F2-1phi}-\eqref{eq:F2-1g1} are actually well-posed and define square-integrable functions in the corresponding variables, with the desired properties. This is in fact the price to pay for the present strategy, whereas for the functions \eqref{eq:badF2-1} it was clear a priori that $\mathcal{F}_2^{-1}$ is applicable, thanks to Lemma \ref{eq:repreDHstar}.

As we shall comment further on (Subsect.~\ref{sec:singular_decomposition_adjoint}), such a strategy will lead to the following somewhat awkward circumstance:  whereas Lemma \ref{eq:repreDHstar} guarantees that applying $\mathcal{F}_2^{-1}$ on $(g_k)_{k\in\mathbb{Z}}$ represented as in \eqref{eq:gkrepresentation} yields three distinct functions, each of which belongs to $\mathcal{F}_2^{-1}\mathcal{D}(\mathscr{H}_\alpha^*)=\mathcal{D}(\mathsf{H}_\alpha^*)$, the three summands in the r.h.s.~of \eqref{eq:afterF2-1} will be proved to belong to $L^2(\mathbb{R}\times\mathbb{S}^1,\ud x\ud y)$, \emph{none} of which being however in $\mathcal{D}(\mathsf{H}_\alpha^*)$ in general! -- only their sum is, due to cancellations of singularities. This explains why the analysis is going to be onerous.

\subsection{Integrability and Sobolev regularity of $g_0$ and $g_1$}\label{subsec:g0g1Sobolev}~

Following the programme outlined in the previous Subsection, let us show that \eqref{e1:F2-1g0} and \eqref{eq:F2-1g1} indeed define functions in $L^2(\mathbb{S}^1,\ud y)$ with suitable regularity.

\begin{proposition}\label{prop:g0g1Sobolev}
 Let $\alpha\in[0,1)$ and let $(g_k)_{k\in\mathbb{Z}}\in\mathcal{D}(\mathscr{H}_\alpha^{\mathrm{u.f.}})$, where $\mathscr{H}_\alpha^{\mathrm{u.f.}}$ is one of the operators \eqref{eq:HalphaFriedrichs_unif-fibred} or \eqref{eq:HalphaR_unif-fibred}-\eqref{eq:Halpha-III_unif-fibred}, for given parameters $\gamma\in\mathbb{R}$, $a\in\mathbb{C}$, $\Gamma\in\mathbb{R}^4$, depending on the type. With respect to the representation \eqref{eq:gwithPweight_k} of each $g_k$, one has the following. 
\begin{itemize}
	\item[(i)] If $\mathscr{H}_{\alpha}^{\mathrm{u.f.}}$ is the Friedrichs extension, then
	\begin{equation}\label{eq:g0g1Sobolev_typeF}
		\sum_{k \in \mathbb{Z}} |k|^{\frac{1-\alpha}{1+\alpha}}|g_{1,k}^\pm|^2 \; < \;+\infty \,,\qquad g_{0,k}^\pm\;=\;0\,.
	\end{equation}
	\item[(ii)] If $\mathscr{H}_{\alpha}^{\mathrm{u.f.}}$ is of type $\mathrm{I}_{\mathrm{R}}$, then
	\begin{equation}\label{eq:g0g1Sobolev_typeIR}
		\begin{split}
		\sum_{k \in \mathbb{Z}} |k|^{\frac{1-\alpha}{1+\alpha}}|g_{1,k}^-|^2  \; < \; +\infty& \,, \qquad g_{0,k}^- \;=\;0\,, \\
		\sum_{k \in \mathbb{Z}}  |k|^{\frac{3+\alpha}{1+\alpha}}|g_{1,k}^+|^2 \; < \; + \infty& \,, \qquad \sum_{k \in \mathbb{Z}}  |k|^{\frac{3+\alpha}{1+\alpha}}|g_{0,k}^+|^2 \; < \; + \infty\,.
		\end{split}
	\end{equation}
	\item[(iii)] If $\mathscr{H}_{\alpha}^{\mathrm{u.f.}}$ is of type $\mathrm{I}_{\mathrm{L}}$, then
	\begin{equation}\label{eq:g0g1Sobolev_typeIL}
		\begin{split}
		\sum_{k \in \mathbb{Z}} |k|^{\frac{1-\alpha}{1+\alpha}}|g_{1,k}^+|^2  \; < \; +\infty& \,, \qquad g_{0,k}^+ \;=\;0\,, \\
		\sum_{k \in \mathbb{Z}} |k|^{\frac{3+\alpha}{1+\alpha}} |g_{1,k}^-|^2  \; < \; + \infty& \,, \qquad \sum_{k \in \mathbb{Z}} |k|^{\frac{3+\alpha}{1+\alpha}}|g_{0,k}^-|^2  \; < \; + \infty\,.
		\end{split}
	\end{equation}
	\item[(iv)] If $\mathscr{H}_{\alpha}^{\mathrm{u.f.}}$ is of type $\mathrm{II}_a$, then
	\begin{equation}\label{eq:g0g1Sobolev_typeIIa}
		\begin{split}
		&\sum_{k \in \mathbb{Z}} |k|^{\frac{1-\alpha}{1+\alpha}}|g_{1,k}^\pm|^2  \; < \; +\infty \,, \qquad \sum_{k \in \mathbb{Z}}  |k|^{\frac{3+\alpha}{1+\alpha}} |g_{0,k}^\pm|^2\; < \; + \infty\,, \\
		&\sum_{k \in \mathbb{Z}} |k|^{\frac{3+\alpha}{1+\alpha}}|g_{1,k}^-+ \overline{a} g_{1,k}^+|^2  \; < \; + \infty\,. 
		\end{split}
	\end{equation}
	\item[(v)] If $\mathscr{H}_{\alpha}^{\mathrm{u.f.}}$ is of type $\mathrm{III}$, then
	\begin{equation}\label{eq:g0g1Sobolev_typeIII}
	   \sum_{k\in\mathbb{Z}}|k|^{\frac{3+\alpha}{1+\alpha}}|g^\pm_{0,k}|^2\;<\;+\infty\,,\qquad \sum_{k\in\mathbb{Z}}|k|^{\frac{3+\alpha}{1+\alpha}}|g^\pm_{1,k}|^2\;<\;+\infty\,.
	\end{equation}
\end{itemize} 
\end{proposition}

\begin{corollary}\label{cor:(g0)k_(g1)k_in_Hs}
 Under the assumptions of Proposition \ref{prop:g0g1Sobolev}, $(g^\pm_{0,k})_{k\in\mathbb{Z}}$ and $(g^\pm_{1,k})_{k\in\mathbb{Z}}$ belong $\ell^2(\mathbb{Z})$. Hence, \eqref{e1:F2-1g0} and \eqref{eq:F2-1g1} define functions $y\mapsto g_0^\pm(y)$ and $y\mapsto g_1^\pm(y)$ that belong to $L^2(\mathbb{S}^1,\ud y)$. In particular, the summability properties \eqref{eq:g0g1Sobolev_typeF}-\eqref{eq:g0g1Sobolev_typeIII} imply that $g_0^\pm \in H^{s_{0,\pm}}(\mathbb{S}^1, \ud y)$ and $g_1^\pm\in H^{s_{1,\pm}}(\mathbb{S}^1,\ud y)$, where the order of such Sobolev spaces is, respectively,
 \begin{itemize}
 	\item[(i)] $s_{1,\pm}=\frac{1}{2}\frac{1-\alpha}{1+\alpha}$\qquad\qquad\qquad\qquad\qquad\; for the Friedrichs extension,
 	\item[(ii)] $s_{1,-}=\frac{1}{2}\frac{1-\alpha}{1+\alpha}$, $s_{0,+}=s_{1,+}=\frac{1}{2}\frac{3+\alpha}{1+\alpha}$ \quad for extensions of type $\mathrm{I}_R$,
 	\item[(iii)]  $s_{1,+}=\frac{1}{2}\frac{1-\alpha}{1+\alpha}$, $s_{0,-}=s_{1,-}=\frac{1}{2}\frac{3+\alpha}{1+\alpha}$ \quad for extensions of type $\mathrm{I}_L$,
 	\item[(iv)] $s_{1,\pm}=s_{0,\pm}=\frac{1}{2}\frac{1-\alpha}{1+\alpha}$ \qquad\qquad\qquad \;\;\;\,for extensions of type $\mathrm{II}_a$,
 	\item[(v)] $s_{1,\pm}=s_{0,\pm}=\frac{1}{2}\frac{3+\alpha}{1+\alpha}$ \qquad\qquad\qquad \;\; for extensions of type $\mathrm{III}$.
 \end{itemize}
 \end{corollary}

\begin{proof}[Proof of Proposition \ref{prop:g0g1Sobolev}]
For each case, the proof is organised in two stages. First, we consider each family of extensions on fibre Hilbert space as characterised by Theorem \ref{thm:bifibre-extensionsc0c1} in terms of certain self-adjointness constraints between the coefficients $c_0^\pm$ and $c_1^\pm$ of the representation \eqref{eq:repreDHstar}-\eqref{eq:gkrepresentation} of the elements of $\mathcal{D}(\mathscr{H}_\alpha^*)$, and we show that owing to such constraints the a priori summability \eqref{eq:pileupcond2}-\eqref{eq:pileupcond3} of the $c_0^\pm$'s and $c_1^\pm$'s is actually enhanced (see also Remark \ref{rem:enhanced_summability} below). Then, we export the resulting summability of the $c_0^\pm$'s and $c_1^\pm$'s on to the $g_0^\pm$'s and $g_1^\pm$'s by means of the relations
\begin{eqnarray}
 g_{0,k}^\pm \!& = &\!c_{0,k}^\pm \textstyle{\sqrt{\frac{\pi(1+\alpha)}{2 |k|}}}\,, \label{eq:Relationg0c0} \\
 g_{1,k}^\pm \!& = &\! c_{1,k}^\pm \textstyle{\sqrt{\frac{2 |k|}{\pi (1+\alpha)^3}}} \Vert \Phi_{\alpha,k} \Vert_{L^2(\mathbb{R}^+)}^2 - c_{0,k}^\pm \textstyle{\sqrt{\frac{\pi |k|}{2(1+\alpha)}}}\label{eq:Relationg1c1}
\end{eqnarray}
valid for $k\neq 0$ (see \eqref{eq:limitsg0g1} above). Obviously, it suffices to prove the final summability properties for $k\in\mathbb{Z}\setminus\{0\}$. Let us also recall from \eqref{eq:Phinorm} that
\[
 \Vert \Phi_{\alpha,k} \Vert_{L^2(\mathbb{R}^+)}^2\;\sim\; |k|^{-\frac{2}{1+\alpha}}\,,
\]
namely for some multiplicative constant depending only on $\alpha$.


(i) Theorem \ref{thm:bifibre-extensionsc0c1} states that for this case $c_{0,k}^\pm =0$. This, together with  \eqref{eq:pileupcond2} and \eqref{eq:Relationg1c1}, yields
\[
	+ \infty \; > \sum_{k\in\mathbb{Z}\setminus\{0\}} |k|^{-\frac{2}{1+\alpha}} |c_{1,k}^\pm|^2 \;\sim\sum_{k\in\mathbb{Z}\setminus\{0\}} |k|^{\frac{1-\alpha}{1+\alpha}} |g_{1,k}^\pm |^2\,.
\]

(ii) Theorem \ref{thm:bifibre-extensionsc0c1} states that for this case $c_{0,k}^+=0$ and $c_{1,k}^+ =\beta_k c_{0,k}^+$ with $\beta_k$ given for $k\neq 0$ by 
\[
\gamma \; = \; \textstyle\frac{|k|}{1+\alpha}\Big(\frac{\, 2 \|\Phi_{\alpha,k}\|_{L^2}^2 \,}{\pi(1+\alpha)}\, \beta_k-1\Big)
\]
(see \eqref{eq:g1gammag0} above), that is, $\beta_k\sim |k|^{\frac{2}{1+\alpha}}$ at the leading order in $k$. (Here the operator of multiplication by $\beta_k$ is what we denoted in abstract by $S(k)$ in the discussion following Theorem \eqref{thm:Halphageneralext} -- see \eqref{eq:fibredS} above.) This, together with  \eqref{eq:pileupcond3} and \eqref{eq:Relationg0c0} yields
\[
 +\infty\;>\sum_{k\in\mathbb{Z}\setminus\{0\}} |k|^{-\frac{2}{1+\alpha}} |c_{1,k}^+|^2\;=\sum_{k\in\mathbb{Z}\setminus\{0\}} |k|^{-\frac{2}{1+\alpha}} |\beta_k c_{0,k}^+|^2\;\sim\sum_{k\in\mathbb{Z}\setminus\{0\}} |k|^{\frac{3+\alpha}{1+\alpha}}|g_{0,k}^+|^2\,.
\]
From this one also obtains
\[
 \sum_{k\in\mathbb{Z}\setminus\{0\}} |k|^{\frac{3+\alpha}{1+\alpha}}|g_{1,k}^+|^2\;<\;+\infty\,,
\]
owing to the self-adjointness condition in the form $g_{1,k}^+=\gamma g_{0,k}^+$ (Theorem \ref{thm:bifibre-extensions}). As for the summability of the $c_{1,k}^-$, one proceeds precisely as in case (i).

(iii) The reasoning for this case is completely analogous as for case (ii), upon exchanging the `$+$' coefficients with the `$-$' coefficients.

(iv) Theorem \ref{thm:bifibre-extensionsc0c1} states for this case 
\[
	\begin{split}
		c_{0,k}^- \;&=\; c_{0,k} \, , \qquad\qquad\! c_{1,k}^- \; = \; \tau_k c_{0,k} + \widetilde{c}_{0,k} \, , \\
		c_{0,k}^+ \;& = \; a  c_{0,k}\,, \qquad\quad\; c_{1,k}^+ \;=\; \tau_k a c_{0,k} - \overline{a}^{-1} \widetilde{c}_{0,k}\,,
	\end{split}
\]
with $\tau_k$ given for $k \neq 0$ by
\[
			  \gamma\;:=\;\textstyle{(1+|a|^2) \frac{|k|}{1+\alpha}} \Big( \frac{\,2 \|\Phi_{\alpha,k}\|_{L^2}^2 \,}{\pi (1+\alpha)} \, \tau_k -1 \Big)\,,
\]
(see \eqref{eq:IkTheorem51} above), that is, $\tau_k \sim |k|^{\frac{2}{1+\alpha}}$ at the leading order in $k$. This, together with the a priori bounds \eqref{eq:pileupcond3}, and with \eqref{eq:Relationg0c0}, yields
\[
\begin{split}
 + \infty  \; &> \; \sum_{k\in\mathbb{Z}\setminus\{0\}} |k|^{-\frac{2}{1+\alpha}} |c_{1,k}^- + \overline{a} c_{1,k}^+ |^2 \;=\; \sum_{k\in\mathbb{Z}\setminus\{0\}} |k|^{-\frac{2}{1+\alpha}} |(1+|a|^2) \tau_k c_{0,k} |^2 \\
 &\sim \; \sum_{k\in\mathbb{Z}\setminus\{0\}}|k|^{\frac{3+\alpha}{1+\alpha}} |g_{0,k}^-|^2 \, .
\end{split}
\]
From this, and self-adjointness conditions $g_{0,k}^+= a g_{0,k}^-$ and $g_{1,k}^-+\overline{a} g_{1,k}^+= \gamma g_{0,k}^-$  (Theorem \ref{thm:bifibre-extensions}), one obtains the last two conditions in \eqref{eq:g0g1Sobolev_typeIIa}. As for establishing the first condition in \eqref{eq:g0g1Sobolev_typeIIa}, one has
\[
 \begin{split}
  \sum_{k\in\mathbb{Z}\setminus\{0\}} &|k|^{\frac{1-\alpha}{1+\alpha}}\,|g_{1,k}^{\pm}|^2 \\
  &\leqslant\; \sum_{k\in\mathbb{Z}\setminus\{0\}}|k|^{\frac{1-\alpha}{1+\alpha}}\,|c_{1,k}^{\pm}|^2{\textstyle{\frac{4|k|}{\pi(1+\alpha^3)}}}\|\Phi_{\alpha,k}\|_{L^2(\mathbb{R}^+)}^4+ \!\!\sum_{k\in\mathbb{Z}\setminus\{0\}}|k|^{\frac{1-\alpha}{1+\alpha}}\,|c_{0,k}^{\pm}|^2\textstyle{\frac{\pi|k|}{(1+\alpha)}} \\
  &\sim\;\sum_{k\in\mathbb{Z}\setminus\{0\}}|k|^{-\frac{2}{1+\alpha}}|c_{1,k}^{\pm}|^2+\!\!\sum_{k\in\mathbb{Z}\setminus\{0\}}|k|^{\frac{3+\alpha}{1+\alpha}}|g_{0,k}^{\pm}|^2\;<\;+\infty\,,
 \end{split}
\]
having used \eqref{eq:Relationg1c1} for the first step, \eqref{eq:Relationg0c0} for the second step, and the a priori bounds \eqref{eq:pileupcond3} as well as the already proved second condition in \eqref{eq:g0g1Sobolev_typeIIa} for the last step.

(v) Theorem \ref{thm:bifibre-extensionsc0c1} states for this case
\[
	\begin{pmatrix}
		c_{1,k}^- \\
		c_{1,k}^+
	\end{pmatrix} \; = \; \begin{pmatrix}
		\tau_{1,k} & \tau_{2,k} + \ii \tau_{3,k} \\
		\tau_{2,k} - \ii \tau_{3,k} & \tau_{4,k}
	\end{pmatrix}
	\begin{pmatrix}
		c_{0,k}^- \\
		c_{0,k}^+
	\end{pmatrix}
\]
with
 \begin{equation*}
  \begin{split}
   \gamma_1\;&=\; \textstyle \frac{|k|}{1+\alpha} \Big( \frac{\,2 \|\Phi_{\alpha,k}\|_{L^2}^2 \,}{\pi (1+\alpha)} \, \tau_{1,k} -1 \Big)\,, \\
   \gamma_2+\ii\gamma_3\;&=\; (\tau_{2,k}+\ii\tau_{3,k})\textstyle\frac{2|k|}{\pi(1+\alpha)^2}\|\Phi_{\alpha,k}\|_{L^2}^2\,, \\
   \gamma_4\;&=\; \textstyle \frac{|k|}{1+\alpha} \Big( \frac{\,2 \|\Phi_{\alpha,k}\|_{L^2}^2 \,}{\pi (1+\alpha)} \, \tau_{4,k} -1 \Big)  \end{split}
 \end{equation*}
(see \eqref{eq:IIkTheorem51} above). Thus,
\[
\tau_{1,k}\;\sim\;|k|^{\frac{2}{1+\alpha}}\,,\qquad   \tau_{2,k}\pm\ii\tau_{3,k}\;\sim\;|k|^{\frac{1-\alpha}{1+\alpha}} \,,\qquad \tau_{4,k}\;\sim\;|k|^{\frac{2}{1+\alpha}}\,,
\]
and 
\[
	\begin{pmatrix}
		c_{1,k}^- \\
		c_{1,k}^+
	\end{pmatrix} \; \sim \; \begin{pmatrix}
		|k|^{\frac{2}{1+\alpha}} & |k|^{\frac{1-\alpha}{1+\alpha}} \\
		|k|^{\frac{1-\alpha}{1+\alpha}} & |k|^{\frac{2}{1+\alpha}}
	\end{pmatrix}
	\begin{pmatrix}
		c_{0,k}^- \\
		c_{0,k}^+
	\end{pmatrix}\;\sim\;
	 \begin{pmatrix}
		|k|^{\frac{5+\alpha}{2(1+\alpha)}} & |k|^{\frac{3-\alpha}{2(1+\alpha)}} \\
		|k|^{\frac{3-\alpha}{2(1+\alpha)}} & |k|^{\frac{5+\alpha}{2(1+\alpha)}}
	\end{pmatrix}
	\begin{pmatrix}
		g_{0,k}^- \\
		g_{0,k}^+
	\end{pmatrix}
\]
at the leading order in $k$, having used \eqref{eq:Relationg0c0} in the last asymptotics. As the above matrix has determinant of leading order $|k|^{\frac{5+\alpha}{1+\alpha}}$, a standard inversion formula yields
\[
 \begin{pmatrix}
		g_{0,k}^- \\
		g_{0,k}^+
	\end{pmatrix}\;\sim\;|k|^{-\frac{5+\alpha}{2(1+\alpha)}}
	 \begin{pmatrix}
		1 & -|k|^{-1} \\
		-|k|^{-1} & 1
	\end{pmatrix}\begin{pmatrix}
		c_{1,k}^- \\
		c_{1,k}^+
	\end{pmatrix},
\]
whence
\[
 |g_{0,k}^-|^2+|g_{0,k}^+|^2\;\lesssim\;|k|^{-\frac{5+\alpha}{1+\alpha}}\big(|c_{0,k}^-|^2+|c_{0,k}^+|^2\big)
\]
at the leading order in $k$. Therefore,
\[
 \begin{split}
   \sum_{k\in\mathbb{Z}\setminus\{0\}}|k|^{\frac{3+\alpha}{1+\alpha}}\, |g_{0,k}^\pm|^2\;\lesssim \sum_{k\in\mathbb{Z}\setminus\{0\}}|k|^{-\frac{2}{1+\alpha}}\, \big(|c_{0,k}^-|^2+|c_{0,k}^+|^2\big)\;<\;+\infty\,,
 \end{split}
\]
having used the a priori bound  \eqref{eq:pileupcond3} for the last step. This establishes the first condition in \eqref{eq:g0g1Sobolev_typeIII}. The second condition follows at once from the first by means of the self-adjointness constraints 
\[
 \begin{split}
  g_{1,k}^-\;&=\;\gamma_1 g_{0,k}^-+(\gamma_2+\ii\gamma_3)g_{0,k}^+ \\
  g_{1,k}^+\;&=\;(\gamma_2-\ii\gamma_3)g_{0,k}^-+\gamma_4 g_{0,k}^+
 \end{split}
\]
from Theorem \ref{thm:bifibre-extensions}.
\end{proof}


\begin{remark}[Enhanced summability]\label{rem:enhanced_summability}
 Let $(g_k)_{k\in\mathbb{Z}}\in\mathcal{D}(\mathscr{H}_\alpha^*)$. As established in Lemma \ref{lem:gkkrepr}, the coefficients $c_{0,k}$ given by the representation \eqref{eq:repreDHstar}-\eqref{eq:gkrepresentation} of $g_k$ satisfy
 \[
  \sum_{k\in\mathbb{Z}\setminus\{0\}}|k|^{-\frac{2}{1+\alpha}}|c_{0,k}^\pm|^2\,\;<\;+\infty\,.
 \]
 If \emph{in addition}  $(g_k)_{k\in\mathbb{Z}}\in\mathcal{D}(\mathscr{H}_\alpha^{\mathrm{u.f.}})$ for some uniformly-fibred extension of $\mathscr{H}_\alpha$, then Prop.~\ref{prop:g0g1Sobolev} above shows that the coefficients $g_{0,k}$ given by the representation \eqref{eq:gwithPweight_k} of $g_k$ satisfy
 \[
  \sum_{k\in\mathbb{Z}\setminus\{0\}}|k|^{\frac{3+\alpha}{1+\alpha}}|g_{0,k}^\pm|^2\,\;<\;+\infty
 \]
 (this covers also the case when the $g_{0,k}^+$'s or the $g_{0,k}^-$'s are all zero, depending on the considered type of extension). The latter condition, owing to \eqref{eq:Relationg0c0} and hence $g_{0,k}^\pm\sim |k|^{-\frac{1}{2}}c_{0,k}^\pm$, implies
  \[
  \sum_{k\in\mathbb{Z}\setminus\{0\}}|k|^{\frac{2}{1+\alpha}}|c_{0,k}^\pm|^2\,\;<\;+\infty\,.
 \]
 Thus, the condition of belonging to $\mathcal{D}(\mathscr{H}_\alpha^{\mathrm{u.f.}})$, instead of generically to $\mathcal{D}(\mathscr{H}_\alpha^*)$, enhances the summability of the sequence $(c_{0,k}^\pm)_{k\in\mathbb{Z}}$. 
%
\end{remark}

\subsection{Decomposition of the adjoint into singular terms}\label{sec:singular_decomposition_adjoint}~

As alluded to at the end of Subsect.~\ref{sec:genstrategy}, let us show that the decomposition induced by \eqref{eq:gwithPweight_k} of a generic element in the domain of a uniformly fibred extension $\mathscr{H}_\alpha^{\mathrm{u.f.}}$, namely
\begin{equation}\label{eq:singular_decomposition_adjoint}
 (g_k)_{k\in\mathbb{Z}}\;=\;(\varphi_k)_{k\in\mathbb{Z}}+\big(g_{1,k}|x|^{1+\frac{\alpha}{2}}P\big)_{k\in\mathbb{Z}}+\big(g_{0,k}|x|^{-\frac{\alpha}{2}}P\big)_{k\in\mathbb{Z}}\,,
\end{equation}
unavoidably displays an annoying form of singularity, which affects our subsequent analysis, in the following sense.

\begin{lemma}\label{lem:singular_decomposition_adjoint}
Let $\alpha\in[0,1)$ and let $\mathscr{H}_\alpha^{\mathrm{u.f.}}$ be a uniformly fibred self-adjoint extension. There exists $ (g_k)_{k\in\mathbb{Z}}\in\mathcal{D}(\mathscr{H}_\alpha^{\mathrm{u.f.}})$ such that, with respect to the decomposition \eqref{eq:singular_decomposition_adjoint},
\[
 \begin{split}
   \big(g_{1,k}|x|^{1+\frac{\alpha}{2}}P\big)_{k\in\mathbb{Z}}\;&\notin\;\mathcal{D}(\mathscr{H}_\alpha^*)\,, \\
   \big(g_{0,k}|x|^{-\frac{\alpha}{2}}P\big)_{k\in\mathbb{Z}}\;&\notin\;\mathcal{D}(\mathscr{H}_\alpha^*)\,,
 \end{split}
\]
with the obvious exception of those objects above that are prescribed to be identically zero for all elements of the domain of the considered uniformly fibred extension.
%
%
\end{lemma}


 Clearly, the fact that
 \begin{equation}\label{eq:varhisquareintegrable}
  (\varphi_k)_{k\in\mathbb{Z}}\in\ell^2(\mathbb{Z},L^2(\mathbb{R},\ud x))
 \end{equation}
 follows at once by difference from \eqref{eq:singular_decomposition_adjoint}, because owing to Corollary \ref{cor:(g0)k_(g1)k_in_Hs} both $(g_{1,k}|x|^{1+\frac{\alpha}{2}}P)_{k\in\mathbb{Z}}$ and $(g_{0,k}|x|^{-\frac{\alpha}{2}}P)_{k\in\mathbb{Z}}$ belong to $\ell^2(\mathbb{Z},L^2(\mathbb{R},\ud x))$. However, whereas in \eqref{eq:gwithPweight_k}/\eqref{eq:singular_decomposition_adjoint} each $\varphi_k$ belongs to $\mathcal{D}(\overline{A_{\alpha}(k)})$, their collection $(\varphi_k)_{k\in\mathbb{Z}}$ may fail to belong to $\mathcal{D}(\overline{\mathscr{H}_\alpha})$ because it may even fail to belong to $\mathcal{D}(\mathscr{H}_\alpha^*)$!

In preparation for the proof of Lemma \ref{lem:singular_decomposition_adjoint}, a simple computation shows that
\[
 \begin{split}
   A_\alpha^\pm(k)^*\big(|x|^{-\frac{\alpha}{2}}P\big)\;&=\; \alpha|x|^{-(1+\frac{\alpha}{2})}P'-|x|^{-\frac{\alpha}{2}}P''+k^2|x|^{\frac{3\alpha}{2}} P\,, \\
  A_\alpha^\pm(k)^*\big(|x|^{1+\frac{\alpha}{2}}P\big)\;&=\;-(2+\alpha)|x|^{\frac{\alpha}{2}}P'-|x|^{1+\frac{\alpha}{2}}P''+k^2|x|^{1+\frac{5\alpha}{2}} P
 \end{split}
\]
for any $k\in\mathbb{Z}$ and $x\gtrless 0$ depending on the `+' or the `$-$' case. In particular, as the cut-off function $P$ is constantly equal to one when $|x|<1$, 
\begin{equation}\label{eq:action_on_short_x}
 \begin{split}
 \mathbf{1}_{I^{\pm}}(x)A_\alpha^\pm(k)^*\big(|x|^{-\frac{\alpha}{2}}P\big)\;&=\;\mathbf{1}_{I^{\pm}}(x)k^2|x|^{\frac{3\alpha}{2}}\,, \\
 \mathbf{1}_{I^{\pm}}(x) A_\alpha^\pm(k)^*\big(|x|^{1+\frac{\alpha}{2}}P\big)\;&=\;\mathbf{1}_{I^{\pm}}(x) k^2|x|^{1+\frac{5\alpha}{2}}\,,
 \end{split}
\end{equation}
where $I^-:=(-1,0)$ and $I^+:=(0,1)$. We can see that this implies
\begin{eqnarray}
 \big\|(\mathscr{H}_\alpha^\pm)^*\big(g_{0,k}^{\pm}|x|^{-\frac{\alpha}{2}}P)_{k\in\mathbb{Z}}\big\|^2_{\cH^\pm}\!&\geqslant&\!\sum_{k\in\mathbb{Z}}k^4|g_{0,k}^\pm|^2\,, \label{eq:Hastar-shortdist-g0}\\
 \big\|(\mathscr{H}_\alpha^\pm)^*\big(g_{1,k}^{\pm}|x|^{1+\frac{\alpha}{2}}P)_{k\in\mathbb{Z}}\big\|^2_{\cH^\pm}\!&\geqslant&\!\sum_{k\in\mathbb{Z}}k^4|g_{1,k}^\pm|^2\,.  \label{eq:Hastar-shortdist-g1}
\end{eqnarray}
Indeed,
\[
 \begin{split}
  \big\|(\mathscr{H}_\alpha^+)^*\big(g_{1,k}^{+}x^{1+\frac{\alpha}{2}}P)_{k\in\mathbb{Z}}\big\|_{\cH^+}^2\;&=\;\sum_{k\in\mathbb{Z}}\big\|A_\alpha^+(k)^*\big(g_{1,k}^+x^{1+\frac{\alpha}{2}}P\big)\big\|^2_{L^2(\mathbb{R}^+,\ud x)} \\
  &\geqslant\;\sum_{k\in\mathbb{Z}}\big\|g_{1,k}^+ k^2 x^{1+\frac{5\alpha}{2}}\big\|^2_{L^2((0,1),\ud x)} \\
  &=\;(3+5\alpha)^{-1}\sum_{k\in\mathbb{Z}}k^4|g_{1,k}^+|^2\,,
 \end{split}
\]
where we used \eqref{eq:Halphaadj-decomposable} in the first step and \eqref{eq:action_on_short_x} in the second; all other cases for  \eqref{eq:Hastar-shortdist-g0}-\eqref{eq:Hastar-shortdist-g1} are obtained in a completely analogous way.

%
%
%
%
 
\begin{proof}[Proof of Lemma \ref{lem:singular_decomposition_adjoint}]
Let us discuss case by case all possible types of uniformly fibred extensions. For arbitrary $\varepsilon>0$ let
\[
 \begin{split}
  a_k(\varepsilon)\;&:=\;
  \begin{cases}
   \;|k|^{\frac{1}{1+\alpha}-\frac{1}{2}(1+\varepsilon)}\,, & \;\textrm{ if }k\in\mathbb{Z}\setminus\{0\}\,, \\
   \qquad 0\,, & \;\textrm{ if }k=0\,,
  \end{cases} \\
  b_k(\varepsilon)\;&:=\;
  \begin{cases}
   \;|k|^{-\frac{1}{1+\alpha}-\frac{1}{2}(1+\varepsilon)}\,, & \textrm{if }k\in\mathbb{Z}\setminus\{0\} \,,\\
   \qquad 0\,, & \textrm{if }k=0\,.
  \end{cases}
 \end{split}
\]

(i) Friedrichs extension $\mathscr{H}_{\alpha,F}=\bigoplus_{k\in\mathbb{Z}}A_{\alpha,F}(k)$.
For this case we choose $(g_k)_{k\in\mathbb{Z}}$ with
\[
 g_k\;:=\;\begin{pmatrix}
  a_k(\varepsilon)\,\widetilde{\Psi}_{\alpha,k} \\
  a_k(\varepsilon)\,\widetilde{\Psi}_{\alpha,k}
 \end{pmatrix}.
\]
With respect to the representation \eqref{eq:gkrepresentation}, $c_{0,k}^\pm=0$ and $c_{1,k}^\pm=a_k(\varepsilon)$. Therefore,
\[
 \sum_{k\in\mathbb{Z}}|k|^{-\frac{2}{1+\alpha}}|c_{1,k}^\pm|^2\;=\;\sum_{k\in\mathbb{Z}\setminus\{0\}}|k|^{-1-\varepsilon}\;<\;+\infty
\]
and, owing to Lemma \ref{lem:gkkrepr}, $(g_k)_{k\in\mathbb{Z}}\in\mathcal{D}(\mathscr{H}_\alpha^*)$.
Moreover, by construction $g_k$ satisfies the conditions of self-adjointness characterising $\mathcal{D}(A_{\alpha,F}(k))$ stated in Theorem \ref{thm:bifibre-extensionsc0c1}; thus, $(g_k)_{k\in\mathbb{Z}}\in\mathcal{D}(\mathscr{H}_{\alpha,F})$. Expressing now $(g_k)_{k\in\mathbb{Z}}$ in the representation \eqref{eq:singular_decomposition_adjoint}, formulas \eqref{eq:Relationg0c0}-\eqref{eq:Relationg1c1} yield
\[
 g_{0,k}^\pm\;=\;0\,,\qquad g_{1,k}^{\pm}\;\sim\;|k|^{-\frac{1}{2}(\frac{2}{1+\alpha}+\varepsilon)}\quad(k\neq 0)\,,
\]
whence
\[
 \sum_{k\in\mathbb{Z}\setminus\{0\}}k^4|g_{1,k}^\pm|^2\;\sim\sum_{k\in\mathbb{Z}\setminus\{0\}}|k|^{\frac{2+4\alpha}{1+\alpha}-\varepsilon}\;=\;+\infty\quad\Leftrightarrow\quad\varepsilon\in(0,{\textstyle\frac{3+5\alpha}{1+\alpha}}]\,.
\]
Thus, for $\varepsilon\in(0,{\textstyle\frac{3+5\alpha}{1+\alpha}}]$, we deduce from \eqref{eq:Hastar-shortdist-g1} that $(g_{1,k}|x|^{1+\frac{\alpha}{2}}P)_{k\in\mathbb{Z}}\notin\mathcal{D}(\mathscr{H}_\alpha^*)$.

(ii) Extensions of type $\mathrm{I}_{\mathrm{R}}$: for $\gamma\in\mathbb{R}$ let us consider $\mathscr{H}_{\alpha,R}^{[\gamma]}=\bigoplus_{k\in\mathbb{Z}}A_{\alpha,R}^{[\gamma]}(k)$. For this case we choose $(g_k)_{k\in\mathbb{Z}}$ with
\[
 g_k\;:=\;\begin{pmatrix}
		a_k(\varepsilon) \widetilde{\Psi}_{\alpha,k}\\
		\beta_k b_k(\varepsilon) \widetilde{\Psi}_{\alpha,k}+b_k(\varepsilon) \widetilde{\Phi}_{\alpha,k}
	\end{pmatrix}
\]
and $\beta_k$ given by
\[
 \gamma \; = \; \textstyle\frac{|k|}{1+\alpha}\Big(\frac{\, 2 \|\Phi_{\alpha,k}\|_{L^2(\mathbb{R}^+)}^2 \,}{\pi(1+\alpha)}\, \beta_k-1\Big) \, .
\]
From \eqref{eq:Phinorm}, $\|\Phi_{\alpha,k}\|_{L^2(\mathbb{R}^+)}^2\sim|k|^{-\frac{2}{1+\alpha}}$ (for some multiplicative $\alpha$-dependent constant), whence $\beta_k\sim |k|^{\frac{2}{1+\alpha}}$ at the leading order in $k\in\mathbb{Z}\setminus\{0\}$. With respect to the representation \eqref{eq:gkrepresentation},
\[
 \begin{array}{rclcrcl}
  c_{0,k}^- \!\! &=& \!\! 0\,, & \quad & c_{1,k}^- \!\! &=& \!\! a_k(\varepsilon)\;=\;|k|^{\frac{1}{1+\alpha}-\frac{1}{2}(1+\varepsilon)}\,, \\
  c_{0,k}^+ \!\! &=& \!\! b_k(\varepsilon)\;=\;|k|^{-\frac{1}{1+\alpha}-\frac{1}{2}(1+\varepsilon)}\,, & \; & c_{1,k}^+ \!\! &=& \!\! \beta_k b_k(\varepsilon)\;\sim\;|k|^{\frac{1}{1+\alpha}-\frac{1}{2}(1+\varepsilon)}\,,
 \end{array}
\]
at the leading order in $k\in\mathbb{Z}\setminus\{0\}$, whereas all the above coefficients vanish for $k=0$. Therefore,
\[
 \begin{split}
  \sum_{k\in\mathbb{Z}}|k|^{-\frac{2}{1+\alpha}}|c_{0,k}^+|^2\;&=\sum_{k\in\mathbb{Z}\setminus\{0\}}|k|^{-\frac{4}{1+\alpha}-1-\varepsilon}\;<\;+\infty\,, \\
  \sum_{k\in\mathbb{Z}}|k|^{-\frac{2}{1+\alpha}}|c_{1,k}^\pm|^2\;&= \sum_{k\in\mathbb{Z}\setminus\{0\}}|k|^{-1-\varepsilon}\;<\;+\infty\,,
 \end{split}
\]
which implies, owing to Lemma \ref{lem:gkkrepr}, that $(g_k)_{k\in\mathbb{Z}}\in\mathcal{D}(\mathscr{H}_\alpha^*)$.
Moreover, by construction $g_k$ satisfies the conditions of self-adjointness characterising $\mathcal{D}(A_{\alpha,R}^{[\gamma]}(k))$ stated in Theorem \ref{thm:bifibre-extensionsc0c1}; thus, $(g_k)_{k\in\mathbb{Z}}\in\mathcal{D}(\mathscr{H}_{\alpha,R}^{[\gamma]})$. Expressing now $(g_k)_{k\in\mathbb{Z}}$ in the representation \eqref{eq:singular_decomposition_adjoint}, formulas \eqref{eq:Relationg0c0}-\eqref{eq:Relationg1c1} yield
\[
 \begin{array}{rclcrcl}
  g_{0,k}^- \!\! &=& \!\! 0\,, & \quad & g_{1,k}^- \!\! &\sim& \!\! |k|^{-\frac{1}{2}(\frac{2}{1+\alpha}+\varepsilon)} \,, \\
  g_{0,k}^+ \!\! &\sim& \!\!|k|^{-\frac{1}{2}(\frac{4+2\alpha}{1+\alpha}+\varepsilon)}\,, & \; & g_{1,k}^+ \!\! &\sim& \!\!|k|^{-\frac{1}{2}(\frac{4+2\alpha}{1+\alpha}+\varepsilon)}\,,
 \end{array}
\]
for $k\in\mathbb{Z}\setminus\{0\}$, up to multiplicative pre-factors depending on $\alpha$ and $\gamma$ only, all the above coefficients vanishing for $k=0$. From this one obtains
\[
 \begin{split}
  \sum_{k\in\mathbb{Z}}k^4|g_{0,k}^+|^2\;&\sim\sum_{k\in\mathbb{Z}\setminus\{0\}}|k|^{\frac{2\alpha}{1+\alpha}-\varepsilon}\;=\;+\infty\quad\Leftrightarrow\quad\varepsilon\in(0,{\textstyle\frac{1+3\alpha}{1+\alpha}}]\,, \\
 \sum_{k\in\mathbb{Z}}k^4|g_{1,k}^+|^2\;&\sim\sum_{k\in\mathbb{Z}\setminus\{0\}}|k|^{\frac{2\alpha}{1+\alpha}-\varepsilon}\;=\;+\infty\quad\Leftrightarrow\quad\varepsilon\in(0,{\textstyle\frac{1+3\alpha}{1+\alpha}}]\,, \\
  \sum_{k\in\mathbb{Z}}k^4|g_{1,k}^-|^2\;&\sim\sum_{k\in\mathbb{Z}\setminus\{0\}}|k|^{\frac{2+4\alpha}{1+\alpha}-\varepsilon}\;=\;+\infty\quad\Leftrightarrow\quad\varepsilon\in(0,{\textstyle\frac{3+5\alpha}{1+\alpha}}]\,.
 \end{split}
\]
Thus, for $\varepsilon\in(0,{\textstyle\frac{1+3\alpha}{1+\alpha}}]$, we deduce from \eqref{eq:Hastar-shortdist-g0}-\eqref{eq:Hastar-shortdist-g1} that $(g_{0,k}|x|^{-\frac{\alpha}{2}}P)_{k\in\mathbb{Z}}\notin\mathcal{D}(\mathscr{H}_\alpha^*)$ and $(g_{1,k}|x|^{1+\frac{\alpha}{2}}P)_{k\in\mathbb{Z}}\notin\mathcal{D}(\mathscr{H}_\alpha^*)$.

(iii)  Extensions of type $\mathrm{I}_{\mathrm{L}}$: for $\gamma\in\mathbb{R}$ let us consider $\mathscr{H}_{\alpha,L}^{[\gamma]}=\bigoplus_{k\in\mathbb{Z}}A_{\alpha,L}^{[\gamma]}(k)$. For this case we choose $(g_k)_{k\in\mathbb{Z}}$ with
\[
 g_k\;:=\;\begin{pmatrix}
		\beta_k b_k(\varepsilon) \widetilde{\Psi}_{\alpha,k}+b_k(\varepsilon) \widetilde{\Phi}_{\alpha,k} \\
		a_k(\varepsilon) \widetilde{\Psi}_{\alpha,k}
	\end{pmatrix},
\]
with the same $\beta_k$ as in case (ii). With the obvious inversion between `-' and `+' components, the reasoning is the same as in case (ii).

(iv) Extensions of type $\mathrm{II}_a$ for given $a\in\mathbb{C}\setminus\{0\}$: for $\gamma\in\mathbb{R}$ let us consider $ \mathscr{H}_{\alpha,a}^{[\gamma]}=\bigoplus_{k\in\mathbb{Z}}A_{\alpha,a}^{[\gamma]}(k)$.
For this case we choose $(g_k)_{k\in\mathbb{Z}}$ with
\[
 g_k\;:=\;\begin{pmatrix}
		\big(\tau_k b_k(\varepsilon) + a_k(\varepsilon)\big) \widetilde{\Psi}_{\alpha,k}+b_{k}(\varepsilon)\widetilde{\Phi}_{\alpha,k} \\
		\big(\tau_k a b_k(\varepsilon)- \overline{a}^{-1} a_k(\varepsilon)\big)\widetilde{\Psi}_{\alpha,k}+a b_{k}(\varepsilon)\widetilde{\Phi}_{\alpha,k}
	\end{pmatrix}
\]
and $\tau_k$ given by 
\[
			  \gamma\;:=\;\textstyle{(1+|a|^2) \frac{|k|}{1+\alpha}} \Big( \frac{\,2 \|\Phi_{\alpha,k}\|_{L^2(\mathbb{R}^+)}^2 \,}{\pi (1+\alpha)} \, \tau_k -1 \Big)\,.
\]
In particular, $\tau_k\sim |k|^{\frac{2}{1+\alpha}}$ at the leading order in $k\in\mathbb{Z}\setminus\{0\}$. With respect to the representation \eqref{eq:gkrepresentation},
\[
  c_{0,k}^\pm\;\sim\;|k|^{-\frac{1}{1+\alpha}-\frac{1}{2}(1+\varepsilon)} \,,\qquad c_{1,k}^\pm\;\sim\;|k|^{\frac{1}{1+\alpha}-\frac{1}{2}(1+\varepsilon)}
\]
at the leading order in $k\in\mathbb{Z}\setminus\{0\}$, whereas all the above coefficients vanish for $k=0$. Therefore,
\[
 \begin{split}
  \sum_{k\in\mathbb{Z}}|k|^{-\frac{2}{1+\alpha}}|c_{0,k}^\pm|^2\;&=\sum_{k\in\mathbb{Z}\setminus\{0\}}|k|^{-\frac{4}{1+\alpha}-1-\varepsilon}\;<\;+\infty\,, \\
  \sum_{k\in\mathbb{Z}}|k|^{-\frac{2}{1+\alpha}}|c_{1,k}^\pm|^2\;&= \sum_{k\in\mathbb{Z}\setminus\{0\}}|k|^{-1-\varepsilon}\;<\;+\infty\,, 
 \end{split}
\]
which implies, owing to Lemma \ref{lem:gkkrepr}, that $(g_k)_{k\in\mathbb{Z}}\in\mathcal{D}(\mathscr{H}_\alpha^*)$.
Moreover, by construction $g_k$ satisfies the conditions of self-adjointness characterising $\mathcal{D}(A_{\alpha,a}^{[\gamma]}(k))$ stated in Theorem \ref{thm:bifibre-extensionsc0c1}; thus, $(g_k)_{k\in\mathbb{Z}}\in\mathcal{D}(\mathscr{H}_{\alpha,a}^{[\gamma]})$. Expressing now $(g_k)_{k\in\mathbb{Z}}$ in the representation \eqref{eq:singular_decomposition_adjoint}, formulas \eqref{eq:Relationg0c0}-\eqref{eq:Relationg1c1} yield
\[
 g_{0,k}^\pm\;\sim\;|k|^{-\frac{1}{2}(\frac{4+2\alpha}{1+\alpha}-\varepsilon)}\,,\qquad g_{1,k}^\pm\;\sim\;|k|^{-\frac{1}{2}(\frac{2}{1+\alpha}+\varepsilon)}
\]
at the leading order in $k\in\mathbb{Z}\setminus\{0\}$, all the above coefficients vanishing for $k=0$. From this one obtains
\[
 \begin{split}
  \sum_{k\in\mathbb{Z}}k^4|g_{0,k}^\pm|^2\;&\sim\sum_{k\in\mathbb{Z}\setminus\{0\}}|k|^{\frac{2\alpha}{1+\alpha}-\varepsilon}\;=\;+\infty\quad\Leftrightarrow\quad\varepsilon\in(0,{\textstyle\frac{1+3\alpha}{1+\alpha}}]\,, \\
  \sum_{k\in\mathbb{Z}}k^4|g_{1,k}^\pm|^2\;&\sim\sum_{k\in\mathbb{Z}\setminus\{0\}}|k|^{\frac{2+4\alpha}{1+\alpha}-\varepsilon}\;=\;+\infty\quad\Leftrightarrow\quad\varepsilon\in(0,{\textstyle\frac{3+5\alpha}{1+\alpha}}]\,.
 \end{split}
\]
Thus, for $\varepsilon\in(0,{\textstyle\frac{1+3\alpha}{1+\alpha}}]$, we deduce from \eqref{eq:Hastar-shortdist-g0}-\eqref{eq:Hastar-shortdist-g1} that $(g_{0,k}|x|^{-\frac{\alpha}{2}}P)_{k\in\mathbb{Z}}\notin\mathcal{D}(\mathscr{H}_\alpha^*)$ and $(g_{1,k}|x|^{1+\frac{\alpha}{2}}P)_{k\in\mathbb{Z}}\notin\mathcal{D}(\mathscr{H}_\alpha^*)$.

(v) Extensions of type $\mathrm{III}$: for $\Gamma\in\mathbb{R}^4$ let us consider $  \mathscr{H}_{\alpha}^{[\Gamma]}=\bigoplus_{k\in\mathbb{Z}}A_{\alpha}^{[\Gamma]}(k)$.
For this case we choose $(g_k)_{k\in\mathbb{Z}}$ with
\[
 g_k\;:=\;\begin{pmatrix}
		\big(\tau_{1,k}+\tau_{2,k}+ \ii \tau_{3,k}\big)b_{k}(\varepsilon) \widetilde{\Psi}_{\alpha,k}+b_{k}(\varepsilon)\widetilde{\Phi}_{\alpha,k} \\
		\big(\tau_{2,k}- \ii \tau_{3,k}+\tau_{4,k}\big)b_{k}(\varepsilon)\widetilde{\Psi}_{\alpha,k}+b_{k}(\varepsilon)\widetilde{\Phi}_{\alpha,k}
	\end{pmatrix}
\]
and $(\tau_{1,k},\tau_{2,k},\tau_{3,k},\tau_{4,k})$ given by
\[
  \begin{split}
   \gamma_1\;&=\; \textstyle \frac{|k|}{1+\alpha} \Big( \frac{\,2 \|\Phi_{\alpha,k}\|_{L^2(\mathbb{R}^+)}^2 \,}{\pi (1+\alpha)} \, \tau_{1,k} -1 \Big) \,,\\
   \gamma_2+\ii\gamma_3\;&=\; (\tau_{2,k}+\ii\tau_{3,k})\textstyle\frac{2|k|}{\pi(1+\alpha)^2}\|\Phi_{\alpha,k}\|_{L^2(\mathbb{R}^+)}^2 \,,\\
   \gamma_4\;&=\; \textstyle \frac{|k|}{1+\alpha} \Big( \frac{\,2 \|\Phi_{\alpha,k}\|_{L^2(\mathbb{R}^+)}^2 \,}{\pi (1+\alpha)} \, \tau_{4,k} -1 \Big).  \end{split}
\]
In particular, 
\[
\tau_{1,k}\;\sim\;|k|^{\frac{2}{1+\alpha}}\,,\qquad   \tau_{2,k}\pm\ii\tau_{3,k}\;\sim\;|k|^{\frac{1-\alpha}{1+\alpha}} \,,\qquad \tau_{4,k}\;\sim\;|k|^{\frac{2}{1+\alpha}}\,,
\]
at the leading order in $k\in\mathbb{Z}\setminus\{0\}$. With respect to the representation \eqref{eq:gkrepresentation},
\[
  c_{0,k}^\pm\;\sim\;|k|^{-\frac{1}{1+\alpha}-\frac{1}{2}(1+\varepsilon)} \,,\qquad c_{1,k}^\pm\;\sim\;|k|^{\frac{1}{1+\alpha}-\frac{1}{2}(1+\varepsilon)}
\]
at the leading order in $k\in\mathbb{Z}\setminus\{0\}$, whereas all the above coefficients vanish for $k=0$. 
From this point one repeats verbatim the reasoning of part (iv).
%
\end{proof}

\subsection{Detecting short-scale asymptotics and regularity}\label{subsec:auxiliaryLemma}~

 As observed with \eqref{eq:varhisquareintegrable}, $\mathcal{F}_2^{-1}$ is applicable to $(\varphi_k)_{k\in\mathbb{Z}}$ and thus \eqref{eq:F2-1phi} defines a function $\varphi\in L^2(\mathbb{R}\times\mathbb{S}^1,\ud x\ud y)$. The next step in the strategy outlined in Subsect.~\ref{sec:genstrategy} is to show convenient short-scale asymptotics as $x\to 0$ for $\varphi(x,y)$ and $\partial_x\varphi(x,y)$.

 Evidently, the possibility that $\varphi\notin\mathcal{F}_2^{-1}\mathcal{D}(\mathscr{H}_\alpha^*)=\mathcal{D}(\mathsf{H}_\alpha^*)$ (Lemma \ref{lem:singular_decomposition_adjoint}) complicates this analysis: no regularity or short-scale asymptotics of the elements of $\mathcal{D}(\mathsf{H}_\alpha^*)$ can be claimed a priori for $\varphi$.

 For the above purposes we shall make use of the following auxiliary result.

 \begin{lemma}\label{lem:grand_auxiliary_lemma}
  Let $\alpha\in[0,1)$ and let $R:(0,1)\times\mathbb{S}^1\to\mathbb{C}$ be a function such that
  \begin{itemize}
   \item[(a)] $\big\| x^{-(\frac{3}{2}+\frac{\alpha}{2})}R\,\big\|_{L^2((0,1)\times\mathbb{S}^1,\ud x\ud y)}\;<\;+\infty$\,,
   \item[(b)] $\big\| \partial_x^2R\,\big\|_{L^2((0,1)\times\mathbb{S}^1,\ud x\ud y)}\;<\;+\infty$\,.
  \end{itemize}
 Then for almost every $y\in\mathbb{S}^1$ the function $(0,1)\ni x\mapsto R(x,y)$ belongs to $H^2_0((0,1])$ and as such it satisfies the following properties:
 \begin{itemize}
  \item[(i)] $R(\cdot,y)\in C^1(0,1)$,
  \item[(ii)] $R(x,y)\stackrel{x\downarrow 0}{=}o(x^{\frac{3}{2}})$,
  \item[(iii)] $\partial_x R(x,y)\stackrel{x\downarrow 0}{=}o(x^{\frac{1}{2}})$.
 \end{itemize}  
 \end{lemma}

 \begin{remark}
  $H^2_0((0,1])$ in the statement of Lemma \ref{lem:grand_auxiliary_lemma} denotes as usual the closure of $C^\infty_0((0,1])$ in the $H^2$-norm. The edge $x=1$ is included so as to mean that there is no vanishing constraint at $x=1$ for the elements of $H^2_0((0,1])$ and their derivatives: only vanishing as $x\downarrow 0$ emerges, in the form of conditions (ii) and (iii).
 \end{remark}

 \begin{proof}[Proof of Lemma \ref{lem:grand_auxiliary_lemma}]
  Assumption (a) in Lemma \ref{lem:grand_auxiliary_lemma} implies that $R(\cdot,y)\in L^2((0,1))$, and hence together with (b) it implies that $R(\cdot,y)\in H^2((0,1))$ for a.e.~$y\in\mathbb{S}^1$.
  Therefore $R(\cdot,y)=a_y+b_yx+r_y(x)$ for a.e.~$y\in\mathbb{S}^1$, for some $a_y,b_y\in\mathbb{C}$ and $r_y\in H^2_0((0,1])$. For compatibility with assumption (a), necessarily $a_y=b_y=0$, whence $R(\cdot,y)\in H^2_0((0,1])$ for a.e.~$y\in\mathbb{S}^1$.  
 \end{proof}

 Let us discuss the application of Lemma \ref{lem:grand_auxiliary_lemma} to our context.

 As we are interested in characterising for fixed $y\in\mathbb{S}^1$ the behaviour and the regularity of $x\mapsto\varphi(x,y)$ as $x\to 0$ from \emph{each side} of the singular point $x=0$, it suffices to analyse the case $x>0$; then completely analogous conclusions are obtained for $x<0$. Lemma \ref{lem:grand_auxiliary_lemma} is thus meant to be applied to the restriction $R(x,y)=\varphi(x,y)\mathbf{1}_{(0,1)}(x)$.

 In fact, since in general $\varphi\in L^2(\mathbb{R}\times\mathbb{S}^1,\ud x\ud y)\setminus\mathcal{D}(\mathsf{H}_\alpha^*)$, we are not able to check the assumptions (a) and (b) of Lemma \ref{lem:grand_auxiliary_lemma} directly for $\varphi$. We opt instead for splitting $\varphi$ into a component in $\mathcal{D}(\overline{\mathsf{H}_\alpha})$ plus a remainder, the explicit form of which will allow to apply Lemma \ref{lem:grand_auxiliary_lemma}.

 This idea is implicit in the very choice of $(\varphi_k)_{k\in\mathbb{Z}}$ made in \eqref{eq:gwithPweight_k}. Let us recall that for given $(g_k)_{k\in\mathbb{Z}}$ we could represent
 \[
  g_k^{\pm}\;=\;\varphi_k^{\pm}+g_{1,k}^{\pm}|x|^{1+\frac{\alpha}{2}}P+g_{0,k}^{\pm}|x|^{-\frac{\alpha}{2}}P
 \]
 and also
 \[
  g_k^{\pm}\;=\;\widetilde{\varphi}_k^{\pm}+c_{1,k}^{\pm}\widetilde{\Psi}_{\alpha,k}+c_{0,k}^{\pm}\widetilde{\Phi}_{\alpha,k}\,,
 \]
 where
  \begin{equation}\label{eq:phitildeinDclosure}
  (\widetilde{\varphi}_k^{\pm})_{k\in\mathbb{Z}}\;\in\;\mathcal{D}\bigg(\bigoplus_{k\in\mathbb{Z}}\overline{A_\alpha^\pm(k)}\bigg)\;=\;\mathcal{D}\big(\overline{\mathscr{H}_\alpha^\pm}\big)
 \end{equation}
 Moreover, as argued in the proof of Theorem \ref{prop:g_with_Pweight}, for each $k\in\mathbb{Z}\setminus\{0\}$ we can split
 \begin{equation}\label{eq:phi=phitilde+q}
  \varphi_k^{\pm}\;=\;\widetilde{\varphi}_k^{\pm}+\vartheta_k^{\pm}\,,
 \end{equation}
 while keeping
 \begin{equation}\label{eq:varthetazeromode}
  \widetilde{\varphi}_0^{\pm}\;\equiv\; \varphi_0^{\pm}\qquad\textrm{ and }\qquad \vartheta_0^{\pm}\;\equiv\; 0\qquad\textrm{ when }\qquad k=0\,,
 \end{equation}
 where
 \begin{equation}\label{eq:q=q0+q1}
  \vartheta_k^{\pm}\;=\; \vartheta_{0,k}^{\pm}+\vartheta_{1,k}^{\pm} 
 \end{equation}
 with
 \begin{eqnarray}
  & & \vartheta_{0,k}^{\pm}\;:=\;c_{0,k}^\pm\Big(\widetilde{\Phi}_{\alpha,k}-{\textstyle\sqrt{\frac{\pi(1+\alpha)}{2|k|}}|x|^{-\frac{\alpha}{2}}P+\sqrt{\frac{\pi|k|}{2(1+\alpha)}}\,|x|^{1+\frac{\alpha}{2}} P}\Big)\,, \label{eq:defq0q1decomp-q0} \\
  & & \vartheta_{1,k}^{\pm}\;:=\;c_{1,k}^\pm\Big(\widetilde{\Psi}_{\alpha,k}-{\textstyle\sqrt{\frac{2|k|}{\pi(1+\alpha)^3}}}\,\|\Phi_{\alpha,k}\|_{L^2(\mathbb{R}^+)}^2\,|x|^{1+\frac{\alpha}{2}}P\Big) \label{eq:defq0q1decomp-q1}
 \end{eqnarray}
 and 
  \begin{equation}\label{eq:regularityoftheta01}
  \vartheta_{0,k}^{\pm},\vartheta_{1,k}^{\pm}\;\in\;\mathcal{D}\big(\overline{A^\pm_\alpha(k)}\big)\;=\;H^2_0(\mathbb{R}^\pm)\cap L^2(\mathbb{R}^\pm,\langle x\rangle^{4\alpha}\,\ud x)\,.
 \end{equation}
 It is important to remember that for later convenience the zero mode is all cast into $\widetilde{\varphi}_0^{\pm}\equiv \varphi_0^{\pm}$, hence $(\vartheta_k)_{k\in\mathbb{Z}}\equiv(\vartheta_k)_{k\in\mathbb{Z}\setminus\{0\}}$.

 The decomposition \eqref{eq:phi=phitilde+q}-\eqref{eq:defq0q1decomp-q1} induces the splitting
 \begin{equation}\label{eq:fftk}
  (\varphi_k)_{k\in\mathbb{Z}}\;=\;(\widetilde{\varphi}_k)_{k\in\mathbb{Z}}+(\vartheta_k)_{k\in\mathbb{Z}}
 \end{equation}
 as an identity in $\ell^2(\mathbb{Z},L^2(\mathbb{R}^+,\ud x))$, where $(\vartheta_k)_{k\in\mathbb{Z}}$ does not necessarily belong to $\mathcal{D}(\mathscr{H}_\alpha^*)$, as $(\varphi_k)_{k\in\mathbb{Z}}$ does not either (Lemma \ref{lem:singular_decomposition_adjoint}).
 In turn, owing to \eqref{eq:varhisquareintegrable} and \eqref{eq:phitildeinDclosure}, the identity \eqref{eq:fftk} yields the splitting
 \begin{equation}\label{eq:splittingphiphitildetheta}
  \varphi(x,y)\;=\;\widetilde{\varphi}(x,y)+\vartheta(x,y)\,,\qquad(x,y)\in\mathbb{R}\times\mathbb{S}^1\,,
 \end{equation}
 with
 \begin{eqnarray}
    & & \widetilde{\varphi}\;:=\;\mathcal{F}_2^{-1}(\widetilde{\varphi}_k)_{k\in\mathbb{Z}}\;\in\;\mathcal{F}_2^{-1}\mathcal{D}\big(\overline{\mathscr{H}_\alpha^\pm}\big)\;=\;\mathcal{D}(\overline{\mathsf{H}_\alpha}) \\
    & & \vartheta\, \;:=\; \mathcal{F}_2^{-1} (\vartheta_k)_{k\in\mathbb{Z}}\;\in\;L^2(\mathbb{R}\times\mathbb{S}^1,\ud x\ud y)\,. \label{eq:deffunctionvartheta}
 \end{eqnarray}
  Here $\vartheta$ may fail to belong to $\mathcal{D}(\mathsf{H}_\alpha^*)$, precisely as $\varphi$.

 The explicit information that $\widetilde{\varphi}\in\mathcal{D}(\overline{\mathsf{H}_\alpha})$ and the explicit expression for $\vartheta$ will finally allow us to apply Lemma \ref{lem:grand_auxiliary_lemma} separately to each of them. This will be the object of Subsect.~\ref{sec:control-of-tildephi} and \ref{sec:q0q1}.

\subsection{Control of $\widetilde{\varphi}$}\label{sec:control-of-tildephi}~

We are concerned now with the regularity and the behaviour as $x\to 0^\pm$ of the functions belonging to $\mathcal{D}\big(\overline{\mathsf{H}_\alpha^\pm}\big)$.

Clearly, from \eqref{eq:explicit-tildeHalpha},
\begin{equation}\label{eq:DHaclosedclosure}
 \mathcal{D}\big(\overline{\mathsf{H}_\alpha^\pm}\big)\;=\;\overline{C^\infty_c(\mathbb{R}^\pm_x\times\mathbb{S}^1_y)}^{\|\,\|_{\mathsf{H}_\alpha}}\,,
\end{equation}
where $\|h\|_{\mathsf{H}_\alpha}:=\big(\|h\|_{L^2(\mathbb{R}^\pm_x\times\mathbb{S}^1_y)}^2+\|\mathsf{H}_\alpha^\pm h\|_{L^2(\mathbb{R}^\pm_x\times\mathbb{S}^1_y)}^2\big)^{\frac{1}{2}}$.

We also recall, from $\overline{\mathsf{H}_\alpha^\pm}\subset(\mathsf{H}_\alpha^\pm)^*$ and from \eqref{eq:HHalphaadjoint}, that 
\begin{equation}\label{eq:actionofHaclosed}
 \overline{\mathsf{H}_\alpha^\pm}\,\widetilde{\varphi}^\pm\;=\;\Big(-\frac{\partial^2}{\partial x^2}- |x|^{2\alpha}\frac{\partial^2}{\partial y^2}+\frac{\,C_\alpha}{|x|^2}\Big)\widetilde{\varphi}^\pm\qquad\forall\widetilde{\varphi}^\pm\in\mathcal{D}\big(\overline{\mathsf{H}_\alpha^\pm}\big)\,.
\end{equation}

The main result here is the following.

\begin{proposition}\label{prop:Hclosurecontrol}
 Let $\alpha\in[0,1)$. There exists a constant $K_\alpha>0$ such that for any $\widetilde{\varphi}^\pm\in\mathcal{D}(\overline{\mathsf{H}_\alpha^\pm})$ one has
 \begin{equation}\label{eq:Hclosurecontrol}
  \big\|\,|x|^{-2}\widetilde{\varphi}^\pm\,\big\|_{L^2(\mathbb{R}^\pm_x\times\mathbb{S}^1_y)}+\big\|\partial_x^2\widetilde{\varphi}^\pm\big\|_{L^2(\mathbb{R}^\pm_x\times\mathbb{S}^1_y)}\;\leqslant\;K_\alpha\,\big\|\overline{\mathsf{H}_\alpha^\pm}\,\widetilde{\varphi}^\pm\big\|_{L^2(\mathbb{R}^\pm_x\times\mathbb{S}^1_y)}\,.
 \end{equation}
 When $\alpha\uparrow 1$, then $K_\alpha\to +\infty$. As a consequence, $\widetilde{\varphi}^\pm$ satisfies the assumptions of Lemma \ref{lem:grand_auxiliary_lemma} and therefore, for almost every $y\in\mathbb{S}^1$,
 \begin{itemize}
  \item[(i)] the function $x\mapsto\widetilde{\varphi}^\pm(x,y)$ belongs to $C^1(0,1)$,
  \item[(ii)] $\widetilde{\varphi}^\pm(x,y)=o(|x|^{\frac{3}{2}})$ as $x\to 0^\pm$,  
  \item[(iii)] $\partial_x\widetilde{\varphi}^\pm(x,y)=o(|x|^{\frac{1}{2}})$ as $x\to 0^\pm$.
 \end{itemize} 
\end{proposition}

As we only need information on the limit separately from each side of the singularity, it is enough to consider the `+' case: the same conclusions will apply also to the `-' case. Thus, in the remaining part of this Subsection, we shall simply write $\widetilde{\varphi}$ for $\widetilde{\varphi}^+\in\mathcal{D}(\overline{\mathsf{H}_\alpha^+})$.

The proof of Proposition \ref{prop:Hclosurecontrol} relies on two technical estimates. The first is an iterated version of the standard one-dimensional inequality by Hardy
\begin{equation}\label{eq:hardy1}
\Vert r^{-1} h \Vert_{L^2(\mathbb{R}^+,\ud r)} \;\leqslant\; 2 \,\Vert h' \Vert_{L^2(\mathbb{R}^+,\ud r)}\qquad \forall\,h \in C^\infty_0(\mathbb{R}^+)\,.
\end{equation}

\begin{lemma}[Rellich inequality, \cite{Rellich-ICM1954}]\label{lem:doubleHardy}
For any $h\in C^\infty_c(\mathbb{R}^+)$ one has
\begin{equation}\label{eq:hardy2}
 \|r^{-2}h\|_{L^2(\mathbb{R}^+,\ud r)}\;\leqslant\;\frac{4}{3}\,\|h''\|_{L^2(\mathbb{R}^+,\ud r)}\,.
 \end{equation} 
\end{lemma}

\begin{corollary}\label{cor:doubleHardy}
 Let $\widetilde{\varphi}\in C^\infty_c(\mathbb{R}^+_x\times \mathbb{S}^1_y)$. Then
 \begin{equation}
  \|x^{-2}\widetilde{\varphi}\|_{L^2(\mathbb{R}^+_x\times\mathbb{S}^1_y)}\;\leqslant\;\frac{4}{3}\|\partial_x^2 \widetilde{\varphi}\|_{L^2(\mathbb{R}^+_x\times\mathbb{S}^1_y)}\,.
 \end{equation}
\end{corollary}

 For completeness of presentation, we include a direct, fast proof of Lemma \ref{lem:doubleHardy}, which is simpler than the general demonstration \cite[Chapter 6]{Balinsky-Evans-Lewis-BOOK-Hardy} in the $d$-dimensional case.

\begin{proof}[Proof of Lemma \ref{lem:doubleHardy}]
 As $h\in C^\infty_c(\mathbb{R}^+)$, all the considered integrals are finite, because the integrand functions are supported away from zero, and moreover integration by parts produces no boundary terms. One has
 \[
  \begin{split}
   \|r^{-2}h\|_{L^2}^2\;&=\;\int_0^{+\infty}\frac{|h(r)|^2}{r^4}\,\ud r\;=\;-\frac{1}{3}\int_0^{+\infty}\Big(\frac{1}{r^3}\Big)'\,\overline{h(r)}\,h(r)\,\ud r \\
   &=\;\frac{1}{3}\int_0^{+\infty}\frac{1}{r^3}\,\big(\overline{h(r)}\,h(r)\big)'\,\ud r\;=\;\frac{2}{3}\,\mathfrak{Re}\int_0^{+\infty}\frac{\overline{h(r)}\,h'(r)}{r^3}\,\ud r\,,
  \end{split}
 \]
 and in turn, by means of a weighted Cauchy-Schwarz inequality and Hardy's inequality,
 \[
  \begin{split}   
  \Big|\int_0^{+\infty}\frac{\overline{h(r)}\,h'(r)}{r^3}\,\ud r\Big|\;&\leqslant\;{\textstyle\frac{1}{2}}a\|r^{-2}h\|_{L^2}^2+{\textstyle\frac{1}{2}} a^{-1}\|r^{-1}h'\|_{L^2}^2 \\
  &\leqslant\;{\textstyle\frac{1}{2}}a\|r^{-2}h\|_{L^2}^2+2 a^{-1}\|h''\|_{L^2}^2
  \end{split}
 \]
 for some $a>0$.
 Thus,
 \[
  \|r^{-2}h\|_{L^2}^2\;\leqslant\;{\textstyle\frac{1}{3}}a\,\|r^{-2}h\|_{L^2}^2+{\textstyle\frac{4}{3}}a^{-1}\,\|h''\|_{L^2}^2\,,
 \]
 whence
 \[
   \|r^{-2}h\|_{L^2}^2\;\leqslant\;\frac{4}{a(3-a)}\,\|h''\|_{L^2}^2\,.
 \]
 Optimising over $a\in(0,3)$ yields $a=\frac{3}{2}$, which corresponds to $\|r^{-2}h\|_{L^2}^2\leqslant\frac{16}{9}\|h''\|_{L^2}^2$. This is precisely \eqref{eq:hardy2}.
 \end{proof}

The second estimate is meant to control the term $x^{2\alpha}\partial_y^2$ of $\overline{\mathsf{H}_\alpha}$ and reads as follows.

\begin{lemma}\label{lem:boundedness_x2alphad2y}
 Let $\alpha\in[0,1)$. There exists a constant $D_\alpha>0$ such that for any $\widetilde{\varphi}\in\mathcal{D}(\overline{\mathsf{H}_\alpha^+})$ one has
 \begin{equation}\label{eq:boundedness_x2alphad2y}
  \|x^{2\alpha}\partial_y^2\widetilde{\varphi}\|_{L^2(\mathbb{R}^+_x\times\mathbb{S}^1_y)}\;\leqslant\;D_\alpha\big\|\overline{\mathsf{H}_\alpha^+}\,\widetilde{\varphi}\big\|_{L^2(\mathbb{R}^+_x\times\mathbb{S}^1_y)}\,.
 \end{equation}
\end{lemma}

\begin{proof}
 It is enough to prove \eqref{eq:boundedness_x2alphad2y} for any $\widetilde{\varphi}\in C^\infty_c(\mathbb{R}^+_x\times\mathbb{S}^1_y)$; then the general inequality is merely obtained by closure, owing to \eqref{eq:DHaclosedclosure}. To this aim, let $(\widetilde{\varphi}_k)_{k\in\mathbb{Z}}:=\mathcal{F}_2^+\widetilde{\varphi}\in\cH^+\cong\ell^2(\mathbb{Z},L^2(\mathbb{R}^+,\ud x))$. One has
 \begin{equation*}
  \begin{split}
   \big\|x^{2\alpha}\partial_y^2 \widetilde{\varphi}\big\|^2_{L^2(\mathbb{R}^+_x\times\mathbb{S}^1_y)}\;&=\;\sum_{k\in\mathbb{Z}}\|x^{2\alpha}k^2\,\widetilde{\varphi}_k\|_{L^2(\mathbb{R}^+)}^2 \\
   &=\;\sum_{k\in\mathbb{Z}\setminus\{0\}}\|x^{2\alpha}k^2\,R_{G_{\alpha,k}}A_{\alpha,F}(k)\widetilde{\varphi}_k\|_{L^2(\mathbb{R}^+)}^2 \\
   &\leqslant\;\sum_{k\in\mathbb{Z}\setminus\{0\}}\|x^{2\alpha}k^2\,R_{G_{\alpha,k}}\big\|_{\mathrm{op}}^2\big\|\overline{A_{\alpha}^+(k)}\widetilde{\varphi}_k\|_{L^2(\mathbb{R}^+)}^2\,,
  \end{split}
 \end{equation*}
where we used Plancherel's formula in the first identity and Proposition \ref{eq:RGisSFinv} in the second identity. Owing to Lemma \ref{lem:RGbddsa}(ii), $\|x^{2\alpha}k^2\,R_{G_{\alpha,k}}\|_{\mathrm{op}}\leqslant D_\alpha$ uniformly in $k$ for some $D_\alpha>0$. Based on this fact, and on Lemma \ref{lem:Halphaadj-decomposable} (formula \eqref{eq:Halphaclosure-decomposable}), one then has
\[
 \begin{split}
   \big\|x^{2\alpha}\partial_y^2 \widetilde{\varphi}\big\|^2_{L^2(\mathbb{R}^+_x\times\mathbb{S}^1_y)}\;&\leqslant\; D_\alpha^2\sum_{k\in\mathbb{Z}}\big\|\overline{A_{\alpha}^+(k)}\widetilde{\varphi}_k\|_{L^2(\mathbb{R}^+)}^2\;=\;D_\alpha^2\,\big\|\overline{\mathscr{H}_\alpha^+}(\widetilde{\varphi}_k)_{k\in\mathbb{Z}}\big\|_{\cH}^2 \\
   &=\;D_\alpha^2\,\big\|\overline{\mathsf{H}_\alpha^+}\,\widetilde{\varphi}\big\|_{L^2(\mathbb{R}^+_x\times\mathbb{S}^1_y)}^2\,,
 \end{split}
\]
which completes the proof.
\end{proof}

Based upon the above estimates, we can prove Proposition \ref{prop:Hclosurecontrol}.

\begin{proof}[Proof of Proposition \ref{prop:Hclosurecontrol}]
 Again, it suffices to establish \eqref{eq:Hclosurecontrol} when $\widetilde{\varphi}\in C^\infty_c(\mathbb{R}^+_x\times\mathbb{S}^1_y)$, and then conclude by density from \eqref{eq:DHaclosedclosure}.
 
 One has
 \begin{equation*}
 \begin{split}
  \|\partial_x^2\widetilde{\varphi}\|_{L^2(\mathbb{R}_x\times\mathbb{S}^1_y)}\;&\leqslant\;\big\|\overline{\mathsf{H}_\alpha^+}\,\widetilde{\varphi}\big\|_{L^2(\mathbb{R}_x\times\mathbb{S}^1_y)}+\big\|x^{2\alpha}\partial_y^2 \widetilde{\varphi}\big\|_{L^2(\mathbb{R}^+_x\times\mathbb{S}^1_y)}+C_\alpha\|x^{-2}\widetilde{\varphi}\|_{L^2(\mathbb{R}_x^+\times\mathbb{S}^1_y)} \\
  &\leqslant\;\big\|\overline{\mathsf{H}_\alpha^+}\,\widetilde{\varphi}\big\|_{L^2(\mathbb{R}_x^+\times\mathbb{S}^1_y)}+D_\alpha\big\|\overline{\mathsf{H}_\alpha^+}\,\widetilde{\varphi}\big\|_{L^2(\mathbb{R}^+_x\times\mathbb{S}^1_y)}+\frac{4C_\alpha}{3}\|\partial_x^2\widetilde{\varphi}\|_{L^2(\mathbb{R}_x^+\times\mathbb{S}^1_y)}\,,
 \end{split}
 \end{equation*}
 where the first inequality is a triangular inequality based on \eqref{eq:actionofHaclosed}, whereas the second inequality follows directly from Corollary \ref{cor:doubleHardy} and Lemma \ref{lem:boundedness_x2alphad2y}.
 
 Therefore,
 \[
  \|\partial_x^2\widetilde{\varphi}\|_{L^2(\mathbb{R}_x^+\times\mathbb{S}^1_y)}\;\leqslant\;\frac{1+D_\alpha}{\,1-\frac{4}{3}C_\alpha}\,\big\|\overline{\mathsf{H}_\alpha^+}\,\widetilde{\varphi}\big\|_{L^2(\mathbb{R}_x^+\times\mathbb{S}^1_y)}\,.
 \]
 As $C_\alpha=\frac{1}{4}\alpha(2+\alpha)$, the constant $K_\alpha:=(1+D_\alpha)(1-\frac{4}{3}C_\alpha)^{-1}$ is strictly positive for any $\alpha$ of interest, namely, $\alpha\in(0,1)$. Moreover, $K_\alpha\to +\infty$ as $\alpha\uparrow 1$ (indeed, tracing back the constant $D_\alpha$ through the proof of Lemma 3.3 where it was imported from in Lemma \ref{lem:boundedness_x2alphad2y}, it is easy to see that $D_\alpha$ does not diverge when $\alpha\uparrow 1$). The proof is thus completed.  
\end{proof}

\subsection{Control of $\vartheta$}\label{sec:q0q1}~

 As a counterpart to Subsect.~\ref{sec:control-of-tildephi}, let us now prove the needed short-scale behaviour of the function $\vartheta\in L^2(\mathbb{R}\times\mathbb{S}^1,\ud x\ud y)$ defined in  \eqref{eq:deffunctionvartheta}.

 Let us recall that $\vartheta^\pm$ may well fail to belong to $\mathcal{D}(\overline{\mathsf{H}_\alpha^\pm})$ and therefore cannot be controlled by means of Prop.~\ref{prop:Hclosurecontrol}: a separate analysis is needed, and we base it on the explicit expression and homogeneity properties of $\vartheta$.

 Our main result here is the following.

 \begin{proposition}\label{prop:varthetacontrol}
   Let $\alpha\in[0,1)$. For almost every $y\in\mathbb{S}^1$,
 \begin{itemize}
  \item[(i)] the function $x\mapsto\vartheta^\pm(x,y)$ belongs to $C^1(0,1)$,
  \item[(ii)] $\vartheta^\pm(x,y)=o(|x|^{\frac{3}{2}})$ as $x\to 0^\pm$,  
  \item[(iii)] $\partial_x\vartheta^\pm(x,y)=o(|x|^{\frac{1}{2}})$ as $x\to 0^\pm$.
 \end{itemize} 
 \end{proposition}

 In preparation for the proof of this result, in terms of the functions
\begin{equation}\label{eq:funtionsh0h1}
 \begin{split}
  h_{0,k}\;&:=\; {\textstyle{\sqrt{\frac{2}{\pi(1+\alpha)}}}}\,|k|^{\frac{1}{2(1+\alpha)}}\Big(\Phi_{\alpha,k}-{\textstyle\sqrt{\frac{\pi(1+\alpha)}{2|k|}}\,x^{-\frac{\alpha}{2}}+\sqrt{\frac{\pi|k|}{2(1+\alpha)}}\,x^{1+\frac{\alpha}{2}}}\Big)\,,\\
  h_{1,k}\;&:=\;{\textstyle{\sqrt{\frac{2}{\pi(1+\alpha)}}}}\,|k|^{\frac{5}{2(1+\alpha)}}\Big(\Psi_{\alpha,k}-{\textstyle\sqrt{\frac{2|k|}{\pi(1+\alpha)^3}}}\,\|\Phi_{\alpha,k}\|_{L^2(\mathbb{R}^+)}^2\,|x|^{1+\frac{\alpha}{2}}\Big)
 \end{split}
\end{equation}
 defined on $\mathbb{R}^+$ for each $k\in\mathbb{Z}\setminus\{0\}$, one sees from \eqref{eq:defq0q1decomp-q0}-\eqref{eq:defq0q1decomp-q1} that 
  \begin{equation}\label{eq:q0q1scaling}
  \begin{split}
   \vartheta_{0,k}^\pm(x)\;&=\;c_{0,k}^{\pm}{\textstyle\sqrt{\frac{\pi(1+\alpha)}{2}}}\,|k|^{-\frac{1}{2(1+\alpha)}}\,h_{0,k}(|x|)\,,\qquad 0<\pm x<1\,, \\
   \vartheta_{1,k}^\pm(x)\;&=\;c_{1,k}^{\pm}{\textstyle\sqrt{\frac{\pi(1+\alpha)}{2}}}\,|k|^{-\frac{5}{2(1+\alpha)}}\,h_{1,k}(|x|)\,,\qquad 0<\pm x<1\,.
  \end{split}
 \end{equation}
 Clearly the above identities are not valid when $|x|>1$.

\begin{lemma}\label{lem:homogeneity}
 Let $\alpha\in[0,1)$ and $k\in\mathbb{Z}\setminus\{0\}$. For $x\in\mathbb{R}^+$ one has
\begin{equation}\label{eq:q0q1scaling2}
  h_{0,k}(x)\;:=\;w_0\big(|k|x^{1+\alpha}\big)\,,\qquad h_{1,k}(x)\;:=\;w_1\big(|k|x^{1+\alpha}\big)
\end{equation}
with
\begin{equation}\label{eq:q0q1scalingw0}
 w_0(x)\;:=\;x^{-\frac{\alpha}{2(1+\alpha)}}\big(e^{-\frac{x}{1+\alpha}}-1+{\textstyle\frac{x}{1+\alpha}}\big)
\end{equation}
and 
\begin{equation}\label{eq:q0q1scalingw1}
 \begin{split}
  w_1(x)\;&:=\;x^{-\frac{\alpha}{2(1+\alpha)}}\,e^{-\frac{x}{1+\alpha}}\int_0^{x^{\frac{1}{1+\alpha}}}\ud\rho\,\rho^{-\alpha}\sinh{\textstyle(\frac{\rho^{1+\alpha}}{1+\alpha})}\,e^{-\frac{\,\rho^{1+\alpha}}{1+\alpha}} \\
  &\qquad +x^{-\frac{\alpha}{2(1+\alpha)}}\,\sinh{\textstyle(\frac{x}{1+\alpha})}\int_{x^{\frac{1}{1+\alpha}}}^{+\infty}\!\!\ud\rho\,\rho^{-\alpha}\,e^{-\frac{2\rho^{1+\alpha}}{1+\alpha}} \\
  &\qquad-2^{-\frac{1-\alpha}{1+\alpha}}\,(1+\alpha)^{-\frac{1+3\alpha}{1+\alpha}}\,\Gamma\big({\textstyle\frac{1-\alpha}{1+\alpha}}\big)\,x^{\frac{2+\alpha}{2(1+\alpha)}}\,.
 \end{split}
\end{equation} 
\end{lemma}

\begin{proof}
 Plugging the explicit expression \eqref{eq:Phi_and_F_explicit} for $\Phi_{\alpha,k}$ into the first formula in \eqref{eq:funtionsh0h1}  one finds
 \[
  h_{0,k}(x)\;=\;\big(|k|^{\frac{1}{1+\alpha}}x)^{-\frac{\alpha}{2}}\Big(e^{-\frac{|k|}{1+\alpha}x^{1+\alpha}}-1+{\textstyle\frac{|k|x^{1+\alpha}}{1+\alpha}}\Big)\;=\;w_0\big(|k|x^{1+\alpha}\big)
 \]
 with $w_0$ defined by \eqref{eq:q0q1scalingw0}. Analogously, inserting the expression \eqref{eq:explicitPsika} for $\Psi_{\alpha,k}$ and the expression \eqref{eq:Phinorm} for $\| \Phi_{\alpha,k} \|_{L^2(\mathbb{R}^+)}^2$ into the second formula in \eqref{eq:funtionsh0h1}, one obtains
 \[
 \begin{split}
  h_{1,k}^{\pm}(x)\;&=\; \big(|k|^{\frac{1}{1+\alpha}}x\big)^{-\frac{\alpha}{2}}e^{-\frac{|k|x^{1+\alpha}}{1+\alpha}}\!\int_0^{x|k|^{\frac{1}{1+\alpha}}}\ud\rho\,\rho^{-\alpha}\sinh{\textstyle(\frac{\rho^{1+\alpha}}{1+\alpha})}\,e^{-\frac{\,\rho^{1+\alpha}}{1+\alpha}} \\
  &\qquad\quad +\big(|k|^{\frac{1}{1+\alpha}}x\big)^{-\frac{\alpha}{2}}\sinh{\textstyle\big(\frac{|k|x^{1+\alpha}}{1+\alpha}\big)}\!\int_{x|k|^{\frac{1}{1+\alpha}}}^{+\infty}\!\!\ud\rho\,\rho^{-\alpha}\,e^{-\frac{2\rho^{1+\alpha}}{1+\alpha}} \\
  &\qquad\quad-2^{-\frac{1-\alpha}{1+\alpha}}\,(1+\alpha)^{-\frac{1+3\alpha}{1+\alpha}}\,\Gamma\big({\textstyle\frac{1-\alpha}{1+\alpha}}\big)\,\big(|k|^{\frac{1}{1+\alpha}}x\big)^{1+\frac{\alpha}{2}}  \\
  &=\;w_1\big(|k|x^{1+\alpha}\big)
 \end{split}
 \]
 with $w_1$ defined by \eqref{eq:q0q1scalingw1}.
\end{proof}

\begin{lemma}\label{lem:scaling_on_norms_hjk}
 Let $\alpha\in[0,1)$ and $k\in\mathbb{Z}\setminus\{0\}$. The functions $h_{0,k}$ and $h_{1,k}$ defined in \eqref{eq:funtionsh0h1} satisfy
 \begin{eqnarray}
  \big\|x^{-2}h_{j,k}\big\|_{L^2((0,1))}^2\!\!&\leqslant&\!|k|^{\frac{3}{1+\alpha}}\,\big\|x^{-2}h_{j,1}\big\|_{L^2(\mathbb{R}^+)}^2\,, \label{eq:scalingx-2h} \\
  \big\|h_{j,k}''\big\|_{L^2((0,1))}^2\!\!&\leqslant&\!|k|^{\frac{3}{1+\alpha}}\,\big\|h_{j,1}''\big\|_{L^2(\mathbb{R}^+)}^2 \label{eq:scalingd2h}
 \end{eqnarray}
 for $j\in\{0,1\}$. 
\end{lemma}

\begin{proof}
 By means of the homogeneity properties \eqref{eq:q0q1scaling2} one finds
 \[
  \begin{split}
    \big\|x^{-2}h_{j,k}\big\|_{L^2((0,1))}^2\;&=\;\int_0^1\big|x^{-2} w_j\big(|k|x^{1+\alpha}\big)\big|^2\,\ud x \\
    &=\;|k|^{\frac{3}{1+\alpha}}\int_0^{|k|^{\frac{1}{1+\alpha}}}|x^{-2}w_j(x^{1+\alpha})|^2\,\ud x \\
    &\leqslant\;|k|^{\frac{3}{1+\alpha}}\int_0^{+\infty}|x^{-2}h_{j,1}(x)|^2\,\ud x
  \end{split}
 \]
 and 
\[
 \begin{split}
  \big\|&h_{j,k}''\big\|_{L^2((0,1))}^2\;=\;\int_0^1\Big| \frac{\ud^2}{\ud x^2}\,w_j\big(|k|x^{1+\alpha}\big) \Big|^2\,\ud x \\
  &=\;\int_0^1\big| (1+\alpha)^2|k|^2 x^{2\alpha}w_j''\big(|k|x^{1+\alpha}\big)+\alpha(1+\alpha)|k|x^{-(1-\alpha)}w_j'\big(|k|x^{1+\alpha}\big) \big|^2\,\ud x \\
  &=\;|k|^{\frac{3}{1+\alpha}}\int_0^{|k|^{\frac{1}{1+\alpha}}}\big| (1+\alpha)^2 x^{2\alpha}w_j''\big(x^{1+\alpha}\big)+\alpha(1+\alpha)x^{-(1-\alpha)}w_j'\big(x^{1+\alpha}\big) \big|^2\,\ud x \\
  &=\;|k|^{\frac{3}{1+\alpha}}\int_0^{|k|^{\frac{1}{1+\alpha}}}\Big| \frac{\ud^2}{\ud x^2}\,w_j\big(x^{1+\alpha}\big) \Big|^2\,\ud x\;\leqslant\;|k|^{\frac{3}{1+\alpha}}\int_0^{+\infty}|h_{j,1}''(x)|^2\,\ud x\,,
 \end{split}
\]
which proves, respectively, \eqref{eq:scalingx-2h} and \eqref{eq:scalingd2h}.
\end{proof}

\begin{lemma}\label{lem:finitenessnormsh0h1}
 Let $\alpha\in[0,1)$. The functions $h_{0,1}$ and $h_{1,1}$ defined in \eqref{eq:funtionsh0h1} satisfy
  \begin{eqnarray}
  \big\|x^{-2}h_{j,1}\big\|_{L^2(\mathbb{R}^+)}^2\!\!&<&\!+\infty \label{eq:x-2h-finite}\\
  \big\|h_{j,1}''\big\|_{L^2(\mathbb{R}^+)}^2\!\!&<&\!+\infty \label{eq:D2h-finite}
 \end{eqnarray}
 for $j\in\{0,1\}$. 
\end{lemma}

\begin{proof}
As $h_{0,1}$ (resp., $h_{1,1}$) only agrees with $\vartheta^+_{0,1}$ (resp., $\vartheta^+_{1,1}$) over the interval $(0,1)$, apart from a $\alpha$-dependent pre-factor, one cannot deduce \eqref{eq:x-2h-finite}-\eqref{eq:D2h-finite} from \eqref{eq:regularityoftheta01}, because the considered norms are over the whole $\mathbb{R}^+$. However, the reasoning made in the proof of Theorem \ref{prop:g_with_Pweight}, which led to \eqref{eq:regularityoftheta01}, can be essentially repeated here. Clearly, both $h_{0,1}$ and $h_{1,1}$ are $C^\infty(\mathbb{R}^+)$-functions; therefore, the finiteness of the norms in \eqref{eq:x-2h-finite}-\eqref{eq:D2h-finite} is only to be checked as $x\downarrow 0$ and $x\to+\infty$. In fact, for
\[
 h_{0,1}\;=\;x^{-\frac{\alpha}{2}}\big(e^{-\frac{x^{1+\alpha}}{1+\alpha}}-1+{\textstyle\frac{x^{1+\alpha}}{1+\alpha}}\big)
\]
one can perform a straightforward computation and find
\[
 \begin{split}
  & \;\,h_{0,1}(x)\;\stackrel{x\downarrow 0}{=}\;x^{2+\frac{3}{2}\alpha}(1+O(x^{1+\alpha}))\,, \\
  & h_{0,1}(x)\stackrel{x\to +\infty}{=}{\textstyle\frac{1}{1+\alpha}}x^{1-\frac{\alpha}{2}}(1+O(x^{-1})) \,,
 \end{split}
\]
and 
\[
 \begin{split}
  &h_{0,1}''(x)(x)\;\stackrel{x\downarrow 0}{=}\;x^{\frac{3}{2}\alpha}\big({\textstyle\frac{9}{8}-\frac{1}{8(1+\alpha)^2}}\big)(1+O(x^{1+\alpha}))\,, \\
  &\quad h_{0,1}''(x)\stackrel{x\to +\infty}{=}{\textstyle\frac{\alpha(2+\alpha)}{4(1+\alpha)}}x^{-(1+\frac{\alpha}{2})}(1+o(1))\,.
  \end{split}
\]
Such asymptotics imply \eqref{eq:x-2h-finite}-\eqref{eq:D2h-finite} when $j=0$, as $\alpha\in(0,1)$. Concerning
\[
  h_{1,1}\;=\;{\textstyle{\sqrt{\frac{2}{\pi(1+\alpha)}}}}\,\Psi_{\alpha,1}-{\textstyle\frac{2}{\pi(1+\alpha)^2}}\,\|\Phi_{\alpha,1}\|_{L^2(\mathbb{R}^+)}^2\,x^{1+\frac{\alpha}{2}}\,,
\]
the square-integrability of $x^{-2}h_{1,1}$ is controlled analogously to the proof of Theorem \ref{prop:g_with_Pweight}: the short-distance asymptotics \eqref{eq:Psi_asymptotics} for $\Psi_{\alpha,1}$ gives a convenient compensation in $h_{1,1}$ as $x\downarrow 0$, whereas at infinity the control can be simply made term by term, as $\Psi_{\alpha,1}\in L^2(\mathbb{R}^+)$. Thus, \eqref{eq:x-2h-finite} is also proved for $j=1$. Next, we consider
\[
 h_{1,1}''\;=\;{\textstyle{\sqrt{\frac{2}{\pi(1+\alpha)}}}}\,\Psi_{\alpha,1}''-{\textstyle\frac{2}{\pi(1+\alpha)^2}}\,\|\Phi_{\alpha,1}\|_{L^2(\mathbb{R}^+)}^2\,{\textstyle\frac{\alpha(2+\alpha)}{2}}\,x^{-(1-\frac{\alpha}{2})}\,.
\]
As $\Psi_{\alpha,1}=R_{G_{\alpha,1}}\Phi_{\alpha,1}$ and $R_{G_{\alpha,1}}=(A_{\alpha,F}^+(1))^{-1}$ (see \eqref{eq:defPsi} and Prop.~\ref{eq:RGisSFinv} above), then
\[
 \begin{split}
  \Psi_{\alpha,1}''\;&=\;-\big({\textstyle -\frac{\ud^2}{\ud x^2}}+x^{2\alpha}+{\textstyle\frac{\alpha(2+\alpha)}{2}}x^{-2}\big)R_{G_{\alpha,1}}\Phi_{\alpha,1}+\big(x^{2\alpha}+{\textstyle\frac{\alpha(2+\alpha)}{2}}x^{-2}\big)\Psi_{\alpha,1} \\
  &=\;-\Phi_{\alpha,1}+\big(x^{2\alpha}+{\textstyle\frac{\alpha(2+\alpha)}{2}}x^{-2}\big)\Psi_{\alpha,1}\,,
 \end{split}
\]
whence
\[
 \begin{split}
  h_{1,1}''\;&=\;-{\textstyle{\sqrt{\frac{2}{\pi(1+\alpha)}}}}\,\Phi_{1,\alpha}+{\textstyle{\sqrt{\frac{2}{\pi(1+\alpha)}}}}\big(x^{2\alpha}+{\textstyle\frac{\alpha(2+\alpha)}{2}}x^{-2}\big)\Psi_{\alpha,1} \\ 
 &\qquad\qquad  -{\textstyle\frac{2}{\pi(1+\alpha)^2}}\,\|\Phi_{\alpha,1}\|_{L^2(\mathbb{R}^+)}^2\,{\textstyle\frac{\alpha(2+\alpha)}{2}}\,x^{-(1-\frac{\alpha}{2})} \\
 &=\;-{\textstyle{\sqrt{\frac{2}{\pi(1+\alpha)}}}}\,\Phi_{1,\alpha}+{\textstyle{\sqrt{\frac{2}{\pi(1+\alpha)}}}}\,x^{2\alpha}\,\Psi_{\alpha,1}+{\textstyle\frac{\alpha(2+\alpha)}{2}}\,x^{-2}\,h_{1,1}\,.
 \end{split}
\]
Each of the three summands in the r.h.s.~above belongs to $L^2(\mathbb{R}^+)$: in particular, the second does so because $\Psi_{\alpha,1}\in\mathrm{ran}\,R_{G_{\alpha,k}}\subset L^2(\mathbb{R}^+,\langle x\rangle^{4\alpha} \ud x)$ (Corollary \ref{cor:RGtoWeightedL2}). This proves \eqref{eq:D2h-finite} for $j=1$.
\end{proof}

From \eqref{eq:q0q1scaling}, and from Lemmas \ref{lem:scaling_on_norms_hjk} and \ref{lem:finitenessnormsh0h1}, one immediately deduces:

\begin{corollary}\label{cor:varthetasummability}
 Let $\alpha\in[0,1)$ and $k\in\mathbb{Z}\setminus\{0\}$. Then
 \begin{equation}\label{eq:varthetasummability0}
  \begin{split}
   \big\|x^{-2}\vartheta_{0,k}^\pm\big\|_{L^2(I^\pm)}^2\;&\lesssim\;|c_{0,k}^\pm|^2\,|k|^{\frac{2}{1+\alpha}}\,, \\
   \big\|(\vartheta_{0,k}^\pm)''\big\|_{L^2(I^\pm)}^2\;&\lesssim\;|c_{0,k}^\pm|^2\,|k|^{\frac{2}{1+\alpha}}\,,
  \end{split}
 \end{equation}
 and 
  \begin{equation}\label{eq:varthetasummability1}
  \begin{split}
   \big\|x^{-2}\vartheta_{1,k}^\pm\big\|_{L^2(I^\pm)}^2\;&\lesssim\;|c_{1,k}^\pm|^2\,|k|^{-\frac{2}{1+\alpha}}\,, \\
   \big\|(\vartheta_{1,k}^\pm)''\big\|_{L^2(I^\pm)}^2\;&\lesssim\;|c_{1,k}^\pm|^2\,|k|^{-\frac{2}{1+\alpha}}\,,
  \end{split}
 \end{equation}
 with $I^+=(0,1)$ and $I^-=(-1,0)$, where the constants in the above inequalities only depend on $\alpha$. 
\end{corollary}

In fact, \eqref{eq:varthetasummability0}-\eqref{eq:varthetasummability1} are trivially true also for $k=0$: recall indeed (see \eqref{eq:varthetazeromode} above) that $\vartheta_0\equiv 0$.

 \begin{proof}[Proof of Proposition \ref{prop:varthetacontrol}]
  It clearly suffices to discuss the proof for the `+' component $\vartheta^+=\mathcal{F}_2^{-1}(\vartheta^+_k)_{k\in\mathbb{Z}}$. Recall also that $\vartheta_0^+\equiv 0$.

  Now, owing to Corollary \ref{cor:varthetasummability},
  \[
   \begin{split}
    \big\|x^{-2}(\vartheta^+_{0,k})_{k\in\mathbb{Z}}\big\|^2_{\ell^2(\mathbb{Z},L^2((0,1),\ud x))}\;&\lesssim\;\sum_{k\in\mathbb{Z}\setminus\{0\}}|c_{0,k}^\pm|^2\,|k|^{\frac{2}{1+\alpha}} \,,\\
    \big\|((\vartheta_{0,k}^\pm)'')_{k\in\mathbb{Z}}\big\|^2_{\ell^2(\mathbb{Z},L^2((0,1),\ud x))}\;&\lesssim\;\sum_{k\in\mathbb{Z}\setminus\{0\}}|c_{0,k}^\pm|^2\,|k|^{\frac{2}{1+\alpha}}\,.
   \end{split}
  \]
  The series in the r.h.s.~above are \emph{finite}, because of the enhanced summability of the $c_{0,k}$'s due to the fact that the initially considered $(g_k)_{k\in\mathbb{Z}}$ belongs to the domain of a uniformly fibred extension (as observed already in Remark \ref{rem:enhanced_summability}).

  As a first consequence, $(\vartheta^+_{0,k})_{k\in\mathbb{Z}}$ belongs to $\ell^2(\mathbb{Z},L^2((0,1),\ud x))$, and so too does $(\vartheta^+_{1,k})_{k\in\mathbb{Z}}$ by difference from $(\vartheta^+_{k})_{k\in\mathbb{Z}}$: therefore, the inverse Fourier transform can be separately applied to
  \[
   \vartheta^+\;=\;\mathcal{F}_2^{-1}(\vartheta^+_k)_{k\in\mathbb{Z}}\;=\;\mathcal{F}_2^{-1}(\vartheta^+_{0,k})_{k\in\mathbb{Z}}+\mathcal{F}_2^{-1}(\vartheta^+_{1,k})_{k\in\mathbb{Z}}\,.
  \]

  As a further consequence, the above estimates imply, by means of Plancherel's formula, 
  \[
   \begin{split}
    \big\|x^{-2}\mathcal{F}_2^{-1}(\vartheta^+_{0,k})_{k\in\mathbb{Z}}\big\|^2_{L^2((0,1)\times\mathbb{S}^1,\ud x \ud y)}\;=\;\big\|x^{-2}(\vartheta^+_{0,k})_{k\in\mathbb{Z}}\big\|^2_{\ell^2(\mathbb{Z},L^2((0,1),\ud x))}\;&<\;+\infty \,,\\
    \big\|\partial^2_{x}\mathcal{F}_2^{-1}(\vartheta^+_{0,k})_{k\in\mathbb{Z}}\big\|^2_{L^2((0,1)\times\mathbb{S}^1,\ud x \ud y)}\;=\;\big\|(\partial_x^2\vartheta^+_{0,k})_{k\in\mathbb{Z}}\big\|^2_{\ell^2(\mathbb{Z},L^2((0,1),\ud x))}\;&<\;+\infty\,.
   \end{split}
  \]

  Analogously, Corollary \ref{cor:varthetasummability} also implies
  \[
   \begin{split}
    \big\|x^{-2}(\vartheta^+_{1,k})_{k\in\mathbb{Z}}\big\|^2_{\ell^2(\mathbb{Z},L^2((0,1),\ud x))}\;&\lesssim\;\sum_{k\in\mathbb{Z}\setminus\{0\}}|c_{1,k}^\pm|^2\,|k|^{-\frac{2}{1+\alpha}} \,,\\
    \big\|((\vartheta_{1,k}^\pm)'')_{k\in\mathbb{Z}}\big\|^2_{\ell^2(\mathbb{Z},L^2((0,1),\ud x))}\;&\lesssim\;\sum_{k\in\mathbb{Z}\setminus\{0\}}|c_{1,k}^\pm|^2\,|k|^{-\frac{2}{1+\alpha}}\,,
   \end{split}
  \]
  and the series in the r.h.s.~above are \emph{finite} because of the general summability for elements in $\mathcal{D}(\mathscr{H}_{\alpha}^*)$ established in Lemma \ref{lem:gkkrepr}, formula \eqref{eq:pileupcond3}. Thus, for almost every $y\in\mathbb{S}^1$,
  \[
   \begin{split}
    \big\|x^{-2}\mathcal{F}_2^{-1}(\vartheta^+_{1,k})_{k\in\mathbb{Z}}\big\|^2_{L^2((0,1)\times\mathbb{S}^1,\ud x \ud y)}\;&<\;+\infty\,, \\
    \big\|\partial^2_{x}\mathcal{F}_2^{-1}(\vartheta^+_{1,k})_{k\in\mathbb{Z}}\big\|^2_{L^2((0,1)\times\mathbb{S}^1,\ud x \ud y)}\;&<\;+\infty\,.
   \end{split}
  \]
  
  Summarising, 
  \[
   \|x^{-2}\vartheta^+\|_{L^2((0,1)\times\mathbb{S}^1,\ud x \ud y)}+ \|\partial_x^2\vartheta^+\|_{L^2((0,1)\times\mathbb{S}^1,\ud x \ud y)}\;<\;+\infty\,.
  \]
  Therefore, $\vartheta^+$ satisfies the assumptions (a) and (b) of Lemma \ref{lem:grand_auxiliary_lemma} (for, obviously, $|x^{-(\frac{3}{2}+\frac{\alpha}{2})}\vartheta^+(x,y)|\leqslant|x^{-2}\vartheta^+(x,y)|$ when $x\in(0,1)$, since $\alpha\in(0,1)$). The thesis then follows by applying Lemma \ref{lem:grand_auxiliary_lemma}.  
 \end{proof}

 \subsection{Proof of the classification theorem}\label{sec:proofclassifthm}

 \begin{proof}[Proof of Theorem \ref{thm:classificationUF}]
 Let us characterise the domain $\mathcal{D}(\mathcal{F}_2^{-1}\mathscr{H}_\alpha^{\mathrm{u.f.}}\mathcal{F}_2)$ of the various uniformly fibred extensions of $\mathsf{H}_\alpha=\mathcal{F}_2^{-1}\mathscr{H}_\alpha\mathcal{F}_2$.

The expression \eqref{eq:HHalphaadjointagain} for $\mathsf{H}_\alpha^*$ provided in the statement of the theorem was already found in \eqref{eq:HHalphaadjoint}.

Next, let us consider a generic $\phi=\mathcal{F}_2^{-1}(g_k)_{k\in\mathbb{Z}}\in\mathcal{D}(\mathcal{F}_2^{-1}\mathscr{H}_\alpha^{\mathrm{u.f.}}\mathcal{F}_2)$, where $(g_k)_{k\in\mathbb{Z}}\in\mathcal{D}(\mathscr{H}_\alpha^{\mathrm{u.f.}})$. Owing to the definitions \eqref{eq:afterF2-1}-\eqref{eq:F2-1g1} and to Corollary \ref{cor:(g0)k_(g1)k_in_Hs}, 
\begin{equation}\label{eq:nowsafespliting}
 \phi(x,y)\;=\;\varphi(x,y)+g_1(y)|x|^{1+\frac{\alpha}{2}}P(x)+g_0(y)|x|^{-\frac{\alpha}{2}}P(x)\,,
\end{equation}
where $P$ is a smooth cut-off which is identically equal to one for $|x|<1$ and zero for $|x|>2$, and $g_0,g_1\in L^2(\mathbb{S}^1)$ with further Sobolev regularity as specified therein.

Moreover, upon splitting $\varphi=\widetilde{\varphi}+\vartheta$ as in \eqref{eq:splittingphiphitildetheta}, and using Prop.~\ref{prop:Hclosurecontrol} for $\widetilde{\varphi}$ and Prop.~\ref{prop:varthetacontrol} for $\vartheta$, we deduce that for almost every $y\in\mathbb{S}^1$
 \begin{itemize}
  \item the function $x\mapsto\varphi^\pm(x,y)$ belongs to $C^1(0,1)$,
  \item $\varphi^\pm(x,y)=o(|x|^{3/2})$ as $x\to 0^\pm$,  
  \item $\partial_x\varphi^\pm(x,y)=o(|x|^{1/2})$ as $x\to 0^\pm$.
 \end{itemize} 
Plugging this information into \eqref{eq:nowsafespliting} yields
\[
 \begin{split}
  \lim_{x\to 0^\pm} |x|^{\frac{\alpha}{2}}\,\phi^\pm(x,y)\;&=\;g_0^\pm(y)\,, \\
  \lim_{x\to 0^\pm} |x|^{-(1+\frac{\alpha}{2})}\big(\phi^\pm(x,y)-g_0^\pm(y)|x|^{-\frac{\alpha}{2}}\big)\;&=\;  g_1^\pm(y)+\lim_{x\to 0^\pm} |x|^{-(1+\frac{\alpha}{2})}\varphi^\pm(x,y)\\
  &=\;g_1^\pm(y)\,,
 \end{split}
\]
namely
\begin{equation}\label{eq:g0g1f0f1}
 g_0\;=\;\phi_0\,,\qquad g_1\;=\;\phi_1\,,
\end{equation}
proving also that the limits \eqref{eq:limitphi0}, as well as the limits of the first line of \eqref{eq:limitphi1}, do exist. Also, the Sobolev regularity stated for $\phi_0$ and $\phi_1$ follows directly from Corollary \ref{cor:(g0)k_(g1)k_in_Hs}.

The second identity in \eqref{eq:limitphi1} is obtained as follows. By means of \eqref{eq:nowsafespliting} we compute
 \[
  \begin{split}
   \pm(1+\alpha)^{-1}&\lim_{x\to 0^\pm} |x|^{-\alpha}\partial_x\big(|x|^{\frac{\alpha}{2}}\phi^\pm(x,y)\big)\;=\; \\
   &=\;\pm(1+\alpha)^{-1}\lim_{x\to 0^\pm} |x|^{-\alpha}\partial_x\big(|x|^{\frac{\alpha}{2}}\varphi^\pm(x,y)+g_1^\pm(y)|x|^{1+\alpha}+g_0^\pm(y)\big) \\
   &=\;g_1^\pm(y)\pm(1+\alpha)^{-1}\lim_{x\to 0^\pm} |x|^{-\alpha}\partial_x\big(|x|^{\frac{\alpha}{2}}\varphi^\pm(x,y)\big)\,.
  \end{split}
 \]
 On the other hand,
 \[
   \lim_{x\to 0^\pm} |x|^{-\alpha}\partial_x \big(|x|^{\frac{\alpha}{2}}\varphi^\pm(x,y)\big)= \lim_{x\to 0^\pm} \big({\textstyle\frac{\alpha}{2}|x|^{-(1+\frac{\alpha}{2})}}\varphi^\pm(x,y)+|x|^{-\frac{\alpha}{2}}\partial_x\varphi^\pm(x,y)\big)=0\,,
 \]
 having used the properties  $\varphi^\pm(x,y)=o(|x|^{\frac{3}{2}})$ and $\partial_x\varphi^\pm(x,y)=o(|x|^{\frac{1}{2}})$ as $x\to 0^\pm$. This yields the second identity in \eqref{eq:limitphi1}.

It remains to show that for each type of extension, the stated boundary conditions of self-adjointness do hold for $\phi_0$ and $\phi_1$. As, by  \eqref{e1:F2-1g0}-\eqref{eq:F2-1g1} and by \eqref{eq:g0g1f0f1}
\[
 \begin{split}
  \phi_0^{\pm}(y)\;&=\;\frac{1}{\sqrt{2\pi}}\sum_{k\in\mathbb{Z}} e^{\ii k y}g_{0,k}^{\pm} \\
  \phi_1^{\pm}(y)\;&=\;\frac{1}{\sqrt{2\pi}}\sum_{k\in\mathbb{Z}} e^{\ii k y}g_{1,k}^{\pm}\,,
 \end{split}
\]
the above series converging in $L^2(\mathbb{S}^1)$,
and since for each \emph{uniformly fibred} extension $\mathscr{H}_\alpha^{\mathrm{u.f.}}$ the boundary conditions are expressed by \emph{the same linear combinations} of the $g_{0,k}^{\pm}$'s and $g_{1,k}^{\pm}$'s for each $k$, then now the boundary conditions of self-adjointness in terms of $\phi_0$ and $\phi_1$ are immediately read out from those of the classification Theorem \ref{thm:bifibre-extensions} for bilateral-fibre extensions (see also Table \ref{tab:extensions}) in terms of $g_{0,k}^{\pm}$ and $g_{1,k}^{\pm}$.
\end{proof}

\section{Putting all together}\label{sec:proof_xy_Euclidean}

We can finally get back to the statements made in the introduction, Subsect.~\ref{sec:scheme-and-main-results}, that are still to be proved.

 \begin{proof}[Proof of Proposition \ref{prop:adjoint_on_M}]
  The thesis is actually immediate from the analogous statement \eqref{eq:HHalphaadjointagain} in Theorem \ref{thm:classificationUF} for $(\mathsf{H}_\alpha^\pm)^*$, by exploiting the unitary equivalence \eqref{eq:tildeHalpha}, namely
  \begin{equation*}
  \begin{split}
   H_\alpha^\pm\;&=\;(U_\alpha^\pm)^{-1}\mathsf{H}_\alpha^\pm\,U_\alpha^\pm \,,\\
   (H_\alpha^\pm)^*\;&=\;(U_\alpha^\pm)^{-1}(\mathsf{H}_\alpha^\pm)^*\,U_\alpha^\pm\,,
  \end{split}
  \end{equation*}
  where, as set in \eqref{eq:unit1}, $\phi^\pm=U_\alpha^\pm f^\pm=|x|^{-\frac{\alpha}{2}}f^\pm$. Tacitly we used the well-known fact, which is trivial for a finite sum and we also reviewed in Lemma \ref{lem:sumstar-sumclosure} for an infinite sum, that the adjoint of the direct sum of two operators is the direct sum of the adjoints.  
  \end{proof}

\begin{proof}[Proof of Theorem \ref{thm:H_alpha_fibred_extensions}]
 Also in this case, the proof is a matter of exporting the classification of Theorem \ref{thm:classificationUF} for the uniformly fibred self-adjoint extensions of $\mathsf{H}_\alpha$, via unitary equivalence, to the corresponding extensions of
 \[
  H_\alpha\;=\;U_\alpha^{-1}\mathsf{H}_\alpha\,U_\alpha\,.
 \]

 We then define
 \[
  \begin{split}
   H_{\alpha,F}\;&:=\;U_\alpha^{-1}\,\mathsf{H}_{\alpha,F}\,U_\alpha\,, \\
   H_{\alpha,R}^{[\gamma]}\;&:=\;U_\alpha^{-1}\,\mathsf{H}_{\alpha,R}^{[\gamma]}\,U_\alpha\,, \\
   H_{\alpha,L}^{[\gamma]}\;&:=\;U_\alpha^{-1}\,\mathsf{H}_{\alpha,L}^{[\gamma]}\,U_\alpha\,, \\
   H_{\alpha,a}^{[\gamma]}\;&:=\;U_\alpha^{-1}\,\mathsf{H}_{\alpha,a}^{[\gamma]}\,U_\alpha\,, \\
   H_{\alpha}^{[\Gamma]}\;&:=\;U_\alpha^{-1}\,\mathsf{H}_{\alpha}^{[\Gamma]}\,U_\alpha\,.
  \end{split}
 \]
 By construction, the above operators are self-adjoint and extend $H_\alpha$. They are restrictions of $H_\alpha^*$ and as such, in view of Prop.~\ref{prop:adjoint_on_M}, each element in their domain satisfy the integrability and regularity condition \eqref{eq:DHalpha_cond1}.

 A generic function $f$ in the domain of one of the above extensions is by construction, owing to \eqref{eq:unit1}, of the form
 \[
  f\;=\;|x|^{\frac{\alpha}{2}}\phi
 \]
 for some $\phi$ in the domain of the corresponding unitarily equivalent operator. This and  \eqref{eq:limitphi0}-\eqref{eq:limitphi1} then yield
 \[
  \begin{split}
   \phi_0^\pm(y)\;&=\;\lim_{x\to 0^\pm}f(x,y)\;=:\;f_0^\pm(y) \,\\
   \phi_1^\pm(y)\;&=\;\pm(1+\alpha)^{-1}\lim_{x\to 0^\pm}|x|^{-\alpha}\partial_xf(x,y)\;=:\;f_1^\pm(y)\,.
  \end{split}
 \]
 We thus see that the limits \eqref{eq:DHalpha_cond2_limits-1}-\eqref{eq:DHalpha_cond2_limits-2} do exists, and are finite for a.e.~$y$ because both $\phi_0$ and $\phi_1$ belong to $L^2(\mathbb{S}^1)$.

 In fact, the additional Sobolev regularity of $f_0$ and $f_1$ is the same as for $\phi_0$ and $\phi_1$, and it is immediately imported from Theorem \ref{thm:classificationUF}.

 The very same applies to the expression of the boundary conditions of self-adjointness for each family of extensions: \eqref{eq:DHalpha_cond3_Friedrichs-NOWEIGHTS}-\eqref{eq:DHalpha_cond3_III-NOWEIGHTS} immediately imply \eqref{eq:DHalpha_cond3_Friedrichs}-\eqref{eq:DHalpha_cond3_III}.
\end{proof}

\section*{Acknowledgements}

This work was partially supported by the Istituto Nazionale di Alta Matematica (INdAM), the MIUR-PRIN 2017 project MaQuMA cod.~2017ASFLJR, the Alexander von Humboldt Foundation, and the European Union's Horizon 2020 Research and Innovation Programme under the Marie Sklodowska-Curie grant agreement no.~765267 (QuSCo). We are warmly grateful to U.~Boscain, D.~Dimonte, D.~Noja, and A.~Posilicano for very fruitful and instructive discussions on this subject.


\def\cprime{$'$}

\end{document}